\documentclass[reqno]{amsart}

\usepackage{hyperref, bbm, amssymb, mathtools, url, enumerate}
\usepackage{amsmath}
\usepackage{amsthm}
\usepackage{tikz}
\usepackage{tikz-cd}
\usepackage[scr]{rsfso}
\usepackage{comment}
\usepackage{bm}
\usepackage{amssymb}
\usetikzlibrary{matrix,arrows,decorations.pathmorphing, cd}
\usepackage[capitalise]{cleveref}
\usepackage{aliascnt}

\newcommand{\qednow}{\pushQED{\qed}\qedhere\popQED}

\def\twocell[#1]{\arrow[#1, dash, phantom, "\Rightarrow"{scale=1.125, yshift=-.4pt, description, allow upside down, sloped, inner sep=0pt}]}

\makeatletter
\let\@uldcite=\cite
\def\st@rcite#1#2{\@uldcite[#2]{#1}}
\def\cite#1{\@ifstar{\st@rcite{#1}}{\@uldcite{#1}}}
\makeatother
\let\cites=\cite

\pgfdeclarelayer{bg}
\pgfsetlayers{bg,main}
\usetikzlibrary{calc}

\newcommand{\mynewtheorem}[2]{\newaliascnt{#1}{theorem}\newtheorem{#1}[#1]{#2}\aliascntresetthe{#1}}
\newtheorem{theorem}{Theorem}[subsection]
\mynewtheorem{conjecture}{Conjecture}
\mynewtheorem{corollary}{Corollary}
\mynewtheorem{lemma}{Lemma}
\mynewtheorem{proposition}{Proposition}

\theoremstyle{definition}
\newtheorem*{claim*}{Claim}
\mynewtheorem{construction}{Construction}
\mynewtheorem{convention}{Convention}
\mynewtheorem{definition}{Definition}

\mynewtheorem{discussion}{Discussion}
\mynewtheorem{example}{Example}
\mynewtheorem{hypothesis}{Hypothesis}
\mynewtheorem{notation}{Notation}
\mynewtheorem{observation}{Observation}
\mynewtheorem{question}{Question}
\mynewtheorem{remark}{Remark}
\newtheorem*{remark*}{Remark}
\mynewtheorem{warn}{Warning}
\newtheorem{introthm}{Theorem}

\DeclareMathOperator{\Cc}{\mathcal{C}}
\DeclareMathOperator{\Dd}{\mathcal{D}}
\DeclareMathOperator{\Ee}{\mathcal{E}}

\renewcommand{\phi}{\varphi}
\renewcommand{\epsilon}{\varepsilon}

\renewcommand{\S}{{{\mathscr S}}}
\DeclareMathOperator{\Sp}{Sp}
\DeclareMathOperator{\Spc}{Spc}
\DeclareMathOperator{\Cat}{Cat}

\newcommand{\PrL}{\textup{Pr}^{\textup{L}}}
\newcommand{\PrLT}{\textup{Pr}^{\textup{L}}_{T}}
\newcommand{\PrRT}{\textup{Pr}^{\textup{R}}_{T}}
\newcommand{\PrLTst}{\textup{Pr}^{\textup{L,st}}_{T}}
\newcommand{\PrRTst}{\textup{Pr}^{\textup{R,st}}_{T}}
\newcommand{\PrLTsemi}{\textup{Pr}^{\textup{L,$P$-$\oplus$}}_{T}}
\newcommand{\PrRTsemi}{\textup{Pr}^{\textup{R,$P$-$\oplus$}}_{T}}
\newcommand{\PrLTPst}{\textup{Pr}^{\textup{L,$P$-st}}_{T}}

\DeclareMathOperator{\Hom}{Hom}
\DeclareMathOperator{\Fun}{Fun}
\DeclareMathOperator{\PSh}{PSh}
\DeclareMathOperator{\Ar}{Ar}

\DeclareMathOperator{\CMon}{CMon}
\DeclareMathOperator{\RelCat}{RelCat}

\newcommand{\tcat}{{\mathcal D}}
\newcommand{\cat}[1]{\textbf{\textup{#1}}}
\newcommand{\HOM}{\cat{Hom}}
\newcommand{\PSH}{\cat{PSh}}
\newcommand{\FUN}{\cat{Fun}}
\newcommand{\maps}{\textup{maps}}
\newcommand{\catop}{^{\mathrm{op}}}
\newcommand{\op}{{\textup{op}}}

\renewcommand{\smallint}{{\textstyle\int}}

\DeclareMathOperator{\core}{\iota}
\DeclareMathOperator{\colim}{colim}
\DeclareMathOperator{\const}{const}
\DeclareMathOperator{\id}{id}
\DeclareMathOperator{\pr}{pr}
\DeclareMathOperator{\fgt}{fgt}
\DeclareMathOperator{\diag}{diag}
\DeclareMathOperator{\cocart}{cocart}
\DeclareMathOperator{\BC}{BC}

\DeclareMathOperator{\Orb}{Orb}
\DeclareMathOperator{\Glo}{Glo}

\DeclareMathOperator{\ind}{ind}
\DeclareMathOperator{\coind}{coind}
\DeclareMathOperator{\Nm}{Nm}
\DeclareMathOperator{\Nmadj}{\widetilde{\Nm}}
\DeclareMathOperator{\Nmadjdual}{\overline{\Nm}}

\newcommand{\ulhelper}[2]{\underline{\setbox0=\hbox{$#1#2$}\dp0=1pt \box0\relax}}
\newcommand{\ul}[1]{{\mathpalette\ulhelper{#1}}\hbox{\rule[-2pt]{0pt}{0pt}}}
\newcommand{\ulFun}{\ul{\Fun}}
\newcommand{\bbU}{\mathbf{U}}

\newcommand{\finSets}{\mathbb{F}}
\newcommand{\finTsets}{{\finSets_{T}}}
\newcommand{\finPsets}{\finSets_{T}^{P}}
\newcommand{\ulFinSets}[1]{\ul{\finSets}^{#1}_T}
\newcommand{\ulfinPsets}{\ul{\finSets}_T^{P}}

\newcommand{\ulfinptdPsets}{\ul{\finSets}_{T,*}^{P}}

\newcommand{\ulPCMon}{\ul{\CMon}^{P}}

\newcommand{\CatTPSemi}{\Cat_T^{P\text{-}\oplus}}
\newcommand{\CatTPProd}{\Cat_T^{P\text{-}\times}}

\newcommand{\blank}{{\textup{--}}}
\newcommand{\rOgl}{{\mathfrak O}^\textup{gl}}

\newcommand{\Ogl}{\textbf{\textup O}^\textup{gl}}

\newcommand{\nerve}{\textup{N}}
\newcommand{\OGglGamma}{{\cat{O}^\textup{$G$-gl}_\Gamma}}
\newcommand{\sGlo}{\cat{Glo}}
\newcommand{\tcatUn}[1]{\mathop{\hfuzz=10pt\hbox to 0pt{$\textstyle\bm\int$}\kern.3pt\raise.2pt\hbox to 0pt{$\textstyle\bm\int$}\lower.2pt\hbox to 0pt{$\textstyle\bm\int$}\kern.3pt\hbox to 0pt{$\textstyle\bm\int$}\kern-.1pt\raise.1pt\hbox{\color{white}$\textstyle\int$}}}
\newcommand{\GammaS}{{\Gamma\kern-1.5pt\mathscr S}}
\newcommand{\mySp}{{\mathscr S\kern-2ptp}}
\newcommand{\mathscrGr}{{\mathscr G\kern-1.25ptr}}

\DeclareMathOperator{\pt}{pt}
\DeclareMathOperator{\lex}{lex}
\DeclareMathOperator{\ex}{ex}
\DeclareMathOperator{\st}{st}
\DeclareMathOperator{\fib}{fib}
\DeclareMathOperator{\cofib}{cofib}
\newcommand{\ev}{\mathrm{ev}}
\let\smashp=\wedge

\newcommand{\projGgl}[1]{\textup{$#1$-gl~proj}}
\newcommand{\injGgl}[1]{\textup{$#1$-gl~inj}}
\newcommand{\FinOrbSets}[1]{{\mathfrak{F}^{\mathrm{gl},+}_{#1}}}
\newcommand{\combFinOrbSets}[1]{{\tcatUn{}\mathfrak{F}^{\mathrm{gl},+}_{#1}}}

\newcommand{\Plex}{{\textup{$P$-lex}}}
\newcommand{\Pex}{{\textup{$P$-ex}}}
\newcommand{\Pst}{{\textup{$P$-st}}}
\newcommand{\Pprod}{{\textup{$P$-$\times$}}}
\newcommand{\Pbiprod}{{\textup{$P$-$\oplus$}}}
\newcommand{\Poplus}{{\textup{$P$-$\oplus$}}}
\newcommand{\Pcoprod}{{\textup{$P$-$\sqcup$}}}

\newcommand{\iso}{\xrightarrow{\;\smash{\raisebox{-0.5ex}{\ensuremath{\scriptstyle\sim}}}\;}}

\newcommand{\pullbacksign}{\hspace{-0.325ex}\tikz[baseline=(pb.base)]{\draw[line width=rule_thickness, line cap=round] (0,0) ++ (-2.45ex,0.45ex) -- ++ (1ex,0ex) -- ++ (0ex,1ex);\node (pb) at (0,0) {\phantom{x}};}}
\newcommand{\pushoutsign}{\hspace{0.2ex}\tikz[baseline=(po.base)]{\draw (0,0) ++ (2.45ex,-1.45ex) -- ++(0ex,1ex) -- ++ (1ex,0ex);\node (po) at (0,0) {\phantom{x}};}}
\tikzcdset{
	pullback/.style = {dash, phantom, "\vphantom{x}\smash{\pullbacksign}", start anchor=center, end anchor=center},
	pushout/.style = {dash, phantom, "\vphantom{x}\smash{\pushoutsign}", start anchor=center, end anchor=center}
}

\newcommand\noloc{%
	\nobreak
	\mspace{6mu plus 1mu}
	{:}
	\nonscript\mkern-\thinmuskip
	\mathpunct{}
	\mspace{2mu}
}

\makeatletter
\newif\ifhe@d
\def\@tocline#1#2#3#4#5#6#7{\begingroup\relax%
	\he@dfalse\ifcase#1\relax\or\he@dtrue\fi%
	\ifnum#1<3\ifhe@d\else\hspace{2em}\hspace{-2pt}\fi#6\hfill#7\par\fi\endgroup}
\makeatother

\begin{document}
\title[Parametrized stability and the univ.~property of global spectra]{Parametrized stability\\ and the universal property of global spectra}
\author{Bastiaan Cnossen}
\address{B.C.: Mathematisches Institut, Rheinische Friedrich-Wilhelms-Universit\"at Bonn, Endenicher Allee 60, 53115 Bonn, Germany}
\curraddr{\scshape Fakultät für Mathematik, Universität Regensburg, 93040 Regensburg, Germany}
\author{Tobias Lenz}
\address{T.L.: Mathematical Institute, University of Utrecht, Budapestlaan 6, 3584 CD Utrecht, The Netherlands}
\curraddr{\scshape Mathematisches Institut, Rheinische Friedrich-Wilhelms-Universit\"at Bonn, Endenicher Allee 60, 53115 Bonn, Germany}
\author{Sil Linskens}
\address{S.L.: Mathematisches Institut, Rheinische Friedrich-Wilhelms-Universit\"at Bonn, Endenicher Allee 60, 53115 Bonn, Germany}
\curraddr{\scshape Fakultät für Mathematik, Universität Regensburg, 93040 Regensburg, Germany}
\maketitle

\begin{abstract}
	We develop a framework of parametrized semiadditivity and stability with respect to so-called atomic orbital subcategories of an indexing $\infty$-category $T$, extending work of Nardin. Specializing this framework, we introduce \emph{global $\infty$-categories} and the notions of {equivariant semiadditivity} and {stability},
	yielding a higher categorical version of the notion of a Mackey 2-functor studied by Balmer-Dell'Ambrogio. As our main result, we identify the free presentable equivariantly stable global $\infty$-category with a natural global $\infty$-category of global spectra for finite groups, in the sense of Schwede and Hausmann.
\end{abstract}

\tableofcontents

\section{Introduction}

Equivariant homotopy theory combines classical homotopy theory with ideas from representation theory to study geometric objects with symmetries. Many constructions from homotopy theory carry over to the equivariant setting, leading for example to equivariant analogues of important cohomology theories like topological $K$-theory and (stable) bordism. The resulting tools and methods have been successfully applied to various other branches of mathematics, for example in the proof of the \textit{Atiyah-Segal Completion Theorem} \cite{AtiyahSegal1969equivariant}, Carlson's proofs of the \textit{Segal} \cite{carlsson1984SegalConjecture} and \textit{Sullivan Conjecture} \cite{carlsson1991SullivanConjecture}, or in the resolution of the Kervaire invariant one problem by Hill, Hopkins, and Ravenel \cite{HHR2016Kervaire}.

While one can study equivariant homotopy theory for a single group $G$ at a time, there are many equivariant phenomena which occur uniformly and compatibly in large families of groups, such as the families of all finite groups or all compact Lie groups. The study of such phenomena is known as \textit{global homotopy theory}  \cites{gepnerhenriques2007orbispaces, bohmann2014global, schwede2018global, hausmann-global, g-global, LNP}. This framework has led to improved understanding of a variety of equivariant phenomena, where previously a direct description for each individual group was either much more opaque or not available, for example for equivariant stable bordism and equivariant formal group laws \cite{hausmann2022global}. The study of global homotopy theory moreover admits connections to the geometry of stacks \cites{gepnerhenriques2007orbispaces, juran2020orbifolds, pardon2020orbifold, sati2020proper}.

Just like non-equivariant and $G$-equivariant homotopy theory, global homotopy theory comes in various different flavours: \emph{unstable global homotopy theory} studies \emph{global spaces} \cite{gepnerhenriques2007orbispaces} while \emph{stable global homotopy theory} is concerned with so-called \emph{global spectra} \cite{schwede2018global}; in-between, one can also consider a variety of algebraic structures on global spaces \cite{barrero2021}, with the most prominent example being \emph{ultra-commutative monoids} or the equivalent notion of \emph{special global $\Gamma$-spaces} \cite{g-global}. The goal of this article is to understand the relationship between these different variants.

\subsubsection*{Stability and equivariant semiadditivity}
Classically the passage from the homotopy theory of spaces to the homotopy theory of spectra is known as \textit{stabilization}. More generally, a homotopy theory $\Cc$ (e.g.\ given in the form of a model category or an $\infty$-category) is said to be \textit{stable} if it is pointed and the suspension-loop adjunction in $\Cc$ is an equivalence. Stability of a homotopy theory leads to a lot of algebraic structure: for example, its homotopy category $\mathrm{Ho}(\Cc)$ is additive, and it canonically admits the structure of a triangulated category. If $\Cc$ is not yet (known to be) stable, there is a universal way to stabilize it by forming a homotopy theory $\Sp(\Cc)$ of suitable \textit{spectrum objects} in $\Cc$.

Although one may apply this stabilization procedure to the homotopy theory of global spaces, the resulting theory is in many ways not sufficient, and in particular does not yield the homotopy theory of global spectra. This issue in fact already arises in the case of equivariant homotopy theory for a fixed group $G$: applying the general stabilization procedure to the homotopy theory of $G$-spaces for some finite group $G$ only results in the homotopy theory of \emph{naive $G$-spectra}, which for example does not support a good theory of duality. 	Instead, one defines the homotopy theory of \emph{genuine $G$-spectra} by stabilizing more generally with respect to the \textit{representation spheres} $S^V$ for each finite-dimensional $G$-representation $V$.
This genuine stabilization leads to a much richer algebraic structure on the associated homotopy category than naive stabilization: for example, the homotopy category of genuine $G$-spectra admits a canonical enrichment in \textit{Mackey functors}, refining the enrichment in abelian groups.

Non-equivariantly, the algebraic structure on hom sets in a stable homotopy theory comes from semiadditivity: finite coproducts agree with finite products. In a similar way, the Mackey enrichment of the homotopy theory of genuine $G$-spectra comes from a form of \textit{equivariant semiadditivity}.
To explain what this means, consider a subgroup $H$ of the finite group $G$; the restriction functor from genuine $G$-spectra to genuine $H$-spectra then admits both a left adjoint $\ind^G_H$ and a right adjoint $\coind^G_H$, called \textit{induction} and \textit{coinduction}, respectively. From the perspective of this article, the main feature of genuine equivariant spectra is that there is a natural equivalence $\ind^G_H \simeq \coind^G_H$ between these two functors, called the \textit{Wirthmüller isomorphism} \cite{wirthmuller1974equivariant}. If we think of $\ind^G_H$ as a `$G$-coproduct over $G/H$' and $\coind^G_H$ as a `$G$-product over $G/H$,' this may be seen as an equivariant analogue of the usual notion of semiadditivity. These Wirthmüller isomorphisms are then precisely what gives rise to the transfer maps in the aforementioned Mackey enrichment.

\subsubsection*{Parametrized higher category theory}
In light of the above, it is natural to ask whether one can modify the stabilization procedure for $G$-spaces in a way that additionally enforces equivariant semiadditivity, and, if so, whether this will result in the homotopy theory of genuine $G$-spectra. One subtlety with this question is that the Wirthmüller isomorphisms described above do not only depend on the homotopy theory of genuine $G$-spectra but also on the homotopy theories of genuine $H$-spectra for every subgroup $H$ of $G$, together with all the restriction functors relating them. Based on suggestions by Mike Hill in 2012, Clark Barwick, Emanuele Dotto, Saul Glasman, Denis Nardin, and Jay Shah \cite{exposeI} began developing the theory of \textit{$G$-$\infty$-categories} for a finite group $G$, in which these ideas could be made precise. More generally, given an $\infty$-category $T$, they introduced the notion of a \textit{$T$-$\infty$-category}, thought of as an family of $\infty$-categories parametrized by $T$, and showed that many concepts and foundational results from the theory of $\infty$-categories have analogues in this parametrized setting. Using this framework, Nardin \cite{nardin2016exposeIV} worked out a notion of \textit{parametrized semiadditivity} which neatly recovers the equivariant Wirthmüller isomorphisms described earlier. He further sketched a proof that the $G$-$\infty$-category of genuine $G$-spectra is obtained from the $G$-$\infty$-category of $G$-spaces by enforcing both stability and parametrized semiadditivity.

\subsection{Global \texorpdfstring{\for{toc}{$\infty$}\except{toc}{$\bm\infty$}}{∞}-categories}
The goal of this article is to develop an analogue of the above story, and in particular of Nardin's result, in the global setting. A distinguishing feature that was not present in the equivariant setting is the appearance of \textit{inflation functors}: restriction functors along \textit{surjective} group homomorphisms $\alpha \colon H \twoheadrightarrow G$. This extra structure leads to the notion of a \textit{global $\infty$-category}. Roughly speaking, such an object consists of
\begin{enumerate}[(i)]
	\item an $\infty$-category $\Cc(G)$ for every finite group $G$;
	\item a restriction functor $\alpha^*\colon \Cc(G) \to \Cc(H)$ for every homomorphism $\alpha\colon H \to G$;
	\item higher structure which in particular witnesses that
	conjugations act as the identity.
\end{enumerate}

Examples of global $\infty$-categories abound in representation theory, and more generally equivariant mathematics; here we only mention categories of representations, genuine equivariant spectra, and equivariant Kasparov categories, referring the reader to \cite{balmerAmbrogio_Mackey} for a detailed discussion of these examples. In this paper, on the other hand, we will be mainly interested in examples coming from \textit{$G$-global homotopy theory} in the sense of \cite{g-global}; namely, we consider:
\begin{itemize}
	\item the global $\infty$-category $\ul{\S}^\textup{gl}$ of \textit{global spaces}, given at a group $G$ by $G$-global spaces (see \Cref{subsec:GlobalCatOfGlobalSpaces} for a precise definition);
	\item the global $\infty$-category $\ul\GammaS^\textup{gl, spc}$ of \textit{special global $\Gamma$-spaces}, given at a group $G$ by special $G$-global $\Gamma$-spaces (see \Cref{subsec:GGlobalGammaSpaces} for a precise definition);
	\item the global $\infty$-category $\ul\mySp^\textup{gl}$ of \textit{global spectra}, given at a group $G$ by $G$-global spectra (see \Cref{subsec:GGlobalSpectra} for a precise definition).
\end{itemize}

As the main results of this paper we establish universal properties for these three global $\infty$-categories:

\subsubsection*{Presentability}
A global $\infty$-category $\Cc$ is said to be \textit{presentable} if $\Cc(G)$ is presentable for all $G$ and the restriction functors $\alpha^*\colon \Cc(G) \to \Cc(H)$ admit left and right adjoints $a_!$ and $\alpha_*$ for all $\alpha\colon H \to G$ satisfying a base change condition, which may be thought of as a categorified version of the Mackey double coset formula. We refer to \Cref{subsec:presentableTcategories} for a precise definition. The universal example is provided by $G$-global homotopy theory:

\begin{introthm}[Universal property of global spaces, \ref{thm:global-spaces}] \label{introthm:universal-prop-spaces}
	The global $\infty$-category $\ul\S^\textup{gl}$ is presentable. For every presentable global $\infty$-category $\mathcal D$, evaluation at the point $* \in \ul\S^{\textup{gl}}(1)$ induces an equivalence
	\begin{equation*}
		\ul{\Fun}^{\textup{L}}_{\Glo}(\ul\S^\textup{gl},\mathcal D)\to\mathcal D
	\end{equation*}
	of global $\infty$-categories. Put differently, $\ul\S^\textup{gl}$ is the free presentable global $\infty$-category on one generator $*$.
\end{introthm}

We will in fact provide a stronger version of \Cref{introthm:universal-prop-spaces} based on a notion of \textit{global cocompleteness}, see \Cref{subsec:ParametrizedAdjunctions}. Our proof of this result can be regarded as a highly coherent version of Schwede's global Elmendorf theorem \cite{schwede_orbispaces_2020}.

\subsubsection*{Equivariant semiadditivity}
Following ideas of \cite{nardin2016exposeIV}, we introduce a notion of \textit{equivariant semiadditivity} in our context; namely, a global $\infty$-category $\Cc$ is equivariantly semiadditive if the following conditions are satisfied:
\begin{itemize}
	\item
	\textit{Fiberwise semiadditivity:}
	The $\infty$-category $\Cc(G)$ is semiadditive for every $G$ and the functor $\alpha^*\colon \Cc(G) \to \Cc(H)$ preserves finite biproducts for every $\alpha \colon H \to G$;
	\item
	\textit{Ambidexterity:}
	For every \emph{injective} homomorphism $i\colon H \to G$, the restriction functor $i^*\colon \Cc(G) \to \Cc(H)$ admits a both left adjoint $i_!$ and a right adjoint $i_*$ satisfying a base change condition as before, and a certain \textit{norm map} $\Nm_{i}\colon i_! \to i_*$ is a natural equivalence between these two adjoints.
\end{itemize}
A 2-categorical analogue of this definition was studied under the name \textit{Mackey 2-functor} by Balmer-Dell'Ambrogio \cite{balmerAmbrogio_Mackey}. The examples of representations, equivariant spectra, and Kasparov categories referred to above are all equivariantly semiadditive -- for example, in the case of equivariant spectra, ambidexterity precisely comes from the Wirthmüller isomorphism. Once again, $G$-global homotopy theory provides the universal example in this setting:

\begin{introthm}[Universal property of global $\Gamma$-spaces, \ref{thm:global-gamma}]
	\label{introthm:universal-prop-gamma}
	The global $\infty$-category $\ul\GammaS^\textup{gl, spc}$ is presentable and equivariantly semiadditive. For every presentable equivariantly semiadditive global $\infty$-category $\Dd$, evaluation at the free special global $\Gamma$-space $\mathbb P(*)$ induces an equivalence
	\[
	\ul\Fun^\textup{L}_{\Glo}(\ul\GammaS^\textup{gl, spc},\mathcal D)\iso \mathcal D
	\]
	of global $\infty$-categories. Put differently, $\ul{\GammaS}^\textup{gl, spc}$ is the free presentable equivariantly semiadditive global $\infty$-category on one generator $\mathbb P(*)$.
\end{introthm}

\subsubsection*{Equivariant stability}
A global $\infty$-category $\Cc$ is called \textit{equivariantly stable} if it is equivariantly semiadditive and \textit{fiberwise stable}, meaning that the $\infty$-category $\Cc(G)$ is stable for every finite group $G$ and the restriction functors $\alpha^*\colon \Cc(G) \to \Cc(H)$ are exact for all $\alpha \colon H \to G$.

\begin{introthm}[Universal property of global spectra, \ref{thm:global-spectra}] \label{introthm:universal-prop-spectra}
	The global $\infty$-category $\ul\mySp^\textup{gl}$ is presentable and equivariantly stable. For every presentable equivariantly stable global $\infty$-category $\Dd$, evaluation at the global sphere spectrum $\mathbb S$ defines an equivalence
	\begin{equation*}
		\ul\Fun^\textup{L}_{\Glo}(\ul\mySp^\textup{gl},\mathcal D)\xrightarrow{\;\simeq\;}\mathcal D
	\end{equation*}
	of global $\infty$-categories. Put differently, $\ul\mySp^\textup{gl}$ is the free presentable equivariantly stable global $\infty$-category on one generator $\mathbb S$.
\end{introthm}

Combining this with \Cref{introthm:universal-prop-spaces}, this makes precise that $\ul\mySp^\textup{gl}$ is obtained from $\ul\S^\textup{gl}$ by universally stabilizing and enforcing Wirthmüller isomorphisms, answering the question from the beginning. In particular, global $\infty$-categories provide a natural and convenient home for studying global homotopy theory. Conversely, once one is interested in global $\infty$-categories, global (and more generally $G$-global) homotopy theory appears naturally in the form of the universal examples. For example one can show using the above that for every equivariantly stable global $\infty$-category $\Cc$, the $\infty$-category $\Cc(G)$ is canonically enriched over $G$-global spectra, with strong compatibilities as the group $G$ varies.

\subsection{Parametrized higher category theory}

In setting up the formalism of equivariant semiadditivity and stability, we work in the more general context of $T$-$\infty$-categories for an arbitrary $\infty$-category $T$, in the sense of \cite{exposeI}. Global $\infty$-categories arise as the special case where $T$ is the $(2,1)$-category $\Glo$ of finite connected groupoids, see \Cref{ex:globalCategory}. We introduce the notion of an \textit{atomic orbital} subcategory $P \subseteq T$, generalizing a notion due to \cite{nardin2016exposeIV}; in this setting, we can then more generally define \textit{$P$-semiadditivity}
and \textit{$P$-stability}, which for the subcategory $\Orb \subseteq \Glo$ of \textit{faithful morphisms} specializes to the notions of equivariant semiadditivity/stability discussed before.

Given a $T$-$\infty$-category $\Cc$ with sufficiently many parametrized limits, we provide a universal way to turn it into a $P$-semiadditive $T$-$\infty$-category by passing to the $T$-$\infty$-category $\ulPCMon(\Cc)$ of \textit{$P$-commutative monoids}, a parametrized version of commutative monoid objects in higher algebra. In a similar way, we construct a universal $P$-stabilization $\ul{\Sp}^{P}(\Cc)$ of $\Cc$. Combining this with \Cref{introthm:universal-prop-spaces}, the key step in the proof of
\Cref{introthm:universal-prop-gamma} and \Cref{introthm:universal-prop-spectra} is then to produce equivalences of global $\infty$-categories
\[
\ul\GammaS^\textup{gl,spc} \simeq \ul{\CMon}^{\Orb}(\ul\S^\textup{gl}), \qquad \qquad \ul\mySp^\textup{gl} \simeq \ul{\Sp}^{\Orb}(\ul\S^\textup{gl}).
\]

\subsection{Acknowledgements} The authors would like to thank Branko Juran for pointing out an omission in a draft of this article, which led to the inclusion of Appendix~\ref{appendix:slices}.
B.C.\ would like to thank Louis Martini and Sebastian Wolf for many helpful discussions about parametrized higher category theory.
T.L.~would like to thank Markus Hausmann who first suggested to him to use $G$-global homotopy theory to obtain the parametrized incarnation of global homotopy theory. All three authors would like to thank the anonymous referee for a very careful reading and helpful feedback.

This article is partially based on work supported by the Swedish Research Council under grant no. 2016-06596 while T.L.~was in residence at Institut Mittag-Leffler in Djursholm, Sweden in early 2022.

While the first version of this article was written, B.C.\ was supported by the Max Planck Institute for Mathematics in Bonn and S.L.\ was supported by the DFG Schwerpunktprogramm 1786 `Homotopy Theory and Algebraic Geometry' (project ID SCHW 860/1-1). Moreover, both B.C.\ and S.L.\ were associate members of the Hausdorff Center for Mathematics at the University of Bonn during that time, as was T.L. when this article was revised. B.C.\ and S.L.\ are associate members of the SFB 1085 Higher Invariants.

\section{Parametrized higher category theory}
\label{sec:parametrizedHigherCategoryTheory} In this section, we will recall some of the basic notions of parametrized higher category theory. A first development of such theory was given by Clark Barwick,  Emanuele Dotto, Saul Glasman, Denis Nardin and Jay Shah, cf.\  \cites{exposeI, shah2021parametrized, nardin2016exposeIV}. From the perspective of categories internal to $\infty$-topoi, an alternative development was given by Louis Martini and Sebastian Wolf \cites{martini2021yoneda, martiniwolf2021limits, martiniwolf2022presentable}.

\subsection{\texorpdfstring{\except{toc}{$\bm T$-$\bm\infty$}\for{toc}{$T$-$\infty$}}{T-∞}-categories}
We introduce the notion of a $T$-$\infty$-category for a small $\infty$-category $T$ and discuss various constructions and examples.

\begin{definition}
	Let $T$ be a small $\infty$-category. A \textit{$T$-$\infty$-category} is a functor $\Cc\colon T\catop \to \Cat_{\infty}$. If $\Cc$ and $\Dd$ are $T$-$\infty$-categories, then a \textit{$T$-functor} $F\colon \Cc \to \Dd$ is a natural transformation from $\Cc$ to $\Dd$. The $\infty$-category $\Cat_T$ of $T$-$\infty$-categories is defined as the functor category $\Cat_T := \Fun(T\catop, \Cat_{\infty})$.

	If $\Cc$ is a $T$-$\infty$-category and $f\colon A \to B$ is a morphism in $T$, we will write $f^*$ for the functor $\Cc(f)\colon \Cc(B) \to \Cc(A)$ and refer to this as \textit{restriction along $f$}.
\end{definition}

\begin{example}
	\label{ex:GCategory}
	Let $G$ be a finite group and let $\Orb_G$ denote the \textit{orbit category of $G$}, defined as the full subcategory of the category of $G$-sets spanned by the orbits $G/H$ for subgroups $H \leqslant G$. When $T = \Orb_G$, $T$-$\infty$-categories are referred to as \textit{$G$-$\infty$-categories}, c.f.\ \cite{exposeI}.
\end{example}

We will be mainly interested in the following example.

\begin{example}
	\label{ex:globalCategory}
	Define $\sGlo$ as the strict $(2,1)$-category of finite groups, group homomorphisms, and conjugations. In particular, $\sGlo$ comes with a fully faithful functor $B\colon \sGlo\hookrightarrow \cat{Grpd}$ into the $(2,1)$-category of groupoids given by sending a finite group $G$ to the corresponding $1$-object groupoid $BG$, a homomorphism $f\colon G\to H$ to the functor $Bf\colon BG\to BH$ given on homomorphisms by $f$, and a conjugation $h\colon f\Rightarrow f'$ (i.e.~an $h\in H$ such that $f'(g)=hf(g)h^{-1}$ for all $g\in G$) to the natural transformation $Bf\Rightarrow Bf'$ whose value at the unique object of $BG$ is the edge $h$.

	We define the $\infty$-category $\Glo$ as the Duskin nerve of the $(2,1)$-category $\sGlo$. We will use the term \emph{global $\infty$-category} for a $\Glo$-$\infty$-category, \textit{global functor} for a $\Glo$-functor, etc.
\end{example}

\begin{remark}
	\label{rmk:TCategoriesAsCocartesianFibrations}
	The straightening-unstraightening equivalence (see \cite{HTT}*{Theorem~3.2.0.1}) provides an equivalence of $\infty$-categories $\Cat_T \simeq (\Cat_{\infty})_{/T\catop}^{\cocart}$, where $(\Cat_{\infty})_{/T\catop}^{\cocart}$ denotes the (non-full) subcategory of the slice $(\Cat_{\infty})_{/T\catop}$ spanned by the cocartesian fibrations over $T\catop$ and the functors over $T\catop$ that preserve cocartesian edges. The cocartesian fibration over $T\catop$ corresponding to a $T$-$\infty$-category $\Cc\colon T\catop \to \Cat_{\infty}$ is denoted by $\smallint \Cc \to T\catop$ and is referred to as the cocartesian unstraightening of $\Cc$. A $T$-functor $F\colon \Cc \to \Dd$ corresponds to a functor $\smallint F\colon \smallint \Cc \to \smallint \Dd$ over $T\catop$ which preserves cocartesian edges. In fact, in the articles \cite{exposeI}, \cite{shah2021parametrized} and \cite{nardin2016exposeIV}, a $T$-$\infty$-category is \textit{defined} as a cocartesian fibration over $T\catop$.
\end{remark}

\begin{definition}
	\label{def:UnderlyingInfinityCategory}
	Let $\Cc\colon T\catop \to \Cat_{\infty}$ be a $T$-$\infty$-category. We define the \textit{underlying $\infty$-category} $\Gamma(\Cc)$ of $\Cc$ as the limit of $\Cc$:
	\begin{align*}
		\Gamma(\Cc) := \lim_{B \in T\catop} \Cc(B).
	\end{align*}
	This defines a functor $\Gamma\colon \Cat_T \to \Cat_{\infty}$. Note that when $T$ has a final object, $\Gamma(\Cc)$ is obtained by evaluating $\Cc$ at the final object.
\end{definition}

\begin{remark}
	By \cite{HTT}*{Corollary 3.3.3.2}, the $\infty$-category $\Gamma(\Cc)$ is equivalent to the $\infty$-category of cocartesian sections of $\smallint \Cc \to T\catop$.
\end{remark}

We discuss some important examples of $T$-$\infty$-categories.
\begin{example}
	\label{ex:ConstantTCategory}
	Every $\infty$-category $\Ee$ gives rise to a $T$-$\infty$-category $\const_{\Ee}\colon T\catop \to \Cat_{\infty}$ given by $\const_{\Ee}(t) = \Ee$ for all $t \in T$. This provides a functor $\const\colon \Cat_{\infty} \to \Cat_T$. We will refer to $T$-$\infty$-categories in the essential image of this functor as \textit{constant $T$-$\infty$-categories}.
\end{example}

\begin{remark}
	\label{rmk:AdjunctionConstantUnderlying}
	Note that the functor $\const\colon \Cat_{\infty} \to \Cat_{T}$ is left adjoint to the underlying $\infty$-category functor $\Gamma\colon \Cat_T \to \Cat_{\infty}$: for every $T$-$\infty$-category $\Cc$ and every $\infty$-category $\Ee$ there is an equivalence
	\begin{align*}
		\Hom_{\Cat_T}(\const_{\Ee},\Cc) \simeq \Hom_{\Cat_{\infty}}(\Ee,\Gamma(\Cc)).
	\end{align*}
\end{remark}

\begin{example}
	\label{ex:TGroupoid}
	Every presheaf $B\colon T\catop \to \Spc$ on $T$ gives rise to a $T$-$\infty$-category $\ul{B}\colon T\catop \to \Cat_{\infty}$ by composing it with the inclusion $\Spc \subseteq \Cat_{\infty}$ of $\infty$-groupoids into all $\infty$-categories, and we obtain a fully faithful inclusion
	\begin{align*}
		\PSh(T) = \Fun(T\catop,\Spc) \hookrightarrow \Fun(T\catop,\Cat_{\infty}) = \Cat_T.
	\end{align*}
	The $T$-$\infty$-categories in the essential image of this functor will be referred to as \textit{$T$-$\infty$-groupoids}.
\end{example}

In particular, every object $B \in T$ gives rise to a $T$-$\infty$-category $\ul{B}$ via the Yoneda embedding $T \hookrightarrow \PSh(T)$.

\begin{remark}
	\label{rmk:CoreAdjoint}
	The inclusion $\PSh(T) \subseteq \Cat_T$ admits a right adjoint $\core\colon \Cat_T \to \PSh(T)$. It is given on $\Cc$ by $\core\circ \Cc$, where $\core\colon \Cat_{\infty} \to \Spc$ is the functor which assigns to an $\infty$-category its core, the largest $\infty$-groupoid contained in it.
\end{remark}

\begin{example}
	\label{ex:TObjectsInCategory}
	Let $\Ee$ be an $\infty$-category. A \textit{$T$-object in $\Ee$} is a functor $T\catop \to \Ee$. We obtain a $T$-$\infty$-category $\ul{\Ee}_T$ of \textit{$T$-objects in $\Ee$} by assigning to an object $B \in T$ the $\infty$-category $\Fun((T_{/B})\catop,\Ee)$ of $T_{/B}$-objects in $\Ee$. More precisely, the $T$-$\infty$-category $\ul{\Ee}_T$ is defined as the following composite
	\begin{align*}
		T\catop \xrightarrow{B \mapsto (T_{/B})\catop} (\Cat_{\infty})\catop \xrightarrow{\Fun(-,\Ee)} \Cat_{\infty},
	\end{align*}
	where the functoriality of the first functor is via post-composition in $T$, i.e.\ the straightening of the cocartesian fibration $\ev_1\colon T^{[1]} \to T$. It is clear that sending $\Ee$ to $\ul{\Ee}_T$ gives rise to a functor $\Cat_{\infty} \to \Cat_T$.

	As a special case, we obtain the following $T$-$\infty$-categories:
	\begin{enumerate}[(1)]
		\item \label{it:TSpaces}
		taking $\Ee = \Spc$ gives a $T$-$\infty$-category $\ul{\Spc}_T$ of \textit{$T$-spaces} or \textit{$T$-$\infty$-groupoids}.
		\item \label{it:PointedTSpaces}
		taking $\Ee = \Spc_*$ gives a $T$-$\infty$-category $\ul{\Spc}_{T,*} \coloneqq \ul{\Spc_*}_T$ of \textit{pointed $T$-spaces}.
		\item \label{it:NaiveTSpectra}
		taking $\Ee = \Sp$ gives a $T$-$\infty$-category $\ul{\Sp}_T$ of \textit{naive $T$-spectra}.\footnote{The term `naive spectra' is used in equivariant homotopy theory to contrast it with `genuine spectra'.}
		\item \label{it:TCategories}
		taking $\Ee = \mathrm{cat}_{\infty}$, the $\infty$-category of small $\infty$-categories, gives a $T$-$\infty$-category $\ul{\mathrm{cat}}_T \coloneqq \ul{\mathrm{cat}_{\infty}}_T$ of \textit{small $T$-$\infty$-categories}.
	\end{enumerate}
\end{example}

\begin{remark}
	\label{rmk:AdjunctionGrothendieckConstructionTObjects}
	For every $T$-$\infty$-category $\Cc$ and every $\infty$-category $\Ee$, there is an equivalence
	\begin{align*}
		\Hom_{\Cat_T}(\Cc,\ul{\Ee}_T) \simeq \Hom_{\Cat_{\infty}}(\smallint \Cc, \Ee)
	\end{align*}
	which is natural in $\Cc$ and $\Ee$. In other words, the construction of \cref{ex:TObjectsInCategory} provides a right adjoint to the cocartesian unstraightening $\smallint\colon \Cat_T \to \Cat_{\infty}$ which assigns to a $T$-$\infty$-category $\Cc\colon T\catop \to \Cat_{\infty}$ the total category $\smallint \Cc$ of its unstraightening $\smallint \Cc \to T\catop$. We will prove this in \cref{lem:AdjunctionGrothendieckConstructionTObjects} below.
\end{remark}

\begin{remark}
	\label{rmk:limitExtension}
	One may alternatively describe $T$-$\infty$-categories as $\Cat_{\infty}$-valued sheaves on the presheaf $\infty$-topos $\PSh(T) = \Fun(T\catop,\Spc)$, i.e., as limit-preserving functors $\PSh(T)\catop \to \Cat_{\infty}$. Indeed, the functor
	\begin{align*}
		\Fun(\PSh(T)\catop,\Cat_{\infty}) \to \Fun(T\catop,\Cat_{\infty})
	\end{align*}
	given by precomposition with the Yoneda embedding $T \hookrightarrow \PSh(T)$ becomes an equivalence when restricting the domain to the full subcategory of limit-preserving functors, see \cite{HTT}*{Theorem~5.1.5.6}.
\end{remark}

\begin{remark}
	\label{rmk:ParametrizedCategoriesAsInternalCategories}
	For an $\infty$-topos $\mathcal{B}$, the $\infty$-category $\Fun^{\mathrm{R}}(\mathcal{B}\catop, \Cat_{\infty})$ of sheaves of $\infty$-categories on $\mathcal{B}$ is equivalent to the full subcategory of $\Fun(\Delta\catop,\mathcal{B})$ spanned by the \textit{internal $\infty$-categories} (or \textit{complete Segal objects}). We refer to \cite{martini2021yoneda}*{Definition~3.2.4} for a precise definition of an internal $\infty$-category, and to \cite{martini2021yoneda}*{Section~3.5} for a proof of this equivalence. By \cref{rmk:limitExtension}, the study of $T$-$\infty$-categories is thus equivalent to the study of $\infty$-categories internal to the presheaf topos $\PSh(T)$. Although we will never use this perspective in this article, we will not hesitate to cite results from \cites{martini2021yoneda, martiniwolf2021limits, martiniwolf2022presentable} which are formulated in the language of internal $\infty$-categories.
\end{remark}

\begin{convention}
	Henceforth, we will abuse notation and denote the extension of a $T$-$\infty$-category $\Cc$ to a limit preserving functor $\PSh(T)\catop \to \Cat_{\infty}$ again by $\Cc$. At various points in this article, we will write expressions such as $A \times A$ or $A \times_B A$ for objects $A,B \in T$, meaning implicitly that this pullback is taken in the presheaf $\infty$-category $\PSh(T)$. In particular, when we write $\Cc(A \times B)$ or $\Cc(A \times_B A)$, we are referring to the values of the limit-extension $\Cc\colon \PSh(T)\catop \to \Cat_{\infty}$ at the relevant objects. This abuse of notation is justified by the fact that the Yoneda embedding $T \hookrightarrow \PSh(T)$ preserves all limits that exist in $T$, cf.\ \cite{HTT}*{Proposition 5.1.3.2}. In a similar way, all colimits of objects of $T$ are understood to be taken in the presheaf $\infty$-category $\PSh(T)$: expressions such as $\bigsqcup_{i=1}^n A_i$ will always mean formal disjoint union. 
\end{convention}

\begin{remark}\label{rmk:sheaf_associated_to_Spc_T}
	It will be useful to observe that the limit-extension of the $T$-$\infty$-category $\ul{\Spc}_T$ of $T$-spaces is equivalent to the slice functor
	\begin{align*}
		\PSh(T)_{/-}\colon \PSh(T)\catop &\rightarrow \Cat_\infty, \\
		A &\mapsto \PSh(T)_{/A}, \\
		[f\colon A\rightarrow B] &\mapsto f^*\colon \PSh(T)_{/B}\rightarrow \PSh(T)_{/A},
	\end{align*}
	which is defined as the functor which classifies the cartesian fibration
	\[
	t\colon \Ar(\PSh(T)) \to \PSh(T)\colon (A \to B) \mapsto B.
	\]
	Indeed, since this slice functor preserves limits by \cite{HTT}*{Theorem~6.1.3.9, Proposition~6.1.3.10}, it suffices to show that its restriction to $T\catop$ is equivalent to $\ul{\Spc}_T$. 	Consider the Yoneda embedding $T\hookrightarrow \PSh(T)$. By considering the functoriality in over-categories on both sides we obtain a natural transformation
	\[T_{/-}\rightarrow \PSh(T)_{/-}\]
    of functors in $T$. The universal property of presheaves implies that this extends to a natural equivalence
	\[\PSh(T_{/-})\iso \PSh(T)_{/-}.\]
	By the naturality of the Yoneda embedding (see \cite{HHLNa}*{Theorem 8.1} or \cite{Ramzi2022Yoneda}*{Theorem 2.4}) we get that upon passing to right adjoints the diagram $\PSh(T_{/-})$ agrees with $\ul{\Spc}_T$, completing the proof.
\end{remark}

\begin{example}
	\label{ex:baseChangeCategory}
	For an object $B \in T$ there is an adjunction
	\begin{equation*}
		\pi_B\colon \PSh(T)_{/B} \rightleftarrows
		\PSh(T)\noloc - \times B,
	\end{equation*}
	where $\pi_B$ is the forgetful functor. Since both functors preserve colimits we obtain by precomposition an adjunction
	\begin{align*}
		\pi^*_B\colon \Cat_T \rightleftarrows
		\Cat_{T_{/B}}\noloc (\pi_B)_* = (- \times B)^*.
	\end{align*}
\end{example}

\begin{lemma}
	\label{lem:baseChangeOfTObjects}
	Consider an object $B \in T$. Then there is for every $\infty$-category $\Ee$ an equivalence of $T_{/B}$-$\infty$-categories
	\[
	\pi_B^*\ul{\Ee}_T \simeq \ul{\Ee}_{T_{/B}},
	\]
	natural in $\Ee$.
\end{lemma}
\begin{proof}
	It will suffice to prove that the composite
	\[
	T_{/B} \xrightarrow{\pi_B} T \xrightarrow{A \mapsto (T_{/A})\catop} \Cat_{\infty}
	\]
	is equivalent to the slice functor of $T_{/B}$. This is immediate from the observation that the target map $\ev_1 \colon (T_{/B})^{[1]} \to T_{/B}$ of $T_{/B}$ is the pullback along $\pi_B$ of the target map $\ev_1\colon T^{[1]} \to T$.
\end{proof}

\subsection{Parametrized functor categories}
In this subsection, we establish a variety of basic results on parametrized functor categories.
\begin{definition}
	\label{def:TCategoryOfTFunctors}
	Since $T$ is small and $\Cat_{\infty}$ is cartesian closed, the $\infty$-category $\Cat_T = \Fun(T\catop,\Cat_\infty)$ is again cartesian closed. Given two $T$-$\infty$-categories $\Cc$ and $\Dd$, we define the \textit{$T$-$\infty$-category of $T$-functors $\Cc \to \Dd$}, denoted $\ulFun_T(\Cc,\Dd)$, as the internal hom-object between $\Cc$ and $\Dd$ in the $\infty$-category $\Cat_T$. In particular, for any triple of $T$-$\infty$-categories $\Cc$, $\Dd$ and $\Ee$ there is a natural equivalence
	\begin{align*}
		\ulFun_T(\Cc \times \Dd, \Ee) \simeq \ulFun_T(\Cc,\ulFun_T(\Dd,\Ee)).
	\end{align*}
\end{definition}

\begin{definition}
	\label{def:CategoryOfTFunctors}
	Given two $T$-$\infty$-categories $\Cc$ and $\Dd$, we define the $\infty$-category $\Fun_T(\Cc,\Dd)$ of $T$-functors $\Cc \to \Dd$ as the underlying $\infty$-category of the $T$-$\infty$-category $\ulFun_T(\Cc,\Dd)$:
	\begin{align*}
		\Fun_T(\Cc,\Dd) := \Gamma(\ulFun_T(\Cc,\Dd)).
	\end{align*}
\end{definition}

\begin{remark}
	The objects of $\Fun_T(\Cc,\Dd)$ may be identified with $T$-functors $\Cc \to \Dd$. If $F$ and $F'$ are two such $T$-functors, we refer to a morphism $\alpha\colon F \to F'$ in $\Fun_T(\Cc,\Dd)$ as a \textit{natural transformation of $T$-functors}. A natural transformation of $T$-functors is given by a collection of natural transformations $\eta_A\colon F(A)\rightarrow F'(A)$ together with a coherent collection of 3-cells which fill the cylinders
	\[
	\begin{tikzcd}[column sep=3.0pc, row sep=4pc]
		\Cc(B)\arrow[d, "f^*"{left, name=s3}]\ar[r, bend left, "F(B)", ""{name=s1, below}]\ar[r, bend right, "{F'(B)}"{below}, ""{above, name=t1}] & \Dd(B)\arrow[d, "f^*"{right, name=t3}]\arrow[Rightarrow, from=s1, to=t1, "\eta_B"]\\
		\Cc(A)\ar[r, bend left, "F(A)", ""{name=s2, below}]\ar[r, bend right, "{F'(A)}"{below}, ""{above, name=t2}] & {\Dd(A),}\arrow[Rightarrow, from=s2, to=t2, "\eta_{A}"]
	\end{tikzcd}
	\]
	for every morphism $f\colon A\to B$ in $T$.
\end{remark}

\begin{remark}\label{rk:YonedaEmbedding}
	Let $\Cc$ be a $T$-$\infty$-category. In \cite{martini2021yoneda}*{Section 4.7}, Martini constructs a \emph{parametrized hom functor}
	\[\ul{\Hom}_{\Cc}(-,-)\colon \Cc\catop \times \Cc \rightarrow \ul{\Spc}_T\] 
	via a parametrized version of straightening for left fibrations, also see \cite{exposeI}*{Section 10}. While we will never need the precise definition, let us record for motivational purposes that, given two objects $X,Y\in \Cc(B)$, the presheaf $\ul{\Hom}_{\Cc}(X,Y)$ is informally given by
	\[\ul{\Hom}_{\Cc}(X,Y)\colon (T_{/B})^{\op}\to \Spc, \quad [f\colon A\to B] \mapsto \Hom_{\Cc(A)}(f^* X,f^* Y).\] 
	Adjoining over results in the \emph{parametrized Yoneda embedding}
	\[
		y\colon \Cc \to \ul{\Fun}_T(\Cc\catop,\ul{\Spc}_T),
	\]
	which is fully faithful\footnote{A $T$-functor $F\colon \Cc \to \Dd$ is called \emph{fully faithful} if $F(A) \colon \Cc(A) \to \Dd(A)$ is fully faithful for all $A \in T$, or equivalently (as limits of fully faithful functors are fully faithful) for all $A\in \PSh(T)$.} by \cite{martini2021yoneda}*{Theorem 4.7.8}. The $T$-functors of the form $\Cc\catop \to \ul{\Spc}_T$ are called \textit{$T$-presheaves on $\Cc$}.
\end{remark}

Natural transformations between ordinary categories induce natural transformations between their associated $T$-$\infty$-categories of $T$-objects.
\begin{construction}
	\label{cstr:NaturalTrafosGiveParametrizedNaturalTrafos}
	Given $\infty$-categories $\Ee$ and $\Ee'$, we will construct a functor
	\begin{align*}
		\Fun(\Ee,\Ee') \to \Fun_T(\ul{\Ee}_T,\ul{\Ee'}_T)
	\end{align*}
	which on groupoid cores reduces to the functoriality of the construction $\Ee \mapsto \ul{\Ee}_T$ of \Cref{ex:TObjectsInCategory}. By adjunction we may equivalently specify a $T$-functor of the form
	\begin{align*}
		\const_{\Fun(\Ee,\Ee')} \times \ul{\Ee}_T \to \ul{\Ee'}_T.
	\end{align*}
	At level $B \in T$, we define this as the composition functor
	\begin{align*}
		\Fun(\Ee,\Ee') \times \Fun((T_{/B})\catop,\Ee) \to \Fun((T_{/B})\catop,\Ee').
	\end{align*}
	By precomposing with the functors $T_{/A}\catop \to T_{/B}\catop$ this specifies a $T$-functor.
\end{construction}

The following result of \cite{martiniwolf2021limits} relates the $\infty$-category of $T$-functors from \Cref{def:CategoryOfTFunctors} to the identically named $\infty$-category of $T$-functors from \cite{exposeI}*{p.3}.

\begin{proposition}[\cite{martiniwolf2021limits}*{Proposition~3.2.1}]
	For any two $T$-$\infty$-categories $\Cc$ and $\Dd$ there is a natural equivalence
	\begin{align*}
		\Fun_T(\Cc,\Dd) \simeq \Fun^{\cocart}_{/T\catop}(\smallint \Cc, \smallint \Dd),
	\end{align*}
	where the right-hand side denotes the full subcategory of $\Fun_{/T\catop}(\smallint \Cc, \smallint \Dd)$ spanned by those functors $\smallint \Cc \to \smallint \Dd$ over $T\catop$ that preserve cocartesian edges.\qed
\end{proposition}

To give a pointwise description of the parametrized functor category $\ulFun_T(\Cc,\Dd)$, we need the following enhanced version of the Yoneda lemma.

\begin{lemma}[Categorical Yoneda lemma]
	\label{lem:YonedaLemma}
	For every $B \in T$ and $\Cc \in \Cat_T$, evaluation at the identity $\id_B \in \Hom_T(B,B) = \ul{B}(B)$ induces a natural equivalence of $\infty$-categories
	\[
	\Fun_T(\ul{B},\Cc) \iso \Cc(B).
	\]
\end{lemma}
\begin{proof}
	By the Yoneda lemma and \cref{rmk:CoreAdjoint} there is a natural equivalence
	\begin{align*}
		\Hom_{\Cat_T}(\ul{B},\Cc) \simeq \core(\Cc(B))
	\end{align*}
	between the $\infty$-groupoid of $T$-functors $\ul{B} \to \Cc$ and the groupoid core of the $\infty$-category $\Cc(B)$, so the statement holds on groupoid cores. To obtain the statement on the level of categories, we use that the $\infty$-category $\Cat_T$ is cotensored over $\Cat_{\infty}$: for every $T$-$\infty$-category $\Cc$ and every $\infty$-category $\Ee$, the cotensor $\Cc^{\Ee}$ is given at $B \in T$ by $\Cc^{\Ee}(B) \simeq \Fun(\Ee,\Cc(B))$. It follows that for any $\infty$-category $\Ee$ we have natural equivalences
	\begin{align*}
		\Hom_{\Cat_{\infty}}(\Ee,\Fun_T(\ul{B},\Cc)) &\simeq \Hom_{\Cat_T}(\ul{B},\Cc^{\Ee}) \simeq \core(\Cc^{\Ee}(B)) \\
		&\simeq \core(\Fun(\Ee,\Cc(B)) = \Hom_{\Cat_{\infty}}(\Ee,\Cc(B)),
	\end{align*}
	and thus the claim follows from the Yoneda lemma.
\end{proof}

By limit-extending the previous equivalence to presheaves, we immediately obtain:

\begin{corollary}
\label{cor:YonedaLemmaPresheaves}
	There is a unique natural equivalence $\Fun_T(\ul X,\Cc)\simeq\Cc(X)$ of functors $\PSh(T)\times\Cat_T\to\Cat$ that for representable presheaves recovers the equivalence from the previous lemma.\qed
\end{corollary}

\begin{corollary}
\label{cor:TCategoryOfTFunctors2}
Let $X \in\PSh(T)$ and let $\Cc$ and $\Dd$ be $T$-$\infty$-categories. Then there are natural equivalences of $\infty$-categories
\[
    \ulFun_T(\Cc,\Dd)(X) \simeq \Fun_T(\ul{X},\ulFun_T(\Cc,\Dd)) \simeq \Fun_T(\ul{X} \times \Cc, \Dd) \simeq \Fun_T(\Cc,\ulFun_T(\ul{X},\Dd)).
\]
\end{corollary}
\begin{proof}
The first equivalence is \Cref{cor:YonedaLemmaPresheaves}, while the others are immediate.
\end{proof}

Our next goal is to give an alternative description of the functor $T$-$\infty$-category $\ulFun_T(\Cc,\Dd)$.

\begin{construction}
\label{cons:FunctorCategories}
Let $B \in T$ and let $\Cc$ and $\Dd$ be $T$-$\infty$-categories. We define a $T_{/B}$-functor
\[
    \pi_B^*\colon \pi_B^*\ulFun_{T}(\Cc,\Dd) \to \ulFun_{T_{/B}}(\pi_B^*\Cc,\pi_B^*\Dd)
\]
as adjoint to the composite
\[
    \pi_B^*\ulFun_{T}(\Cc,\Dd) \times \pi_B^*\Cc \simeq \pi_B^*(\ulFun_{T}(\Cc,\Dd) \times \Cc) \xrightarrow{\pi_B^*(\ev)} \pi_B^*\Dd.
\]
We obtain $T$-functors
\[
    \ulFun_T(\ul{B},\Dd) \to (\pi_B)_*\pi_B^*\Dd \qquad \text{ and } \qquad (\pi_B)_!\pi_B^*\Cc \to \ul{B} \times \Cc,
\]
where $\pi_{B!}\colon\Cat_{T}\to\Cat_{T_{/B}}$ denotes the left adjoint of $\pi_B^*$, i.e.\ left Kan extension along $\pi_B$.

The first one is adjoint to the composite $T_{/B}$-functor
\[
    \pi_B^*\ulFun_T(\ul{B},\Dd) \xrightarrow{\pi_B^*} \ulFun_{T_{/B}}(\pi_B^*\ul{B},\pi_B^*\Dd) \xrightarrow{\ev_{\id_B}} \pi_B^*\Dd,
\]
while the second one is adjoint to the composite
\[
    \pi_B^*\Cc \xrightarrow{(\iota, \id)} \pi_B^*\ul{B} \times \pi_B^*\Cc \simeq \pi_B^*(\ul{B} \times \Cc).
\]
Here $\ev_{\id_B}$ denotes evaluation at the object $\id_B \in \Gamma(\pi_B^*\ul{B}) = \Hom_T(B,B)$, and $\iota$ picks out $\id_B \in \Gamma(\pi_B^*\ul{B}) = \Hom_T(B,B)$. Observe that the resulting map $(\pi_B)_!\pi_B^*\Cc \to \ul{B} \times \Cc$ is the total mate of the map $\ulFun_T(\ul{B},\Dd) \to (\pi_B)_*\pi_B^*\Dd$.
\end{construction}

\begin{corollary}
	\label{cor:TCategoryOfTFunctors}
	Let $\Cc$ and $\Dd$ be $T$-$\infty$-categories and let $B \in T$. The following hold:
	\begin{enumerate}[(1)]
		\item \label{it:TFun2} \label{it:TFun3} The $T$-functors
		\[
		    \ulFun_T(\ul{B},\Dd) \iso (\pi_B)_*\pi_B^*\Dd \qquad \text{ and } \qquad (\pi_B)_!\pi_B^*\Cc \iso \ul{B} \times \Cc
		\]
		from \Cref{cons:FunctorCategories} are equivalences of $T$-$\infty$-categories.
		\item \label{it:TFun5} The $T_{/B}$-functor
		\[
		    \pi_B^*\colon \pi_B^*\ulFun_T(\Cc,\Dd) \iso \ulFun_{T_{/B}}(\pi_B^*\Cc,\pi_B^*\Dd)
		\]
		from \Cref{cons:FunctorCategories} is an equivalence of $T_{/B}$-$\infty$-categories.
		\item \label{it:TFun4} In particular, passing to global sections gives an equivalence of $\infty$-categories
		\[
		    \ulFun_T(\Cc,\Dd)(B) \iso \Fun_{T_{/B}}(\pi_B^*\Cc,\pi_B^*\Dd).
		\]
	\end{enumerate}
\end{corollary}
\begin{proof}
For part (1), it suffices to show the first equivalence, since the second equivalence follows by passing to total mates. For this, we have to show that for every object $A \in T$ the induced map $\ulFun_T(\ul{B},\Dd)(A) \to (\pi_B)_*\pi_B^*\Dd(A) = \Dd(A \times B)$ is an equivalence. Given \Cref{lem:YonedaLemma}, it will suffice to show that the following diagram commutes:
\[
\begin{tikzcd}
    \Fun_T(\ul{A},\ulFun_T(\ul{B},\Dd)) \dar{\sim}[swap]{\ref{lem:YonedaLemma}} \rar{\sim} & \Fun_T(\ul{A \times B},\Dd) \dar{\sim}[swap]{\ref{lem:YonedaLemma}} \\
    \ulFun_T(\ul{B},\Dd)(A) \rar & \Dd(A \times B).
\end{tikzcd}
\]
This follows from unwinding the definitions, using the observation that the equivalence $\Fun_T(\ul{B},\Dd) \to \Dd(B)$ from \Cref{lem:YonedaLemma} is the map induced on global sections by the map $\ulFun_T(\ul{B},\Dd) \to (\pi_B)_*\pi_B^*\Dd$.

We will next prove part (3). It suffices to show that the following diagram commutes:
\[
\begin{tikzcd}
    \Fun_T(\ul{B},\ulFun_T(\Cc,\Dd)) \dar{\sim}[swap]{\ref{lem:YonedaLemma}} \rar{\sim} & \Fun_T(\Cc,\ulFun_T(\ul{B},\Dd)) \rar{\sim}[swap]{(1)} & \Fun_T(\Cc,(\pi_B)_*\pi_B^*\Dd) \dar{\sim} \\
    \ulFun_T(\Cc,\Dd)(B) \ar{rr}{\pi_B^*(B)} &&  \Fun_{T_{/B}}(\pi_B^*\Cc,\pi_B^*\Dd).
\end{tikzcd}
\]
This is again a matter of unwinding definitions, using that the equivalence of (1) is defined in terms of the map $\pi_B^*$ and evaluation at $\id_B$.

Finally, for part (2) it remains to show that the map $\pi_B^*$ induces an equivalence when evaluated at every object $A \in T_{/B}$. To see this, consider the following two $T_{/A}$-functors:
\[
    \pi_A^*\ulFun_T(\Cc,\Dd) \to \pi_A^*\ulFun_{T_{/B}}(\pi_B^*\Cc,\pi_B^*\Dd) \to \ulFun_{T_{/A}}(\pi_A^*\Cc,\pi_A^*\Dd).
\]
Here we abuse notation by writing $A$ both for an object in $T_{/B}$ and for its underlying object in $T$. By part (3), both the second map and the composite map induce equivalences on global sections, and therefore so does the first. This finishes the proof.
\end{proof}

For later use let us describe the functoriality of $\ul\Fun_{T_{/B}}(\pi_B^*\Cc,\pi_B^*\Dd)$ in $B$. While this can be done in a fully coherent fashion using the results and techniques of \cite{HHLNb}, for the purposes of the present paper the following more elementary lemma will be sufficient:

\begin{lemma}\label{lemma:restriction_functor_cat}
	Let $f\colon A\to B$ be a map in $T$, and let $\sigma\colon \pi_B\circ T_{/f}\simeq \pi_A$ be the usual equivalence. Then the diagram
	\begin{equation*}
		\begin{tikzcd}[cramped]
			\Hom_{T_{/B}}(\pi_B^*\Cc,\pi_B^*\Dd)\arrow[d,hook]\arrow[r, "T_{/f}^*"] & \Hom_{T_{/A}}(T_{/f}^*\pi_B^*\Cc,T_{/f}^*\pi_B^*\Dd)\arrow[r, "\sigma^*"] &
			\Hom_{T_{/A}}(\pi_A^*\Cc,\pi_A^*\Dd)\arrow[d,hook]\\
			\Fun_{T_{/B}}(\pi_B^*\Cc,\pi_B^*\Dd)\arrow[d,"\sim"',"\ref{cor:TCategoryOfTFunctors}"]&&\Fun_{T_{/A}}(\pi_A^*\Cc,\pi_A^*\Dd)\arrow[d, "\sim"', "\ref{cor:TCategoryOfTFunctors}"]\\
			\ul\Fun_T(\Cc,\Dd)(B)\arrow[rr, "f^*"'] && \ul\Fun_T(\Cc,\Dd)(A)
		\end{tikzcd}
	\end{equation*}
	of natural transformations of functors $\Cat_T^\op\times\Cat_T\to\Cat_\infty$ commutes up to homotopy.
	\begin{proof}
		Unravelling the definitions and using the naturality of the equivalences from the previous corollary, it suffices to construct a homotopy filling
		\begin{equation*}
			\begin{tikzcd}[cramped]
				\Hom_{T_{/B}}(\pi_B^*\Cc,\pi_B^*\Dd)\arrow[d,"\sim"']\arrow[r, "T_{/f}^*"] & \Hom_{T_{/A}}(T_{/f}^*\pi_B^*\Cc,T_{/f}^*\pi_B^*\Dd)
				\arrow[r, "\sigma^*"] &
				\Hom_{T_{/A}}(\pi_A^*\Cc,\pi_A^*\Dd)\arrow[d,"\sim"]\\
				\Hom_T(\pi_{B!}\pi_B^*\Cc,\Dd)\arrow[d, "j_B^*"'] && \Hom_T(\pi_{A!}\pi_A^*\Cc,\Dd)\arrow[d, "j_A^*"]\\
				\Hom_T(\ul B\times \Cc,\Dd)\arrow[rr, "(f\times\Cc)^*"'] && \Hom_T(\ul A\times \Cc,\Dd)
			\end{tikzcd}
		\end{equation*}
		where the unlabelled equivalences come from adjunction and $j_A\colon\pi_{A!}\pi_A^*\Cc\to\ul A\times\Cc$, $j_B\colon\pi_{B!}\pi_B^*\Cc\to\ul B\times\Cc$ are as in Corollary~\ref{cor:TCategoryOfTFunctors} again.

		Obviously, we can make the lower half commute by adding the restriction along the composite $\pi_{A!}\pi_A^*\Cc\simeq\ul A\times \Cc\to\ul B\times\Cc\simeq\pi_{B!}\pi_B^*$ as the middle arrow. Similarly, we can make the upper portion commute by taking the restriction along
		\begin{equation*}
			f_\lozenge\colon (\pi_A)_!\pi_A^*\Cc \xrightarrow[\smash{\raise2pt\hbox{$\scriptstyle\sim$}}]{\;\sigma^*\,} (\pi_A)_! T_{/f}^*\pi_B^*\Cc \longrightarrow (\pi_B)_! \pi_B^* \Cc
		\end{equation*}
		instead, where the second map is the mate of $\sigma^*$. It will therefore suffice to show that these two natural transformations $\pi_{A!}\pi_A^*\Rightarrow\pi_{B!}\pi_B^*$ are in fact homotopic. Plugging in the definitions of the equivalences $j_A$ and $j_B$, this amounts to saying that $j_Bf_\lozenge$ is adjunct to the map $\pi_A^*\Cc\to\pi_A^*\ul B\times\pi_A^*\Cc$ picking out $f\in(\pi_A^*\ul B)(\id_A)=\Hom(A,B)$.

		Further plugging in definitions, the adjunct of $jf_\lozenge$ is the top right composite in
		\begin{equation*}
			\begin{tikzcd}
				\pi_A^*\Cc\arrow[r, "\sigma^*"] & T_{/f}^*\pi_B^*\Cc\arrow[r,"\eta"]\arrow[dr, bend right=10pt, dashed] & T_{/f}^*\pi_B^*(\pi_B)_!\pi_B^*\Cc\arrow[r, "(\sigma^{-1})^*"]\arrow[d, "j_B"'] & \pi_A^*(\pi_B)_!\pi_B^*\Cc\arrow[d, "j_B"]\\
				&& T_{/f}^*\pi_B^*(\ul B\times\Cc)\arrow[r, "(\sigma^{-1})^*"'] & \pi_A^*(\ul B\times\Cc).
			\end{tikzcd}
		\end{equation*}
		The square on the right commutes by naturality and the dashed composite is by definition the adjunct of $j_B$, i.e.~it is induced by the map $\pi_B^*\Cc\to \pi_B^*(\ul B\times\Cc)$ classifying $\id_B\in\ul B(B)$. It follows that after postcomposing with the projection $\pi_A^*(\ul B\times\Cc)\to\pi_A^*\Cc$ the above is simply the identity, and it remains to show that the natural map $\pi_A^*\Cc\to\pi_A^*\ul B$ obtained by postcomposing with the other projection classifies $f\in\ul B(A)$. But indeed, as a functor of $\Cat_T$ the right hand side is constant, so any natural transformation into it is determined by its value on the terminal object. Together with application of the Yoneda lemma it therefore suffices that for $\Cc=*$ the map $(\pi_A^**)(\id_A)\to (\pi_A^*\ul B)(\id_A)$ hits $f$. However, by the above commutative diagram this can be identified with $(\pi_B^**)(f)\to(\pi_B^*\ul B)(f)$, which hits $f=f^*(\id_B)$ for formal reasons.
	\end{proof}
\end{lemma}

We will now prove the adjunction between $\Ee \mapsto \ul{\Ee}_T$ and $\Cc \mapsto \smallint \Cc$ promised in \Cref{rmk:AdjunctionGrothendieckConstructionTObjects}.

\begin{lemma}
	\label{lem:AdjunctionGrothendieckConstructionTObjects}
	The functor $\smallint\colon \Cat_T \to \Cat_{\infty}$, sending a $T$-$\infty$-category $\Cc\colon T\catop \to \Cat_{\infty}$ to the total space $\smallint \Cc$ of the cocartesian fibration $\int \Cc \to T\catop$ it classifies, admits a right adjoint given by the construction $\Ee \mapsto \ul{\Ee}_T$ of \cref{ex:TObjectsInCategory}.
\end{lemma}
\begin{proof}
	The functor $\smallint\colon \Cat_T \to \Cat_{\infty}$ can be expanded into the following composite functor:
	\begin{align*}
		\Cat_T \overset{\ref{rmk:TCategoriesAsCocartesianFibrations}}{\simeq} (\Cat_{\infty})_{/T\catop}^{\cocart} \hookrightarrow (\Cat_{\infty})_{/T\catop} \xrightarrow{\fgt} \Cat_{\infty}
	\end{align*}
	By \cite{Ramzi2022Grothendieck}*{Proposition~A.1}, the functor in the middle is cocontinuous, as is the forgetful functor for trivial reasons. It follows from the Special Adjoint Functor Theorem that $\smallint\colon \Cat_T \to \Cat_{\infty}$ admits a right adjoint $R\colon \Cat_{\infty} \to \Cat_T$.

	As a formal consequence we obtain for each $T$-$\infty$-category $\Cc$ and for each $\infty$-category $\Ee$ a natural equivalence
	\begin{align*}
		\Fun_T(\Cc,R(\Ee)) \simeq \Fun(\smallint \Cc,\Ee)
	\end{align*}
	between the $\infty$-category $T$-functors $\Cc \to R(\Ee)$ and the $\infty$-category of functors $\smallint \Cc \to \Ee$: for every other $\infty$-category $\Ee'$ there is a natural equivalence
	\begin{align*}
		\Hom_{\Cat_{T}}(\Ee',\Fun_T(\Cc,R(\Ee))) &\simeq \Hom_{\Cat_T}(\Cc \times \const_{\Ee'}, R(\Ee)) \\
		&\simeq \Hom_{\Cat_{\infty}}(\smallint(\Cc \times \const_{\Ee'}), \Ee)\\ &\simeq \Hom_{\Cat_{\infty}}(\smallint\Cc \times \Ee', \Ee) \\
		&\simeq \Hom_{\Cat_{\infty}}(\Ee', \Fun(\smallint\Cc, \Ee)),
	\end{align*}
	where we use that the cocartesian unstraightening of $\const_{\Ee'}$ is $T\catop \times \Ee'$ and that the inclusion $(\Cat_{\infty})_{/T\catop}^{\cocart} \hookrightarrow (\Cat_{\infty})_{/T\catop}$ preserves finite products. The claim now follows from the Yoneda lemma.

	The description of $R$ as the functor $\Ee \mapsto \ul{\Ee}_T$ from \cref{ex:TObjectsInCategory} now follows immediately by recalling that the cocartesian unstraightening of the functor $\ul{B} = \Hom_T(-, B)\colon T\catop \to \Spc$ is by definition given by the target functor $(T_{/B})\catop \to T\catop$. Namely for any $\Ee \in \Cat_{\infty}$ and $B \in T$ we have a natural equivalence
	\begin{align}\label{eq:identifying-R}
		R(\Ee)(B) \overset{\ref{lem:YonedaLemma}}{\simeq} \Fun_T(\ul{B},R(\Ee)) \simeq \Fun(\smallint \ul{B},\Ee) \simeq \Fun((T_{/B})\catop,\Ee) = \ul{\Ee}_T(B).
	\end{align}
	This finishes the proof.
\end{proof}

\begin{remark}\label{Rmk:FunTobjects}
	Combining the previous lemma and \Cref{cor:TCategoryOfTFunctors2} we obtain a natural equivalence
	\[\ulFun_T(\Cc,\Ee_T) \simeq \Fun(\smallint \Cc\times \ul{(-)},\Ee).\]
\end{remark}

\begin{remark}\label{rk:grothendieck-vs-evaluation}
	Let $B\in T$ arbitrary. Unravelling the chain of equivalences $(\ref{eq:identifying-R})$ we see that the diagram
	\begin{equation*}
		\begin{tikzcd}
			\Fun_T(\ul B,\ul{\mathcal E}_T)\arrow[r, "\text{adjunction}"]\arrow[d, "\text{Yoneda}"'] &[2em] \Fun(\smallint\ul B,\mathcal E)\arrow[d, "f^*"]\\
			\ul{\mathcal E}_T(B)\arrow[r, equal] & \Fun(T_{/B},\mathcal E)
		\end{tikzcd}
	\end{equation*}
	of equivalences commutes up to natural equivalence where $f$ is the chosen identification of $\smallint \ul B$ with $T_{/B}$ over $T$.

	Now assume $T$ has a final object $1$. Specializing the above to $B=1$ (and identifying $T_{/1}$ with $T$ as usual), we see that
	\begin{equation*}
		\begin{tikzcd}
			\Fun_T(\ul 1,\ul{\mathcal E}_T)\arrow[d, "\text{Yoneda}"', "\simeq"] \arrow[r, "\textup{adjunction}", "\simeq"'] &[2em] \Fun(\int\ul1,\mathcal E)\\
			\mathcal E_T(1)\arrow[r, equal] & \Fun(T^\op, \mathcal E)\arrow[u, "\pi^*"']
		\end{tikzcd}
	\end{equation*}
	commutes up to natural equivalence, where $\pi\colon\int\ul1\to T^\op$ is the cocartesian projection. Combining this with the naturality of the adjunction equivalence, we conclude that we have for every $T$-$\infty$-category $\mathcal C$ and $c\in\mathcal C(1)$ a natural equivalence filling
	\begin{equation*}
		\begin{tikzcd}
			\Fun_T(\mathcal C,\ul{\mathcal E}_T)\arrow[d, "\ev_c"'] \arrow[r, "\textup{adjunction}", "\simeq"'] &[2em] \Fun(\smallint\mathcal C,\mathcal E)\arrow[d, "\hat c^*"]\\
			\ul{\mathcal E}_T(1) \arrow[r, equal] & \Fun(T^\op,\mathcal E)
		\end{tikzcd}
	\end{equation*}
	where $\hat c\colon T^\op\to \smallint\mathcal C$ is the essentially unique map over $T^\op$ sending the fiber over $1\in T$ to $c$ (i.e.~the unstraightening of $c$ viewed as a $T$-functor $\ul1\to\mathcal C$).
\end{remark}

\begin{remark}\label{rmk:unwinding_T_object_adj}
	We can make the equivalence $\Fun_T(\Cc,\ul{\Ee}_T) \simeq \Fun(\smallint \Cc,\Ee)$ of \Cref{lem:AdjunctionGrothendieckConstructionTObjects} more explicit. Consider a functor $\tilde{F}\colon \int \Cc\rightarrow \Ee$. The associated $T$-functor $F\colon \Cc\rightarrow \ul{\Ee}_T$ is given at $B\in T$ by the functor
	\[F_B\colon \Cc(B)\rightarrow \Fun(T_{/B}^\op,\Ee),\] where $F_B(X)(h\colon C\rightarrow B) = \tilde{F}(h^*(X))$, the value of $\tilde{F}$ on the cocartesian pushforward of $X\in \Cc(B)$ along $h$ to $\Cc(C)$. The value of $F_B(X)$ on a triangle
	\[
	\begin{tikzcd}
		C \arrow[rd, "h"']\arrow[rr,"f"] &   &  D \arrow[ld,"g"] \\
		& B
	\end{tikzcd}
	\] is given by applying $\tilde{F}$ to the cocartesian edge over $f$ from $g^*(X)$ to $h^*(X)$. More generally, for another object $B' \in T$ and a functor $\tilde{F}\colon \smallint (\Cc \times \ul{B'}) \to \Ee$, the associated $T$-functor $F\colon \Cc \to \ulFun_T(\ul{B}',\Ee)$ is given at $B \in T$ by the functor $F_B\colon \Cc(B) \to \Fun(T_{/B \times B'}\catop,\Ee)$ given by
	\[
	F_B(X)(A \xrightarrow{(f_B,f_{B'})} B \times B') = \tilde{F}(A,f_B^*X,f_{B'}).
	\]
\end{remark}

\subsection{Parametrized adjunctions, limits and colimits}
\label{subsec:ParametrizedAdjunctions}
We will briefly recall the parametrized versions of adjunctions, limits and colimits, following Sections 3 and 4 of \cite{martiniwolf2021limits} (see \cref{rmk:ParametrizedCategoriesAsInternalCategories}). An alternative treatment in the language of cocartesian fibrations over $T\catop$ is given by \cite{shah2021parametrized}*{Sections 8 and 9}.

\begin{definition}[\cite{martiniwolf2021limits}*{Definition~3.1.1}]
	Let $\Cc$ and $\Dd$ be $T$-$\infty$-categories. An \textit{adjunction} between $\Cc$ and $\Dd$ is a tuple $(L, R, \eta, \epsilon)$, where
	$L\colon \Cc \to \Dd$ and $R\colon \Dd \to \Cc$ are $T$-functors and where $\eta\colon \id_{\Dd} \to RL$ and $\epsilon\colon LR \to \id_{\Cc}$ are natural transformations of $T$-functors fitting in commutative triangles
	\[\begin{tikzcd}
		L \drar[equal] \rar{L\eta} & LRL \dar{\epsilon L} \\
		& L
	\end{tikzcd}
	\qquad \text{ and } \qquad
	\begin{tikzcd}
		RLR \dar[swap]{R\epsilon} & R \dlar[equal] \lar[swap]{\eta R} \\
		R.
	\end{tikzcd}\]
\end{definition}
Note that the notion of an adjunction between two $T$-$\infty$-categories only depends on the (homotopy) 2-category associated to $\Cat_T$ and in particular many of the standard 2-categorical results about adjunctions hold in this setting.

\begin{example}
	Every adjunction $\Ee \rightleftarrows \Ee'$ of $\infty$-categories gives rise to an adjunction $\const_{\Ee} \rightleftarrows \const_{\Ee'}$ on associated constant $T$-$\infty$-categories.
\end{example}

\begin{example}
	\label{ex:AdjunctionGivesParametrizedAdjunctionOnTObjects}
	By \Cref{cstr:NaturalTrafosGiveParametrizedNaturalTrafos}, every adjunction $\Ee \rightleftarrows \Ee'$ of $\infty$-categories gives rise to an adjunction $\ul{\Ee}_T \rightleftarrows \ul{\Ee'}_T$ on associated $T$-$\infty$-categories of $T$-objects.
\end{example}

Important will be the following `pointwise' criterion for checking that a $T$-functor has a parametrized adjoint.

\begin{proposition}[{\cite{martiniwolf2021limits}*{Proposition~3.2.9 and Corollary~3.2.11}}]
	\label{prop:pointwisecriterionadjoints}
	A $T$-functor $F\colon \Cc \to \Dd$ admits a (parametrized) right adjoint if and only if the following two conditions hold:
	\begin{enumerate}[(1)]
		\item For every object $B \in T$, the induced functor $F(B)\colon \Cc(B) \to \Dd(B)$ admits a right adjoint $G(B)\colon \Dd(B) \to \Cc(B)$;
		\item For every morphism $f\colon A \to B$ in $T$, the Beck-Chevalley transformation
		\begin{align*}
			f^* \circ G(B) \implies G(A) \circ f^*
		\end{align*}
		given as the mate of the naturality square
		\[\begin{tikzcd}
			\Cc(B) \rar{F(B)} \dar[swap]{f^*} & \Dd(A) \dar{f^*} \\
			\Cc(A) \rar{F(A)} & \Dd(A)
		\end{tikzcd}\]
		is an equivalence.
	\end{enumerate}
	If this is the case, the right adjoint $G\colon \Dd \to \Cc$ of $F$ is given on an object $B \in T$ by the functor $G(B)\colon \Dd(B) \to \Cc(B)$. Moreover, if $Y\in\PSh(T)$ is any presheaf, then also the functor $F(Y)\colon\mathcal C(Y)\to\mathcal D(Y)$ admits a right adjoint $G(Y)$ in this case, and for any map $f\colon X\to Y$ in $\PSh(T)$ the Beck-Chevalley map $f^*\circ G(Y)\Rightarrow G(X)\circ f^*$ is an equivalence.

	The dual statement for parametrized left adjoints also holds. \qed
\end{proposition}

We will now move to parametrized limits and colimits, of which we will only give a brief treatment sufficient for the purposes of the present article.

\begin{definition}
	Let $K$ and $\Cc$ be $T$-$\infty$-categories. We say that \textit{$\Cc$ admits $K$-indexed colimits} if the diagonal functor $\diag\colon \Cc \to \ulFun_T(K,\Cc)$ given by precomposing with $K \to \ul{1}$ admits a left adjoint $\colim_K\colon \ulFun_T(K,\Cc) \to \Cc$. Similarly we say that \textit{$\Cc$ admits $K$-indexed limits} if $\diag$ admits a right adjoint $\lim_K\colon \ulFun_T(K,\Cc) \to \Cc$.
\end{definition}

\begin{definition}
	\label{def:FunctorPreservesParametrizedColimits}
	Let $K$, $\Cc$ and $\Dd$ be $T$-$\infty$-categories and assume that $\Cc$ and $\Dd$ admit $K$-indexed colimits. We will say that a $T$-functor $F\colon \Cc \to \Dd$ \textit{preserves $K$-indexed colimits} if the Beck-Chevalley transformation $\colim_K \circ \;\ulFun_T(K,F) \implies F \circ \colim_K$ of the naturality square
	\[\begin{tikzcd}
		\Cc \rar{\diag} \dar[swap]{F} & \ulFun_T(K,\Cc) \dar{\ulFun_T(K,F)} \\
		\Dd \rar{\diag} & \ulFun_T(K,\Dd)
	\end{tikzcd}\]
	is an equivalence.
\end{definition}

In the non-parametrized context, one often asks an $\infty$-category to admit (co)limits for a certain class of indexing diagrams. In the parametrized setting, one should work with the following parametrized notion of `class of indexing diagrams.'

\begin{definition}
	\label{def:ClassOfTCategories}
	Let $T$ be an $\infty$-category. A \textit{class of $T$-$\infty$-categories} is a full parametrized subcategory $\bbU \subseteq \ul{\mathrm{cat}}_T$ of the $T$-$\infty$-category of small $T$-$\infty$-categories.
\end{definition}

\begin{definition}[\cite{martiniwolf2021limits}*{Definition~5.2.1 and Remark~5.2.4}]
	\label{def:UColimits}
	Let $\bbU$ be a class of $T$-$\infty$-categories and let $\Cc$ and $\Dd$ be $T$-$\infty$-categories.
	\begin{enumerate}[(1)]
		\item We will say that $\Cc$ \textit{admits $\bbU$-colimits} if the $T_{/B}$-$\infty$-category $\pi_B^*\Cc$ of \cref{ex:baseChangeCategory} admits $K$-indexed $T_{/B}$-colimits for every $B \in T$ and $K \in \bbU(B) \subseteq \Cat(T_{/B})$.
		\item If $\Cc$ and $\Dd$ admit $\bbU$-colimits, a $T$-functor $F\colon \Cc \to \Dd$ is said to \textit{preserve $\bbU$-colimits} if $\pi_B^*F$ preserves $K$-indexed $T_{/B}$-colimits for every $B \in T$ and $K \in \bbU(B)$.
	\end{enumerate}
	Dually, $\Cc$ is said to \textit{admit $\bbU$-limits} if for every $B \in T$ and $K \in \bbU(B)$, the $T_{/B}$-$\infty$-category $\pi_B^*\Cc$ admits $K$-indexed $T_{/B}$-limits. A $T$-functor $F\colon \Cc \to \Dd$ is said to \textit{preserve $\bbU$-limits} if $\pi_B^*F$ preserved $K$-indexed $T_{/B}$-limits for every $B \in T$ and $K \in \bbU(B)$.

	If $\bbU = \ul{\mathrm{cat}}_T$ consists of \textit{all} $T$-$\infty$-categories, we will say that $\Cc$ is \textit{$T$-cocomplete} or \textit{$T$-complete} respectively.
\end{definition}

From the pointwise criterion, \cref{prop:pointwisecriterionadjoints}, of parametrized adjunctions, we immediately obtain characterizations of $T$-(co)limits indexed by constant $T$-$\infty$-categories and $T$-$\infty$-groupoids, respectively. We start with the case of constant $T$-$\infty$-categories.

\begin{lemma}[cf.\ \cite{martiniwolf2021limits}*{Example~4.1.14 and Remark~4.1.15}]
	\label{lem:ColimitsIndexedByConstantTCategories}
	Let $\Cc$ be a $T$-$\infty$-category, let $K$ be an $\infty$-category, and let $\const_K$ be the associated constant $T$-$\infty$-category. Then the following conditions are equivalent:
	\begin{enumerate}[(1)]
		\item The $T$-$\infty$-category $\Cc$ admits $\const_{K}$-indexed colimits;
		\item For every object $B \in T$ the $\infty$-category $\Cc(B)$ admits $K$-indexed colimits, and for every morphism $\beta\colon B' \to B$ in $T$ the restriction functor $\beta^*\colon \Cc(B) \to \Cc(B')$ preserves $K$-indexed colimits.
		\item For every presheaf $Y\in\PSh(T)$, the $\infty$-category $\Cc(Y)$ admits $K$-indexed colimits, and for every morphism $\beta\colon Y'\to Y$ in $\PSh(T)$ the restriction $\beta^*\colon \Cc(Y)\to \Cc(Y')$ preserves $K$-indexed colimits.
	\end{enumerate}
	The dual statement for limits also holds.
\end{lemma}
\begin{proof}
	We apply the natural identification
	\begin{align*}
		\ulFun_T(\const_K,\Cc)(B) &\overset{\ref{cor:TCategoryOfTFunctors2}}{\simeq} \Fun_T(\const_K,\ulFun_T(\ul{B},\Cc)) \\
		&\overset{\ref{rmk:AdjunctionConstantUnderlying}}{\simeq} \Fun(K,\Fun_T(\ul{B},\Cc)) \\
		&\overset{\ref{lem:YonedaLemma}}{\simeq} \Fun(K,\Cc(B)).
	\end{align*}
	Because each equivalence above is natural in $K$, we find that under this identification the $T$-functor $\diag\colon\Cc \rightarrow \ul{\Fun}_{T}(K,\Cc)$ corresponds at $B\in T$ to the standard diagonal functor. Furthermore the Beck-Chevalley transformation associated to the naturality square
	% https://q.uiver.app/?q=WzAsNCxbMCwwLCJcXENjKEIpIl0sWzAsMSwiXFxDYyhBKSJdLFsxLDAsIlxcRnVuKEssXFxDYykiXSxbMSwxLCJcXEZ1bihLLFxcQ2MpIl0sWzAsMSwiXFxDQyhmKSIsMl0sWzAsMiwiXFxkaWFnIl0sWzEsMywiXFxkaWFnIl0sWzIsMywiXFxDYyhmKV4qIl1d
	\[\begin{tikzcd}
		{\Cc(B)} & {\Fun(K,\Cc(B))} \\
		{\Cc(B')} & {\Fun(K,\Cc(B'))}
		\arrow["{\beta^*}"', from=1-1, to=2-1]
		\arrow["\diag", from=1-1, to=1-2]
		\arrow["\diag", from=2-1, to=2-2]
		\arrow["{\Fun(K,\beta^*)}", from=1-2, to=2-2]
	\end{tikzcd}\]
	is the standard comparison $\colim \circ \Fun(K,\beta^*) \Rightarrow F \circ \colim_K$. Therefore the equivalence of the first two statements is an instance of \cref{prop:pointwisecriterionadjoints}.

	The equivalence between the first and the third statement is proven in exactly the same way.
\end{proof}

The following result is proved similarly and will be left to the reader.

\begin{lemma}
	Let $K$ be an $\infty$-category and let $\Cc$ and $\Dd$ be two $T$-$\infty$-categories that admit $\const_K$-indexed $T$-colimits. Then a $T$-functor $F\colon \Cc \to \Dd$ preserves $\const_K$-indexed $T$-colimits if and only if for each $B \in T$ the functor $F(B)\colon \Cc(B) \to \Dd(B)$ preserves $K$-indexed colimits. Moreover, in this case $F(Y)\colon\Cc(Y)\to\Dd(Y)$ preserves $K$-indexed colimits for all $Y\in\PSh(T)$.

	The dual statement for limits also holds.\qed
\end{lemma}

\begin{definition}
	\label{def:fiberwiseCocompleteness}
	If the equivalent conditions of \cref{lem:ColimitsIndexedByConstantTCategories} are satisfied, we say that $\Cc$ \textit{admits fiberwise $K$-indexed colimits.} If $S$ is a collection of small $\infty$-categories such that $\Cc$ admits fiberwise $K$-indexed colimits for every $K \in S$, we say that $\Cc$ \textit{admits fiberwise $S$-indexed colimits}. We say that $\Cc$ is \textit{fiberwise cocomplete} if $\Cc$ admits fiberwise $K$-indexed colimits for every small $\infty$-category $K$.

	Dually one defines when $\Cc$ \textit{admits fiberwise $K$-indexed limits} or is \textit{fiberwise complete}.
\end{definition}

We next describe parametrized colimits indexed by $T$-$\infty$-groupoids.

\begin{definition}
	\label{def:ClassOfTGroupoids}
	A \textit{class of $T$-$\infty$-groupoids}\footnote{This is called a `subuniverse' in \cite{martini2021yoneda}*{Definition 3.9.13}} is a full parametrized subcategory $\bbU \subseteq \ul{\Spc}_T$ of the $T$-$\infty$-category of $T$-$\infty$-groupoids. A morphism $f\colon X \to Y$ in $\PSh(T)$ is said to be in $\bbU$ if it is an object in the full subcategory $\bbU(Y) \subseteq \PSh(T)_{/Y}$.
\end{definition}

\begin{remark}
	In the above definition, we have again viewed $\bbU$ as a sheaf on $\PSh(T)$ via limit extension. For later use, let us make explicit what this means in terms of the original functor $T^\op\to\Cat_\infty$: a map $f\colon X\to Y$ in $\PSh(T)$ belongs to $\bbU$ if and only if for every map $\beta\colon B\to Y$ from a representable presheaf $B\in T$ the pulled back map $\beta^*f\colon\beta^*X\to B$ is an object of $\bbU(B)$.
\end{remark}

\begin{lemma}[cf.\ \cite{martiniwolf2021limits}*{Example~4.1.13}, \cite{shah2021parametrized}*{Proposition 5.12}]
	\label{lem:UColimitsVsAdjointable}
	Let $\bbU$ be a class of $T$-$\infty$-groupoids. Then a $T$-$\infty$-category $\Cc$ admits $\bbU$-colimits if and only if for every morphism $p\colon A \to B$ in $\bbU$, with $B \in T$, the restriction functor $p^*\colon \Cc(B) \to \Cc(A)$ admits a left adjoint $p_!\colon \Cc(A) \to \Cc(B)$, and for every pullback square
	\begin{equation}
		\label{eq:PullbackSquareTCocompleteness}
		\begin{tikzcd}
			A' \dar[swap]{p'} \rar{\alpha} \drar[pullback] & A \dar{p} \\
			B' \rar{\beta} & B
		\end{tikzcd}
	\end{equation}
	in $\PSh(T)$ with $\beta\colon B' \to B$ in $T$ and $p\colon A \to B$ in $\bbU$, the Beck-Chevalley transformation $p'_! \circ \alpha^* \Rightarrow \beta^* \circ p_!$ associated to the commutative diagram
	\[\begin{tikzcd}
		\Cc(B) \rar{\beta^*} \dar[swap]{p^*} & \Cc(B') \dar{{p'}^*} \\
		\Cc(A) \rar{\alpha^*} & \Cc(A')
	\end{tikzcd}\]
	is a natural equivalence.

	Dually, $\Cc$ admits $\bbU$-limits if and only if $p^*\colon \Cc(B) \to \Cc(A)$ admits a right adjoint $p_*\colon \Cc(A) \to \Cc(B)$ for every morphism $p\colon A \to B$ in $\bbU$ and for every pullback square \eqref{eq:PullbackSquareTCocompleteness}, the Beck-Chevalley transformation $\beta^* \circ p_* \Rightarrow p'_* \circ \alpha^*$ is a natural equivalence.
\end{lemma}
\begin{proof}
	Let $(p\colon A \to B) \in \bbU(B) \subseteq \PSh(T)_{/B}$ be a morphism in $\bbU$. It suffices to show that the $T_{/B}$-$\infty$-category $\pi_B^*\Cc$ admits $\ul{A}$-indexed colimits if and only if for every pullback diagram
	\[
	\begin{tikzcd}
		A'' \rar{\alpha'} \dar[swap]{p''} \drar[pullback] & A' \dar[swap]{p'} \rar{\alpha} \drar[pullback] & A \dar{p} \\
		B'' \rar{\beta'} & B' \rar{\beta} & B
	\end{tikzcd}
	\]
	the functors ${p'}^*$ and ${p''}^*$ admit left adjoints $p'_!$ and $p''_!$, and the Beck-Chevalley transformation $p''_! \circ {\alpha'}^* \Rightarrow {\beta'}^* \circ p'_!$ is a natural equivalence. By replacing $T$ by $T_{/B}$, we may assume $B = 1$ is a terminal object of $T$. Using the natural identifications
	\begin{align*}
		\ulFun_T(\ul{A},\Cc)(B') \overset{\ref{cor:TCategoryOfTFunctors2}}{\simeq} \Fun_T(\ul{A} \times \ul{B'},\Cc) \simeq \Fun_T(\ul{A \times B'},\Cc) \overset{\ref{lem:YonedaLemma}}{\simeq} \Cc(A \times B'),
	\end{align*}
	this is an instance of \cref{prop:pointwisecriterionadjoints} applied to the $T$-$\infty$-category $\ulFun_T(\ul{A},\Cc)$.
\end{proof}

\begin{remark}
	\label{rmk:limitExtensionRestrictedClass}
	If $\Cc$ is $\bbU$-cocomplete, then \cite{martiniwolf2021limits}*{Remark~5.2.4} shows that the left adjoint $p_!\colon \Cc(A) \to \Cc(B)$ exists more generally for any presheaf $B\in\PSh(T)$ and any $p\in\bbU(B)$; similarly, the Beck-Chevalley condition holds for any pullback square \eqref{eq:PullbackSquareTCocompleteness} in which $p\in\bbU(B)$, $B\in\PSh(T)$. Put differently, we can drop all representability conditions in the above lemma.
\end{remark}

The following lemma is proved in a similar way and is left to the reader.

\begin{lemma}\label{lemma:BC-vs-cocontinuity}
	Let $\bbU$ be a class of $T$-$\infty$-groupoids and let $\Cc$ and $\Dd$ be two $T$-$\infty$-categories which admit $\bbU$-colimits. Then a $T$-functor $F$ preserves $\bbU$-colimits if and only if for every $B\in T$ and morphism $f\colon A \to B$ in $\bbU$, the Beck-Chevalley transformation $f_! \circ F(A) \Rightarrow F(B) \circ f_!$ is an equivalence. Moreover, in this case the Beck-Chevalley map is an equivalence more generally for any presheaf $B\in\PSh(T)$ and any $f\in\bbU(B)$.

	The dual statement for preserving $\bbU$-limits also holds.\qednow
\end{lemma}

Using this we now give an easy criterion ensuring that cocontinuity is preserved under changing the indexing category:

\begin{lemma}\label{lemma:restriction-preserves-cocompl}
	Let $f\colon\PSh(S)\to\PSh(T)$ be a cocontinuous functor preserving pullbacks, let $\bbU$ be a class of $S$-$\infty$-groupoids, and let $\cat V$ be a class of $T$-$\infty$-groupoids such that $f(u)\in \cat V(f(B))$ for any $B\in S$ and $(u\colon A\to B)\in\bbU(B)$.

	Then $f^*\colon\Cat_T\to\Cat_S$ sends $\cat V$-cocomplete $T$-$\infty$-categories to $\cat U$-cocomplete $S$-$\infty$-categories and $\cat V$-cocontinuous $T$-functors to $\cat U$-cocontinuous $S$-functors.
	\begin{proof}
		Let $\Cc$ be a $\cat V$-cocomplete $T$-$\infty$-category. If $B\in S$, $(u\colon A\to B)\in\bbU(S)$, then $u^*\colon (f^*\Cc)(B)\to(f^*\Cc)(A)$ agrees with $(f(u))^*\colon\Cc(f(B))\to\Cc(f(A))$, so it admits a left adjoint by Remark~\ref{rmk:limitExtensionRestrictedClass} and $\cat V$-cocompleteness of $\Cc$. Similarly, given any pullback in $\PSh(S)$ as on the left
		\begin{equation*}
			\begin{tikzcd}
				A'\arrow[r,"\alpha"]\arrow[d,"p'"']\arrow[dr,"\lrcorner"{very near start},phantom] & A\arrow[d, "p"]\\
				B'\arrow[r, "\beta"'] & B
			\end{tikzcd}
			\qquad\qquad
			\begin{tikzcd}
				f(A')\arrow[r,"f(\alpha)"]\arrow[d,"f(p')"']\arrow[dr,"\lrcorner"{very near start},phantom] & f(A)\arrow[d, "f(p)"]\\
				f(B')\arrow[r, "f(\beta)"'] & f(B)
			\end{tikzcd}
		\end{equation*}
		with $B,B'$ representable and $p\in\bbU(B)$, also the diagram on the right is a pullback by assumption, and the Beck-Chevalley map $p'_!\alpha^*\Rightarrow\beta^*p_!$ for $f^*\Cc$ agrees with the Beck-Chevalley map $f(p')_!f(\alpha)^*\Rightarrow f(\beta)^*f(p)_!$ for $\Cc$. In particular, it is an equivalence again, so Lemma~\ref{lem:UColimitsVsAdjointable} shows that $f^*\Cc$ is $\bbU$-cocomplete.

		The statement about cocontinuity follows similarly from the previous lemma.
		\end{proof}
\end{lemma}

It turns out that the parametrized colimits indexed by the constant $T$-$\infty$-categories and the $T$-$\infty$-groupoids already determine all parametrized colimits.

\begin{proposition}[\cite{martiniwolf2021limits}*{Proposition~4.7.1}]\label{prop:charact-T-cc}
	A $T$-$\infty$-category is $T$-cocomplete if and only if it admits fiberwise colimits and $\ul{\Spc}_T$-colimits. A $T$-functor between $T$-cocomplete $T$-$\infty$-categories preserves $T$-colimits if and only if it preserves fiberwise colimits and $\ul{\Spc}_T$-colimits.\qed
\end{proposition}

\begin{corollary}\label{cor:cocontinuity-restr}
	Let $f\colon\PSh(S)\to\PSh(T)$ as in Lemma~\ref{lemma:restriction-preserves-cocompl}. Then the restriction $f^*\colon\Cat_T\to\Cat_S$ sends $T$-cocomplete $T$-$\infty$-categories to $S$-cocomplete $S$-$\infty$-categories and $T$-cocontinuous $T$-functors to $S$-cocontinuous $S$-functors.
	\begin{proof}
		Clearly, $f^*$ preserves fiberwise cocompleteness and cocontinuity. The claim therefore follows from the previous proposition together with Lemma~\ref{lemma:restriction-preserves-cocompl}.
	\end{proof}
\end{corollary}

An important example of a $T$-(co)complete $T$-$\infty$-category is the $T$-$\infty$-category of $T$-spaces.

\begin{example}
	\label{ex:TSpacesPAdjointable}
	The $T$-$\infty$-category $\ul{\Spc}_T$ is both $T$-cocomplete and $T$-complete. Recall from \cref{rmk:sheaf_associated_to_Spc_T} that $\ul{\Spc}_T(B) \simeq \PSh(T)_{/B}$ for every $B \in T$, with functoriality given via pullback in $\PSh(T)$. The functor $f^*\colon \PSh(T)_{/B} \to \PSh(T)_{/A}$ admits a left adjoint given by postcomposition with $f$, and since $\PSh(T)$ is locally cartesian closed it also admits a right adjoint. It follows that $\ul{\Spc}_T$ admits all fiberwise limits and colimits. The left Beck-Chevalley condition is a consequence of the pasting law of pullback squares. The right Beck-Chevalley condition follows from this by passing to total mates.
\end{example}

\begin{example}
	\label{ex:PointedTSpacesPAdjointable}
	It follows directly from \Cref{ex:TSpacesPAdjointable} that also the $T$-$\infty$-categories $\ul{\Spc}_{T,*} $ and $\ul{\Sp}_T$ of pointed $T$-spaces and naive $T$-spectra are both $T$-cocomplete and $T$-complete, since they may be obtained from $\ul{\Spc}_T$ by pointwise tensoring with $\Spc_*$ and $\Sp$ inside $\PrL$, respectively. For later use, we will make the left adjoint functors $p_!$ of $\ul{\Spc}_{T,*} $ explicit. First note that giving a basepoint to an object $(X,f:X \to A) \in \ul{\Spc}_T(A) \simeq \PSh(T)_{/A}$ amounts to providing a section $s\colon A \to X$ of the map $f$, so that we can identify objects of $\ul{\Spc}_{T,*}(A)$ with triples $(X,f,s)$. Given a morphism $p\colon A \to B$ in $\PSh(T)$, we get $p_!(X,f,s) \simeq (X',f',s')$ defined via the following pushout diagram:
	\[\begin{tikzcd}
		A \rar{s} \dar[swap]{p} \drar[pushout] & X \dar[dashed] \rar{f} & A \dar{p} \\
		B \rar[dashed]{s'} & X' \rar[dashed]{f'} & B.
	\end{tikzcd}\]
\end{example}

We end this subsection with a discussion of categories of $T$-cocontinuous functors.

\begin{definition}
	Let $\Cc,\Dd$ be $T$-cocomplete $T$-$\infty$-categories, and let $A\in T$. We write $\ul\Fun_T^\textup{L}(\Cc,\Dd)(A)\subset\ul\Fun_T(\Cc,\Dd)(A)$ for the full subcategory spanned by the $T_{/A}$-cocontinuous functors $\pi_A^*\Cc\to\pi_A^*\Dd$.
\end{definition}

\begin{lemma}\label{lemma:FunL-T-subcat}
	This defines a $T$-subcategory $\ul\Fun_T^\textup{L}(\Cc,\Dd)\subset\ul\Fun_T(\Cc,\Dd)$.
	\begin{proof}
		By Lemma~\ref{lemma:restriction_functor_cat} it suffices to show that for any $f\colon A\to B$ in $T$ and $T_{/B}$-cocontinuous $F\colon\Cc'\to\Dd'$, the restriction $T_{/f}^*F$ is $T_{/A}$-cocontinuous. This follows at once from Corollary~\ref{cor:cocontinuity-restr} as $(T_{/f})_!\colon\PSh(T_{/A})\to\PSh(T_{/B})$ agrees up to equivalence with the pullback preserving functor $\PSh(T)_{/f}\colon\PSh(T)_{/A}\to\PSh(T)_{/B}$.
	\end{proof}
\end{lemma}

We will now give an alternative description in terms of the adjunct functors $F\colon\Cc\to\pi_{A*}\pi_A^*\Dd\simeq\ul\Fun_T(\ul A,\Dd)$, which will in particular allow us to describe the value of $\ul\Fun^\text{L}_T(\Cc,\Dd)$ at non-representable presheaves. For this we will need:

\begin{proposition}[\cite{martiniwolf2021limits}*{Proposition 4.3.1}]
	\label{prop:CoLimitsInFunctorCategories}
	Let $K$ and $\Dd$ be $T$-$\infty$-categories such that $\Dd$ admits all $K$-indexed parametrized limits. Then $\ulFun_T(\Cc,\Dd)$ admits all $K$-indexed limits for any $T$-$\infty$-category $\Cc$. Furthermore, the precomposition functor $i^*\colon \ulFun_T(\Cc',\Dd) \to \ulFun_T(\Cc,\Dd)$ preserves $K$-indexed limits for every $T$-functor $i\colon \Cc \to \Cc'$. The dual statement for colimits is true as well.\qed
\end{proposition}

Combining this with Lemma~\ref{lem:ColimitsIndexedByConstantTCategories} we get:

\begin{corollary}
	Let $\Cc,\Dd$ be $T$-$\infty$-categories.
	\begin{enumerate}
		\item If $\Dd$ is $\bbU$-(co)complete for some $\bbU\subset\ul\Spc_T$, then so is $\ul\Fun_T(\Cc,\Dd)$. Moreover, if $F\colon\Cc\to\Cc'$ is any functor, then $F^*\colon\ul\Fun_T(\Cc',\Dd)\to\ul\Fun_T(\Cc,\Dd)$ is $\bbU$-(co)continuous.
		\item If $\Dd$ is fiberwise (co)complete, then so is $\ul\Fun_T(\Cc,\Dd)$. For any $\Cc\to\Cc'$, the restriction $\ul\Fun_T(\Cc',\Dd)\to\ul\Fun_T(\Cc,\Dd)$ is fiberwise (co)continuous.\qed
	\end{enumerate}
\end{corollary}

\begin{proposition}\label{prop:slice-adj-cc}
	Let $\Cc$ be a $T$-cocomplete $T$-$\infty$-category and let $\Dd$ be a $T_{/A}$-cocomplete $T$-$\infty$-category. Then $\pi_{A*}\Dd$ is $T$-cocomplete, and a functor $F\colon\pi_A^*\Cc\to\Dd$ is $T_{/A}$-cocontinuous if and only if its adjunct $\tilde F\colon\Cc\to\pi_{A*}\Dd$ is $T$-cocontinuous.
	\begin{proof}
		Assume first that $F$ is $T_{/A}$-cocontinuous. Its adjunct $\tilde F$ is then given by
		\begin{equation*}
			\Cc\xrightarrow{\;\eta\;}\pi_{A*}\pi_{A}^*\Cc\xrightarrow{\pi_{A*}F}\pi_{A*}\Dd.
		\end{equation*}
		Applying Corollary~\ref{cor:cocontinuity-restr} to $A\times\blank\colon\PSh(T)\to\PSh(T)_{/A}\simeq\PSh(T_{/A})$, we see that $\pi_{A*}\Dd$ is $T$-cocomplete and $\pi_{A*}F$ is $T$-cocontinuous, so it suffices to show that $\eta$ is $T$-cocontinuous. 
		
		Unravelling definitions, $\eta$ is simply the functor $\Cc\to\Cc(A\times\blank)$ given by restriction along the projections $A\times B\to B$. To show $T$-cocontinuity, it will suffice by Proposition~\ref{prop:charact-T-cc} that $\eta$ is fiberwise cocontinuous and preserves $\ul{\Spc}_T$-colimits. The first statement is clear, while the second one is equivalent by Lemma~\ref{lemma:BC-vs-cocontinuity} to demanding that for every map $p\colon B'\to B$ in $\PSh(T)$ with target in $T$ the Beck-Chevalley map $p_!\eta\Rightarrow\eta p_!$ be an equivalence. However, this is precisely the Beck-Chevalley map $(A\times p)_!\circ{\pr^*}\Rightarrow{\pr^*}\circ p_!$ associated to the pullback
		\begin{equation*}
			\begin{tikzcd}
				A\times B'\arrow[r, "\pr"]\arrow[dr,phantom,"\lrcorner"{very near start}]\arrow[d, "A\times p"'] & B'\arrow[d, "p"]\\
				A\times B\arrow[r, "\pr"'] & B
			\end{tikzcd}
		\end{equation*}
		and hence indeed an equivalence by the characterization of $T$-cocompleteness given in Lemma~\ref{lem:UColimitsVsAdjointable}.

		Conversely, assume $\tilde F$ is $T$-cocontinuous. Then $F$ factors as
		\begin{equation*}
			\pi_A^*\Cc\xrightarrow{\pi_A^*\tilde F}\pi_A^*\pi_{A*}\Dd\xrightarrow{\,\epsilon\,}\Dd,
		\end{equation*}
		where the first functor is $T_{/A}$-cocontinuous by Corollary~\ref{cor:cocontinuity-restr} (or in fact, simply by definition). Similarly to the above, the counit is given by restriction along the unit maps $B\to A\times\pi_A(B)$, and the claim follows by observing that we also have pullbacks
		\begin{equation*}
			\begin{tikzcd}
				B'\arrow[d, "p"']\arrow[dr,phantom,"\lrcorner"{very near start}] \arrow[r, "\eta"] & A\times \pi_A(B')\arrow[d, "A\times\pi_A(p)"]\\
				B\arrow[r, "\eta"'] & A\times \pi_A(B)
			\end{tikzcd}
		\end{equation*}
		in $\PSh(T_{/A})$ for any $p\colon B'\to B$ in $T_{/A}$.
	\end{proof}
\end{proposition}

\begin{remark}
	Analogously one sees that for a class $\bbU$ of $T$-$\infty$-groupoids a functor $F\colon\pi_A^*\Cc\to\Dd$ is $\pi_A^*\bbU$-cocontinuous if and only if its adjunct is $\bbU$-cocontinuous.
\end{remark}

\begin{proposition}\label{prop:adjunct-cocont}
	Let $\Cc,\Dd$ be $T$-cocomplete $T$-$\infty$-categories, let $X\in\PSh(T)$, and let $(F\colon\Cc\to\ul\Fun_T(\ul X,\Dd))\in\ul\Fun_T(\Cc,\Dd)(X)$. Then $F$ belongs to $\ul\Fun^\textup{L}_T(\Cc,\Dd)(X)$ if and only if $F$ is $T$-cocontinuous.
	\begin{proof}
		If $X$ is representable, this is immediate from Proposition~\ref{prop:slice-adj-cc}. It therefore suffices to show that for general $X$ a functor $F\colon\Cc\to\ul\Fun_T(\ul X,\Dd)$ is $T$-cocontinuous if and only if for every map $A\to X$ from a representable the composite $\Cc\to\ul\Fun_T(\ul A,\Dd)$ is $T$-cocontinuous.

		The `only if' part is immediate from Proposition~\ref{prop:CoLimitsInFunctorCategories}. For the converse we observe that the functors $\ul\Fun_T(\ul X,\Dd)\to\ul\Fun_T(\ul A,\Dd)$ exhibit the left hand side as a limit, and so are jointly conservative. It follows immediately that $F$ is fiberwise cocontinuous. For the Beck-Chevalley condition we now let $p\colon C\to D$ be any map in $\PSh(T)$ with target in $T$ and consider
		\begin{equation*}
			\begin{tikzcd}
				\Cc(D)\arrow[d, "p^*"'] \arrow[r, "F"] & \ul\Fun_T(\ul X,\Dd)(D)\arrow[d, "p^*"{description}]\arrow[r] & \ul\Fun_T(\ul A,\Dd)(D)\arrow[d, "p^*"]\\
				\Cc(C)\arrow[r, "F"'] & \ul\Fun_T(\ul X,\Dd)(C)\arrow[r] & \ul\Fun_T(\ul A,\Dd)(C)\rlap.
			\end{tikzcd}
		\end{equation*}
		By cocontinuity of the restriction functors, the mate of the right hand square is an equivalence, and so is the mate of the total rectangle by assumption. By the compatibility of mates with pasting we see that the mate of the left hand square is an equivalence after postcomposing with $\ul\Fun_T(\ul X,\Dd)(C)\to\ul\Fun_T(\ul A,\Dd)(C)$; the claim follows from joint conservativity again.
	\end{proof}
\end{proposition}

\subsection{Presentable \for{toc}{$T$-$\infty$}\except{toc}{\texorpdfstring{$\bm T$-$\bm\infty$}{T-∞}}-categories}
\label{subsec:presentableTcategories}
For the statement of various universal properties we need to restrict to \textit{presentable} $T$-$\infty$-categories. The notion of parametrized presentability was introduced by Nardin \cite{nardin2017thesis} and was subsequently further developed by Hilman \cite{hilman2022parametrised} in the case where the $\infty$-category $T$ is \textit{orbital} (in the sense of \Cref{def:OrbitalSubcategory} below). A more general theory of parametrized presentability which works for arbitrary $T$ was developed by Martini and Wolf \cite{martiniwolf2022presentable} in terms of internal higher category theory. In this subsection, we will recall the main results on parametrized presentability.

\begin{definition}
	\label{def:presentableTCategory}
	A $T$-$\infty$-category $\Cc$ is called \textit{presentable} if the following two conditions hold:
	\begin{enumerate}[(1)]
		\item $\Cc$ is fiberwise presentable, meaning that the functor $\Cc\colon T\catop \to \Cat_{\infty}$ factors (necessarily uniquely) through $\PrL$;
		\item $\Cc$ is $T$-cocomplete.
	\end{enumerate}
	Observe that fiberwise presentability guarantees that $\Cc$ has fiberwise colimits, so that condition (2) holds if and only if $\Cc$ admits $\ul{\Spc}_T$-indexed colimits.
\end{definition}

By \cite{martiniwolf2022presentable}*{Theorem~A}, this definition agrees with the definition of \cite{martiniwolf2022presentable}*{Section~2.4} applied to the $\infty$-topos $\PSh(T)$. When $T$ is orbital, this definition agrees with that of \cite{hilman2022parametrised}*{Section~4}.

\begin{remark}
	\label{rmk:PresentableHasLimits}
	Any presentable $T$-$\infty$-category $\Cc$ is automatically $T$-complete: fiberwise completeness and the existence of right adjoints $f_*\colon \Cc(A) \to \Cc(B)$ follow from fiberwise presentability, and for every pullback square of the form \eqref{eq:PullbackSquareTCocompleteness}, the Beck-Chevalley map $\beta^* \circ p_* \Rightarrow p'_* \circ \alpha^*$ is the total mate of the Beck-Chevalley map $\alpha_! \circ {p'}^* \Rightarrow p^* \circ \beta_!$ and thus an equivalence.
\end{remark}

\begin{definition}
	We define $\PrLT$ to be the (non-full) subcategory of $\Cat_T$ spanned by the presentable $T$-$\infty$-categories and left adjoint $T$-functors between them. Similarly we define $\PrRT$ to be the (non-full) subcategory of $\Cat_T$ spanned by the presentable $T$-$\infty$-categories and right adjoint $T$-functors between them. There is a canonical equivalence $\PrLT \simeq (\PrRT)\catop$, see \cite{martiniwolf2022presentable}*{Proposition~2.4.4.7}.
\end{definition}

\begin{example}
	\label{ex:T_Spaces_Presentable}
	The $T$-$\infty$-category $\ul{\Spc}_T$ of $T$-spaces is presentable: fiberwise presentability follows from presentability of $\PSh(T)$ while $T$-cocompleteness was argued for in \Cref{ex:TSpacesPAdjointable}.
\end{example}

\begin{example}
	Let $K$ be a small $T$-$\infty$-category and let $\Cc$ be a presentable $T$-$\infty$-category. Then the functor $T$-$\infty$-category $\ulFun_T(K,\Cc)$ is again presentable \cite{martiniwolf2022presentable}*{Corollary~2.4.2.7}, \cite{hilman2022parametrised}*{Lemma~4.6.1}.
\end{example}

\begin{example}\label{ex:accessible-Bousfield-presentable}
	Accessible Bousfield localizations of presentable $T$-$\infty$-category are again presentable.

	In more detail, let $\Cc$ be a presentable $T$-$\infty$-category and let $S$ be a parametrized family of morphisms in $\Cc$, i.e.\ a specification of a set $S(B)$ of morphisms of $\Cc(B)$ for every $B \in T$ such that $f^*(u) \in S(A)$ for every $u\in S(B)$ and every morphism $f\colon A\rightarrow B$ in $T$. An object $X \in \Cc(B)$ is said to be \textit{$S$-local} if for every morphism $f\colon A \to B$ in $T$ the object $f^*X \in \Cc(A)$ is $S(A)$-local, meaning that for every morphism $u\colon Y \to Z$ in $S(A)$ the induced map of spaces $\Hom_{\Cc(A)}(Z,f^*X) \to \Hom_{\Cc(A)}(Y,f^*X)$ is an equivalence. We let $\mathrm{Loc}_S(\Cc) \subseteq \Cc$ denote the full subcategory spanned by the $S$-local objects.

	By \cite{martiniwolf2022presentable}*{Proposition~2.4.1.7, Corollary~2.4.2.9} the $T$-$\infty$-category $\mathrm{Loc}_S(\Cc)$ is again presentable and the inclusion $\mathrm{Loc}_S(\Cc) \subset \Cc$ admits a left adjoint.
\end{example}

\begin{remark}
	It follows from the previous three examples that the subcategory of $S$-local objects of a $T$-$\infty$-category of $T$-presheaves $\ul\PSh_T(K) := \ulFun_T(K\catop,\ul{\Spc}_T)$ is presentable whenever $S$ is a parametrized family of morphisms in $\ul\PSh_T(K)$. Conversely, any presentable $T$-$\infty$-category is of this form, see \cite{martiniwolf2022presentable}*{Theorem~A}, \cite{hilman2022parametrised}*{Theorem~4.1.2}
\end{remark}

\begin{proposition}[Adjoint functor theorem, {\cite{martiniwolf2022presentable}*{Proposition~2.4.3.1}}]
	If $\Cc$ and $\Dd$ are large $T$-$\infty$-categories such that $\Cc$ is presentable and $\Dd$ is locally small, a $T$-functor $\Cc \to \Dd$ preserves $T$-colimits if and only if it admits a right adjoint. \qednow
\end{proposition}

Given a small $T$-$\infty$-category $K$, the $T$-$\infty$-category $\ul\PSh_T(K)$ is freely generated under parametrized colimits by $K$:

\begin{theorem}[\cite{martiniwolf2021limits}*{Theorem~7.1.1}]
	\label{prop:Universal_Property_Presheaves}
	Let $K$ be a small $T$-$\infty$-category and let $\Dd$ be a $T$-cocomplete $T$-$\infty$-category. Then restriction along the Yoneda embedding $y\colon K \hookrightarrow \ul\PSh_T(K)$ (Remark~\ref{rk:YonedaEmbedding}) induces an equivalence of $T$-$\infty$-categories
	\[
	\ulFun_T^{\textup{L}}(\ul\PSh_T(K),\Dd) \iso \ulFun_T(K,\Dd).\qednow
	\]
\end{theorem}

\begin{remark}\label{rk:ra-Yoneda-extension}
	Let $A\in T$ and let $f\colon \pi_A^*K\to \pi_A^*\Dd$ define an element of $\ulFun_T(K,\Dd)(A)$, which by the theorem then extends to a left adjoint $T$-functor $F\colon \pi_A^*\ul \PSh_T(K)\to\pi_A^*\Dd$.
	As in the classical non-parametrized situation, the \emph{right adjoint} $G$ of $F$ is actually easy to describe \cite{martiniwolf2021limits}*{Remark~7.1.4}: it is given by the composition
	\begin{equation*}
		\pi_A^*\Dd \xrightarrow{y}\ul\Fun_{T_{/A}}(\pi_A^*\Dd^\op, \ul\Spc_{T_{/A}})\xrightarrow{f^*}\ul\Fun_{T_{/A}}(\pi_A^*K,\ul\Spc_{T_{/A}})\simeq\pi_A^*\ul\PSh_T(K),
	\end{equation*}
	where the unlabelled equivalence on the right is the one from Corollary~\ref{cor:TCategoryOfTFunctors}.
\end{remark}

Applying the theorem to the case where $K$ is the terminal $T$-$\infty$-category $\ul{1}$, we see that the $T$-$\infty$-category $\ul{\Spc}_T$ is the free $T$-cocomplete $T$-$\infty$-category on a single generator:

\begin{corollary}
	\label{cor:Universal_Property_T_Spaces}
	Let $\Dd$ be a $T$-cocomplete $T$-$\infty$-category. Then evaluation at the terminal object $1 \in \PSh(T) = \Gamma(\ul{\Spc}_T)$ induces an equivalence of $T$-$\infty$-categories
	\[
	\ulFun_T^{\textup{L}}(\ul{\Spc}_T,\Dd) \iso \Dd.\qednow
	\]
\end{corollary}

\section{The universal property of global spaces}
In this section we will give a parametrized interpretation of unstable global homotopy theory in the sense of \cite{schwede2018global}*{Chapter 1} with respect to \emph{finite} groups. For this, the key idea will be to more generally consider unstable \emph{$G$-global homotopy theory} in the sense of \cite{g-global}*{Chapter~1} for finite groups $G$, which we recall in Subsection~\ref{subsec:reminder-global-spaces} below. In \ref{subsec:GlobalCatOfGlobalSpaces} we will then explain how these models for varying $G$ assemble into a global $\infty$-category $\ul\S^\text{gl}$ (in the sense of Example~\ref{ex:globalCategory}), and in Subsection~\ref{subsec:UniversalPropertyGlobalSpaces} we will finally provide a universal description of $\ul\S^\text{gl}$ as the free cocomplete global $\infty$-category generated by the terminal object.

\subsection{A reminder on global and \texorpdfstring{\textit{G}}{G}-global homotopy theory}\label{subsec:reminder-global-spaces}
Let $G$ be a finite group; \cite{g-global}*{Chapter 1} studies various models of \emph{unstable $G$-global homotopy theory}. We will recall two of these models that will be particularly convenient for us:

\begin{definition}
	We write $\mathcal M$ for the monoid (under composition) of injective self-maps of the countably infinite set $\omega\mathrel{:=}\{0,1,\dots\}$.
\end{definition}

The functor $\cat{SSet}\to\cat{Set},X\mapsto X_0$ sending a simplicial set to its set of vertices admits a right adjoint $E$, given explicitly by $(EX)_n=X^{1+n}$ with functoriality induced by the identification $X^{1+n}\cong \Hom(\{0,\dots,n\},X)$; equivalently, this is the nerve of the groupoid with objects $X$ and a unique map between any two objects. As a right adjoint, $E$ in particular preserves products, so $E\mathcal M$ inherits a natural monoid structure from $\mathcal M$.

We occasionally call the resulting simplicial monoid $E\mathcal M$ the `universal finite group.' While $E\mathcal M$ is of course neither finite nor a group, this terminology is motivated by the fact that we can embed any finite group into $E\mathcal M$ in a particularly nice way:

\begin{definition}
	Let $H$ be a finite group. A countable $H$-set $\mathcal U$ is called a \emph{complete $H$-set universe} if every other countable $H$-set embeds equivariantly into $\mathcal U$.
\end{definition}

\begin{definition}
	A finite subgroup $H\subset\mathcal M$ is called \emph{universal} if the tautological $H$-action on $\omega$ makes the latter into a complete $H$-set universe.
\end{definition}

\begin{lemma}[See \cite{g-global}*{Lemma~1.2.8}]\label{lemma:universal-embeddings}
	Let $H$ be a finite group. Then there exists an injective homomorphism $i\colon H\to\mathcal M$ with universal image. If $j\colon H\to\mathcal M$ is another such map, then there exists an invertible $\phi\in\mathcal M$ such that $i(h)=\phi j(h)\phi^{-1}$ for all $h\in H$.\qed
\end{lemma}

\begin{remark}
	Somewhat loosely speaking, the reason to pass from the discrete monoid $\mathcal M$ to the simplicial monoid $E\mathcal M$ is to eliminate the indeterminacy of the invertible element $\phi$ in the above lemma, see~\cite{g-global}*{Subsections 1.2.2--1.2.3} for more details.
\end{remark}

\begin{definition}
	Let $G$ be any group. We write $\cat{$\bm{E\mathcal M}$-$\bm G$-SSet}$ for the $1$-category (or simplicially enriched category) of simplicial sets with a strict action of the simplicial monoid $E\mathcal M\times G$, together with the strictly $(E\mathcal M\times G)$-equivariant maps.
\end{definition}

The category $\cat{$\bm{E\mathcal M}$-$\bm G$-SSet}$ will be our first model for $G$-global homotopy theory. In order to define the weak equivalences of this model structure we recall the following notation:

\begin{notation}
	Let $G_1,G_2$ be groups, let $H\subset G_1$, and let $\phi\colon H\to G_2$ be a homomorphism. The \emph{graph subgroup} $\Gamma_{H,\phi}\subset G_1\times G_2$ is the subgroup $\{(h,\phi(h)): h\in H\}$. If $X$ is a $(G_1\times G_2)$-simplicial set, then we abbreviate $X^\phi\mathrel{:=}X^{\Gamma_{H,\phi}}$, and similarly for $(G_1\times G_2)$-equivariant maps.
\end{notation}

\begin{proposition}
	The category $\cat{$\bm{E\mathcal M}$-$\bm G$-SSet}$ carries a (unique) combinatorial model structure in which a map is a weak equivalence or fibration if and only if $f^\phi$ is a weak homotopy equivalence or Kan fibration, respectively, for every universal subgroup $H\subset\mathcal M$ and homomorphism $\phi\colon H\to G$. We call this the \emph{$G$-global model structure} and its weak equivalences the \emph{$G$-global weak equivalences}.

	Moreover, there is also a unique model structure on $\cat{$\bm{E\mathcal M}$-$\bm G$-SSet}$ whose weak equivalences are the $G$-global weak equivalences and whose cofibrations are the \emph{injective cofibrations}, i.e.~the levelwise injections. We call this the \emph{injective $G$-global model structure}.
	\begin{proof}
		These are special cases of~\cite{g-global}*{Propositions~1.1.2 and~1.1.15}, respectively; also see Corollary~1.2.34 of \emph{op.~cit.} for the former model structure.
	\end{proof}
\end{proposition}

For $G=1$ the above recovers a version of Schwede's global homotopy theory where one only considers equivariant information for finite groups (`$\mathcal F\!in$-global homotopy theory'), see Remark~\ref{rk:g-global-vs-global-unstable} below. On the other hand, for general finite $G$ one can exhibit ordinary $G$-equivariant homotopy theory explicitly as a Bousfield localization of $G$-global homotopy theory, see~\cite{g-global}*{Subsection~1.2.6}. In this sense, $G$-global homotopy theory can be thought of as a `synthesis' of the usual equivariant and global approaches.

\begin{lemma}[See~\cite{g-global}*{Corollaries~1.2.76--1.2.79}]\label{lemma:alpha-star-EM}
	Let $\alpha\colon G\to G'$ be any group homomorphism. Then the restriction functor $\alpha^*\colon \cat{$\bm{E\mathcal M}$-$\bm{G'}$-SSet}\to \cat{$\bm{E\mathcal M}$-$\bm{G}$-SSet}$ is homotopical and it takes part in Quillen adjunctions
	\begin{align*}
		\alpha_!\colon\cat{$\bm{E\mathcal M}$-$\bm{G}$-SSet}_\textup{$G$-gl}&\rightleftarrows\cat{$\bm{E\mathcal M}$-$\bm{G'}$-SSet}_\textup{$G'$-gl}\noloc\alpha^*\\
		\alpha^*\colon\cat{$\bm{E\mathcal M}$-$\bm{G'}$-SSet}_\textup{inj.~$G'$-gl}&\rightleftarrows\cat{$\bm{E\mathcal M}$-$\bm{G}$-SSet}_\textup{inj.~$G$-gl}\noloc \alpha_*.
	\end{align*}
	Moreover, if $\alpha$ is \emph{injective}, then we also have Quillen adjunctions
	\begin{align*}
		\alpha_!\colon\cat{$\bm{E\mathcal M}$-$\bm{G}$-SSet}_\textup{inj.~$G$-gl}&\rightleftarrows\cat{$\bm{E\mathcal M}$-$\bm{G'}$-SSet}_\textup{inj.~$G'$-gl}\noloc \alpha^*\\
		\alpha^*\colon\cat{$\bm{E\mathcal M}$-$\bm{G'}$-SSet}_\textup{$G'$-gl}&\rightleftarrows\cat{$\bm{E\mathcal M}$-$\bm{G}$-SSet}_\textup{$G$-gl}\noloc \alpha_*.\qednow
	\end{align*}
\end{lemma}

Next, we come to another model in terms of suitable `diagram spaces' that will become useful later to relate the unstable and stable theory to each other:

\begin{definition}
	We write $I$ for the category of finite sets and injections. Moreover, we write $\mathcal I$ for the simplicially enriched category obtained by applying $E\colon\cat{Set}\to\cat{SSet}$ to all hom-sets.

	We write $\cat{$\bm{\mathcal I}$-SSet}$ for the category $\FUN(\mathcal I,\cat{SSet})$ of simplicially enriched functors $\mathcal I\to\cat{SSet}$. Moreover, if $G$ is any group, then we write $\cat{$\bm G$-$\bm{\mathcal I}$-SSet}$ for the category of $G$-objects in $\cat{$\bm{\mathcal I}$-SSet}$.
\end{definition}

\begin{construction}
	Let $X$ be any $\mathcal I$-simplicial set. Then we define
	\begin{equation*}
		X(\omega)\mathrel{:=}\mathop{\text{colim}}\limits_{\substack{A\subset\omega \\ \text{finite}}} X(A).
	\end{equation*}
	This admits an $E\mathcal M$-action via the original functoriality of $X$ in $\mathcal I$, see~\cite{g-global}*{Construction~1.4.14} for details, giving rise to a functor $\ev_\omega\colon\cat{$\bm{\mathcal I}$-SSet}\to\cat{$\bm{E\mathcal M}$-SSet}$. If $G$ is any group, then we obtain a functor $\ev_\omega\colon\cat{$\bm G$-$\bm{\mathcal I}$-SSet}\to\cat{$\bm{E\mathcal M}$-$\bm G$-SSet}$ by pulling through the $G$-actions.
\end{construction}

\begin{theorem}[See~\cite{g-global}*{Proposition~1.4.3 and Theorem~1.4.30}]\label{thm:global-model-I}
	There is a unique model structure on $\cat{$\bm G$-$\bm{\mathcal I}$-SSet}$ with
	\begin{itemize}
		\item weak equivalences those maps $f$ for which $\ev_\omega f\mathrel{=:}f(\omega)$ is a $G$-global weak equivalence, \emph{and}
		\item acyclic fibrations those maps $f$ for which $f(A)^\phi$ is an acyclic Kan fibration for every finite set $A$, $H\subset\Sigma_A$, and $\phi\colon H\to G$.
	\end{itemize}
	We call this the \emph{$G$-global model structure} and its weak equivalences the \emph{$G$-global weak equivalences} again.

	Moreover, the functor $\ev_\omega$ is the left half of a Quillen equivalence $\cat{$\bm G$-$\bm{\mathcal I}$-SSet}\rightleftarrows\cat{$\bm{E\mathcal M}$-$\bm G$-SSet}_\textup{inj.~$G$-gl}$.\qed
\end{theorem}

\begin{remark}
	One can also define a $G$-global model structure on the category $\cat{$\bm G$-$\bm I$-SSet}$ (whose weak equivalences are somewhat intricate). The forgetful functor $\cat{$\bm G$-$\bm{\mathcal I}$-SSet}\to\cat{$\bm G$-$\bm{I}$-SSet}$ is then the right half of a Quillen equivalence, see~\cite{g-global}*{Theorem~1.4.31}.
\end{remark}

\begin{remark}\label{rk:g-global-vs-global-unstable}
	Schwede \cite{schwede2018global}*{Theorem~1.2.21} originally studied unstable global homotopy theory in terms of so-called \emph{orthogonal spaces}, which are topologically enriched functors from the topological category $L$ of finite dimensional inner product spaces and linear isometric embeddings into $\cat{Top}$. While Schwede's \emph{global equivalences} on $\cat{$\bm L$-Top}$ see equivariant information for all compact Lie groups, there is a natural notion of `$\mathcal F\!in$-global weak equivalences' \cite{g-global}*{Definition~1.5.13}, and with respect to these the evident forgetful functor $\cat{$\bm L$-Top}\to\cat{$\bm I$-SSet}$ becomes an equivalence of homotopy theories, see~\cite{g-global}*{Corollary~1.5.29}. In this sense, the above two models generalize global homotopy theory with respect to \emph{finite} groups.
\end{remark}

Finally, we again have suitable restriction functoriality analogous to Lemma~\ref{lemma:alpha-star-EM}. We will only recall one aspect that we will need later:

\begin{lemma}[See~\cite{g-global}*{Lemma 1.4.40}]\label{lemma:alpha-star-I}
	Let $\alpha\colon G\to G'$ be any group homomorphism. Then the adjunction
	\begin{equation*}
		\alpha_!\colon\cat{$\bm G$-$\bm{\mathcal I}$-SSet}\rightleftarrows\cat{$\bm{G'}$-$\bm{\mathcal I}$-SSet} :\!\alpha^*
	\end{equation*}
	is a Quillen adjunction with homotopical right adjoint.\qed
\end{lemma}

\subsection{The global \except{toc}{\texorpdfstring{$\bm\infty$}{∞}}\for{toc}{$\infty$}-category of global spaces} \label{subsec:GlobalCatOfGlobalSpaces}
We will now bundle the $\infty$-categories associated to the above model categories into a \emph{global $\infty$-category}, i.e.~an $\infty$-category parametrized over the $\infty$-category $\Glo$ from Example~\ref{ex:globalCategory}:

\begin{construction}
	We define the strict $2$-functor $\cat{$\bm{E\mathcal M}$-$\bm\bullet$-SSet}$ as the composition
	\begin{equation}\label{eq:G-to-G-gl-spaces}
		\sGlo^\op\xhookrightarrow{B}\cat{Grpd}^\op\xrightarrow{\Fun(\blank,\cat{$\bm{E\mathcal M}$-SSet})} \cat{Cat};
	\end{equation}
	put differently, this sends a finite group $G$ to the $1$-category $\cat{$\bm{E\mathcal M}$-$\bm G$-SSet}$, a homomorphism $\alpha\colon G\to G'$ to the restriction map $\alpha^*\colon \cat{$\bm{E\mathcal M}$-$\bm{G'}$-SSet}\to \cat{$\bm{E\mathcal M}$-$\bm G$-SSet}$, and a $2$-cell $g'\colon \alpha\Rightarrow\beta$ in $\Glo$ to the transformation $\alpha^*\Rightarrow\beta^*$ given by acting with $g'$.
\end{construction}

We now want to obtain a \emph{global $\infty$-category of global spaces} by pointwise localizing at the $G$-global weak equivalences. To this end we recall:

\begin{definition}
	A \textit{relative category} is a $1$-category $\Cc$ together with a wide subcategory $W \subseteq \Cc$, whose morphisms we call \textit{weak equivalences}. We let $\cat{RelCat}$ denote the $(2, 1)$-category of relative categories, weak equivalence preserving functors, and natural isomorphisms, and we write $\RelCat$ for its Duskin nerve.
\end{definition}

By Lemma~\ref{lemma:alpha-star-EM}, the restriction functor $\alpha^*\colon \cat{$\bm{E\mathcal M}$-$\bm{G'}$-SSet} \to \cat{$\bm{E\mathcal M}$-$\bm{G}$-SSet}$ sends $G'$-global weak equivalences to $G$-global weak equivalences for any homomorphism $\alpha\colon G \to G'$. In particular, $(\ref{eq:G-to-G-gl-spaces})$ lifts to a $2$-functor into $\cat{RelCat}$ this way.

\begin{construction}
	To every relative category $(\Cc,W)$, one can associate an $\infty$-category $\Cc[W^{-1}]$ together with a functor $\Cc \to \Cc[W^{-1}]$ that exhibits it as a Dwyer-Kan localization of $\Cc$ at $W$ in the sense of \cite{HA}*{Definition~1.3.4.1}. We will now recall the argument of \cite{gepnermeier2020equivTMF}*{Section~C.1} that the $\infty$-category $\Cc[W^{-1}]$ is in fact functorial in the pair $(\Cc,W)$.

	Let $\core\colon \Cat_{\infty} \to \Spc$ denote the right adjoint to the inclusion $\Spc \subseteq \Cat_{\infty}$ of $\infty$-groupoids into $\infty$-categories. Sending an $\infty$-category $\Cc$ to the adjunction counit $\core\Cc \hookrightarrow \Cc$ refines to a functor
	\begin{align*}
		R\colon \Cat_{\infty} \to \Fun(\Delta^1,\Cat_{\infty}).
	\end{align*}
	We let $L_{\infty}\colon \Fun(\Delta^1,\Cat_{\infty}) \to \Cat_{\infty}$ denote a left adjoint to this functor. By associating to a relative category $(\Cc,W)$ the inclusion $W \hookrightarrow \Cc$ and regarding both $W$ and $\Cc$ as $\infty$-categories via their nerve, we obtain a functor $\RelCat \to \Fun(\Delta^1,\Cat_\infty)$. In particular we obtain a localization functor
	\begin{align*}
		L\colon \RelCat \to \Fun(\Delta^1,\Cat_{\infty}) \xrightarrow{L_{\infty}} \Cat_{\infty}.
	\end{align*}
	It follows directly from the definition of $L_{\infty}$ that $L$ is on objects given by sending a relative category $(\Cc,W)$ to the Dwyer-Kan localization $L(\Cc,W) \simeq \Cc[W^{-1}]$.
\end{construction}

Postcomposing with this, we get a global $\infty$-category $L{\mathscr C}$ from any global relative category ${\mathscr C}$, and this comes with a global functor ${\mathscr C}\to L{\mathscr C}$ that is pointwise a Dwyer-Kan localization. By uniqueness of adjoints, this actually pins down $L{\mathscr C}$ up to essentially unique equivalence; in particular, we can (and will at times) freely choose a specific construction of the above localization for a given ${\mathscr C}$.

\begin{definition}
	\label{def:GloCatOfGlobalSpaces}
	We define the global $\infty$-category $\ul{\S}^\text{gl}$ of \textit{global spaces} as the composite
	\begin{align*}
		\Glo^\op =\nerve_\Delta(\sGlo)^\op\xrightarrow{\nerve_\Delta(\cat{$\bm{E\mathcal M}$-$\bm\bullet$-SSet})}\nerve_\Delta(\cat{RelCat})= \RelCat \xrightarrow{L} \Cat_{\infty}.
	\end{align*}
	In particular, for a finite group $G$ the $\infty$-category $\ul{\S}^{\text{gl}}(G)\mathrel{=:}\S_G^\text{gl}$ is the $\infty$-category of $G$-global spaces and for a group homomorphism $\alpha\colon G \to G'$, the functor $\ul{\S}^{\text{gl}}(\alpha)$ is induced by the restriction functor $\alpha^*\colon \cat{$\bm{E\mathcal M}$-$\bm{G'}$-SSet}\to \cat{$\bm{E\mathcal M}$-$\bm G$-SSet}$.

	Analogously, we get a global $\infty$-category $\ul\S_{\mathcal I}^\text{gl}$ sending $G$ to the Dwyer-Kan localization of $\cat{$\bm G$-$\bm{\mathcal I}$-SSet}$, with functoriality via restrictions.
\end{definition}

By design, the maps $\ev_\omega$ are homotopical and strictly compatible with restrictions, and so they assemble into a strictly $2$-natural transformation between functors $\sGlo^\op\to\cat{RelCat}$. Upon localization, we therefore get a global functor $\ul\S^\text{gl}_{\mathcal I}\to\ul\S^\text{gl}$ that we again call $\ev_\omega$. Theorem~\ref{thm:global-model-I} then implies:

\begin{corollary}\label{cor:EM-vs-I}
	The global functor $\ev_\omega\colon\ul\S^\textup{gl}_{\mathcal I}\to\ul\S^\textup{gl}$ is an equivalence of global $\infty$-categories.\qed
\end{corollary}

\subsection{Proof of Theorem~\ref{introthm:universal-prop-spaces}} \label{subsec:UniversalPropertyGlobalSpaces}
As a basis for the universal properties of special global $\Gamma$-spaces and global spectra, we will now relate the global $\infty$-category $\ul\S^\textup{gl}$ (defined above in terms of a purely model categorical construction) to the global $\infty$-category $\ul\Spc_{\Glo}$ (constructed using parametrized higher category theory alone). Namely we will prove:

\begin{theorem}\label{thm:global-spaces-comparison}
	The global $\infty$-category $\ul\S^\textup{gl}$ is presentable. Moreover, the essentially unique globally cocontinuous functor  $\ul\Spc_{\Glo}\to\ul\S^\textup{gl}$ that sends the terminal object of $\Spc_{\Glo}(1)$ to the terminal object of $\S^\textup{gl}=\S^\textup{gl}_1$ is an equivalence.
\end{theorem}

Together with \Cref{cor:Universal_Property_T_Spaces} this will then immediately imply Theorem~\ref{introthm:universal-prop-spaces} from the introduction:

\begin{theorem}\label{thm:global-spaces}
	The presentable global $\infty$-category $\ul\S^\textup{gl}$ is freely generated under global colimits by $*\in\ul\S^\textup{gl}$, i.e.~for any globally cocomplete global $\infty$-category $\mathcal D$ evaluating at $*$ induces an equivalence
	\begin{equation*}
		\ul{\Fun}^{\textup{L}}_{\Glo}(\ul\S^\textup{gl},\mathcal D)\to\mathcal D
	\end{equation*}
	of global $\infty$-categories.\qed
\end{theorem}

Corollary~\ref{cor:EM-vs-I} then shows:

\begin{corollary}\label{cor:SglI-univ-prop}
	The global $\infty$-category $\ul\S_{\mathcal I}^\textup{gl}$ is presentable, and it is freely generated under global colimits by $*\in\ul\S^\textup{gl}_{\mathcal I}$, i.e.~for any globally cocomplete global $\infty$-category $\mathcal D$ evaluating at $*$ induces an equivalence
	\begin{equation*}
		\ul{\Fun}^{\textup{L}}_{\Glo}(\ul\S_{\mathcal I}^\textup{gl},\mathcal D)\to\mathcal D
	\end{equation*}
	of global $\infty$-categories.\qed
\end{corollary}

The way Theorem~\ref{thm:global-spaces-comparison} is phrased naturally suggests a proof strategy: show that the (fiberwise presentable) global $\infty$-category $\ul\S^\text{gl}$ is globally cocomplete, use the universal property to construct the map, and then check that it is an equivalence. In fact, one can use the functoriality properties of Lemma~\ref{lemma:alpha-star-EM} together with \cite{g-global}*{Proposition~1.1.22} to verify global cocompleteness, and it is not hard to show using some adjunction yoga that the resulting functor sends corepresented objects to the standard `generators' of $G$-global homotopy theory (see Proposition~\ref{prop:elmendorf} below) while a concrete computation reveals that the mapping spaces on both sides are \emph{abstractly} equivalent. However, proving that actually the universal functor induces equivalences between these mapping spaces is a totally different story, and in fact the authors do not know a direct argument for this.

Instead, our proof of the theorem will proceed backwards: we will construct an equivalence between $\ul\S^\textup{gl}$ and $\ul\Spc_{\Glo}$ by hand, and deduce the remaining statements from this. Since this comparison is somewhat lengthy, let us outline the general strategy first: by definition, $\ul\Spc_{\Glo}$ is levelwise given by $\infty$-categories of presheaves, and the first step will be to likewise express the levels of $\ul\S^\textup{gl}$ in terms of \emph{model categories} of presheaves. To complete the proof, we will then give a comparison between the indexing categories on both sides, as well as a comparison between presheaves in the model categorical and $\infty$-categorical setting.

\subsubsection{The $G$-global Elmendorf Theorem} Recall that the classical \emph{Elmendorf Theorem} \cite{elmendorf} expresses the homotopy theory of $G$-CW-complexes in terms of \emph{fixed point systems}, yielding a presheaf model of unstable $G$-equivariant homotopy theory. We will now recall a $G$-global version of this, which is most easily formulated using the model of $E\mathcal M$-$G$-simplicial sets:

\begin{construction}
	Let $G$ be finite. We write $\Ogl_G$ for the full simplicial subcategory of $\cat{$\bm{E\mathcal M}$-$\bm G$-SSet}$ spanned by the objects $E\mathcal M\times_\phi G\mathrel{:=}(E\mathcal M\times G)/H$ for all universal subgroups $H\subset\mathcal M$ and homomorphisms $\phi\colon H\to G$, where $H$ acts on $E\mathcal M$ from the right in the evident way and on $G$ from the right via $\phi$.

	We now define a functor \[\Phi\colon\cat{$\bm{E\mathcal M}$-$\bm G$-SSet}\to\FUN((\Ogl_G)^\op,\cat{SSet}),\] where $\FUN$ denotes the $1$-category of simplicially enriched functors, via the formula $\Phi(X)(E\mathcal M\times_\phi G)=\maps(E\mathcal M\times_\phi G, X)$ with the evident (enriched) functoriality in each variable, i.e.~$\Phi$ is the composition
	\begin{equation*}
		\cat{$\bm{E\mathcal M}$-$\bm G$-SSet}\xrightarrow{\textup{Yoneda}} \FUN(\cat{$\bm{E\mathcal M}$-$\bm G$-SSet}^\op,\cat{SSet})\xrightarrow{\textup{restriction}}\FUN((\Ogl_G)^\op,\cat{SSet}).
	\end{equation*}
\end{construction}

\begin{proposition}\label{prop:elmendorf}
	For any finite group $G$ the above functor $\Phi$ is homotopical and the right half of a Quillen equivalence for the projective model structure on the target. In particular, it descends to an equivalence between the $\infty$-categorical localization at the $G$-global weak equivalences and the $\infty$-categorical localization at the levelwise weak homotopy equivalences.
	\begin{proof}
		This is a special case of \cite{g-global}*{Corollary~1.1.13}.
	\end{proof}
\end{proposition}

\begin{remark}\label{rk:OGgl-morphism-spaces}
	We can describe the simplicial category $\Ogl_G$ combinatorially as follows, see also~\cite{g-global}*{Remark~1.2.40}: $n$-simplices of $\maps(E\mathcal M\times_\phi G,E\mathcal M\times_\psi G)$ correspond bijectively to $n$-simplices $[u_0,\dots,u_n;g]\in (E\mathcal M\times_\psi G)^\phi$ via evaluation at $[1;1]\in E\mathcal M\times_\phi G$; note that the right hand side is the nerve of a groupoid (as $H$ acts freely on $E\mathcal M$), so $\Ogl_G$ can be equivalently viewed a $(2,1)$-category. Under this correspondence, composition is given by $[u_0,\dots,u_n;g][u'_0,\dots,u'_n;g']=[u_0'u_0,\dots,u_n'u_n;g'g]$ (note the flipped order of multiplication).

	More generally, if $X$ is any $E\mathcal M$-$G$-simplicial set, then evaluation at $[1;1]$ induces a natural isomorphism $\epsilon\colon \Phi(X)(E\mathcal M\times_\phi G)=\maps(E\mathcal M\times_\phi G,X)\to X^\phi$. A direct computation shows that under this isomorphism restriction along an $(n+1)$-cell $[u_0,\dots,u_n;g]:E\mathcal M\times_\phi G\to E\mathcal M\times_\psi G$ in $\Ogl_G$ corresponds to action by the same element, i.e.~the following diagram commutes:
	\begin{equation}\label{diag:Phi-restriction-functoriality}
		\begin{tikzcd}
			\Phi(X)(E\mathcal M\times_\psi G)\arrow[d, "\epsilon"']\arrow[r, "{\Phi[u_0,\dots,u_n;g]}"] &[3em] \Phi(X)(E\mathcal M\times_\phi G)\arrow[d,"\epsilon"]\\
			X^\psi\arrow[r, "{(u_0,\dots,u_n;g).\blank}"'] & X^\phi\rlap.
		\end{tikzcd}
	\end{equation}
\end{remark}

\subsubsection{Comparisons of enriched presheaves} While one can extend the assignment $G\mapsto\Ogl_G$ to a strict $2$-functor in $\Glo$, and so assemble the localizations of the categories $\FUN((\Ogl_G)^\op,\cat{SSet})$ into a global $\infty$-category, the maps $\Phi$ will not be strictly natural with respect to this structure, but only pseudonatural. In order to avoid talking about all the coherences required to make this precise, we will now give a more `combinatorial' version of the simplicial categories $\Ogl_G$ and the functors $\Phi$ that will also become relevant in Section~\ref{sec:univ-prop-global-Gamma}.

\begin{construction}
	Let $G$ be a finite group. We define a strict $(2,1)$-category $\rOgl_G$ as follows: an object of $\rOgl_G$ is a pair $(H,\phi)$ of a universal subgroup $H\subset\mathcal M$ and a homomorphism $\phi\colon H\to G$. For any two such objects $(H,\phi), (K,\psi)$ the hom-category $\HOM((H,\phi),(K,\psi))$ has objects the triples $(u,g,\sigma)$ with $u\in\mathcal M$, $g\in G$ and $\sigma\colon H\to K$ a homomorphism such that $hu=u\sigma(h)$ for all $h\in H$ and moreover $\phi=c_g\psi\sigma$, where $c_g$ denotes conjugation by $g$. If $(u',g',\sigma')$ is another object of the hom-category, then a morphism $(u,g,\sigma)\to(u',g',\sigma')$ is a $k\in K$ such that $\sigma'=c_k\sigma$ and $g'\psi(k)=g$. Composition in $\HOM((H,\phi),(K,\psi))$ is induced by multiplication in $K$; we omit the easy verification that this is a well-defined groupoid.

	If $(L,\zeta)$ is another object and $(u_1,g_1,\sigma_1)\colon (H,\phi)\to (K,\psi)$, $(u_2,g_2,\sigma_2)\colon (K,\psi)\to (L,\zeta)$ are composable maps, then we define their composition as $(u_1u_2,g_1g_2,\sigma_2\sigma_1)$ (note the flipped order of composition in the first two components!); this is indeed a map $(H,\phi)\to(L,\zeta)$ as $hu_1u_2=u_1\sigma_1(h)u_2=u_1u_2\sigma_2\sigma_1(h)$ for all $h\in H$ and moreover $\phi=c_{g_1}\psi \sigma_1=c_{g_1g_2}\zeta\sigma_2\sigma_1$.

	Finally, if $(u_1',g_1',\sigma_1')\colon (H,\phi)\to (K,\psi)$ and $(u_2',g_2',\sigma_2')\colon (K,\psi)\to (L,\zeta)$ are further morphisms and $k_1\colon (u_1,g_1,\sigma_1)\to (u_1',g_1',\sigma_1')$, $k_2\colon (u_2,g_2,\sigma_2)\to (u_2',g_2',\sigma_2')$ are $2$-cells, then the composite of $k_1$ and $k_2$ is $k_2\sigma_2(k_1)$; note that this is indeed well-defined as $\sigma_2'\sigma_1'=c_{k_2}\sigma_2c_{k_1}\sigma_1=c_{k_1\sigma_2(k_2)}\sigma_2\sigma_1$ while $g_1g_2=g_1'\psi(k_1)g_2'\zeta(k_2)=g_1'g_2'\zeta\sigma_2'(k_1)\zeta(k_2)=g_1'g_2'\zeta(\sigma_2'(k_1)k_2)=g_1'g_2'\zeta(k_2\sigma_2(k_1))$ where the second equality uses that $(u_2',g_2',\sigma_2')$ is a morphism and the final equality uses that $k_2$ is a $2$-cell.

	We omit the straight-forward verification that this is suitably associative and unital with units the maps of the form $(1,1,\id)$, making $\rOgl_G$ into a strict $(2,1)$-category.
\end{construction}

\begin{construction}
	We define $\mu\colon\rOgl_G\to\Ogl_G$ as follows: an object $(H,\phi)$ is sent to $E\mathcal M\times_\phi G$, a morphism $(u,g,\sigma)\colon (H,\phi)\to (K,\psi)$ is sent to the map $E\mathcal M\times_\phi G\to E\mathcal M\times_\psi G$ represented by $[u;g]$ while a $2$-cell $k\colon(u,g,\sigma)\to(u',g',\sigma')$ is sent to $[u'k,u;g]$.
\end{construction}

\begin{lemma}\label{lemma:Ogl-vs-rOgl}
	The above $\mu$ is well-defined (i.e.~these are indeed morphisms and $2$-cells in $\Ogl_G$) and an equivalence of $(2,1)$-categories.
	\begin{proof}
		First observe that $[u;g]$ is indeed $\phi$-fixed as $[hu;\phi(h)g]=[u\sigma(h);g\psi\sigma(h)]=[u;g]$ by definition of the morphisms of $\rOgl_G$; moreover, any $1$-cell in the target is of this form by \cite{g-global}*{Lemma~1.2.38}. On the other hand, Lemma~1.2.74 of \emph{op.~cit.} shows that $[u'k,u;g]$ is indeed a $2$-cell $[u;g]\Rightarrow[u';g']$ and that this assignment is bijective. Thus, it only remains to show that $\mu$ is a strict $2$-functor.

		To prove that $\mu\colon\HOM((H,\phi),(K,\psi))\to (E\mathcal M\times_\psi G)^\phi$ is a functor, it suffices to prove compatibility with composition (as both sides are groupoids), for which we note that for all $k\colon (u,g,\sigma)\to (u',g',\sigma')$ and $k'\colon (u',g',\sigma')\to (u'',g'',\sigma'')$
		\begin{align*}
			\mu(k')\mu(k)&=[u''k',u';g'][u'k,u;g]=[u''k'k,u'k;\smash{\underbrace{g'\psi(k)}_{{}=g}}][u'k,u;g]=[u''k'k,u;g]\\&=\mu(k'k).
		\end{align*}
		Next, we have to show that $\mu$ is compatible with horizontal composition of $2$-cells, hence in particular with composition of $1$-cells. For this we note that if $k\colon (u_1,g_1,\sigma_1)\Rightarrow (u_1',g_1',\sigma_1')$ is a $2$-cell between morphisms $(H,\phi)\to(K,\psi)$ and $\ell\colon (u_2,g_2,\sigma_2)\Rightarrow (u_2',g_2',\sigma_2')$ is a $2$-cell between morphisms $(K,\psi)\to(L,\zeta)$, then the horizontal composition $\mu(\ell)\odot\mu(k)$ is given by
		\begin{align*}
			[u_2'\ell,u_2,g_2]\odot[u_1'k,u_1;g_1]&=[u_1'ku_2'\ell,u_1u_2;g_1g_2]
			=[u_1'u_2'\sigma_2'(k)\ell,u_1u_2;g_1g_2]\\
			&= [u_1'u_2'\ell\sigma_2(k), u_1u_2;g_1g_2]
		\end{align*}
		where the final equality uses that $\sigma_2'(k)\ell=\ell\sigma_2(k)$ as $\ell$ is a $2$-cell. On the other hand, by definition $\ell\odot k= \ell\sigma_2(k)\colon (u_1u_2,g_1g_2,\sigma_2\sigma_1)\to (u_1'u_2',g_1'g_2',\sigma_2'\sigma_1')$, so $\mu(\ell\odot k)=\mu(\ell)\odot\mu(k)$ as desired.

		Finally, $\mu(1,1,\id)=[1;1]$ by construction, i.e.~$\mu$ also preserves identity $1$-cells.
	\end{proof}
\end{lemma}

\begin{construction}
\label{cstr:PsiUnstable}
	Let $G$ be a finite group. We define $\Psi\colon\cat{$\bm{E\mathcal M}$-$\bm G$-SSet}\to\PSH(\rOgl_G)\mathrel{:=}\FUN((\rOgl_G)^\op,\cat{SSet})$ as follows: for any $E\mathcal M$-$G$-simplicial set $X$, the enriched functor $\Psi(X)\colon (\rOgl_G)^\op\to\cat{SSet}$ is given on objects by $\Psi(X)(H,\phi)=X^\phi\subset X$; if $(K,\psi)$ is another object, then we send an $n$-simplex
	\begin{equation}\label{eq:composite-k}
		(u_0,g_0,\sigma_0)\stackrel{k_1\,}{\Longrightarrow} (u_1,g_1,\sigma_1)\stackrel{k_2\,}{\Longrightarrow}\cdots\stackrel{k_n\,}{\Longrightarrow} (u_n,g_n,\sigma_n)\in\maps((H,\phi),(K,\psi))_n
	\end{equation}
	to the action of $(u_nk_n\cdots k_1,u_{n-1}k_{n-1}\cdots k_1,\dots,u_1k_1,u_0;g_0)$ on $X$, cf.\ \Cref{rk:OGgl-morphism-spaces}. If $f\colon X\to Y$ is any map of $E\mathcal M$-$G$-simplicial sets, then we define $\Psi(f)$ via $\Psi(f)(H,\phi)=f^\phi$.
\end{construction}

\begin{proposition}\label{prop:Psi-equivalence}
	The assignment $\Psi\colon \cat{$\bm{E\mathcal M}$-$\bm G$-SSet}\to\PSH(\rOgl_G)$ is well-defined (i.e.~$\Psi(X)$ is a simplicially enriched functor and $\Psi(f)$ is an enriched natural transformation) and constitutes a functor. Furthermore, it descends to an equivalence on $\infty$-categorical localizations.
	\begin{proof}
		We will simultaneously prove that $\Psi$ is well-defined and that it is isomorphic to the composite
		\begin{equation*}
			\cat{$\bm{E\mathcal M}$-$\bm G$-SSet}\xrightarrow\Phi\PSH(\Ogl_G)\xrightarrow{\mu^*}\PSH(\rOgl_G);
		\end{equation*}
		the claim then follows from Proposition~\ref{prop:elmendorf} together with Lemma~\ref{lemma:Ogl-vs-rOgl}.

		To prove this, we first fix an $E\mathcal M$-$G$-simplicial set $X$, and we will show that $\Psi(X)$ is a well-defined simplicial functor isomorphic to $\Phi(X)\circ\mu$. To this end, we recall that we have for every $(H,\phi)\in\rOgl_G$ an isomorphism
		\begin{equation*}
			\Phi(X)(\mu(H,\phi))=\maps(E\mathcal M\times_\phi G, X)\xrightarrow{\;\epsilon\;} X^\phi=\Psi(X)(H,\phi)
		\end{equation*}
		given by evaluation at $[1;1]$. It follows formally that there is a unique way to extend the assignment $(H,\phi)\mapsto X^\phi$ to a simplicially enriched functor $(\rOgl_G)^\op\to\cat{SSet}$ in such a way that the $\epsilon$'s assemble into an enriched natural isomorphism from $\Phi(X)\circ\mu$, namely in terms of the composites
		\begin{align*}
			\maps_{\rOgl}\big((H,\phi),(K,\psi)\big)&\xrightarrow\mu \maps_{\Ogl}(E\mathcal M\times_\phi G,E\mathcal M\times_\psi G)\\
			&\xrightarrow{\Phi}\maps_{\cat{SSet}}\big(\maps_{\cat{$\bm{E\mathcal M}$-$\bm G$-SSet}}(E\mathcal M\times_\psi G,X),\\
			&\phantom{{}\xrightarrow{\Phi}\maps_{\cat{SSet}}\big({}} \maps_{\cat{$\bm{E\mathcal M}$-$\bm G$-SSet}}(E\mathcal M\times_\phi G,X)\big)\\
			&\xrightarrow{\epsilon_*(\epsilon^{-1})^*} \maps_{\cat{SSet}}(X^\psi,X^\phi)
		\end{align*}
		and we only have to show that this recovers the above definition of $\Psi$. By commutativity of $(\ref{diag:Phi-restriction-functoriality})$ this then amounts to saying that
		\begin{equation*}
			\maps\big((H,\phi),(K,\psi)\big)\xrightarrow\mu \maps(E\mathcal M\times_\phi G,E\mathcal M\times_\psi G)\xrightarrow{\epsilon} (E\mathcal M\times_\psi G)^\phi\subset E\mathcal M\times_\psi G
		\end{equation*}
		sends $(\ref{eq:composite-k})$ to $(u_nk_n\cdots k_1,\dots, u_1k_1,u_0;g_0)$. As $E\mathcal M\times_\psi G$ is the nerve of a groupoid, it will be enough to show this after restricting to each edge $0\to m$ ($0\le m\le n$), i.e.~that $\mu(k_m\cdots k_0)=(u_mk_m\cdots k_1,u_0;g_0)$. However, this is precisely the definition.

		Thus, we have altogether shown that $\Psi(X)$ is indeed a well-defined simplicial functor and that the maps $\epsilon$ assemble into an isomorphism $\Psi(X)\cong\Phi(X)\circ\mu$. We can now show that $\Psi$ is a well-defined functor: indeed, if $f\colon X\to Y$ is $(E\mathcal M\times G)$-equivariant, then $\Psi(f)$ is enriched natural as the enriched functor structure on both sides is given by acting with simplices of $E\mathcal M\times G$. It is then clear that $\Psi$ preserves composition and identities as this can be checked after evaluating at each $(H,\phi)$.

		Finally, we have to establish that the isomorphisms $\epsilon$ are natural in $X$. However, we can again check this after evaluating at each $(H,\phi)$, where this is obvious.
	\end{proof}
\end{proposition}

\begin{construction}
	We extend the assignment $G\mapsto\rOgl_G$ to a strict $(2,1)$-functor $\rOgl_\bullet\colon \sGlo\to\cat{Cat}_\Delta$ into the $2$-category of simplicial categories as follows: if $\alpha\colon G\to G'$ is a homomorphism, then $\alpha_!\colon\rOgl_G\to\rOgl_{G'}$ is given on objects by $\alpha_!(H,\phi)=(H,\alpha\phi)$, on $1$-cells by $\alpha_!(u,g,\sigma)=(u,\alpha(g),\sigma)$, and on $2$-cells by the identity; we omit the easy verification that this is well-defined and strictly functorial in $\alpha$. Moreover, if $g\in G'$ defines a natural transformation $\alpha_1\Rightarrow\alpha_2$ (i.e.~$\alpha_2=c_g\alpha_1$),  then we define the natural transformation $g_!\colon\alpha_{1!}\Rightarrow\alpha_{2!}$ on $(H,\phi)$ as $(1,g^{-1},\id_H)\colon (H,\alpha_1\phi)\to (H,\alpha_2\phi)$. We again omit the easy verification that this is well-defined and yields a strict $2$-functor.

	This $2$-functor structure then induces a $2$-functor structure on the assignment $G\mapsto (\rOgl_G)^\op$; note that in this the $2$-cells get inverted, i.e.~$g\colon\alpha_{1}\Rightarrow\alpha_{2}$ is now sent to the natural transformation $g_!^\op$ given pointwise by $(1,g,\id)$.
\end{construction}

\begin{proposition}\label{prop:Psi-2-natural}
	The maps $\Psi$ are strictly $2$-natural in $\sGlo$.
	\begin{proof}
		Let us first check $1$-naturality, i.e.~that for every $\alpha\colon G\to G'$ the diagram
		\begin{equation*}
			\begin{tikzcd}
				\cat{$\bm{E\mathcal M}$-$\bm{G'}$-SSet}\arrow[r,"\alpha^*"]\arrow[d, "\Psi"'] & \cat{$\bm{E\mathcal M}$-$\bm{G'}$-SSet}\arrow[d, "\Psi"]\\
				\PSH(\rOgl_{G'})\arrow[r, "{(\alpha_!)^*}"'] & \PSH(\rOgl_{G})
			\end{tikzcd}
		\end{equation*}
		of ordinary categories commutes.

		\textit{The above diagram commutes on the level of objects:} Let $X$ be an $E\mathcal M$-$G$-simplicial set; we have to show that $\Psi(\alpha^*X)=\Psi(X)\circ\alpha_!$. On the level of objects, this just amounts to the relation $(\alpha^*X)^\phi=X^{\alpha\circ\phi}$ for all $(H,\phi\colon H\to G)\in\rOgl_G$. To prove commutativity on the level of morphism spaces, we let $(K,\psi)$ be any other object and we consider an $n$-simplex
		\begin{equation*}
			(u_\bullet,g_\bullet,\sigma_\bullet)\mathrel{:=}\big((u_0,g_0,\sigma_0)\xRightarrow{k_1} (u_1,g_1,\sigma_1)\xRightarrow{k_2}\cdots\xRightarrow{k_n}(u_n,g_n,\sigma_n)\big)
		\end{equation*}
		of $\maps((H,\phi),(K,\psi))$. Then $\Psi(\alpha^*X)(u_\bullet,g_\bullet,\sigma_\bullet)$ is by definition given by acting with $(u_nk_n\cdots k_1,\dots, u_1k_1,u_0;g_0)\in E\mathcal M_n\times G$ on $\alpha^*X$, or equivalently by acting with $(u_nk_n\cdots k_1,\dots, u_1k_1,u_0;\alpha(g_0))\in E\mathcal M_n\times G'$ on $X$. As $\alpha_!\colon \rOgl_G\to\rOgl_{G'}$ sends $(u_\bullet,g_\bullet,\sigma_\bullet)$ to
		\begin{equation*}
			(u_0,\alpha(g_0),\sigma_0)\xRightarrow{k_1} (u_1,\alpha(g_1),\sigma_1)\xRightarrow{k_2}\cdots\xRightarrow{k_n}(u_n,\alpha(g_n),\sigma_n)
		\end{equation*}
		by definition, we see that $\Psi(X)(\alpha_!(u_\bullet,g_\bullet,\sigma_\bullet))$ is given by acting with the same element. Since in addition both $\Psi(X)(\alpha_!(u_\bullet,g_\bullet,\sigma_\bullet))$ and $\Psi(\alpha^*X)(u_\bullet,g_\bullet,\sigma_\bullet)$ are (higher) maps between the same two objects, this completes the proof that they agree, so that $\Psi(\alpha^*X)=\Psi(X)\circ\alpha_!$ as desired.

		\textit{The above diagram commutes on the level of morphisms:} As we already know that the diagram commutes on the level of objects, it is enough to check the claim after evaluating at each $(H,\phi)$. However, in this case both paths through the diagram send a morphism $f\colon X\to Y$ to the restriction $X^{\alpha\phi}\to Y^{\alpha\phi}$ of $f$.

		Finally, we can now very easily prove \textit{$2$-naturality} by the same argument: namely, it only remains to show that for every $2$-cell $g\colon\alpha_1\Rightarrow\alpha_2$ in $\Glo$, every $E\mathcal M$-$G$-simplicial set $X$, and every $(H,\phi)\in\rOgl_G$ the maps $\Psi(X)(g_!^\op\colon (H,\alpha_2\phi)\to (H,\alpha_1\phi))$ and $\Psi(g.\blank\colon \alpha_1^*X\to\alpha_2^*X)(H,\phi)$ agree. However, plugging in the definitions, both are simply given by acting with $g$ on $X$.
	\end{proof}
\end{proposition}

\begin{construction}\label{cons:gamma_rOglvsGlo}
	Let $G$ be a finite group. We define a strict $2$-functor $\gamma\colon\rOgl_G\to\sGlo_{/G}$ into the $2$-categorical slice as follows: an object $(H,\phi)$ is sent to $\phi\colon H\to G$ and a morphism $(u,g,\sigma)\colon (H,\phi)\to (K,\psi)$ is sent to the morphism
	\begin{equation}\label{diag:1-cell}
		\begin{tikzcd}[column sep=small]
			H\arrow[dr,bend right=10pt,"\phi"'{name=A}]\arrow[rr,"\sigma"]&&K;\arrow[dl, bend left=10pt,"\psi"]\twocell[from=A,"\scriptstyle g^{-1}"{yshift=-8pt},yshift=2pt]\\
			&G
		\end{tikzcd}
	\end{equation}
	note that $g^{-1}$ indeed defines such a $2$-cell in $\sGlo$ since $\phi=c_g\sigma\psi$ by assumption, whence $\sigma\psi=c_{g^{-1}}\phi$.
	Finally, a $2$-cell $k\colon (u,g,\sigma)\Rightarrow(u,g,\sigma)$ is sent to the $2$-cell $k\colon \sigma\Rightarrow\sigma'$.
\end{construction}

\begin{lemma}\label{lemma:Glo-vs-rOgl}
	For any finite $G$, $\gamma$ defines an equivalence $\rOgl_G\simeq\sGlo_{/G}$ of strict $(2,1)$-categories.
	\begin{proof}
		One easily checks by plugging in the definitions that $\gamma$ is indeed a strict $2$-functor. Essential surjectivity of $\gamma$ follows from the fact that any finite group is isomorphic to a universal subgroup (Lemma~\ref{lemma:universal-embeddings}). Moreover, given a general $1$-cell as depicted in $(\ref{diag:1-cell})$, there exists by \cite{g-global}*{Corollary~1.2.39} a $u\in\mathcal M$ with $hu=u\sigma(h)$ for all $h\in H$; $(u,g,\sigma)$ then clearly defines a $1$-cell $(H,\phi)\to(K,\psi)$ in $\rOgl_G$, and this is a preimage of $(\ref{diag:1-cell})$. This shows that $\gamma$ is essentially surjective on hom groupoids. Finally, $\gamma$ is bijective on $2$-cells by direct inspection.
	\end{proof}
\end{lemma}

\begin{lemma}\label{lemma:Glo-rOgl-natural}
	The maps $\gamma$ define a strictly $2$-natural transformation $\rOgl_\bullet\Rightarrow\sGlo_{/\bullet}$.
	\begin{proof}
		Let us first show that $\gamma$ is $1$-natural, i.e.~for every homomorphism $\alpha\colon G\to G'$ the diagram
		\begin{equation*}
			\begin{tikzcd}
				\rOgl_G\arrow[r, "\gamma"]\arrow[d, "\alpha_!"'] & \sGlo_{/G}\arrow[d, "\sGlo_{/\alpha}"]\\
				\rOgl_{G'}\arrow[r, "\gamma"'] & \sGlo_{/G'}
			\end{tikzcd}
		\end{equation*}
		of strict $2$-functors commutes strictly. This just amounts to plugging in the definitions: both paths through the diagram send an object $(H,\phi)$ to $\alpha\phi\colon H\to G'$, a $1$-cell as in $(\ref{diag:1-cell})$ to
		\begin{equation*}
			\begin{tikzcd}[column sep=small]
				H\arrow[dr,bend right=10pt,"\alpha\phi"'{name=A}]\arrow[rr,"\sigma"]&&K,\arrow[dl, bend left=10pt,"\alpha\psi"]\twocell[from=A,"\scriptstyle \alpha(g)^{-1}"{yshift=-8pt},yshift=2pt]\\
				&G'
			\end{tikzcd}
		\end{equation*}
		and a $2$-cell $\sigma\Rightarrow\sigma'$ represented by $k$ to a $2$-cell represented by the same $k$.

		For $2$-functoriality it then only remains to show that for any $2$-cell $g\colon\alpha\Rightarrow\alpha'$ of maps $G\to G'$ in $\sGlo$ the two pastings
		\begin{equation*}
			\begin{tikzcd}
				\rOgl_G\arrow[r, bend left=15pt,"\alpha_!"{name=A},shift left=5pt]\arrow[r, bend right=15pt,"\alpha_!'"'{name=B}, shift right=5pt] &\twocell[from=A,to=B, "\scriptstyle g_!"{xshift=8pt},xshift=-4pt] \rOgl_{G'} \arrow[r,"\gamma"] & \sGlo_{/G'}
			\end{tikzcd}
			\qquad\text{and}\qquad
			\begin{tikzcd}
				\rOgl_G\arrow[r,"\gamma"]&
				\sGlo_{/G}\arrow[r, bend left=15pt,"\alpha_!"{name=A},shift left=4pt]\arrow[r, bend right=15pt,"\alpha_!'"'{name=B}, shift right=4pt] &\twocell[from=A,to=B, "\scriptstyle g_!"{xshift=8pt},xshift=-4pt] \sGlo_{/G'}
			\end{tikzcd}
		\end{equation*}
		agree pointwise. However, by direct inspection both are given on an object $(H,\phi)$ of $\rOgl_G$ simply as the $1$-cell
		\begin{equation*}
			\begin{tikzcd}[column sep=small]
				H\arrow[dr,bend right=10pt,"\alpha\phi"'{name=A}]\arrow[rr,"="]&&H\arrow[dl, bend left=10pt,"\alpha'\phi"]\twocell[from=A,"\scriptstyle g"{yshift=-8pt},yshift=2pt]\\
				&G
			\end{tikzcd}
		\end{equation*}
		which completes the proof of the lemma.
	\end{proof}
\end{lemma}

\subsubsection{Putting the pieces together}
Now we are finally ready to deduce our comparison result:

\begin{proof}[Proof of Theorem~\ref{thm:global-spaces}]
	As mentioned in the beginning of this subsection, we will first construct an equivalence $\ul\S^\textup{gl}\simeq\ul\Spc_{\Glo}$ by hand:

	Proposition~\ref{prop:Psi-2-natural} says that the maps $\Psi$ define a $2$-natural transformation between the global categories $\cat{$\bm{E\mathcal M}$-$\bm\bullet$-SSet}$ and $\PSH(\rOgl_\bullet)\colon G\mapsto \PSH(\rOgl_G)$. If we equip $\cat{$\bm{E\mathcal M}$-$\bm G$-SSet}$ with the $G$-global weak equivalences for varying $G$ and each $\PSH(\rOgl_G)$ with the levelwise weak homotopy equivalences, this lifts to a map of global relative categories, which in turn decends to an equivalence between the global $\infty$-categories obtained by pointwise localization according to Proposition~\ref{prop:Psi-equivalence}. Note that the localization of $\cat{$\bm{E\mathcal M}$-$\bm\bullet$-SSet}$ is the global $\infty$-category $\ul\S^\text{gl}$ by definition; it will now be useful to pick a very specific localization of $\PSH(\rOgl_\bullet)$ for the purposes of this proof:

	Namely, we pick a \emph{simplicially enriched} fibrant replacement functor for the Kan-Quillen model structure (for example via the enriched small object argument \cite{cat-htpy}*{Theorem~13.5.2} or simply by using the geometric realization-singular set adjunction), which provides us with an enriched functor $r\colon\cat{SSet}\to\cat{Kan}$ together with enriched natural transformations $\id\Rightarrow ir$ and $\id\Rightarrow ri$ that are levelwise weak homotopy equivalences, where $i\colon\cat{Kan}\hookrightarrow\cat{SSet}$ is the inclusion. As an upshot, if $A$ is any simplicially enriched category, then $r\circ\blank\colon \PSH(A)\to\PSH^\cat{Kan}(A)\mathrel{:=}\FUN(A^\op,\cat{Kan})$ is a homotopy equivalence with respect to the levelwise weak homotopy equivalences, so it induces an equivalence of $\infty$-categorical localizations. Specializing this to our situation, the maps $r$ assemble into a map of global relative categories from $\PSH(\rOgl_\bullet)$ to $\PSH^\cat{Kan}(\rOgl_\bullet)$. Finally, for any simplicial category $A$ the enriched-natural comparison map
	\begin{equation*}
		\nerve\big(\PSH^\cat{Kan}(A)\big)=\nerve\FUN(A^\op,\cat{Kan})\to\Fun(\nerve_\Delta(A^\op),\nerve_\Delta(\cat{Kan}))=\PSh(\nerve_\Delta A)
	\end{equation*}
	is a localization at the levelwise weak homotopy equivalences as a consequence of \cite{HTT}*{Proposition~4.2.4.4}, see also~\cite{g-global}*{Proposition~A.1.18}, where this argument is spelled out in detail. Thus, we altogether get a map of global $\infty$-categories
	\begin{equation*}
		\nerve\big(\PSH(\rOgl_\bullet)\big)\xrightarrow{r\circ\blank}\nerve\big(\PSH^\cat{Kan}(\rOgl_\bullet)\big)\xrightarrow{\text{canonical}} \PSh(\nerve_\Delta(\rOgl_\bullet))
	\end{equation*}
	that is pointwise a localization, whence induces an equivalence $\ul\S^\text{gl}\simeq\PSh(\nerve_\Delta\rOgl_\bullet)$ of global $\infty$-categories.

	Restricting along the strictly $2$-natural equivalence $\gamma\colon\rOgl_\bullet\Rightarrow\sGlo_{/\bullet}$ of $2$-functors $\sGlo\to\cat{Cat}_\Delta$ (see Lemmas~\ref{lemma:Glo-vs-rOgl} and~\ref{lemma:Glo-rOgl-natural}) yields an equivalence of global $\infty$-categories $\PSh(\nerve_\Delta(\sGlo_{/\bullet}))\simeq\PSh(\nerve_\Delta\rOgl_\bullet)$. By Proposition~\ref{prop:comparison-slices} the left hand side is then further equivalent to $\PSh(\Glo_{/\bullet})=\ul{\Spc}_{\Glo}$. This completes the construction of an equivalence $\ul\Spc_{\Glo}\simeq\ul\S^\text{gl}$ of global $\infty$-categories.

	As $\ul\Spc_{\Glo}$ is presentable (Example~\ref{ex:T_Spaces_Presentable}), so is $\ul\S^\text{gl}$. Moreover, the universal property of $\ul\Spc_{\Glo}$ shows that the equivalence $F\colon \ul\Spc_{\Glo}\to\ul\S^\text{gl}$ constructed above is characterized essentially uniquely by the image of the terminal object $*\in\ul\Spc_\text{Glo}(1)$, so it only remains to verify that $F$ sends this to the terminal object of $\S^\text{gl}$. However, this follows simply from the fact that $F(1)\colon \ul\Spc_{\Glo}(1)\to \S^\text{gl}$ is an equivalence of ordinary $\infty$-categories.
\end{proof}

\section{Parametrized semiadditivity}
The goal of this section is to introduce the parametrized analogue of the familiar notion of semiadditivity of an $\infty$-category, following the ideas introduced by Nardin \cite{nardin2016exposeIV}.
In the parametrized setting, the notion of semiadditivity comes in various flavors, parametrized by suitable subcategories $P \subseteq T$: roughly speaking, a $T$-$\infty$-category $\Cc$ is $P$-semiadditive if it is pointwise semiadditive, admits left adjoints $p_!$ and right adjoints $p_*$ for the morphisms $p\colon A \to B$ in $P$ satisfying base change, and a canonical \textit{norm map} $\Nm_p\colon p_! \to p_*$ between these two adjoints is an equivalence.

\subsection{Pointed \for{toc}{$T$-$\infty$}\except{toc}{\texorpdfstring{$\bm T$-$\bm\infty$}{T-∞}}-categories}
\label{subsec:PointedTCategories}
As a first step towards defining parametrized semiadditivity, we introduce the notion of pointedness for $T$-$\infty$-categories. Recall that a \textit{zero object} of an $\infty$-category is an object which is both initial and terminal. An $\infty$-category is called \textit{pointed} if it admits a zero object. This has the following parametrized analogue.

\begin{definition}
	Let $\Cc$ be a $T$-$\infty$-category and let $c\colon \ul{1} \to \Cc$ be a $T$-$\infty$-functor. We say that $c$ is a \textit{$T$-zero object} of $\Cc$ if  $c(B) \in \Cc(B)$ is a zero object for every $B \in T$. We say that $\Cc$ is \textit{pointed} if it admits a $T$-zero object; equivalently, $\Cc(B)$ is a pointed $\infty$-category for every $B \in T$ and $f^*\colon \Cc(B) \to \Cc(A)$ preserves the zero object for every $f\colon A \to B$ in $T$.

	Similarly, we say that $c\colon \ul{1} \to \Cc$ is a  \textit{$T$-initial object} (resp.\ a \textit{$T$-final object}) if $c(B) \in \Cc(B)$ is an initial object (resp.\ a final object) for all $B \in T$.

	Denote by $\Cat_T^{*} \subseteq \Cat_T$ the (non-full) subcategory spanned by the $T$-$\infty$-categories admitting a $T$-final object and the $T$-functors that preserve the $T$-final object. We let $\Cat^{\pt}_T \subseteq \Cat_T^{*}$ denote the full subcategory spanned by the pointed $T$-$\infty$-categories.
\end{definition}

\begin{remark}\label{rk:T-zero}
	By Lemma~\ref{lem:ColimitsIndexedByConstantTCategories}, a $T$-initial object is equivalently a $T$-colimit over the constant $T$-$\infty$-category $\const_\emptyset$. The aforementioned lemma moreover shows that in this case $\Cc(X)$ has an initial object for every presheaf $X\in\PSh(T)$, and that restriction along arbitrary maps of presheaves preserves initial objects. The statement for $T$-final objects is then formally dual, and we in particular conclude that for a pointed $T$-$\infty$-category $\Cc$ and any presheaf $X$, the $T$-$\infty$-category $\Cc(X)$ is pointed.
\end{remark}

\begin{definition}
	For $T$-$\infty$-categories $\Cc$ and $\Dd$ which admit a $T$-final object, we let
	\begin{align*}
		\ulFun_T^*(\Cc,\Dd) \subseteq \ulFun_T(\Cc,\Dd)
	\end{align*}
	be the full parametrized subcategory spanned at $B \in T$ by the \textit{pointed $T_{/B}$-functors}, i.e.\ those $F\colon \pi_B^*\Cc \to \pi_B^*\Dd$ which preserve the $T_{/B}$-final object.
\end{definition}

In the non-parametrized setting, an $\infty$-category is pointed if and only if it admits an initial object $\emptyset$ and a terminal object $1$, and the canonical map $\emptyset \to 1$ is an equivalence. In other words: the limit and colimit of the empty diagram in $\Cc$ exist and are canonically equivalent. For our discussion of parametrized semiadditivity, we will need an enhancement of this statement to the parametrized setting which identifies more generally the (parametrized) limit and colimit corresponding to a \textit{disjoint summand inclusion}.

\begin{definition}
	\label{def:DisjointSummandInclusion}
	A map $f\colon A \to B$ in an $\infty$-category $\Ee$ is called a \textit{disjoint summand inclusion} if there exists another morphism $g\colon C \to B$ in $\Ee$ such that the maps $f$ and $g$ exhibit $B$ as a coproduct of $A$ and $C$ in $\Ee$.
\end{definition}

\begin{lemma}
	\label{lem:NormMapDisjointSummandInclusions}
	Let $\Cc$ be a $T$-$\infty$-category and let $f\colon A \to B$ be a disjoint summand inclusion in $\PSh(T)$.
	\begin{enumerate}[(1)]
		\item If $\Cc$ admits a $T$-initial object, then the restriction functor $f^*\colon \Cc(B) \to \Cc(A)$ admits a fully faithful left adjoint $f_!\colon \Cc(A) \to \Cc(B)$;
		\item If $\Cc$ admits a $T$-final object, then the restriction functor $f^*\colon \Cc(B) \to \Cc(A)$ admits a fully faithful right adjoint $f_*\colon \Cc(A) \to \Cc(B)$;
		\item If $\Cc$ admits both a $T$-initial object and a $T$-final object, then there is a unique map \[\Nm_f\colon f_! \implies f_*\] whose restriction $f^*\Nm_f\colon f^*f_! \Rightarrow f^*f_*$ is the equivalence inverse to the composite \[f^*f_* \xRightarrow[\sim]{c^*_{f}} \id \xRightarrow[\sim]{u^!_f} f^*f_!;\]
		\item If $\Cc$ is pointed, this map $\Nm_f\colon f_! \Rightarrow f_*$ is an equivalence.
	\end{enumerate}
\end{lemma}
\begin{proof}
	Let $g\colon C \to B$ denote the disjoint complement of $f$. As $\Cc\colon \PSh(T)\catop \to \Cat_{\infty}$ sends colimits in $\PSh(T)$ to limits of $\infty$-categories, the maps $f$ and $g$ induce an equivalence
	\begin{align*}
		(f^*,g^*)\colon \Cc(B) \iso \Cc(A) \times \Cc(C),
	\end{align*}
	and under this equivalence the restriction functor $f^*\colon \Cc(B) \to \Cc(A)$ corresponds to the first projection map $\Cc(A) \times \Cc(C) \to \Cc(A)$. If $\Cc$ admits a $T$-initial object, then this projection has a fully faithful left adjoint given by $X \mapsto (X,\emptyset)$, where $\emptyset \in \Cc(C)$ denotes the initial object (see Remark~\ref{rk:T-zero}). It follows that $f^*$ admits a fully faithful left adjoint $f_!$. Similarly if $\Cc$ admits a $T$-final object, the projection has a fully faithful right adjoint given by $X \mapsto (X,1)$, where $1 \in \Cc(C)$ is a final object, and thus $f^*$ admits a right adjoint $f_*$. If $\Cc$ satisfies both, then inserting the unique map $\emptyset \to 1$ in the second variable gives rise to a natural transformation $\Nm_f\colon f_! \Rightarrow f_*$, which is uniquely determined by requiring that its restriction along $f$ is the canonical identification $f^*f_! \simeq f^*f_*$ in (3). It is clear that $\Nm_f$ is an equivalence whenever $\Cc(C)$ is pointed.
\end{proof}

Before moving on, we record a useful characterization of the disjoint summand inclusions in a presheaf category:

\begin{lemma}
\label{lem:DisjointSummandInclusoinsLocalClass}
    Let $f\colon X \to Y$ be a map in $\PSh(T)$. Then the following are equivalent:
    \begin{enumerate}[(1)]
        \item The map $f$ is a disjoint summand inclusion;
        \item For every map $t\colon A \to Y$ from a representable object $A \in T$, the base change $t^*f\colon A \times_Y X \to A$ of $f$ is a disjoint summand inclusion (i.e.~either $t^*f$ is an equivalence or $A \times_Y X = \emptyset$).
    \end{enumerate}
	\begin{proof}
        It is clear that (1) implies (2) as disjoint summand inclusions in $\PSh(T)$ are closed under pullback. We thus assume that (2) is satisfied and prove that it implies (1).

        By the co-Yoneda Lemma, there are equivalences $X \simeq \colim_{A \in T_{/X}} A$ and $Y \simeq \colim_{B \in T_{/Y}} B$. Under these equivalences, the map $f$ corresponds to restriction of indexing diagrams along the functor $T_{/f}\colon T_{/X}\to T_{/Y}$. It will therefore suffice to show that this functor is a disjoint summand inclusion of $\infty$-categories, or equivalently that it is fully faithful and any object of $T_{/X}$ admitting a map to or from the image of $T_{/f}$ already belongs to the essential image.

		For this, we will first show that $f$ is a monomorphism. This will immediately imply that $\PSh(T)_{/f}$ is fully faithful, hence so is $T_{/f}$. To this end, we observe that the natural map $\tau\colon\coprod_{A\in T_{/X}}A\to\colim_{A\in T_{/X}}A\simeq X$ is an effective epimorphism \cite{abfj-left-exact}*{Example~2.3.6}, so pullbacks along it detect monomorphisms by \cite{HTT}*{Proposition~6.2.3.17}. However, by universality of colimits, $\tau^*f$ is simply given as the coproduct of all the pullbacks of $f$ along all the maps $A\to X$, and each of these is in particular a monomorphism by assumption.

		For the closure of the image, consider objects $t\colon A\to X$ in $T_{/X}$ and $u\colon B\to Y$ in $T_{/Y}$. If there is a map $\alpha\colon u\to T_{/f}(t)$ in $T_{/Y}$, then $ft\alpha\sim u$, so $t\alpha$ is a preimage of $u$. Conversely, a map $\beta\colon T_{/f}(t)\to u$ amounts to a commutative square
		\begin{equation}\label{diag:sq-def-beta}
			\begin{tikzcd}
				A\arrow[d, "t"']\arrow[r, "\beta"] & B\arrow[d, "u"]\\
				X\arrow[r, "f"'] & Y
			\end{tikzcd}
		\end{equation}
		in $\PSh(T)$. By assumption, the pullback  $B\times_YX$ is either empty or the projection to $B$ is an equivalence. However, the first case is impossible as $B\times_YX$ receives a map from $A$ induced by $(\ref{diag:sq-def-beta})$, so we see there exists a (pullback) square of the form
		\begin{equation*}
			\begin{tikzcd}
				B\arrow[r, equal]\arrow[d,"v"']\arrow[dr,phantom, "\lrcorner"{very near start}] & B\arrow[d, "u"]\\
				X\arrow[r, "f"'] & Y.
			\end{tikzcd}
		\end{equation*}
		The map $v$ is the desired preimage of $u$, finishing the proof of the lemma.
	\end{proof}
\end{lemma}

Given a $T$-$\infty$-category $\Cc$ admitting a $T$-final object, one may form the $T$-$\infty$-category $\Cc_*$ of pointed objects of $\Cc$. We will need several basic properties of this construction.

\begin{construction}
	Let $\Cc$ be a $T$-$\infty$-category which admits a $T$-final object. We define the $T$-$\infty$-category $\Cc_*$ of \textit{pointed objects of $\Cc$} as the composite
	\begin{align*}
		T\catop \xrightarrow{\Cc} \Cat_{\infty}^{*} \xrightarrow{(-)_*} \Cat_\infty^{\pt},
	\end{align*}
	where the second functor sends an $\infty$-category $\Ee$ with terminal object $*$ to the slice $\Ee_* := \Ee_{*/}$. This construction is functorial in $\Cc$ and assembles into a functor $(-)_*\colon \Cat^{*}_T \to \Cat^{\pt}_T$.
\end{construction}

\begin{corollary}
	\label{cor:AdjunctionPointedness}
	The functor $(-)_*\colon \Cat^{*}_T \to \Cat^{\pt}_T$ is right adjoint to the fully faithful inclusion $\Cat^{\pt}_T \hookrightarrow \Cat^{*}_T$.
\end{corollary}
\begin{proof}
	This is immediate from the fact that the functor $(-)_*\colon \Cat^{*}_\infty \to \Cat^{\pt}_\infty$ is right adjoint to the fully faithful inclusion $\Cat^{\pt}_\infty \subseteq \Cat^{*}_\infty$.
\end{proof}

\begin{corollary}
	\label{cor:AdjunctionPointednessInternal}
	For $\Cc \in \Cat^{\pt}_T$ and $\Dd \in \Cat^{*}_T$, composition with the adjunction counit $\Dd_* \to \Dd$ induces an equivalence of $T$-$\infty$-categories $\ulFun_T^{*}(\Cc,\Dd_*) \iso \ulFun_T^{*}(\Cc,\Dd)$.
\end{corollary}
\begin{proof}
	We will prove that the induced functor $\Fun_T^{*}(\Cc,\Dd_*) \to \Fun_T^{*}(\Cc,\Dd)$ on underlying $\infty$-categories is an equivalence. For every $B \in T$ this thus gives an equivalence $\Fun_{T_{/B}}^{*}(\pi_B^*\Cc,\pi_B^*\Dd_*) \to \Fun_{T_{/B}}^{*}(\pi_B^*\Cc,\pi_B^*\Dd)$ which proves the claim.
	By the Yoneda lemma it suffices to prove that for any $\infty$-category $\Ee$ the above map induces an equivalence
	\begin{align*}
		\Hom_{\Cat_{\infty}}(\Ee,\Fun^*_T(\Cc,\Dd_*)) \to \Hom_{\Cat_{\infty}}(\Ee,\Fun^*_T(\Cc,\Dd)).
	\end{align*}
	Observe that the cotensor $\Dd^{\Ee}$ of $\Dd$ by $\Ee$ again has fiberwise final objects, and that there is a canonical equivalence $(\Dd^{\Ee})_* \simeq (\Dd_*)^{\Ee}$. The cotensoring adjunction gives rise to an equivalence
	\[
	    \Hom_{\Cat_{\infty}}(\Ee,\Fun^*_T(\Cc,\Dd)) \simeq \Hom_{\Cat^*_T}(\Cc,\Dd^{\Ee})
	\]
	and similarly for $\Dd_*$. It thus suffices to show that for every $\infty$-category $\Ee$ the map $(\Dd^{\Ee})_* \to \Dd^{\Ee}$ induces an equivalence
	\begin{align*}
		\Hom_{\Cat^*_T}(\Cc,(\Dd^{\Ee})_*) \to \Hom_{\Cat^*_T}(\Cc,\Dd^{\Ee}),
	\end{align*}
	which is true by the adjunction of \cref{cor:AdjunctionPointedness}.
\end{proof}

It follows that the condition of being pointed is closed under passing to parametrized functor categories.
\begin{corollary}
	\label{cor:FunctorCategoryPointed}
	Consider $T$-$\infty$-categories $\Cc$ and $\Dd$ admitting a $T$-final object. If either $\Cc$ or $\Dd$ is pointed, the $T$-$\infty$-category $\ulFun_T^*(\Cc,\Dd)$ is pointed as well.
\end{corollary}
\begin{proof}
	The case where $\Dd$ is pointed is clear from \Cref{prop:CoLimitsInFunctorCategories} together with Remark~\ref{rk:T-zero}. When $\Cc$ is pointed, we have by \cref{cor:AdjunctionPointednessInternal} an equivalence $\ulFun_T^*(\Cc,\Dd_*) \iso \ulFun_T^*(\Cc,\Dd)$, which reduces to the previous case since $\Dd_*$ is pointed.
\end{proof}

\begin{lemma}
	\label{lem:ForgetfulFunctorPreservesTLimitsPointedCase}
	Let $\bbU$ be a class of $T$-$\infty$-groupoids and let $\Dd$ be a $\bbU$-complete $T$-$\infty$-category admitting a $T$-final object. Then $\Dd_*$ is also $\bbU$-complete and the forgetful functor $\Dd_* \to \Dd$ preserves $\bbU$-limits.
\end{lemma}
\begin{proof}
Let $B \in T$ and let $(f\colon A \to B) \in \bbU(B)$. Consider objects $X \in \Dd(B)$ and $Y \in \Dd(A)_*$, and assume we are given a map $\phi\colon f^*X \to Y$ in $\Dd(A)$ which exhibits $X$ as a right adjoint object to $Y$ under $f^*\colon \Dd(B) \to \Dd(A)$, in the sense that for all $Z \in \Cc(B)$ the composite
	\begin{align*}
		\Hom_{\Cc(B)}(Z,X) \xrightarrow{f^*} \Hom_{\Cc(A)}(f^*Z,f^*X) \xrightarrow{\phi \circ -} \Hom_{\Cc(A)}(f^*Z,Y)
	\end{align*}
	is an equivalence. Taking $Z = *$ gives $X$ a canonical basepoint which turns the map $f^*X \to Y$ into a map in $\Dd(A)_*$. One now observes that this map exhibits $X \in \Dd(A)_*$ as a right adjoint object to $Y \in \Dd(B)_*$ under $f^*\colon \Dd(B)_* \to \Dd(A)_*$. This proves the claim.
\end{proof}

\subsection{Orbital subcategories}
\label{sec:OrbitalSubcategory}
In order to obtain a parametrized analogue of semiadditivity, we first need a parametrized analogue of the notion of \textit{finite (co)products}. In the non-parametrized setting, an $\infty$-category $\Ee$ admits finite (co)products if and only if it admits (co)limits indexed by finite sets (regarded as discrete $\infty$-categories). To generalize this to the parametrized setting, we would thus need a parametrized analogue of the notion of finite set.

In general, there might be various natural choices for such a generalization. A large family of examples comes from certain subcategories $P$ of $T$ that we call \textit{orbital}, extending the terminology of \cite{exposeI}. To every orbital subcategory $P$, we assign a class of $T$-$\infty$-groupoids called the \textit{finite $P$-sets}, and a $T$-$\infty$-category $\Cc$ is said to admit \textit{finite $P$-coproducts} if it admits parametrized colimits indexed by finite $P$-sets.

\begin{definition}
	\label{def:finPsets}
	Let $\finTsets$ be the finite coproduct completion of $T$, defined as the full subcategory of $\PSh(T)$ spanned by the finite disjoint unions $\bigsqcup_{i=1}^n A_i$ of representable presheaves $A_i \in T$. We refer to $\finTsets$ as the \textit{$\infty$-category of finite $T$-sets}.

	For a wide subcategory $P \subseteq T$, we let $\finPsets \subseteq \finTsets$ denote the wide subcategory spanned by all the morphisms which are a disjoint union of morphisms of the form $(p_i)\colon \bigsqcup_{i=1}^n A_i \to B$ where each morphism $p_i\colon A_i \to B$ is in $P$. We refer to $\finPsets$ as the \textit{$\infty$-category of finite $P$-sets}.

    Note that $\finPsets$ is equivalent to the finite coproduct completion of the $\infty$-category $P$.
\end{definition}

\begin{definition}
	\label{def:OrbitalSubcategory}
	A wide subcategory $P \subseteq T$ is called \textit{orbital} if the base change of a morphism in $\finPsets$ along an arbitrary morphism in $\finTsets$ exists and is again in $\finPsets$. Equivalently, for every pullback diagram
	\[\begin{tikzcd}
		A' \dar[swap]{p'} \rar{\alpha} \drar[pullback] & A \dar{p} \\
		B' \rar{\beta} & B
	\end{tikzcd}\]
	in $\PSh(T)$, with $A,B,B' \in T$ and $p\colon A \to B$ in $P$, the morphism $p'\colon A' \to B'$ can be decomposed as a disjoint union $(p_i)_{i=1}^n\colon \bigsqcup_{i=1}^n A_i \to B'$ for morphisms $p_i\colon A_i \to B'$ in $P$.

	The $\infty$-category $T$ is called \textit{orbital} if it is orbital when regarded as a subcategory of itself.
\end{definition}

\begin{remark}
	An $\infty$-category $T$ is orbital in our sense if and only if it is orbital in the sense of \cite{exposeI}, \cite{shah2021parametrized}, \cite{nardin2016exposeIV}*{Definition~4.1}.
\end{remark}

\begin{example}
	Every $\infty$-category $T$ has a minimal orbital subcategory given by $\iota T$, the core of $T$.
\end{example}

The following is the main example of an orbital subcategory in this article.

\begin{example}
	\label{ex:Orb_Orbital}
	We define $\Orb\subset\Glo$ to be the subcategory spanned by all objects and the injective group homomorphisms. We claim that $\Orb$ is an orbital subcategory of $\Glo$. Observe that the $\infty$-category of finite $\Glo$-sets is equivalent to the $(2,1)$-category of finite groupoids, which admits all homotopy-pullbacks. The subcategory of finite $\Orb$-sets is the wide subcategory on the faithful maps of groupoids, and thus the orbitality of $\Orb$ is equivalent to the observation that pullbacks of faithful maps of groupoids are again faithful.
\end{example}

The following two examples are variations of \cref{ex:Orb_Orbital}.

\begin{example}
	\label{ex:OrbG_Orbital}
	The orbit category $\Orb_G$ of a finite group $G$ is orbital. More generally, for a Lie group $G$, let $\Orb^{f.i.}_{G}$ be the wide subcategory of the orbit $\infty$-category $\Orb_G$ spanned by the morphisms equivalent to projections $G/K \to G/H$ for subgroups $K \subseteq H \subseteq G$ where $K$ has finite index in $H$. Then $\Orb^{f.i.}_{G}$ is an orbital subcategory of $\Orb_G$. Indeed, the pullback of $G/K \to G/H$ along a morphism $G/H' \to G/H$ is computed via a double coset formula, namely the finite disjoint union $\bigsqcup_{[g] \in H'\backslash H / K} G/(H' \cap {}^g\!K)$.
\end{example}

\begin{example}
	Mixing \Cref{ex:Orb_Orbital} with \Cref{ex:OrbG_Orbital}, one can define an $\infty$-category $\Glo_{\mathrm{Lie}}$ whose objects are compact Lie groups $G$ and whose morphism space $\Hom_{\Glo}(G,H)$ is given by the homotopy orbit space $\Hom_{\mathrm{TopGrp}}(G,H)_{hH}$, where $H$ acts on the space of continuous homomorphisms $G \to H$ via conjugation. See \cite{gepnerhenriques2007orbispaces}*{Section~4.1} or \cite{rezk2014global}*{Section~2.2}. Let $\Orb^{f.i.}_{\mathrm{Lie}} \subseteq \Glo_{\mathrm{Lie}}$ be the wide subcategory whose morphisms are given by the injective homomorphisms $G \hookrightarrow H$ of finite index. Then $\Orb^{f.i.}_{\mathrm{Lie}}$ is an orbital subcategory. The relevant pullbacks are again computed by a double coset formula.
\end{example}

Orbital subcategories are closed under various constructions:
\begin{example}
	\label{ex:OrbitalClosed}
	\begin{enumerate}[(1)]
		\item \label{it:SliceOfOrbitalIsOrbital} (Slice) Let $P \subseteq T$ be an orbital subcategory and let $B \in T$ be an object. Then the wide subcategory of $T_{/B}$ spanned by those morphisms over $B$ contained in $P$ is again an orbital subcategory. We will often abuse notation and denote this subcategory again by $P$.
		\item \label{it:PreimageOfOrbitalIsOrbital} (Preimage) More generally, if $f\colon S \to T$ is a right fibration, then the preimage $Q := f^{-1}(P) \subseteq S$ of an orbital subcategory $P \subseteq T$ is again orbital. Indeed, note that $\mathbb{F}_Q = f^{-1}(\mathbb{F}_P)$, and that the extension $\mathbb{F}_Q \to \mathbb{F}_P$ of $f$ is still a right fibration. The claim is then an instance of \cite{HHLNa}*{Proposition 2.6}.
		\item (Intersection) The intersection $\bigcap_{i \in I} P_i$ of any non-empty collection of orbital subcategories $P_i \subseteq T$ is again orbital.
	\end{enumerate}
\end{example}

\begin{example}
	Let $G$ be a finite group. Combining part \eqref{it:PreimageOfOrbitalIsOrbital} from \cref{ex:OrbitalClosed} with \cref{ex:OrbG_Orbital}, we find that for a $G$-space $X\colon \Orb_G\catop \to \Spc$, the $\infty$-category $\smallint X$ of points of $X$ (that is, the total category of the right fibration classified by $X$) is orbital.
\end{example}

So far, all the given examples of orbital subcategories are equivariant in nature, being a variation of the orbit category of a group; these are the examples we are most interested in in this article. In the following example we mention some orbital subcategories that appear in completely different contexts.

\begin{example}
	Let $T$ be an $\infty$-category, and assume $P \subseteq T$ is a wide subcategory such that base changes of morphisms in $P$ exist in $T$ and are again in $P$. Then $P$ is orbital.

	In particular, many geometric examples give rise to orbital subcategories. For example:
	\begin{enumerate}[(1)]
		\item If $T = \mathrm{Diff}$ is the ordinary category of smooth manifolds, the wide subcategory on the local diffeomorphisms is orbital.
		\item If $T = \mathrm{Sm}_S$ is the ordinary category of smooth schemes over some base scheme $S$, the wide subcategory on the étale morphisms is orbital.
	\end{enumerate}
\end{example}

For the remainder of this subsection, we will fix an orbital subcategory $P \subseteq T$.

\begin{definition}
	\label{def:ulFinPSets}
	We define the \textit{$T$-$\infty$-category of finite $P$-sets} $\ulfinPsets$. Given $B \in T$, we let
	\[
	\ulfinPsets(B) \subseteq \PSh(T)_{/B}
	\]
	be the full subcategory spanned by those morphisms $p\colon A \to B$ in $\PSh(B)$ which can be decomposed as a coproduct $(p_i)\colon \bigsqcup_{i=1}^{n} A_i \to B$ such that each morphism $p_i\colon A_i \to B$ is in $P$.
	By orbitality of $P$, $\ulfinPsets$ forms a parametrized subcategory of $\ul{\Spc}_T$.
	When $P = T$, we simply write $\ulFinSets{}$ for $\ulFinSets{T}$.
\end{definition}

Since $\ulfinPsets$ forms a class of $T$-$\infty$-groupoids (see \cref{def:ClassOfTGroupoids}) it makes sense to speak of parametrized colimits indexed by $\ulfinPsets$.

\begin{definition}
	\label{def:finitePcoproducts}
	Let $P \subseteq T$ be an orbital subcategory of $T$. We say that a $T$-$\infty$-category $\Cc$ \textit{admits finite $P$-coproducts} if it admits $\ulfinPsets$-colimits, in the sense of \cref{def:UColimits}. Dually, we say $\Cc$ \textit{admits finite $P$-products} if it admits $\ulfinPsets$-limits.
\end{definition}

\begin{definition}
	Let $\Cc$ and $\Dd$ be two $T$-$\infty$-categories which admit $\ulfinPsets$-limits. We define $\ulFun^{\Pprod}(\Cc,\Dd)$ to be the full parametrized subcategory of $\ulFun(\Cc,\Dd)$ spanned in level $B$ by the $T_{/B}$-functors $F\colon \pi_B^* \Cc \to \pi_B^* \Dd$ which preserve $P$-products (i.e.\ preserves $\pi_B^{-1}(P)$-products, c.f.\ \Cref{ex:OrbitalClosed}). Dually we define $\ulFun^{\Pcoprod}(\Cc,\Dd)$.
\end{definition}

When $P = T$, a $T$-$\infty$-category $\Cc$ admits finite $T$-coproducts in the sense of \cref{def:finitePcoproducts} if and only if it has finite $T$-coproducts in the sense of Shah \cite{shah2021parametrized}*{Definition~5.10}.

The following result gives a more explicit characterization of the condition for a $T$-$\infty$-category to admit finite $P$-(co)products.

\begin{proposition}[cf.\ \cite{shah2021parametrized}*{Proposition~5.12}, \cite{nardin2016exposeIV}*{Proposition~2.11}]
	\label{prop:finitePcoproducts}
	Let $P \subseteq T$ be an orbital subcategory
	and let $\Cc$ be a $T$-$\infty$-category. Then $\Cc$ admits finite $P$-coproducts if and only if the following two conditions hold:
	\begin{enumerate}[(1)]
		\item $\Cc$ admits fiberwise finite coproducts, see \cref{def:fiberwiseCocompleteness};
		\item for every morphism $p\colon A \to B$ in $P$, the restriction functor $p^*\colon \Cc(B) \to \Cc(A)$ admits a left adjoint $p_!\colon \Cc(A) \to \Cc(B)$ and for every pullback square as in \cref{lem:UColimitsVsAdjointable}\eqref{eq:PullbackSquareTCocompleteness} with $A,B,B' \in T$ and $f\colon A \to B$ in $P$, the Beck-Chevalley transformation $p'_! \circ \alpha^* \Rightarrow \beta^* \circ p_!$ is an equivalence.
	\end{enumerate}
	Dually, $\Cc$ admits finite $P$-products if and only the dual conditions hold.
\end{proposition}
\begin{proof}
	By definition, every morphism in $\finPsets$ with target $B \in T$ can be written as a composite
	\begin{align}
		\label{eq:decomposedMorphism}
		\bigsqcup_{i=1}^n B_i \xrightarrow{\bigsqcup_{i=1}^n p_i} \bigsqcup_{i=1}^n B \xrightarrow{\nabla} B
	\end{align}
	for morphisms $p_i\colon B_i \to B$ in $P$, where $\nabla\colon \bigsqcup_{i=1}^n B \to B$ denotes the fold map in $\PSh(T)$. As the functor $\Cc\colon \PSh(T)\catop \to \Cat_{\infty}$ sends colimits in $\PSh(T)$ to limits of $\infty$-categories, the condition of left $\finPsets$-adjointability splits up as left adjointability for the maps $\nabla\colon \bigsqcup_{i=1}^n B \xrightarrow{\nabla} B$ and left adjointability for the maps in $P$. Spelling out the definitions, one observes that the former is equivalent to condition (1) while the latter is equivalent to condition (2).
\end{proof}

A similar argument gives the following analogous result for preservation of finite $P$-coproducts:

\begin{proposition}
	\label{prop:PreservingFinitePcoproducts}
	Let $P \subseteq T$ be an orbital subcategory
	and let $\Cc$ and $\Dd$ be $T$-$\infty$-categories that admits finite $P$-coproducts. Then a $T$-functor $F\colon \Cc \to \Dd$ preserves finite $P$-coproducts if and only if it preserves fiberwise finite coproducts and for every morphism $p\colon A \to B$ in $P$, the Beck-Chevalley transformation $p_! \circ F_A \Rightarrow F_B \circ p_!$ is an equivalence.

	The dual statement for preservation of finite $P$-products also holds.\qednow
\end{proposition}

We end this subsection by showing that the $T$-$\infty$-category $\ulfinPsets$ can be characterized by a universal property: it is the free $T$-$\infty$-category admitting finite $P$-coproducts.

\begin{corollary}
	\label{cor:FinPSetsHasFinitePColimits}
	The $T$-$\infty$-category $\ulfinPsets$ admits finite $P$-coproducts and the inclusion $\ulfinPsets \hookrightarrow \ul{\Spc}_T$ preserves finite $P$-coproducts.
\end{corollary}
\begin{proof}
	By \Cref{ex:TSpacesPAdjointable} it suffices to show that the subcategory $\ulfinPsets \hookrightarrow \ul{\Spc}_T$ is closed under finite $P$-coproducts. But this is clear from \Cref{prop:finitePcoproducts} since it is closed under fiberwise coproducts and under composition with morphisms in $P$ by construction.
\end{proof}

\begin{corollary}
	\label{cor:FinPSetsFreeOnFinitePColimits}
	Let $\Dd$ be a $T$-$\infty$-category admitting finite $P$-coproducts. Let $*\colon \ul{1} \to \ulfinPsets$ denote the $T$-final object, given at $B \in T$ by the identity $\id_{B} \in \ulfinPsets(B)$. Then composition with $*\colon \ul{1} \to \ulfinPsets$ induces an equivalence of $T$-$\infty$-categories
	\begin{align*}
		\ulFun_T^{{\Pcoprod}}(\ulfinPsets,\Dd) \to \ulFun_T(\ul{1},\Dd) \simeq \Dd.
	\end{align*}
\end{corollary}
\begin{proof}
	It follows directly from \Cref{cor:FinPSetsHasFinitePColimits} that the subcategory $\ulfinPsets \subseteq \ul{\Spc}_T$ is the smallest subcategory which contains the $T$-final object and is closed under finite $P$-coproducts, meaning it is equivalent to $\ul\PSh_T^{\ulfinPsets}(\ul{1})$ in the notation of \cite{martiniwolf2021limits}*{Definition~7.1.6}. The claim is then an instance of \cite{martiniwolf2021limits}*{Theorem~7.1.13}.
\end{proof}

\subsection{Atomic orbital subcategories and norm maps}
\label{sec:AtomicSubcategory}
Let $P$ be an orbital subcategory of $T$. In this subsection, we will define what it means for $P$ to be an \textit{atomic orbital} subcategory of $T$, generalizing a definition of \cite{nardin2016exposeIV}. The atomicity condition on $P$ will allow us to define \textit{norm maps} $\Nm_p\colon p_! \to p_*$ in a pointed $T$-$\infty$-category $\Cc$, making it possible to compare finite $P$-coproducts in $\Cc$ to finite $P$-products in $\Cc$. We may therefore think of the atomic orbital subcategories as classifying the various potential `levels of semiadditivity' that a $T$-$\infty$-category might have.

\begin{definition}
	\label{def:AtomicSubcategory}
	Suppose $T$ is an $\infty$-category and let $P \subseteq T$ be an orbital subcategory. We say that $P$ is \textit{atomic orbital} if for every morphism $p\colon A \to B$ in $P$ the diagonal $\Delta\colon A \to A \times_B A$ in $\PSh(T)$ is a disjoint summand inclusion in the sense of \cref{def:DisjointSummandInclusion}. An $\infty$-category $T$ is called \textit{atomic orbital} if it is atomic orbital as a subcategory of itself.
\end{definition}

For a subcategory $P \subset T$, being an atomic orbital subcategory is a very restrictive condition: since every disjoint summand inclusion in $\PSh(T)$ is in particular a monomorphism, it implies that all the morphisms in $P$ have to be $0$-truncated.

The following lemma provides an alternative characterization of atomic subcategories in terms of the triviality of certain retracts. The case $P = T$ of this lemma immediately implies that our definition of atomic orbital $\infty$-categories is equivalent to that of \cite{nardin2016exposeIV}*{Definition~4.1}.

\begin{lemma}
	\label{lem:AtomicVsDisjointDiagonals}
	Let $P \subseteq T$ be an orbital subcategory. Then $P$ is atomic orbital if and only if any morphism $p\colon A \to B$ in $P$ which admits a section in $T$ is an equivalence.
\end{lemma}
\begin{proof}
	Assume first that $P$ is atomic orbital. Let $p\colon A \to B$ be a morphism in $P$ and assume that $p$ admits a section $s\colon B \to A$ in $T$. We will show that $p$ is an equivalence with inverse $s$. Since we are given an equivalence $ps \simeq \id_B$, it remains to show that $sp \simeq \id_A$. Equivalently, we may show that the map $(\id_A,sp)\colon A \to A \times_B A$ factors through the diagonal $\Delta_p\colon A \to A \times_B A$. By assumption this diagonal is equivalent to an inclusion $A \hookrightarrow A \sqcup C$ for some presheaf $C \in \PSh(T)$, and since $A$ is a representable presheaf it follows that the map $(\id_A,sp)\colon A \to A \times_B A \simeq A \sqcup C$ must either factor through $\Delta_p\colon A \hookrightarrow A \sqcup C$ or through $C \hookrightarrow A \sqcup C$. Assume for contradiction that we are in the latter situation. Then the pullback of $\Delta_p\colon A \to A\times_B A$ and $(\id_A,sp)\colon A \to A \times_B A$ is the empty presheaf. But this pullback is also equivalent to $B$, by the following pullback diagram:
	\[\begin{tikzcd}
		B \rar{s} \dar[swap]{s} \drar[pullback] & A \arrow[d, "{(\id_A,sp)}"{description}] \rar{p} \drar[pullback] & B \dar{s} \\
		A \rar{\Delta_p} & A \times_B A \rar{\pr_2} \dar[swap]{\pr_1} \drar[pullback] & A \dar{p} \\
		& A \rar{p} & B.
	\end{tikzcd}\]
	Since $B$ is not the empty presheaf, this leads to a contradiction, showing that $(\id_A,sp)\colon A \to A \times_B A$ factors through $\Delta_p$ as desired.

	Conversely, assume that any map in $P$ which admits a section in $T$ is an equivalence. Let $p\colon A \to B$ be a morphism in $P$. Since $P$ is orbital, the projection map $\pr_1\colon A \times_B A \to A$ in $\PSh(T)$ can be decomposed as a disjoint union $(p_i)^n_{i=1}\colon \bigsqcup_{i=1}^n A_i \to A$ of morphisms $p_i\colon A_i \to B$ in $P$. Since $A$ is representable, the diagonal $\Delta_p\colon A \to A \times_B A \simeq \bigsqcup_{i=1}^n A_i$ has to factor through one of the inclusions $A_i \hookrightarrow A \times_B A$, so that the morphism $p_i\colon A_i \to A$ admits a section $A \to A_i$ in $T$. By the assumption on $P$, this means that $p_i$ is an equivalence. It follows that the diagonal $\Delta_p$ of $p$ is the inclusion of a disjoint summand $A \simeq A_i \hookrightarrow \bigsqcup_{i=1}^n A_i$, as desired.
\end{proof}

\begin{example}
	Recall the subcategory $\Orb\subset\Glo$ spanned by the injective homomorphisms. Clearly, any injective homomorphism that admits a section is also surjective, hence an isomorphism. Together with Example~\ref{ex:Orb_Orbital}, we conclude that $\Orb$ is an atomic orbital subcategory of $\Glo$. By direct computation one sees that the diagonal in $\PSh(\Glo)$ of a non-injective group homomorphism is never a disjoint summand inclusion, and thus it follows that $\Orb$ is in fact the \textit{maximal} atomic orbital subcategory of $\Glo$.
\end{example}

\begin{remark}
    There is a classification of the atomic orbital subcategories of $\Glo$ in terms of \textit{global transfer systems} in the sense of Barrero \cite{Barrero2023TransferSystems}. Recall from \textit{op.~cit.}\ that a global transfer system (for the family of finite groups) is a partial order $\leqslant_T$ on the collection of finite groups which refines the subgroup relation and which is closed under preimages, meaning that for a group homomorphism $\alpha\colon G' \to G$, if $H \leqslant_T G$ then $\alpha^{-1}(H) \leqslant G'$. We may assign to $\leqslant_T$ a wide subcategory $\Orb_T \subseteq \Orb$ which contains those injections $i\colon H \hookrightarrow G$ for which $i(H) \leqslant_T G$. It is not difficult to show that $\Orb_T$ is an atomic orbital subcategory of $\Glo$, and that conversely every atomic orbital subcategory of $\Glo$ is of the form $\Orb_T$ for some global transfer system $\leq_T$.
\end{remark}

A convenient feature of atomic orbital subcategories is that they are \textit{left cancellable}, in the sense that for morphisms $f\colon A \to B$ and $g\colon B \to C$ in $T$, if both $g$ and $gf$ are in $P$ then also $f$ is in $P$.

\begin{lemma}
	Every atomic orbital subcategory $P \subseteq T$ is left cancellable.
\end{lemma}
\begin{proof}
	Let $f\colon A \to B$ and $g\colon B \to C$ be morphisms in $T$, and assume that both $g$ and $gf$ are in $P$. We will show that then also $f$ is in $P$. This is a classical argument \cite{ega1}*{Remarque~5.5.12}: we may factor $f$ as a composite
	\[
	A \xrightarrow{(1,f)} A \times_C B \xrightarrow{\pr_B} B
	\]
	in $\finTsets$, and it will suffice to show that both of these morphisms are morphisms in $\finPsets$. The projection $\pr_B\colon A \times_C B \to B$ is the base change of $gf\colon A \to C$ along $B \to C$, so it is in $\finPsets$ by orbitality of $P$ and the assumption on $gf$. In turn, the morphism $(1,f)\colon A \to A \times_C B$ is a base change of the diagonal map $\Delta_g\colon B \to B \times_C B$, which is by assumption a disjoint summand inclusion and thus in particular in $\finPsets$. This finishes the proof.
\end{proof}

\begin{corollary}
	Let $P \subseteq T$ be an atomic orbital subcategory. Then for every $B \in T$, the inclusion $P_{/B} \hookrightarrow T_{/B}$ is fully faithful. In particular, there is an equivalence $\ulfinPsets(B) \simeq (\finPsets)_{/B}$.\qed
\end{corollary}

While atomicity a priori only requires the diagonals of maps in $P$ be disjoint summand inclusions, the next proposition shows that this in fact holds for a more general class of maps in $\PSh(T)$. Recall from \Cref{rmk:limitExtensionRestrictedClass} that, given a presheaf $B$ on $T$, we write $\ulfinPsets(B) \subseteq \PSh(T)_{/B}$ for the full subcategory containing those morphisms $p\colon A \to B$ of presheaves whose base change to any representable $B' \in T$ lives in $\mathbb F^P_T$.

\begin{proposition}
	\label{prop:FinitePSetsHasDisjointDiagonalInclusions}
	Let $P\subseteq T$ atomic orbital, let $Y\in\PSh(T)$ and let $(p\colon X\to Y)\in\ul{\mathbb F}^P_T(Y)$. Then the diagonal $\Delta_p\colon X\to X\times_YX$ in $\PSh(T)$ is a disjoint summand inclusion.
\end{proposition}

\begin{proof}
	By \Cref{lem:DisjointSummandInclusoinsLocalClass}, it will suffice to show that the base change of $\Delta_p\colon  X\to X \times_Y X$ along any map $\alpha = (\alpha_1,\alpha_2)\colon A \to X \times_Y X$ from a representable $A \in T$ is a disjoint summand inclusion. Observe that the map $\alpha$ factors as the following composite:
	\[
		A \xrightarrow{(\id,\alpha_2)} A \times_Y X \xrightarrow{\alpha_1 \times \id} X \times_Y X.
	\]
	As base changes of disjoint summand inclusions are again disjoint summand inclusions, it will thus suffice to show that the base change of $\Delta_p$ along the map $\alpha_1 \times \id$ is a disjoint summand inclusion. To this end, consider the following commutative diagram:
	\[
	\begin{tikzcd}
		A \dar[swap]{(\id,\alpha_1)} \rar{\alpha_1} \drar[pullback, xshift = -10pt] &[1em] X \dar{\Delta_p} \\
		A \times_Y X \dar[swap]{\pr_1} \rar{\alpha_1 \times \id} \drar[pullback, xshift = -10pt] & X \times_Y X \dar[swap]{\pr_1} \rar{\pr_2} \drar[pullback] & X \dar{p} \\
		A \rar{\alpha_1} & X \rar{p} & Y.
	\end{tikzcd}
	\]
	It follows readily from the pasting law of pullback squares that each square is a pullback square. Observe that the projection map $\pr_1\colon A \times_Y X \to A$ is the base change of $p$ along $p \alpha_1$, hence it lies in $\finPsets$ by assumption. We decompose $A\times_YX\simeq\coprod_{i=1}^n A_n$ into a disjoint union of representables, and note that $(\id,\alpha_1)$ has to factor through a (unique) coproduct summand $A_i$. The resulting map $A\to A_i$ is then a section to ${\pr_1}|_{A_i}\in P$, so Lemma~\ref{lem:AtomicVsDisjointDiagonals} shows that ${\pr_1}|_{A_i}\colon A_i\to A$ is an equivalence. We conclude by 2-out-of-3 that $(\id,\alpha_1)\colon A\to A\times_YX$ defines an equivalence onto $A_i$, so it is a disjoint summand inclusion as claimed.
\end{proof}

For the remainder of this subsection, we will fix an atomic orbital subcategory $P \subseteq T$. We are now ready to define the norm map $\Nm_p\colon p_! \to p_*$ for $p$ as in Proposition~\ref{prop:FinitePSetsHasDisjointDiagonalInclusions}.

\begin{construction}[Norm map, cf. \cite{HA}*{Construction~6.1.6.8}, \cite{nikolausscholze2018cyclic}*{Construction~I.1.7}, \cite{hopkins2013ambidexterity}*{Construction~4.1.8}]
	\label{cstr:NormMap}
	Let $\Cc$ be a pointed $T$-$\infty$-category, let $B \in \PSh(T)$ and let $(p\colon A \to B) \in \ulfinPsets(B)$. Consider the following pullback diagram
	\begin{equation}
		\label{eq:PullbackDiagramNormMap}
		\begin{tikzcd}
			A \times_B A \dar[swap]{\pr_1} \rar{\pr_2} \drar[pullback] & A \dar{p} \\
			A \rar{p} & B
		\end{tikzcd}
	\end{equation}
	in $\PSh(T)$, and let $\Delta\colon A \to A \times_B A$ denote the diagonal of $p$. By \Cref{prop:FinitePSetsHasDisjointDiagonalInclusions}, $\Delta$ is a disjoint summand inclusion, so that \cref{lem:NormMapDisjointSummandInclusions} provides adjunctions $\Delta_! \dashv \Delta^* \dashv \Delta_*$ and an equivalence $\Nm_{\Delta}\colon \Delta_! \simeq \Delta_*$.
	\begin{enumerate}[(1)]
		\item \label{it:DefAlpha} Define a natural transformation $\alpha\colon \pr_1^* \Rightarrow \pr_2^*$ as the following composite:
		\begin{align*}
			\pr_2^* \xRightarrow{u^*_{\Delta}} \Delta_*\Delta^*{\pr_2}^* \simeq \Delta_* \overset{\Nm_{\Delta}^{-1}}{\simeq} \Delta_! \simeq \Delta_! \Delta^*\pr_1^* \xRightarrow{c^!_{\Delta}} \pr_1^*.
		\end{align*}
		\item \label{it:AdjointNormMap} Assume that $\Cc$ admits finite $P$-coproducts, so that the pullback square \eqref{eq:PullbackDiagramNormMap} gives a left base change equivalence $p^* p_! \simeq {\pr_1}_! \pr_2^*$. We define the \textit{adjoint norm transformation} $\Nmadj_p\colon p^*p_! \Rightarrow \id$ \textit{of $p$ in $\Cc$} as the composite
		\begin{align*}
			\Nmadj_p\colon p^* p_! \overset{l.b.c.}{\simeq} {\pr_1}_! \pr_2^* \xRightarrow{{\pr_1}_!\alpha}
			{\pr_1}_!\pr_1^* \xRightarrow{c^!_{\pr_1}} \id.
		\end{align*}
		\item \label{it:DualAdjointNormMap} Assume that $\Cc$ admits finite $P$-products, so that the pullback square \eqref{eq:PullbackDiagramNormMap} gives a right base change equivalence $p^* p_* \simeq {\pr_2}_* \pr_1^*$. We define the \textit{dual adjoint norm transformation} $\Nmadjdual_p\colon \id \Rightarrow p^*p_*$ \textit{of $p$ in $\Cc$} as the composite
		\begin{align*}
			\Nmadjdual_p\colon \id \xRightarrow{u^*_{\pr_2}}
			{\pr_2}_*\pr_2^* \xRightarrow{{\pr_2}_*\alpha}
			{\pr_2}_*\pr_1^* \overset{r.b.c.}{\simeq} p^*p_*.
		\end{align*}
		\item \label{it:NormMap} Assume that $\Cc$ admits both finite $P$-products and finite $P$-coproducts. We define the \textit{norm transformation of $p$ in $\Cc$}
		\begin{align*}
			\Nm_p\colon p_! \implies p_*
		\end{align*}
		as the map adjoint to the adjoint norm transformation $\Nmadj_p\colon p^*p_! \Rightarrow \id$.
	\end{enumerate}
	We will sometimes write $\Nmadj^{\Cc}_p$, $\Nmadjdual^{\Cc}_p$ or $\Nm^{\Cc}_p$ to emphasize the dependence on $\Cc$.
\end{construction}

\begin{remark}\label{rem:concise_norm_maps}
	Unwinding the definitions, the map $\Nmadj_p\colon p^*p_! \Rightarrow \id$ may be given more directly as the composite
	\[
	p^* p_! \overset{l.b.c.}{\simeq} {\pr_1}_! \pr_2^* \xRightarrow{u^*_{\Delta}} {\pr_1}_! \Delta_* \Delta^* \pr_2^* \overset{\Nm^{-1}_{\Delta}}{\simeq} {\pr_1}_! \Delta_! \Delta^* \pr_2^* \simeq \id_{\Cc(A)}.
	\]
	Similarly, the map $\Nmadjdual_p\colon \id \Rightarrow p^*p_*$ unwinds to the following composite:
	\[
	\id_{\Cc(A)} \simeq {\pr_2}_*\Delta_*\Delta^*\pr_1^* \overset{\Nm^{-1}_{\Delta}}{\simeq} {\pr_2}_*\Delta_!\Delta^*\pr_1^* \xRightarrow{c^!_{\Delta}} {\pr_2}_* \pr_1^* \overset{r.b.c.}{\simeq} p^* p_*.
	\]
	The description of the adjoint norm map $\Nmadj$ given above is precisely the definition of the map $\nu^{(0)}_p\colon p^*p_! \Rightarrow \id$ of \cite{hopkins2013ambidexterity}*{Construction 4.1.8}, applied to the Beck-Chevalley fibration $\smallint \Cc \to \PSh(T)$ classified by the functor $\Cc\colon \PSh(T)\catop \to \Cat_{\infty}$. In particular, the norm map $\Nm_p\colon p_! \to p_*$ defined above agrees with the norm map $\Nm_p$ of \cite{hopkins2013ambidexterity}*{Construction 4.1.12}.
\end{remark}

\begin{remark}
	Let $(f\colon A \to B) \in \ulfinPsets(B)$ be a morphism in $\PSh(T)$ which happens to be a disjoint summand inclusion. Then the norm map $\Nm_f\colon f_! \Rightarrow f_*$ of \Cref{cstr:NormMap} agrees with the map $\Nm_f\colon f_! \Rightarrow f_*$ constructed in \cref{lem:NormMapDisjointSummandInclusions}.
\end{remark}

The map $\alpha\colon \pr_2^* \Rightarrow \pr_1^*$ defined in \Cref{cstr:NormMap}\eqref{it:DefAlpha} may be thought of as some kind of `diagonal matrix': as the next lemma shows, it restricts to the identity when restricted along the diagonal $\Delta\colon A \hookrightarrow A \times_B A$, and restricts to the zero map on the complement of the diagonal.

\begin{lemma}
	\label{lem:AlphaIsDiagonalMatrix}
	Let $\Cc$ be a pointed $T$-$\infty$-category and let $(p\colon A \to B) \in \ulfinPsets(B)$. Let $j\colon C \hookrightarrow A \times_B A$ denote the disjoint complement of the diagonal inclusion $\Delta\colon A \hookrightarrow A \times_B A$. Then the following hold:
	\begin{enumerate}[(1)]
		\item The composite $\id_{\Cc(A)} \simeq \Delta^*\pr_2^* \xRightarrow{\Delta^*\alpha} \Delta^*\pr_1^* \simeq \id_{\Cc(A)}$ is homotopic to the identity transformation.
		\item The map $j^*\alpha\colon j^*\pr_2^* \Rightarrow j^*\pr_1^*$ is the zero transformation, in the sense that it factors through the zero functor $0\colon \Cc(A) \to \Cc(C)$.
	\end{enumerate}
\end{lemma}
\begin{proof}
	The proof of (1) follows from the following commutative diagram:
	% https://q.uiver.app/?q=WzAsOSxbMCwwLCJcXERlbHRhXipcXHByXzJeKiJdLFswLDEsIlxcRGVsdGFeKlxcRGVsdGFfKlxcRGVsdGFeKlxccHJfMl4qIl0sWzQsMSwiXFxEZWx0YV4qXFxEZWx0YV8hXFxEZWx0YV4qXFxwcl8yXioiXSxbMSwxLCJcXERlbHRhXipcXERlbHRhXyoiXSxbMywxLCJcXERlbHRhXipcXERlbHRhXyEiXSxbMSwwLCJcXERlbHRhXipcXHByXzJeKiJdLFs0LDAsIlxcRGVsdGFeKlxccHJfMV4qIl0sWzMsMCwiXFxEZWx0YV4qXFxwcl8xXioiXSxbMiwwLCJcXGlkIl0sWzMsNCwiXFxObV57LTF9X3tcXERlbHRhfSIsMl0sWzAsNSwiIiwwLHsibGV2ZWwiOjIsInN0eWxlIjp7ImhlYWQiOnsibmFtZSI6Im5vbmUifX19XSxbNyw2LCIiLDAseyJsZXZlbCI6Miwic3R5bGUiOnsiaGVhZCI6eyJuYW1lIjoibm9uZSJ9fX1dLFswLDEsInVfe1xcRGVsdGF9XioiLDJdLFsxLDUsImNfe1xcRGVsdGF9XioiLDFdLFs3LDIsInVfe1xcRGVsdGF9XiEiLDFdLFsyLDYsImNfe1xcRGVsdGF9XiEiLDJdLFsxLDMsIlxcc2ltZXEiLDJdLFs0LDIsIlxcc2ltZXEiLDJdLFs4LDcsIlxcc2ltZXEiXSxbNSw4LCJcXHNpbWVxIl0sWzMsOCwiY197XFxEZWx0YX1eKiIsMV0sWzgsNCwidV97XFxEZWx0YX1eISIsMV1d
	\[\begin{tikzcd}
		{\Delta^*\pr_2^*} & {\Delta^*\pr_2^*} & \id & {\Delta^*\pr_1^*} & {\Delta^*\pr_1^*} \\
		{\Delta^*\Delta_*\Delta^*\pr_2^*} & {\Delta^*\Delta_*} && {\Delta^*\Delta_!} & {\Delta^*\Delta_!\Delta^*\pr_2^*.}
		\arrow["{\Nm^{-1}_{\Delta}}"', from=2-2, to=2-4]
		\arrow[Rightarrow, no head, from=1-1, to=1-2]
		\arrow[Rightarrow, no head, from=1-4, to=1-5]
		\arrow["{u_{\Delta}^*}"', from=1-1, to=2-1]
		\arrow["{c_{\Delta}^*}"{description}, from=2-1, to=1-2]
		\arrow["{u_{\Delta}^!}"{description}, from=1-4, to=2-5]
		\arrow["{c_{\Delta}^!}"', from=2-5, to=1-5]
		\arrow["\simeq"', from=2-1, to=2-2]
		\arrow["\simeq"', from=2-4, to=2-5]
		\arrow["\simeq", from=1-3, to=1-4]
		\arrow["\simeq", from=1-2, to=1-3]
		\arrow["{c_{\Delta}^*}"{description}, from=2-2, to=1-3]
		\arrow["{u_{\Delta}^!}"{description}, from=1-3, to=2-4]
	\end{tikzcd}\]
	The triangles on the two sides commute by the triangle identity, the rhombi commute by naturality and the triangle in the middle commutes by the defining property of the norm map $\Nm_{\Delta}$ of \cref{lem:NormMapDisjointSummandInclusions}.

	For (2), note that by definition of $\alpha$ the map $j^*\alpha$ factors through the functor $j^*\Delta_*$. Since coproducts are disjoint in $\PSh(T)$, the fiber product $C \times_{A \times_B A} A$ is the empty presheaf. It then follows from base change that the functor $j^*\Delta_*$ factors through the $\infty$-category $\Cc(\emptyset) \simeq *$, which forces it to be the zero functor.
\end{proof}

\begin{remark}
	\label{rmk:NormMapBalmerDellAmbrogio}
	In the setting of Mackey 2-functors, Balmer and Dell'Ambrogio \cite{balmerAmbrogio_Mackey}*{Theorem~3.3.4} have produced a similar transformation $\Theta_i\colon i_! \Rightarrow i_*$ for $i$ a faithful map of groupoids, i.e.~ a morphism in $\mathbb{F}^{\Orb}_{\Glo}$. It follows from \Cref{lem:AlphaIsDiagonalMatrix} and \cite{balmerAmbrogio_Mackey}*{Proposition~3.2.1} that the transformation $\Nm_i\colon i_! \Rightarrow i_*$ of \Cref{cstr:NormMap} specializes to the transformation $\Theta_i$ of Balmer and Dell'Ambrogio in the case $T = \Glo$ and $P = \Orb$. In particular, if $\Cc$ is a pointed global $\infty$-category admitting finite $\Orb$-(co)products, it follows from \cite{balmerAmbrogio_Mackey}*{Theorem~3.4.2} that the norm maps $\Nm_i$ are equivalences for every faithful map of groupoids $i\colon H \rightarrow G$ if and only if there exist abstract equivalences $i_! \simeq i_*$ for every such $i$.
\end{remark}

\subsection{Properties of norm maps}
We will next establish a variety of results about the calculus of norm maps.

To start with, we address the obvious asymmetry in the construction of the norm map: we could just as well have considered the map $p_! \Rightarrow p_*$ adjoint to the \textit{dual} adjoint norm map $\Nmadjdual_p\colon \id \Rightarrow p^*p_*$. The following lemma shows that these two maps agree.

\begin{lemma}
	\label{lem:NormMapEqualsDualNormMap}
	Assume that $\Cc$ is a pointed $T$-$\infty$-category which admits both finite $P$-products and finite $P$-coproducts. For every $(p\colon A \to B) \in \ulfinPsets(B)$, the maps $\Nmadj_p\colon p^*p_! \Rightarrow \id$ and $\Nmadjdual_p\colon \id \Rightarrow p^*p_*$ adjoin to the same map $\Nm_p\colon p_! \Rightarrow p_*$.
\end{lemma}
\begin{proof}
	We have to show that dual adjoint norm map $\Nmadjdual_p$ is the total mate of the adjoint norm map $\Nmadj_p$. A mundane exercise in 2-category theory shows that the total mate of the Beck-Chevalley equivalence $p^* p_! \simeq {\pr_1}_! \pr_2^*$ is the Beck-Chevalley equivalence ${\pr_2}_* \pr_1^* \simeq p^* p_*$. Furthermore, it follows directly from the triangle identity that the total mate of the composite
	\[
	{\pr_1}_! \pr_2^* \xRightarrow{{\pr_1}_!\alpha} {\pr_1}_! \pr_1^* \xRightarrow{c^!_{\pr_1}} \id
	\]
	is given by the composite
	\[
	\id \xRightarrow{u^*_{\pr_2}} {\pr_2}_* \pr_2^* \xRightarrow{{\pr_2}_*\alpha} {\pr_2}_* \pr_1^*.
	\]
	Since the total mate of a composite of transformations is given by composing in opposite order the individual total mates of these transformations, this finishes the proof.
\end{proof}

The norm map $\Nm_p$ can be written in terms of the double Beck-Chevalley map $p_!{\pr_2}_* \Rightarrow p_*{\pr_1}_!$ associated to the pullback square \eqref{eq:PullbackDiagramNormMap}:

\begin{lemma}
	\label{lem:NormMapInTermsOfExchangeMap}
	Assume that $\Cc$ is a pointed $T$-$\infty$-category which admits both finite $P$-products and finite $P$-coproducts, and let $(p\colon A \to B) \in \ulfinPsets(B)$. Then the norm map $\Nm_p$ is homotopic to the composite
	\[
	p_! \simeq p_!{\pr_2}_*\Delta_* \xrightarrow{} p_*{\pr_1}_!\Delta_* \xrightarrow{\Nm^{-1}_{\Delta}} p_*{\pr_1}_!\Delta_! \simeq p_*.
	\]
\end{lemma}
\begin{proof}
	By adjunction, it suffices to show that the adjoint norm map $\Nmadj_p\colon p^*p_! \to \id$ is given by the composite
	\[
	p^*p_! \simeq p^*p_!{\pr_2}_*\Delta_* \xrightarrow{} p^*p_*{\pr_1}_!\Delta_* \xrightarrow{\Nm^{-1}_{\Delta}} p^*p_*{\pr_1}_!\Delta_! \simeq p^*p_* \xrightarrow{c^*_p} \id.
	\]
	This follows from the following commutative diagram:
	\[
	% https://q.uiver.app/?q=WzAsMTQsWzAsMCwicF4qcF8hIl0sWzEsMCwicF4qcF8he1xccHJfMn1fKlxcRGVsdGFfKiJdLFswLDEsIntcXHByXzF9XyFcXHByXzJeKiJdLFsxLDEsIntcXHByXzF9XyFcXHByXzJeKntcXHByXzJ9XypcXERlbHRhXyoiXSxbMiwwLCJwXipwXyp7XFxwcl8xfV8hXFxEZWx0YV8qIl0sWzMsMCwicF4qcF8qe1xccHJfMX1fIVxcRGVsdGFfISJdLFszLDEsIntcXHByXzF9XyFcXERlbHRhXyEiXSxbNCwwLCJwXipwXyoiXSxbNCwxLCJcXGlkIl0sWzIsMSwie1xccHJfMX1fIVxcRGVsdGFfKiJdLFswLDIsIntcXHByXzF9XyFcXHByXzJeKiJdLFsxLDIsIntcXHByXzF9XyFcXERlbHRhXypcXERlbHRhXipcXHByXzJeKiJdLFszLDIsIntcXHByXzF9XyFcXERlbHRhXyFcXERlbHRhXipcXHByXzFeKiJdLFs0LDIsIntcXHByXzF9XyFcXHByXzFeKiJdLFswLDEsIlxcc2ltZXEiXSxbMCwyLCJsLmIuYy4iLDJdLFsxLDMsImwuYi5jLiIsMl0sWzIsMywiXFxzaW1lcSJdLFsyLDEwLCIiLDIseyJsZXZlbCI6Miwic3R5bGUiOnsiaGVhZCI6eyJuYW1lIjoibm9uZSJ9fX1dLFszLDksImNeKl97XFxwcl8yfSJdLFsxMSw5LCJcXHNpbWVxIiwyXSxbOSw2LCJcXE5tXnstMX1fe1xcRGVsdGF9Il0sWzQsNSwiXFxObV57LTF9X3tcXERlbHRhfSJdLFsxLDRdLFs0LDksImNeKl9wIiwyXSxbNSw2LCJjXipfcCIsMl0sWzcsOCwiY14qX3AiLDJdLFs1LDcsIlxcc2ltZXEiXSxbNiw4LCJcXHNpbWVxIl0sWzEwLDExLCJ1Xipfe1xcRGVsdGF9Il0sWzEyLDYsIlxcc2ltZXEiLDJdLFsxMyw4LCJjXiFfe1xccHJfMX0iLDJdLFsxMiwxMywiY14hX3tcXERlbHRhfSJdLFsxMCwxMywie1xccHJfMX1fIVxcYWxwaGEiLDIseyJjdXJ2ZSI6NH1dLFszLDQsIigxKSIsMSx7InN0eWxlIjp7ImJvZHkiOnsibmFtZSI6Im5vbmUifSwiaGVhZCI6eyJuYW1lIjoibm9uZSJ9fX1dLFs2LDEzLCIoMykiLDEseyJzdHlsZSI6eyJib2R5Ijp7Im5hbWUiOiJub25lIn0sImhlYWQiOnsibmFtZSI6Im5vbmUifX19XSxbMTAsOSwiKDIpIiwxLHsibGFiZWxfcG9zaXRpb24iOjMwLCJvZmZzZXQiOi0yLCJzdHlsZSI6eyJib2R5Ijp7Im5hbWUiOiJub25lIn0sImhlYWQiOnsibmFtZSI6Im5vbmUifX19XV0=
	\hskip-1.8483pt\hfuzz=1.8483pt\begin{tikzcd}[cramped]
		{p^*p_!} & {p^*p_!{\pr_2}_*\Delta_*} & {p^*p_*{\pr_1}_!\Delta_*} & {p^*p_*{\pr_1}_!\Delta_!} & {p^*p_*} \\
		{{\pr_1}_!\pr_2^*} & {{\pr_1}_!\pr_2^*{\pr_2}_*\Delta_*} & {{\pr_1}_!\Delta_*} & {{\pr_1}_!\Delta_!} & \id \\
		{{\pr_1}_!\pr_2^*} & {{\pr_1}_!\Delta_*\Delta^*\pr_2^*} && {{\pr_1}_!\Delta_!\Delta^*\pr_1^*} & {{\pr_1}_!\pr_1^*}
		\arrow["\simeq", from=1-1, to=1-2]
		\arrow["{l.b.c.}"', from=1-1, to=2-1]
		\arrow["{l.b.c.}"', from=1-2, to=2-2]
		\arrow["\simeq", from=2-1, to=2-2]
		\arrow[Rightarrow, no head, from=2-1, to=3-1]
		\arrow["{c^*_{\pr_2}}", from=2-2, to=2-3]
		\arrow["\simeq"', from=3-2, to=2-3]
		\arrow["{\Nm^{-1}_{\Delta}}", from=2-3, to=2-4]
		\arrow["{\Nm^{-1}_{\Delta}}", from=1-3, to=1-4]
		\arrow[from=1-2, to=1-3]
		\arrow["{c^*_p}"', from=1-3, to=2-3]
		\arrow["{c^*_p}"', from=1-4, to=2-4]
		\arrow["{c^*_p}"', from=1-5, to=2-5]
		\arrow["\simeq", from=1-4, to=1-5]
		\arrow["\simeq", from=2-4, to=2-5]
		\arrow["{u^*_{\Delta}}", from=3-1, to=3-2]
		\arrow["\simeq"', from=3-4, to=2-4]
		\arrow["{c^!_{\pr_1}}"', from=3-5, to=2-5]
		\arrow["{c^!_{\Delta}}", from=3-4, to=3-5]
		\arrow["{{\pr_1}_!\alpha}"', bend right = 10, from=3-1, to=3-5]
		\arrow["{(1)}"{description}, draw=none, from=2-2, to=1-3]
		\arrow["{(3)}"{description}, draw=none, from=2-4, to=3-5]
		\arrow["{(2)}"{description, pos=0.3}, shift left=2, draw=none, from=3-1, to=2-3]
	\end{tikzcd}
	\]
	The unlabeled squares commute by naturality. Commutativity of (1) is by the triangle identity, while commutativity of (2) and (3) follows from the equivalence $\pr_1 \circ \Delta \simeq \id \simeq \pr_2 \circ \Delta$ and the fact that the (co)unit of a composite of adjunctions is the composite of the individual (co)units.
\end{proof}

As was shown by Hopkins and Lurie \cite{hopkins2013ambidexterity}, the norm maps behave well under composition and base change of morphisms in $\ulfinPsets$.

\begin{proposition}[\cite{hopkins2013ambidexterity}*{Proposition~4.2.1}]
	\label{prop:adjNormsVsBaseChange}
	Assume that $\Cc$ is a pointed $T$-$\infty$-category which admits finite $P$-coproducts. Consider a pullback square
	\[
	\begin{tikzcd}
		A' \dar[swap]{p'} \rar{g_{A}} \drar[pullback] & A \dar{p} \\
		B' \rar{g_B} & B
	\end{tikzcd}
	\]
	in $\PSh(T)$ such that $p \in \ulfinPsets(B)$ and (hence) $p' \in \ulfinPsets(B')$.
	Then there is a commutative diagram
	\[
	\begin{tikzcd}[anchor=south,baseline=.65em]
		{p'}^*p'_!g_A^* \dar[swap]{\Nmadj_{p'}} \rar{l.b.c.}[swap]{\sim} & {p'}^*g_B^*p_! \rar{\sim} & g_A^*p^*p_! \dar{\Nmadj_p} \\
		g^*_A \ar[equal]{rr} && g_A^*.
	\end{tikzcd}\qednow
	\]
\end{proposition}

\begin{corollary}[\cite{hopkins2013ambidexterity}*{Remark~4.2.3}]\label{cor:NormsVsBaseChange}
	In the situation of \Cref{prop:adjNormsVsBaseChange}, assume that $\Cc$ furthermore admits finite $P$-products. Then the composite
	\[
	p'_!g_A^* \overset{l.b.c.}{\simeq} g_B^*p_! \xrightarrow{g_B^*\Nm_p} g_B^*p_* \overset{r.b.c.}{\simeq} p'_*g_A^*
	\]
	is homotopic to the map $\Nm_{p'}g_A^*$.\qed
\end{corollary}

\begin{proposition}[\cite{hopkins2013ambidexterity}*{Proposition~4.2.2}]
	\label{prop:adjNormsVsComposition}
	Assume that $\Cc$ is a pointed $T$-$\infty$-category which admits finite $P$-coproducts. Let $(p\colon A \to B) \in \ulfinPsets(B)$ and $(q\colon B \to C) \in \ulfinPsets(C)$, so that also $(qp\colon A \to C) \in \ulfinPsets(C)$. Then the adjoint norm map $\Nmadj_{qp}$ is homotopic to the composite
	\[
	(qp)^*(qp)_! \simeq p^*q^*q_!p_! \xrightarrow{\Nmadj_q} p^*p_! \xrightarrow{\Nmadj_p} \id.\qednow
	\]
\end{proposition}

\begin{corollary}[\cite{hopkins2013ambidexterity}*{Remark~4.2.4}]\label{cor:NormsVsComposition}
	In the situation of \Cref{prop:adjNormsVsComposition}, assume that $\Cc$ furthermore admits finite $P$-products. Then the composite transformation
	\[
	(qp)_! \simeq q_!p_! \xrightarrow{\Nm_q} q_*p_! \xrightarrow{\Nm_p} q_*p_* \simeq (qp)_*
	\]
	is homotopic to the norm map $\Nm_{qp}$.\qed
\end{corollary}

The norm maps are also suitably functorial in the $T$-$\infty$-category $\Cc$: as we will now show, any pointed $T$-functor $G\colon \Cc \to \Dd$ transforms norm maps in $\Cc$ into norm maps in $\Dd$.

\begin{lemma}\label{lem:alpha-compatible}
	Let $G\colon \Cc \to \Dd$ be a pointed $T$-functor of pointed $T$-categories and let $(p\colon A \to B) \in \ulfinPsets(B)$. Then the diagram
	\begin{equation*}
		\begin{tikzcd}
			\pr_2^*G\arrow[r, "\alpha G"]\arrow[d,"\simeq"'] & \pr_1^*G\arrow[d, "\simeq"]\\
			G\pr_2^*\arrow[r, "G\alpha"'] & G\pr_1^*
		\end{tikzcd}
	\end{equation*}
	of transformations between functors $\mathcal C(A)\to\mathcal D(A\times_BA)$ commutes.
\end{lemma}
\begin{proof}
	Spelling out the definition of $\alpha$, this is a direct consequence of the following three commutative diagrams:
	\[
	% https://q.uiver.app/?q=WzAsNCxbMCwwLCJHIl0sWzEsMCwiXFxEZWx0YV8qXFxEZWx0YV4qRyJdLFsxLDEsIlxcRGVsdGFfKkdcXERlbHRhXioiXSxbMCwxLCJHXFxEZWx0YV8qXFxEZWx0YV4qIl0sWzAsMSwidV97XFxEZWx0YX1eKkciXSxbMSwyLCJcXHNpbWVxIl0sWzAsMywiR3VeKl97XFxEZWx0YX0iLDJdLFszLDIsIlxcdGV4dHtCQ31fKiIsMl1d
	\begin{tikzcd}[cramped]
		G & {\Delta_*\Delta^*G} \\
		{G\Delta_*\Delta^*} & {\Delta_*G\Delta^*,}
		\arrow["{u_{\Delta}^*G}", from=1-1, to=1-2]
		\arrow["\simeq", from=1-2, to=2-2]
		\arrow["{Gu^*_{\Delta}}"', from=1-1, to=2-1]
		\arrow["{\text{BC}_*}"', from=2-1, to=2-2]
	\end{tikzcd}
	\qquad
	\begin{tikzcd}[cramped]
		\Delta_!G\arrow[d, "\text{BC}_!"']\arrow[r, "\Nm_{\Delta} G"] &[1em] \Delta_*G\\
		G\Delta_!\arrow[r, "G\Nm_{\Delta}"'] & G\Delta_*, \arrow[u, "\text{BC}_*"']
	\end{tikzcd}
	\qquad
	% https://q.uiver.app/?q=WzAsNCxbMCwwLCJcXERlbHRhXyEgXFxEZWx0YV4qRyJdLFsxLDAsIkciXSxbMCwxLCJcXERlbHRhXyFHXFxEZWx0YV4qIl0sWzEsMSwiR1xcRGVsdGFfISBcXERlbHRhXioiXSxbMCwxLCJjXiFfe1xcRGVsdGF9RyJdLFszLDEsIkdjXiFfe1xcRGVsdGF9IiwyXSxbMiwzLCJcXHRleHR7QkN9XyEiLDJdLFsyLDAsIlxcc2ltZXEiXV0=
	\begin{tikzcd}[cramped]
		{\Delta_! \Delta^*G} & G \\
		{\Delta_!G\Delta^*} & {G\Delta_! \Delta^*.}
		\arrow["{c^!_{\Delta}G}", from=1-1, to=1-2]
		\arrow["{Gc^!_{\Delta}}"', from=2-2, to=1-2]
		\arrow["{\text{BC}_!}"', from=2-1, to=2-2]
		\arrow["\simeq", from=2-1, to=1-1]
	\end{tikzcd}
	\]
	The left and right squares commute by definition of the Beck-Chevalley maps, using the triangle identities. The fact that the middle square commutes follows directly from pointedness of $G$ and the construction of $\Nm_{\Delta}$ in \cref{lem:NormMapDisjointSummandInclusions}.
\end{proof}

\begin{lemma}
	\label{lem:NormVsBeckChevalley}
	Let $G\colon \Cc\rightarrow \Dd$ be a pointed $T$-functor between two pointed $T$-$\infty$-categories which admit finite $P$-coproducts. Then for every $(p\colon A \to B) \in \ulfinPsets(B)$, the diagram
	% https://q.uiver.app/?q=WzAsNSxbMSwwLCJwXipHX0JwXyEiXSxbMiwwLCJHX0FwXipwXyEiXSxbMCwxLCJHX0EiXSxbMCwwLCJwXipwXyFHX0EiXSxbMiwxLCJHX0EiXSxbMCwxLCJcXHNpbWVxIl0sWzMsMCwiXFxCQ18hIl0sWzIsNCwiIiwyLHsibGV2ZWwiOjIsInN0eWxlIjp7ImhlYWQiOnsibmFtZSI6Im5vbmUifX19XSxbMSw0LCJHX0FcXE5tYWRqX3AiXSxbMywyLCJcXE5tYWRqX3BHX0EiLDJdXQ==
	\[\begin{tikzcd}
		{p^*p_!G_A} & {p^*G_Bp_!} & {G_Ap^*p_!} \\
		{G_A} && {G_A}
		\arrow["\simeq", from=1-2, to=1-3]
		\arrow["{\BC_!}", from=1-1, to=1-2]
		\arrow[Rightarrow, no head, from=2-1, to=2-3]
		\arrow["{G_A\Nmadj^{\Cc}_p}", from=1-3, to=2-3]
		\arrow["{\Nmadj^{\Dd}_pG_A}"', from=1-1, to=2-1]
	\end{tikzcd}\]
	commutes.
\end{lemma}

\begin{proof}
	Consider the diagram
	% https://q.uiver.app/?q=WzAsMTEsWzEsMCwicF4qR19CcF8hIl0sWzIsMCwiR19BcF4qcF8hIl0sWzAsMywiR19BIl0sWzAsMCwicF4qcF8hR19BIl0sWzAsMSwie1xccHJfMX1fISBcXHByXzJeKkdfQSJdLFsyLDEsIkdfQXtcXHByXzF9XyEgXFxwcl8yXioiXSxbMSwxLCJ7XFxwcl8xfV8hR197QSBcXHRpbWVzX0IgQX1cXHByXzJeKiJdLFswLDIsIntcXHByXzF9XyEgXFxwcl8xXipHX0EiXSxbMSwyLCJ7XFxwcl8xfV8hR197QSBcXHRpbWVzX0IgQX1cXHByXzFeKiJdLFsyLDIsIkdfQXtcXHByXzF9XyEgXFxwcl8xXioiXSxbMiwzLCJHX0EiXSxbMCwxLCJcXHNpbWVxIl0sWzMsMCwiXFxCQ18hIl0sWzMsNCwibC5iLmMuIl0sWzEsNSwibC5iLmMuIiwyXSxbNiw1LCJcXEJDXyEiXSxbNCw2LCJcXHNpbWVxIl0sWzQsNywie1xccHJfMX1fIVxcYWxwaGEgR19BIl0sWzUsOSwiR19Be1xccHJfMX1fISBcXGFscGhhIiwyXSxbNiw4LCJ7XFxwcl8xfV8hR197QSBcXHRpbWVzX0IgQX1cXGFscGhhIl0sWzgsOSwiXFxCQ18hIl0sWzcsOCwiXFxzaW1lcSJdLFs3LDIsImNeIV97XFxwcl8xfUdfQSJdLFs5LDEwLCJHX0FjXiFfe1xccHJfMX0iLDJdLFsyLDEwLCIiLDIseyJsZXZlbCI6Miwic3R5bGUiOnsiaGVhZCI6eyJuYW1lIjoibm9uZSJ9fX1dLFsxLDEwLCJHX0FcXE5tYWRqX3AiXSxbMywyLCJcXE5tYWRqX3BHX0EiLDJdXQ==
	\[\begin{tikzcd}
		{p^*p_!G_A} & {p^*G_Bp_!} & {G_Ap^*p_!} \\
		{{\pr_1}_! \pr_2^*G_A} & {{\pr_1}_!G_{A \times_B A}\pr_2^*} & {G_A{\pr_1}_! \pr_2^*} \\
		{{\pr_1}_! \pr_1^*G_A} & {{\pr_1}_!G_{A \times_B A}\pr_1^*} & {G_A{\pr_1}_! \pr_1^*} \\
		{G_A} && {G_A.}
		\arrow["\simeq", from=1-2, to=1-3]
		\arrow["{\BC_!}", from=1-1, to=1-2]
		\arrow["{l.b.c.}", from=1-1, to=2-1]
		\arrow["{l.b.c.}"', from=1-3, to=2-3]
		\arrow["{\BC_!}", from=2-2, to=2-3]
		\arrow["\simeq", from=2-1, to=2-2]
		\arrow["{{\pr_1}_!\alpha G_A}", from=2-1, to=3-1]
		\arrow["{G_A{\pr_1}_! \alpha}"', from=2-3, to=3-3]
		\arrow["{{\pr_1}_!G_{A \times_B A}\alpha}"{description}, from=2-2, to=3-2]
		\arrow["{\BC_!}", from=3-2, to=3-3]
		\arrow["\simeq", from=3-1, to=3-2]
		\arrow["{c^!_{\pr_1}G_A}", from=3-1, to=4-1]
		\arrow["{G_Ac^!_{\pr_1}}"', from=3-3, to=4-3]
		\arrow[Rightarrow, no head, from=4-1, to=4-3]
		\arrow["{G_A\Nmadj^{\Cc}_p}", bend left = 50, from=1-3, to=4-3]
		\arrow["{\Nmadj^{\Dd}_pG_A}"', bend right = 50, from=1-1, to=4-1]
	\end{tikzcd}\]
	We are interested in the outer square. The right middle square commute by naturality. The left middle square commutes by \cref{lem:alpha-compatible}. The bottom rectangle commutes by definition of the Beck-Chevalley map, using the triangle identity. Finally, the upper rectangle commutes as the two composites are the Beck-Chevalley transformations associated to the following two equivalent composite squares:
	% https://q.uiver.app/?q=WzAsNixbMCwwLCJcXENjJyhCKSJdLFswLDEsIlxcQ2MnKEEpIl0sWzEsMSwiXFxDYyhBKSJdLFsxLDAsIlxcQ2MoQikiXSxbMiwwLCJcXENjKEEpIl0sWzIsMSwiXFxDYyhBIFxcdGltZXNfQiBBKSJdLFswLDMsIkdfQiJdLFsxLDIsIkdfQSJdLFswLDEsInBeKiIsMl0sWzMsMiwicF4qIl0sWzQsNSwiXFxybXtwcn1fMV4qIl0sWzIsNSwiXFxwcl8yXioiXSxbMyw0LCJwXioiXV0=
	\[\hskip-22.111pt\hfuzz=23pt\begin{tikzcd}[cramped]
		{\Cc(B)} & {\Dd(B)} & {\Dd(A)} \\
		{\Cc(A)} & {\Dd(A)} & {\Dd(A \times_B A)}
		\arrow["{G_B}", from=1-1, to=1-2]
		\arrow["{G_A}", from=2-1, to=2-2]
		\arrow["{p^*}"', from=1-1, to=2-1]
		\arrow["{p^*}", from=1-2, to=2-2]
		\arrow["{\mathrm{pr}_1^*}", from=1-3, to=2-3]
		\arrow["{\pr_2^*}", from=2-2, to=2-3]
		\arrow["{p^*}", from=1-2, to=1-3]
	\end{tikzcd}
	\qquad \text{ and } \qquad
	% https://q.uiver.app/?q=WzAsNixbMCwwLCJcXENjJyhCKSJdLFswLDEsIlxcQ2MnKEEpIl0sWzEsMSwiXFxDYycoQVxcdGltZXNfQiBBKSJdLFsxLDAsIlxcQ2MnKEEpIl0sWzIsMCwiXFxDYyhBKSJdLFsyLDEsIlxcQ2MoQSBcXHRpbWVzX0IgQSkiXSxbMCwzLCJwXioiXSxbMSwyLCJcXHByXzJeKiJdLFswLDEsInBeKiIsMl0sWzMsMiwiXFxybXtwcn1fMV4qIl0sWzQsNSwiXFxybXtwcn1fMV4qIl0sWzIsNSwiR197QSBcXHRpbWVzX0IgQX0iXSxbMyw0LCJHX0EiXV0=
	\begin{tikzcd}[cramped]
		{\Cc(B)} & {\Cc(A)} & {\Dd(A)} \\
		{\Cc(A)} & {\Cc(A\times_B A)} & {\Dd(A \times_B A).}
		\arrow["{p^*}", from=1-1, to=1-2]
		\arrow["{\pr_2^*}", from=2-1, to=2-2]
		\arrow["{p^*}"', from=1-1, to=2-1]
		\arrow["{\mathrm{pr}_1^*}", from=1-2, to=2-2]
		\arrow["{\mathrm{pr}_1^*}", from=1-3, to=2-3]
		\arrow["{G_{A \times_B A}}", from=2-2, to=2-3]
		\arrow["{G_A}", from=1-2, to=1-3]
	\end{tikzcd}\]
	This finishes the proof.
\end{proof}

We end the subsection with the the following technical lemma, needed for the proof of \cref{prop:CharacterizationPSemiadditivity} below. We recommend the reader skip this lemma on first reading.

\begin{lemma}
	\label{lem:TechnicalLemmaRecognitionPSemiadditivity}
	Let $\Cc$ be a pointed $T$-$\infty$-category which admits finite $P$-products. Let $(p\colon A \to B) \in \ulfinPsets(B)$, and assume that $p$ admits a section $s\colon B \to A$ which is a disjoint summand inclusion. Then the composite
	\[
	s^* \iso p_*s_*s^* \xrightarrow{p_*\Nm_s^{-1}s^*} p_*s_!s^* \xrightarrow{p_*c^!_s} p_*
	\]
	is homotopic to the composite
	\[
	s^* \xrightarrow{s^*\Nmadjdual_p} s^*p^*p_* \simeq p_*.
	\]
\end{lemma}
\begin{proof}
	Recall from \cref{rem:concise_norm_maps} that the map $\Nmadjdual_p\colon \id \to p^*p_*$ is given by the following composite:
	\[
	\id \simeq {\pr_2}_*\Delta_*\Delta^*\pr_1^* \xrightarrow{\Nm_{\Delta}^{-1}} {\pr_2}_*\Delta_!\Delta^*\pr_1^* \xrightarrow{c^!_{\Delta}} {\pr_2}_*\pr_1^* \overset{r.b.c.}{\simeq} p^*p_*.
	\]
	We thus see that the composite $s^* \xrightarrow{s^*\Nmadjdual_p} s^*p^*p_* \simeq p_*$ is given by the composite along the left, bottom and right in the following large diagram:
	\[\hskip-9.21pt\hfuzz=10pt
	\begin{tikzcd}[cramped]
		{s^*} & {p_*s_*s^*} & {p_*s_!s^*} & {p_*} \\
		&& {p_*s_!s^*(1,sp)^*\pr_1^*} & {p_*(1,sp)^*\pr_1^*} & {p_*} \\
		{s^*} & {p_*s_*s^*\Delta^*\pr_1^*} & {p_*s_!s^*\Delta^*\pr_1^*} \\
		& {p_*(1,sp)^*\Delta_*\Delta^*\pr_1^*} & {p_*(1,sp)^*\Delta_!\Delta^*\pr_1^*} & {p_*(1,sp)^*\pr_1^*} & {p_*} \\
		{s^*} & {s^*(\pr_2)_*\Delta_*\Delta^*\pr_1^*} & {s^*(\pr_2)_*\Delta_!\Delta^*\pr_1^*} & {s^*(\pr_2)_*\pr_1^*} & {s^*p^*p_*}
		\arrow["{\Nm_{\Delta}^{-1}}", from=5-2, to=5-3]
		\arrow["{c^!_{\Delta}}", from=5-3, to=5-4]
		\arrow["\simeq", from=5-1, to=5-2]
		\arrow["{r.b.c.}", from=5-4, to=5-5]
		\arrow["{s^*\Nmadjdual_p}"{description}, bend right = 13pt, from=5-1, to=5-5]
		\arrow["\simeq", from=3-1, to=3-2]
		\arrow["{r.b.c.}", from=3-2, to=4-2]
		\arrow["{r.b.c.}", from=4-2, to=5-2]
		\arrow[""{name=0, anchor=center, inner sep=0}, Rightarrow, no head, from=3-1, to=5-1]
		\arrow["{\Nm_s^{-1}}", from=3-2, to=3-3]
		\arrow["{l.b.c.}", from=3-3, to=4-3]
		\arrow["{\Nm_{\Delta}^{-1}}", from=4-2, to=4-3]
		\arrow["{r.b.c.}", from=4-3, to=5-3]
		\arrow["{r.b.c.}", from=4-4, to=5-4]
		\arrow["\simeq", from=4-4, to=4-5]
		\arrow["\simeq", from=4-5, to=5-5]
		\arrow[Rightarrow, no head, from=2-5, to=4-5]
		\arrow["{c^!_{\Delta}}", from=4-3, to=4-4]
		\arrow["\simeq", from=1-1, to=1-2]
		\arrow[Rightarrow, no head, from=1-1, to=3-1]
		\arrow["{\Nm_s^{-1}}", from=1-2, to=1-3]
		\arrow["\simeq", from=1-2, to=3-2]
		\arrow["{c_s^!}", from=1-3, to=1-4]
		\arrow[Rightarrow, no head, from=1-4, to=2-5]
		\arrow["{(1)}"{description}, draw=none, from=4-2, to=3-3]
		\arrow["\simeq", from=2-3, to=3-3]
		\arrow["\simeq", from=1-3, to=2-3]
		\arrow["\simeq", from=2-4, to=2-5]
		\arrow["\simeq", from=1-4, to=2-4]
		\arrow["{c^!_s}", from=2-3, to=2-4]
		\arrow[Rightarrow, no head, from=2-4, to=4-4]
		\arrow["{(2)}"{description}, draw=none, from=4-3, to=2-4]
		\arrow["{(3)}"{description}, draw=none, from=4-4, to=5-5]
		\arrow["{(3)}"{description}, draw=none, from=0, to=3-2]
	\end{tikzcd}
	\]
	The composite along the top of this diagram is the other map appearing in the statement of the lemma, so it will suffice to prove that the diagram commutes. All unlabeled equivalences in this diagram come from identifications on the level of maps in $\PSh(T)$, e.g.\ we have $p_*s_* \simeq (ps)_* \simeq \id_* \simeq \id$, etcetera. The maps labeled \textit{l.b.c.} and \textit{r.b.c.} are the left/right base change equivalences associated with one of the following three pullback squares in $\finPsets$:
	\[
	% https://q.uiver.app/?q=WzAsOCxbMSwwLCJBIl0sWzEsMSwiQSBcXHRpbWVzX0IgQSJdLFsxLDIsIkEiXSxbMiwxLCJBIl0sWzIsMiwiQi4iXSxbMiwwLCJCIl0sWzAsMCwiQiJdLFswLDEsIkEiXSxbMCw1LCJwIl0sWzIsNCwicCJdLFswLDEsIigxLHNwKSIsMV0sWzEsMiwiXFxwcl8xIl0sWzEsMywiXFxwcl8yIl0sWzMsNCwicCJdLFs1LDMsInMiXSxbNywxLCJcXERlbHRhIl0sWzYsMCwicyJdLFs2LDcsInMiLDJdLFs3LDIsIiIsMSx7ImxldmVsIjoyLCJzdHlsZSI6eyJoZWFkIjp7Im5hbWUiOiJub25lIn19fV1d
	\begin{tikzcd}
		B \drar[pullback] & A \drar[pullback] & B \\
		A & {A \times_B A} \drar[pullback] & A \\
		& A & {B.}
		\arrow["p", from=1-2, to=1-3]
		\arrow["p", from=3-2, to=3-3]
		\arrow["{(1,sp)}"{description}, from=1-2, to=2-2]
		\arrow["{\pr_1}", from=2-2, to=3-2]
		\arrow["{\pr_2}", from=2-2, to=2-3]
		\arrow["p", from=2-3, to=3-3]
		\arrow["s", from=1-3, to=2-3]
		\arrow["\Delta", from=2-1, to=2-2]
		\arrow["s", from=1-1, to=1-2]
		\arrow["s"', from=1-1, to=2-1]
		\arrow[Rightarrow, no head, from=2-1, to=3-2]
	\end{tikzcd}
	\]
	Except for the squares labelled (1), (2) and (3), all squares in the above diagram commute by naturality. The commutativity of (1) is an instance of \cref{cor:NormsVsBaseChange} applied to the previous pullback square exhibiting $s$ as a base change of $\Delta$ along $(1,sp)\colon A \to A \times_B A$.
	The commutativity of (2) follows directly from the definition of the left base change equivalence $s_!s^* \iso (1,sp)^*\Delta_!$, using the triangle identity. Finally, the two squares labeled (3) use that the composite of two right base change equivalences is the right base change equivalence for the composite, which in both cases is just equivalent to the identity.
\end{proof}

\subsection{\texorpdfstring{\for{toc}{$P$}\except{toc}{$\bm P$}}{P}-semiadditive \for{toc}{$T$-$\infty$}\except{toc}{\texorpdfstring{$\bm T$-$\bm\infty$}{T-∞}}-categories}
\label{subsec:PSemiadditiveTCategories}

In this section, we will introduce and discuss the notion of a $P$-semiadditive $T$-$\infty$-category for a fixed atomic orbital subcategory $P \subseteq T$.

\begin{definition}[cf. \cite{nardin2016exposeIV}*{Definition~5.3}]
	\label{def:PSemiadditivity}
	Let $\Cc$ be a pointed $T$-$\infty$-category which admits both finite $P$-products and finite $P$-coproducts. We say that $\Cc$ is \textit{$P$-semiadditive} if for every morphism $p\colon A \to B$ in $\finPsets$ the norm map $\Nm_p\colon p_! \Rightarrow p_*$ is an equivalence.

	We let $\CatTPProd \subseteq \Cat_T$ denote the (non-full) subcategory spanned by the $T$-$\infty$-categories which admit finite $P$-products and the $T$-functors which preserve finite $P$-products. We let $\CatTPSemi \subseteq \CatTPProd$ denote the full subcategory spanned by the $P$-semiadditive $T$-$\infty$-categories.
\end{definition}

\begin{example}\label{ex:eqsemi}
The previous definition applied to the pair $\Orb \subset \Glo$ gives a notion of $\Orb$-semiadditivity for global $\infty$-categories. We will refer to this as \emph{equivariant semiadditivity}.
\end{example}

It follows directly that also the norm maps for more general morphisms in $\ulfinPsets$ are equivalences:
\begin{corollary}
    \label{cor:NormEquivalencesForUlFinPSets}
    Let $\Cc$ be a $P$-semiadditive $T$-$\infty$-category, let $B \in \PSh(T)$ and let $(p\colon A \to B) \in \ulfinPsets(B)$. Then the norm map $\Nm_p\colon p_! \Rightarrow p_*$ is an equivalence.
\end{corollary}
\begin{proof}
    We may write the presheaf $B$ as a colimit $\colim_i B_i$ of representables $B_i \in T$, which gives rise to an equivalence of $\infty$-categories $\Cc(B) \simeq \lim_i \Cc(B_i)$. It will thus suffice to show that for every representable $B' \in T$ and any morphism $g\colon B' \to B$ of presheaves, the transformation $g^*\Nm_p\colon g^*p_! \Rightarrow g^*p_*$ is an equivalence. By \Cref{cor:NormsVsBaseChange}, it will suffice to show that the transformation $\Nm_{p'}\colon p'_! \Rightarrow p'_*$ is an equivalence, where $p'\colon A \times_B B' \to B'$ is the base change of $p$ along $g$. Since this base change is a morphism in $\finPsets$, this holds by assumption on $\Cc$.
\end{proof}

We will next discuss various alternative characterizations of $P$-semiadditivity. We start by observing that this condition is self-dual.

\begin{lemma}
	\label{lem:CharacterizationPSemiadditiveTCategories}
	Let $\Cc$ be a pointed $T$-$\infty$-category. Then the following conditions are equivalent:
	\begin{enumerate}[(1)]
		\item The $T$-$\infty$-category $\Cc$ is $P$-semiadditive;
		\item The opposite $T$-$\infty$-category $\Cc\catop$ is $P$-semiadditive;
		\item The $T$-$\infty$-category $\Cc$ admits finite $P$-coproducts and for every morphism $p\colon A \to B$ in $\finPsets$ the adjoint norm map $\Nmadj_p\colon p^*p_! \Rightarrow \id$ is the counit of an adjunction $p^* \dashv p_!$;
		\item The $T$-$\infty$-category $\Cc$ admits finite $P$-products and for every morphism $p\colon A \to B$ in $\finPsets$ the dual adjoint norm map $\Nmadjdual_p\colon \id \Rightarrow p^*p_*$ is the unit of an adjunction $p_* \dashv p^*$.
	\end{enumerate}
\end{lemma}
\begin{proof}
	Observe that the dual adjoint norm map $\Nmadjdual_p\colon \id \Rightarrow p^*p_*$ may be obtained by applying the construction of the adjoint norm map $\Nmadj\colon p^*p_! \Rightarrow \id$ to the $T$-$\infty$-category $\Cc\catop$. The equivalence between (1) and (2) is then immediate from \cref{lem:NormMapEqualsDualNormMap}. The fact that (1) implies (3) is clear, since the norm map $\Nm_p\colon p_! \Rightarrow p_*$ is adjoint to $\Nmadj\colon p^*p_! \Rightarrow \id$. For the implication (3) $\implies$ (1), it remains to show that the right adjoints given by (3) satisfy the Beck-Chevalley condition, which is a consequence of \Cref{prop:adjNormsVsBaseChange}. The equivalence between (2) and (4) is obtained dually by replacing $\Cc$ with $\Cc\catop$.
\end{proof}

Every choice of an atomic orbital subcategory $P \subseteq T$ gives a different notion of parametrized semiadditivity for a $T$-$\infty$-category $\Cc$. The weakest form of parame\-trized semiadditivity is fiberwise semiadditivity:

\begin{definition}
	A $T$-$\infty$-category $\Cc$ is called \textit{fiberwise semiadditive} if for every $B \in T$ the $\infty$-category $\Cc(B)$ is semiadditive and for every morphism $f\colon A \to B$ in $T$ the restriction functor $f^*\colon \Cc(B) \to \Cc(A)$ preserves finite biproducts.
\end{definition}

\begin{lemma}
	\label{lem:PSemiadditiveImpliesFiberwiseSemiadditive}
	Let $\Cc$ be a pointed $T$-$\infty$-category which admits fiberwise finite products and coproducts. Then the following three conditions are equivalent:
	\begin{enumerate}[(1)]
		\item The $T$-$\infty$-category $\Cc$ is fiberwise semiadditive;
		\item The norm map $\Nm_{\nabla}\colon \nabla_! \to \nabla_*$ associated to the fold map $\nabla\colon \bigsqcup_{i=1}^n B \to B$ is an equivalence for every $n \geq 0$ and every $B \in T$;
		\item The $T$-$\infty$-category $\Cc$ is $P$-semiadditive for $P = \core T$, the core of $T$.
	\end{enumerate}
\end{lemma}

\begin{proof}
	When $P = \core T$ is the core of $T$, any map in $\finPsets$ is equivalent to a fold map $\nabla\colon \bigsqcup_{i=1}^n B \to B$ for some $B \in T$, and thus the equivalence between (2) and (3) is clear. It remains to show that (1) and (2) are equivalent. The $\infty$-category $\Cc(\bigsqcup_{i=1}^n B)$ is equivalent to the $n$-fold product $\prod_{i=1}^n \Cc(B)$ of $\Cc(B)$. Given an object $X = (X_i) \in \prod_{i=1}^n \Cc(B)$, there are equivalences $\nabla_!(X) \simeq \bigoplus_{i=1}^n X_i$ and $\nabla_*(X) \simeq \prod_{i=1}^n X_i$. By \Cref{lem:AlphaIsDiagonalMatrix}, the map $\alpha(X)$ is a morphism in $\prod_{i=1}^n\prod_{j=1}^n \Cc(B)$ which we may visually display as
	\begin{align*}
		\begin{pmatrix}
			1 & 0 & \dots & 0 \\
			0 & 1 & \dots & 0 \\
			\vdots & \vdots& \ddots & \vdots \\
			0 & 0 & \dots & 1
		\end{pmatrix}
		\colon
		\begin{pmatrix}
			X_1 & X_2 & \dots & X_n \\
			X_1 & X_2 & \dots & X_n \\
			\vdots & \vdots & \ddots&\vdots \\
			X_1 & X_2 & \dots & X_n
		\end{pmatrix}
		\to
		\begin{pmatrix}
			X_1 & X_1 & \dots & X_1 \\
			X_2 & X_2 & \dots & X_2 \\
			\vdots & \vdots& \ddots& \vdots \\
			X_n & X_n & \dots & X_n
		\end{pmatrix},
	\end{align*}
	where $1$ denotes an identity map while $0$ denotes the zero map. In particular, the induced norm map $\Nm_p\colon \bigoplus_{i=1}^n X_i \to \prod_{j=1}^n X_j$ is induced by the family of maps $\{X_i \to X_j\}_{i,j}$ given by the identity when $i = j$ and the zero-map when $i \neq j$. This is precisely the norm map defining ordinary semiadditivity for $\infty$-categories, finishing the proof.
\end{proof}

As the next result shows, the condition of $P$-semiadditivity for general $P$ is a combination of fiberwise semiadditivity and norm equivalences $\Nm_p\colon p_! \simeq p_*$ for morphisms $p$ in $P$.

\begin{corollary}
	\label{cor:PSemiadditivity}
	Let $\Cc$ be a $T$-$\infty$-category. Then $\Cc$ is $P$-semiadditive if and only if it is fiberwise semiadditive and for every morphism $p\colon A \to B$ in $P$ the norm map $\Nm_p\colon p_! \Rightarrow p_*$ is an equivalence.
\end{corollary}

\begin{proof}
	As in the proof of \Cref{prop:finitePcoproducts}, every morphism in $\finPsets$ with representable domain $B \in T$ can be written as a composite $\bigsqcup_{i=1}^n A_i \xrightarrow{\bigsqcup_{i=1}^n p_i} \bigsqcup_{i=1}^n A_i \xrightarrow{\nabla} B$ for morphisms $p_i\colon A_i \to B$ in $P$, where $\nabla$ denotes the fold map. The norm map of $\bigsqcup_{i=1}^n p_i\colon \bigsqcup_{i=1}^n A_i \to \bigsqcup_{i=1}^n B$ is equivalent to the product of the norm maps for each individual $p_i\colon A_i \to B$. By \Cref{cor:NormsVsComposition}, the norm map of a composite morphism can be written as a composite of norm maps, and it follows that $\Cc$ is $P$-semiadditive if and only if the norm maps of all the fold maps $\nabla\colon \bigsqcup_{i=1}^n B \to B$ and of all morphisms $p\colon A \to B$ in $P$ are equivalences. But by \Cref{lem:PSemiadditiveImpliesFiberwiseSemiadditive} the norm maps for the fold maps are equivalences if and only if $\Cc$ is fiberwise semiadditive.
\end{proof}

We finish this subsection with a recognition criterion for $P$-semiadditivity along the lines of \cite{HA}*{Proposition 2.4.3.19}.

\begin{proposition}
	\label{prop:CharacterizationPSemiadditivity}
	Let $\Cc$ be a pointed $T$-$\infty$-category admitting finite $P$-products. Assume that for every morphism $p\colon A \to B$ in $\finPsets$, there is a natural transformation $\mu_p\colon p_*p^* \Rightarrow \id_{\Cc(B)}$ of functors $\Cc(B) \to \Cc(B)$ satisfying the following two conditions:
	\begin{enumerate}[(a)]
		\item for every $X \in \Cc(B)$, the composite
		\begin{align*}
			p^*X \xrightarrow{\Nmadjdual_{p}p^*X} p^*p_*p^*X \xrightarrow{p^*\mu_pX} p^*X
		\end{align*}
		is homotopic to the identity;
		\item for every $Y \in \Cc(A)$, the following diagram commutes
		\[\begin{tikzcd}
			p_*p^*p_* Y \drar[swap, bend right = 10]{\mu_pp_*Y} \rar{\simeq}[swap]{r.b.c.} & p_*(\pr_2)_*\pr_1^* Y \rar{\simeq} & p_*(\pr_1)_*\pr_1^*Y \dlar[bend left = 10]{p_*\mu_{\pr_1}Y} \\
			& p_*Y.
		\end{tikzcd}\]
	\end{enumerate}
	Then the $T$-$\infty$-category $\Cc$ is $P$-semiadditive.
\end{proposition}
\begin{proof}
	To show that $\Cc$ is $P$-semiadditive, we may by \Cref{lem:CharacterizationPSemiadditiveTCategories} equivalently show that for every map $p\colon A \to B$ in $\finPsets$ and every object $Y \in \Cc(A)$, the dual adjoint norm map $\Nmadjdual_pY \colon Y \Rightarrow p^*p_*Y$ exhibits $p_*Y$ as a left adjoint object to $Y$ under the functor $p^*\colon \Cc(B) \to \Cc(A)$, i.e.\ that for every $X \in \Cc(B)$ the composite
	\begin{align*}
		\Hom_{\Cc(B)}(p_*Y,X) \xrightarrow{p^*} \Hom_{\Cc(A)}(p^*p_*Y,p^*X) \xrightarrow{- \circ \Nmadjdual_p Y} \Hom_{\Cc(A)}(Y,p^*X)
	\end{align*}
	is an equivalence. We claim that an inverse is given by
	\begin{align*}
		\Hom_{\Cc(A)}(Y,p^*X) \xrightarrow{p_*} \Hom_{\Cc(B)}(p_*Y,p_*p^*X) \xrightarrow{\mu_p X \circ -} \Hom_{\Cc(B)}(p_*Y,X).
	\end{align*}
	By naturality of $\mu_p$ and $\Nm_p$, it suffices to prove that the following two composites are homotopic to the identity for every \textit{fixed} $X$:\footnote{While this suffices to show that $\Nm_p$ is a unit of an adjunction, it does not show that $\mu_p$ is the corresponding counit, as we do not provide homotopies that are functorial in $X$ and $Y$.}
	\begin{align*}
		p^*X \xrightarrow{\Nmadjdual_{p}p^*X} p^*p_*p^*X \xrightarrow{p^*\mu_pX} p^*X, \\
		p_*Y \xrightarrow{p_*\Nmadjdual_pY} p_*p^*p_*Y \xrightarrow{\mu_p p_*Y} p_*Y.
	\end{align*}
	The first composite is homotopic to the identity by condition (a), so we focus on the second composite. Plugging in the description of $\Nmadjdual_p$ given in \cref{rem:concise_norm_maps}, this composite expands to
	% https://q.uiver.app/?q=WzAsNixbMCwwLCJwXypZIl0sWzEsMCwicF8qe1xccHJfMn1fKlxcRGVsdGFfKlxcRGVsdGFeKlxccHJfMV4qWSJdLFsyLDAsInBfKntcXHByXzJ9XypcXERlbHRhXyFcXERlbHRhXipcXHByXzFeKlkiXSxbMiwxLCJwXyp7XFxwcl8yfV8qIFxccHJfMV4qWSJdLFszLDEsInBfKnBeKiBwXypZIl0sWzQsMSwicF8qWSwiXSxbMCwxLCJcXHNpbSJdLFsxLDIsIlxcTm1eey0xfV97XFxEZWx0YX0iXSxbMiwzLCJjXiFfe1xcRGVsdGF9Il0sWzMsNCwici5iLmMuIl0sWzQsNSwiXFxtdV9wIHBfKiJdLFsyLDEsIlxcc2ltIiwxLHsib2Zmc2V0IjotMiwiY29sb3VyIjpbMCwwLDEwMF19XSxbNCwzLCJcXHNpbSIsMSx7Im9mZnNldCI6LTIsImNvbG91ciI6WzAsMCwxMDBdfV1d
	\[\begin{tikzcd}[cramped]
		{p_*Y} & {p_*{\pr_2}_*\Delta_*\Delta^*\pr_1^*Y} & {p_*{\pr_2}_*\Delta_!\Delta^*\pr_1^*Y} \\
		&& {p_*{\pr_2}_* \pr_1^*Y} & {p_*p^* p_*Y} & {p_*Y,}
		\arrow["\sim", from=1-1, to=1-2]
		\arrow["{\Nm^{-1}_{\Delta}}", from=1-2, to=1-3]
		\arrow["{c^!_{\Delta}}", from=1-3, to=2-3]
		\arrow["{r.b.c.}", from=2-3, to=2-4]
		\arrow["{\mu_p p_*Y}", from=2-4, to=2-5]
		\arrow["\sim"{description}, shift left=2, draw=white, from=1-3, to=1-2]
		\arrow["\sim"{description}, shift left=2, draw=white, from=2-4, to=2-3]
	\end{tikzcd}\]
	which, using condition (b) and the equivalence $p \circ \pr_1 \simeq p \circ \pr_2$, is homotopic to the composite
	\[
	p_*Y \simeq p_*{\pr_1}_*\Delta_*\Delta^*\pr_1^*Y \xrightarrow{\Nm_{\Delta}^{-1}} p_*{\pr_1}_*\Delta_!\Delta^*\pr_1^*Y \xrightarrow{c^!_{\Delta}} p_*{\pr_1}_*\pr_1^*Y \xrightarrow{p_*\mu_{\pr_1}Y} p_*Y.
	\]
	Applying \cref{lem:TechnicalLemmaRecognitionPSemiadditivity} to the map $\pr_1\colon A \times_B A \to A$ with section $\Delta\colon A \to A \times_B A$, we see that this map is homotopic to the following composite:
	\[
	\hskip-13.5785pt\hfuzz=13.579pt p_*Y \simeq p_*\Delta^*\pr_1^*Y \xrightarrow{p_*\Delta^*\Nmadjdual_{\pr_1}\pr_1^*Y} p_*\Delta^*\pr_1^*{\pr_1}_*\pr_1^*Y \xrightarrow{p_*\Delta^*\pr_1^*\mu_{\pr_1}Y} p_*\Delta^*\pr_1^*Y \simeq p_*Y.
	\]
	This map is homotopic to the identity by assumption (a) applied to the map $\pr_1\colon A \times_B A \to A$, finishing the proof.
\end{proof}

\subsection{\texorpdfstring{\for{toc}{$P$}\except{toc}{$\bm P$}}{P}-semiadditive \for{toc}{$T$}\except{toc}{\texorpdfstring{$\bm T$}{T}}-functors}
\label{subsec:PSemiadditiveTFunctors}
We continue to fix an atomic orbital subcategory $P \subseteq T$. In this subsection we will define what it means for a $T$-functor $F\colon \Cc \to \Dd$ to be \textit{$P$-semiadditive}: roughly speaking, it means that $F$ turns finite $P$-coproducts in $\Cc$ into finite $P$-products in $\Dd$. The main result of this subsection is \cref{prop:SemiadditiveFunctorsFormSemiadditiveCategory}, which states that the $T$-subcategory $\ulFun^{\Poplus}_T(\Cc,\Dd)$ of $\ulFun_T(\Cc,\Dd)$ spanned by the $P$-semiadditive $T$-functors is $P$-semiadditive.

We start by constructing a `relative' variant of the norm map.

\begin{construction}
	\label{cstr:NormMapFunctor}
	Let $F\colon \Cc \to \Dd$ be a $T$-functor such that $\Cc$ is pointed and admits finite $P$-coproducts and $\Dd$ admits finite $P$-products, let $B \in \PSh(T)$ and let $(p\colon A \to B) \in \ulfinPsets(B)$. We define the \textit{norm transformation of $p$ relative to $F$}
	\[
	\Nm^F_{p}\colon F_B \circ p_! \implies p_* \circ F_A \vspace{-10pt}
	\]
	as the transformation adjoint to the composite $p^*F_B p_! \simeq F_A p^* p_! \xRightarrow{F_A\Nmadj^{\Cc}_p} F_A$, where the first equivalence uses that the parametrized functor $F\colon \Cc \to \Dd$ commutes with the restriction functors.
\end{construction}

Note that when $\Dd$ is equal to $\Cc$ and $F$ is the identity on $\Cc$, the transformation $\Nm^F_p$ reduces to the norm transformation $\Nm^{\Cc}_p\colon p_! \Rightarrow p_*$ of \cref{cstr:NormMap}.

\begin{definition}
	\label{def:PSemiadditiveFunctor}
	Let $F\colon \Cc \to \Dd$ be a $T$-functor such that $\Cc$ is pointed and admits finite $P$-coproducts and $\Dd$ admits finite $P$-products. We will say that $F$ is \textit{$P$-semiadditive} if it satisfies the following condition:
	\begin{enumerate}[($\ast$)]
		\item For each morphism $p\colon A \to B$ in $\finPsets$, the transformation $\Nm^F_{p}\colon F_B \circ p_! \Rightarrow p_* \circ F_A$ defined in \cref{cstr:NormMapFunctor} is a natural equivalence.
	\end{enumerate}
	By \Cref{ex:OrbitalClosed}\eqref{it:SliceOfOrbitalIsOrbital} we also obtain a notion of $P$-semiadditive $T_{/B}$-functors for all $B \in T$. Note that $\Cc$ is $P$-semiadditive if and only if the identity $\id\colon \Cc \to \Cc$ is $P$-semiadditive. Also note that condition $(\ast)$ specializes for $A = \emptyset$ to the condition that the functor $F_B\colon \Cc(B) \to \Dd(B)$ sends the zero object of $\Cc(B)$ to the final object of $\Dd(B)$.
\end{definition}

Just like in \Cref{cor:NormEquivalencesForUlFinPSets}, one immediately deduces that the relative norm maps are equivalences for arbitrary morphisms in $\ulfinPsets$:

\begin{corollary}
    \label{cor:RelativeNormEquivalencesForUlFinPSets}
    Let $F\colon \Cc \to \Dd$ as in \Cref{cstr:NormMapFunctor} and assume that $F$ is $P$-semiadditive. Then for every $B \in \PSh(T)$ and every $(p\colon A \to B) \in \ulfinPsets(B)$, the transformation $\Nm^F_{p}\colon F_B \circ p_! \Rightarrow p_* \circ F_A$ is an equivalence.\qed
\end{corollary}

While not necessary for our work, we show for completeness that our norm map generalizes the analogous construction in \cite{nardin2016exposeIV}.

\begin{proposition}
	\label{prop:OurNormMapVsNardinsNormMap}
	Let $T$ be an atomic orbital $\infty$-category, let $B \in T$ and let $p\colon A \to B$ be a morphism in $\finTsets$. Let $F\colon \Cc \to \Dd$ be a $T$-functor with $\Cc$ and $\Dd$ satisfying the assumptions of \cref{cstr:NormMapFunctor}. Then the norm transformation $\Nm^F_p\colon F_B \circ p_! \Rightarrow p_*\circ F_A$ of \cref{cstr:NormMapFunctor} is homotopic to the transformation defined in \cite{nardin2016exposeIV}*{Construction~5.2}.
\end{proposition}
\begin{proof}
	We will first give an alternative description of the norm map in this special case, and then argue why it agrees with the construction of Nardin. By definition of $\finTsets$, we may assume $p\colon A \to B$ to be of the form $p= (p_i)\colon \bigsqcup_{i=1}^n A_i \to B$, where each $A_i \in T$ is representable. Let $\iota_i\colon A_i \hookrightarrow \bigsqcup_{i=1}^n A_i = A$ denote the canonical inclusion, so that $p_i = p \circ \iota_i\colon A_i \to B$. The functor $p_*\colon \Dd(A) \to \Dd(B)$ may be decomposed as
	\begin{align*}
		\Dd(A) = \Dd(\bigsqcup_{i=1}^n A_i) \xrightarrow[\simeq]{(\iota_i^*)_i} \prod_{i=1}^n \Dd(A_i) \xrightarrow{\prod_{i=1}^n {p_i}_*} \prod_{i=1}^n \Dd(B) \xrightarrow{\prod} \Dd(B),
	\end{align*}
	where the last map denotes the multiplication in $\Dd(B)$. For an object $X = (X_i) \in \Cc(A) \simeq \prod_{i=1}^n \Cc(A_i)$ the norm map $\Nm^F_p(X)\colon F_B(p_!(X)) \to  p_*F_A(X) \simeq \prod_{i=1}^n {p_i}_*(F_{A_i}(X_i))$ is the product of $n$ maps $F_B(p_!(X)) \to {p_i}_*(F_{A_i}(X_i))$, where the $i$-th one is obtained by adjunction from the composite
	\begin{align*}     p_i^*F_B(p_!(X)) \simeq F_{A_i}p_i^*p_!X \simeq F_{A_i}\iota_i^*p^*p_!X \xRightarrow{F_{A_i}\iota_i^*\Nmadj_p} F_{A_i}\iota_i^*X = F_{A_i}X_i.
	\end{align*}
	We will now expand the definition of the map $\iota_i^*\Nmadj_p\colon p_i^*p_!X \to X_i$. First notice that the map $\Nmadj_p\colon p^*p_!X \to X$ is given by the following composite:
	\begin{align*}
		p^* p_!X \overset{l.b.c.}{\simeq} {\pr_1}_! \pr_2^*X \xRightarrow{u^*_{\Delta}}
		{\pr_1}_! \Delta_* \Delta^* \pr_2^*X \xRightarrow[\simeq]{\Nm_{\Delta}^{-1}}
		{\pr_1}_! \Delta_! \Delta^* \pr_2^*X \simeq X.
	\end{align*}
	Applying left base change to the pullback diagram
	\[\begin{tikzcd}
		A_i \times_B A \rar{\iota_i \times_B A} \dar[swap]{\pr_1} \drar[pullback] &
		A_i \times_B A \dar[swap]{\pr_1} \rar{\pr_2} \drar[pullback] & A \dar{p} \\
		A_i \rar{\iota_i} &
		A_i \rar{p_i} & B
	\end{tikzcd}\]
	gives an equivalence $p_i^*p_!X \simeq {\pr_1}_!\pr_2^*X$. Since $T$ is atomic, the diagonal $\Delta_{p_i}\colon A_i \to A_i \times_B A_i \hookrightarrow \bigsqcup_{i=1}^n A_i \times_B A_i = A_i \times_B A$ is a disjoint summand inclusion. Writing $g\colon C \to A_i \times_B A$ for the complement summand, we observe that $\Cc(A_i \times_B A) = \Cc(A_i \sqcup C) \simeq \Cc(A_i) \times \Cc(C)$ and that the object $\pr_2^*X \in \Cc(A_i \times_B A)$ corresponds to the pair $(X_i,X_C)$ for some $X_C \in \Cc(C)$. Plugging in the map $X_C \to *$ to the zero-object $*$ of $\Cc(C)$ thus gives a map ${\pr_1}_!(X_i,X_C) \to {\pr_1}_!(X_i,*) \simeq X_i$. Looking at the construction of $\Nm_{\Delta}$ in \cref{lem:NormMapDisjointSummandInclusions}, one sees that the resulting composite $p_i^*p_!X \to X_i$ is precisely $\iota_i^*\Nmadj_p$.

	One may now observe that this second description of the norm map is precisely the construction of \cite{nardin2016exposeIV}, after making the following translations in notation:
	\begin{align*}
		&B \leftrightarrow V, && A \leftrightarrow U, && p \leftrightarrow I, && \bigsqcup_{i=1}^n A_i \leftrightarrow \bigsqcup_{W \in \rm{Orbit}(U)} W, \\
		&p_! \leftrightarrow \bigsqcup_I, && p_i^* \leftrightarrow \delta_{W/V}, && \iota_i \leftrightarrow (W \subseteq U), && \iota_i^*\Nmadj_p \leftrightarrow (\chi_{[W \subseteq U]})_*.
	\end{align*}
	This finishes the proof.
\end{proof}

Next, we will show that the $P$-semiadditive $T$-functors from $\Cc$ to $\Dd$ form a parame\-trized subcategory of $\ulFun_T(\Cc,\Dd)$. This will rely on the following general criterion in the spirit of Lemma~\ref{lemma:restriction-preserves-cocompl}:

\begin{lemma}\label{lem:SemiadditiveFunctorsClosedUnderBaseChange}
Let $f\colon\PSh(S)\to\PSh(T)$ be a cocontinuous functor that preserves pullbacks. Let $P\subset S$ and $Q\subset T$ be atomic orbital subcategories and assume that for every $p\colon A\to B$ in $P$ we have $(f(p)\colon f(A) \to f(B))\in\ul{\mathbb F}^Q_T(f(B))$. Then:
\begin{enumerate}
    \item The functor $f^*\colon\Cat_T\to\Cat_S$ sends (pointed) $T$-$\infty$-categories with finite $Q$-coproducts to (pointed) $S$-$\infty$-categories with finite $P$-coproducts, and dually for finite $Q$-products and finite $P$-products.
    \item If $F\colon \Cc \to \Dd$ is a $T$-functor such that $\Cc$ is pointed with finite $Q$-coproducts and $\Dd$ has finite $Q$-products, then the relative norm map $\Nm_p^{f^*F}$ for any $B\in\PSh(S),p\in \ul{\mathbb F}^P_S(B)$ agrees with the relative norm map $\Nm_{f(p)}^F$.
    \item The functor $f^*\colon\Cat_T\to\Cat_S$ sends $Q$-semiadditive $T$-categories to $P$-semiadditive $S$-categories and $Q$-semiadditive $T$-functors to $P$-semiadditive $S$-functors.
\end{enumerate}
\end{lemma}
\begin{proof}
	It is clear that $f^*$ preserves pointedness. Moreover, as $f$ preserves coproducts, it more generally maps $\ul{\mathbb F}^P_S(B)$ into $\ul{\mathbb F}^Q_T(f(B))$, so part (1) is an instance of Lemma~\ref{lemma:restriction-preserves-cocompl} and its dual. Part (2) follows similarly by direct inspection of the construction of the norm maps, and (3) is an immediate consequence of (2).
\end{proof}

\begin{definition}
	Let $\Cc$ and $\Dd$ be $T$-$\infty$-categories such that $\Cc$ is pointed and admits finite $P$-coproducts and $\Dd$ admits finite $P$-products. We define $\ulFun^{\Poplus}_T(\Cc,\Dd)$ as the full subcategory $\ulFun_T(\Cc,\Dd)$ spanned at level $B \in T$ by the $P$-semiadditive $T_{/B}$-functors $F\colon \pi^*_B\Cc \to \pi_B^*\Dd$ for $B \in T$.
\end{definition}

This does indeed form a $T$-subcategory by the previous lemma applied to the maps $T_{/f}\colon T_{/A}\to T_{/B}$ for all $f\colon A\to B$ in $T$, cf.~the proof of Lemma~\ref{lemma:FunL-T-subcat}.

We think of a $P$-semiadditive $T$-functor as a functor which sends finite $P$-coproducts to finite $P$-products. Hence we expect that this condition should be preserved when precomposing (resp.\ postcomposing) with a $T$-functor which preserves finite $P$-coproducts (resp.\ finite $P$-products). The following result shows that this is indeed the case.

\begin{proposition}
	\label{prop:CompositionSemiadditiveFunctors}
	Let $F\colon \Cc \to \Dd$ be a $P$-semiadditive $T$-functor, where $\Cc$ and $\Dd$ are as in \cref{def:PSemiadditiveFunctor}, and let $(p\colon A \to B) \in \ulfinPsets(B)$.
	\begin{enumerate}[(1)]
		\item \label{it:PrecompositionWithColimitPreserving} Let $\Cc'$ be another pointed $T$-category admitting finite $P$-coproducts and let $G\colon \Cc' \to \Cc$ be a pointed $T$-functor which preserves finite $P$-coproducts. Then the norm map $\Nm^{FG}_{p}\colon F_BG_Bp_! \Rightarrow p_*F_BG_B$ of $FG$ with respect to $p$ is given by the composite
		\[\begin{tikzcd}
			{\Cc'(A)} & {\Cc(A)} & {\Dd(A)} \\
			{\Cc'(B)} & {\Cc(B)} & {\Dd(B),}
			\arrow["{F_A}", from=1-2, to=1-3]
			\arrow["{G_A}", from=1-1, to=1-2]
			\arrow["{p_!}"', from=1-1, to=2-1]
			\arrow["{p_!}"', from=1-2, to=2-2]
			\arrow["{p_*}", from=1-3, to=2-3]
			\arrow["{F_B}"', from=2-2, to=2-3]
			\arrow["{G_B}"', from=2-1, to=2-2]
			\arrow["{\BC_!^{-1}}"{description}, Rightarrow, from=2-1, to=1-2]
			\arrow["{\Nm^F_p}"{description}, Rightarrow, from=2-2, to=1-3]
		\end{tikzcd}\]
		where $\BC_!\colon p_!G(A)\iso G(B)p_!$ denotes the Beck-Chevalley equivalence of $G$. In particular the composite $F \circ G\colon \Cc' \to \Dd$ is again $P$-semiadditive.
		\item \label{it:PostcompositionWithLimitPreserving} Let $\Dd'$ be another $T$-$\infty$-category which admits finite $P$-products and let $H\colon \Dd \to \Dd'$ be a $T$-functor which preserves finite $P$-products. Then the norm map $\Nm^{HF}_p\colon H_BF_Bp_! \Rightarrow p_*H_BF_B$ of $HF$ at $p$ is given by the composite
		\[
		\begin{tikzcd}
			{\Cc(A)} & {\Dd(A)} & {\Dd'(A)} \\
			{\Cc(B)} & {\Dd(B)} & {\Dd'(B),}
			\arrow["{F_A}", from=1-1, to=1-2]
			\arrow["{p_!}"', from=1-1, to=2-1]
			\arrow["{p_*}", from=1-2, to=2-2]
			\arrow["{F_B}"', from=2-1, to=2-2]
			\arrow["{\Nm^F_p}"{description}, Rightarrow, from=2-1, to=1-2]
			\arrow["{p_*}", from=1-3, to=2-3]
			\arrow["{H_A}", from=1-2, to=1-3]
			\arrow["{H_B}"', from=2-2, to=2-3]
			\arrow["{\BC_*}"{description}, Rightarrow, from=2-2, to=1-3]
		\end{tikzcd}\]
		where $\BC_*\colon H(A) p_* \iso p_* H(A)$ denotes the Beck-Chevalley equivalence of $H$. In particular the composite $H \circ F\colon \Cc \to \Dd$ is again $P$-semiadditive.
	\end{enumerate}
\end{proposition}

\begin{proof}
	The description of $\Nm^{FG}_p$ follows from the commutative diagram
	% https://q.uiver.app/?q=WzAsOCxbMCwwLCJGX0JHX0JwXyEiXSxbMSwwLCJwXypwXipGX0JHX0JwXyEiXSxbMiwwLCJwXypGX0FwXipHX0JwXyEiXSxbMywwLCJwXypGX0FHX0FwXipwXyEiXSxbMywxLCJwXypGX0FHX0EiXSxbMiwxLCJwXypGX0FwXipwXyFHX0EiXSxbMSwxLCJwXypwXipGX0JwXyFHX0EiXSxbMCwxLCJGX0JwXyFHX0EiXSxbMCwxLCJ1XipfcCJdLFsxLDIsIlxcc2ltZXEiXSxbMiwzLCJcXHNpbWVxIl0sWzMsNCwicF8qRl9BR19BXFxObWFkal9wIl0sWzIsNSwiXFxCQ18hXnstMX0iXSxbNSw0LCJwXypGX0FcXE5tYWRqX3BHX0EiLDJdLFsxLDYsIlxcQkNfIV57LTF9Il0sWzAsNywiXFxCQ18hXnstMX0iLDJdLFs3LDYsInVeKl9wIiwyXSxbNiw1LCJcXHNpbWVxIiwyXSxbNyw0LCJcXE5tX3tGLHB9IiwxXV0=
	\[\begin{tikzcd}[cramped]
		{F_BG_Bp_!} & {p_*p^*F_BG_Bp_!} & {p_*F_Ap^*G_Bp_!} &[30] {p_*F_AG_Ap^*p_!} \\
		{F_Bp_!G_A} & {p_*p^*F_Bp_!G_A} & {p_*F_Ap^*p_!G_A} & {p_*F_AG_A.}
		\arrow["{u^*_p}", from=1-1, to=1-2]
		\arrow["\simeq", from=1-2, to=1-3]
		\arrow["\simeq", from=1-3, to=1-4]
		\arrow["{p_*F_AG_A\Nmadj_p}", from=1-4, to=2-4]
		\arrow["{\BC_!^{-1}}", from=1-3, to=2-3]
		\arrow["{p_*F_A\Nmadj_pG_A}", from=2-3, to=2-4]
		\arrow["{\BC_!^{-1}}", from=1-2, to=2-2]
		\arrow["{\BC_!^{-1}}"', from=1-1, to=2-1]
		\arrow["{u^*_p}", from=2-1, to=2-2]
		\arrow["\simeq", from=2-2, to=2-3]
		\arrow["{\Nm^F_pG_A}"', bend right = 10, from=2-1, to=2-4]
	\end{tikzcd}\]
	The middle and left square commute by naturality, and the right square by Lemma \ref{lem:NormVsBeckChevalley}.
	The description of $\Nm^{HF}_p$ follows from the commutative diagram
	% https://q.uiver.app/?q=WzAsOCxbMCwwLCJIX0JGX0JwXyEiXSxbMCwxLCJwXypwXipIX0JGX0JwXyEiXSxbMSwxLCJwXypIX0FwXipGX0JwXyEiXSxbMiwxLCJwXypIX0FGX0FwXipwXyEiXSxbMSwwLCJIX0JwXypwXipGX0JwXyEiXSxbMiwwLCJIX0JwXypGX0FwXipwXyEiXSxbMywwLCJIX0JwXypGX0EiXSxbMywxLCJwXypIX0FGX0EiXSxbMCwxLCJ1XipfcCIsMl0sWzAsNCwiSF9CdV4qX3AiXSxbMSwyLCJcXHNpbWVxIiwyXSxbMiwzLCJcXHNpbWVxIiwyXSxbNCwyLCJcXEJDXyoiXSxbNCw1LCJcXHNpbWVxIl0sWzUsMywiXFxCQ18qIiwyXSxbNSw2LCJIX0JwXypGX0FcXE5tYWRqX3AiXSxbMyw3LCJwXypIX0FGX0FcXE5tYWRqX3AiLDJdLFs2LDcsIlxcQkNfKiJdLFswLDYsIkhfQlxcTm1fe0YscH0iLDFdXQ==
	\[\begin{tikzcd}
		{H_BF_Bp_!} & {H_Bp_*p^*F_Bp_!} & {H_Bp_*F_Ap^*p_!} &[30] {H_Bp_*F_A} \\
		{p_*p^*H_BF_Bp_!} & {p_*H_Ap^*F_Bp_!} & {p_*H_AF_Ap^*p_!} & {p_*H_AF_A,}
		\arrow["{u^*_p}"', from=1-1, to=2-1]
		\arrow["{H_Bu^*_p}"', from=1-1, to=1-2]
		\arrow["\simeq"', from=2-1, to=2-2]
		\arrow["\simeq"', from=2-2, to=2-3]
		\arrow["{\BC_*}"', "\simeq", from=1-2, to=2-2]
		\arrow["\simeq", from=1-2, to=1-3]
		\arrow["{\BC_*}"', "\simeq", from=1-3, to=2-3]
		\arrow["{H_Bp_*F_A\Nmadj_p}"', from=1-3, to=1-4]
		\arrow["{p_*H_AF_A\Nmadj_p}"', from=2-3, to=2-4]
		\arrow["{\BC_*}", "\simeq"', from=1-4, to=2-4]
		\arrow["{H_B\Nm^F_p}", bend left = 10, from=1-1, to=1-4]
	\end{tikzcd}\]
	where the middle and right square commute by naturality while the left-most square commutes by definition of the Beck-Chevalley equivalence $\BC_*$ and the triangle identity.
\end{proof}

\begin{corollary}
	\label{cor:AdjunctionSemiadditiveFunctorsIntoPointedObjects}
	Let $\Cc$ and $\Dd$ be $T$-$\infty$-categories such that $\Cc$ is pointed and admits finite $P$-coproducts, and $\Dd$ admits finite $P$-products. Then post-composition with the forgetful functor $\Dd_* \to \Dd$ induces an equivalence of $T$-$\infty$-categories
	\begin{align*}
		\ulFun_T^{\Poplus}(\Cc,\Dd_*) \iso \ulFun_T^{\Poplus}(\Cc,\Dd).
	\end{align*}
\end{corollary}
\begin{proof}
	By \cref{cor:AdjunctionPointednessInternal} it remains to show that a pointed $T$-functor $\Cc \to \Dd_*$ is $P$-semiadditive if and only if its composition with $\Dd_* \to \Dd$ is $P$-semiadditive. This follows from \cref{prop:CompositionSemiadditiveFunctors} since the $T$-functor $\Dd_* \to \Dd$ is conservative and preserves $T$-limits by \cref{lem:ForgetfulFunctorPreservesTLimitsPointedCase}.
\end{proof}

\begin{corollary}
	\label{cor:SemiadditiveFunctorsUnderBaseChangeAdjunction}
	Let $\Cc$ be a pointed $T$-$\infty$-category which admits finite $P$-coproducts and let $\Dd$ be a $T_{/B}$-$\infty$-category which admits finite $P$-products. Let $B \in T$ and consider a $T_{/B}$-functor $F\colon \pi_B^*\Cc \to\Dd$. Then $\pi_B^*\Cc$ is pointed with finite $P$-coproducts, $\pi_{B*}\Dd$ has finite $P$-products, and $F$ is $P$-semiadditive if and only if the corresponding functor $\widetilde{F}\colon \Cc \to {\pi_B}_*\Dd$ is $P$-semiadditive.
\end{corollary}
\begin{proof}
	This follows from Lemma~\ref{lem:SemiadditiveFunctorsClosedUnderBaseChange} and Proposition~\ref{prop:CompositionSemiadditiveFunctors} by the same arguments as in the proof of Proposition~\ref{prop:slice-adj-cc}.
\end{proof}

\begin{corollary}
	Let $\Cc$ be a pointed $T$-$\infty$-category with finite coproducts, let $\Dd$ be a $T$-$\infty$-category with finite products, and let $X\in\PSh(T)$ arbitrary. Then $(F\colon\Cc\to\ul\Fun_T(\ul X,\Dd))\in\ul\Fun_T(\Cc,\Dd)(X)$ defines an object of $\ul\Fun_T^{P\text-\oplus}(\Cc,\Dd)(X)$ if and only if it is $P$-semiadditive.
	\begin{proof}
		If $X$ is representable, this is an instance of the previous proposition. In the general case, we then simply observe analogously to the proof of Proposition~\ref{prop:adjunct-cocont} that the functors $\ul\Fun_T(\ul X,\Dd)\to\ul\Fun_T(\ul A,\Dd)$ for maps $\ul A\to\ul X$ with $A\in T$ are jointly conservative and preserve finite $P$-products, so that the claim follows from Proposition~\ref{prop:CompositionSemiadditiveFunctors}.
	\end{proof}
\end{corollary}

\begin{lemma}
	\label{lem:SemiadditiveFunctorsClosedUnderLimits}
	Let $\Cc$ and $\Dd$ be $T$-$\infty$-categories such that $\Cc$ is pointed and admits finite $P$-coproducts and $\Dd$ admits finite $P$-products. Let $\bbU$ be a class of $T$-$\infty$-categories, and assume that $\Dd$ admits $\bbU$-limits. Then the $T$-$\infty$-category $\ulFun^{\Poplus}_T(\Cc, \Dd)$ also admits $\bbU$-limits and the inclusion $\ulFun^{\Poplus}_T(\Cc, \Dd) \hookrightarrow \ulFun_T(\Cc, \Dd)$ preserves $\bbU$-limits.
\end{lemma}

\begin{proof}
	First note that the $T$-$\infty$-category $\ulFun_T(\Cc, \Dd)$ admits $\bbU$-limits by \cref{prop:CoLimitsInFunctorCategories}. Let $K \in \bbU(B)$ be a $T_{/B}$-$\infty$-category in $\bbU$, and let $F\colon \pi_B^*\Cc \to \ulFun_T(K,\pi_B^*\Dd)$ be a $P$-semiadditive $T_{/A}$-functor. We need to show that the $T_{/B}$-functor $\lim_KF\colon \pi_B^*\Cc \to \pi_B^*\Dd$ is again $P$-semiadditive. To simplify the notation, we will assume that $B$ is the final object of $T$ by replacing $T$ by $T_{/B}$, and thus we may identify $\pi_B^*\Cc$ and $\pi_B^*\Dd$ with $\Cc$ and $\Dd$, respectively.
	Since parametrized limits in $\ulFun_T(\Cc, \Dd)$ are computed pointwise by \cref{prop:CoLimitsInFunctorCategories}, the functor $\lim_KF\colon \Cc \to \Dd$ is given by the composite
	\begin{align*}
		\Cc \xrightarrow{F} \ulFun_T(K,\Dd) \xrightarrow{\lim_K} \Dd.
	\end{align*}
	Note that the $T$-functor $\lim_K\colon \ulFun_T(K,\Dd) \to \Dd$, being right adjoint to the diagonal $\Dd \to \ulFun_T(K,\Dd)$, preserves all parametrized limits and thus in particular all finite $P$-products. It then follows from \cref{prop:CompositionSemiadditiveFunctors} that $\lim_KF$ is $P$-semiadditive as desired.
\end{proof}

\begin{corollary}\label{cor:BilinearFunctors}
	Let $\Cc$ and $\Dd$ be pointed $T$-$\infty$-categories admitting finite $P$-coproducts, and let $\Ee$ be a $T$-$\infty$-category admitting finite $P$-products. Then the composite equivalence
	\[
	\ulFun_T(\Cc,\ulFun_T(\Dd,\Ee)) \simeq \ulFun_T(\Cc \times \Dd,\Ee) \simeq \ulFun_T(\Dd,\ulFun_T(\Cc,\Ee))
	\]
	restricts to an equivalence
	\[
	\ulFun^{\Poplus}_T(\Cc,\ulFun^{\Poplus}_T(\Dd,\Ee)) \simeq \ulFun^{\Poplus}_T(\Dd,\ulFun^{\Poplus}_T(\Cc,\Ee)).
	\]
\end{corollary}
\begin{proof}
	It follows immediately from \Cref{lem:SemiadditiveFunctorsClosedUnderLimits} and \Cref{prop:CoLimitsInFunctorCategories} that both sides correspond to the full subcategory of $\ulFun_T(\Cc \times \Dd,\Ee)$ spanned by those $T$-functors which are $P$-semiadditive in both variables. Here we say a $T$-functor $F\colon \Cc \times \Dd \to \Ee$ is \textit{$P$-semiadditive in both variables} if for every $B \in T$ and $X\colon \ul{B} \to \Cc$, the $T$-functor
	\[
	F(X,-)\colon \Dd \to \ulFun_T(\ul{B},\Ee)
	\]
	adjoint to the composite $\ul{B} \times \Dd \xrightarrow{X \times \Dd} \Cc \times \Dd \xrightarrow{F} \Ee$
	is $P$-semiadditive and similarly for every $Y\colon \ul{B} \to \Dd$ the $T$-functor
	\[
	F(-,Y)\colon \Cc \to \ulFun_T(\ul{B},\Ee)
	\]
	adjoint to $\Cc \times \ul{B} \xrightarrow{\Cc \times Y} \Cc \times \Dd \xrightarrow{F} \Ee$ is $P$-semiadditive.
\end{proof}

We now come to the main result of this subsection: the $P$-semiadditivity of the $T$-$\infty$-category $\ulFun_T^{\Poplus}(\Cc,\Dd)$.

\begin{proposition}[cf.\ \cite{nardin2016exposeIV}*{Proposition~5.8}]
	\label{prop:SemiadditiveFunctorsFormSemiadditiveCategory}
	Let $\Cc$ and $\Dd$ be $T$-$\infty$-categories such that $\Cc$ is pointed and admits finite $P$-coproducts and $\Dd$ admits finite $P$-products. Then the $T$-$\infty$-category $\ulFun^{\Poplus}_T(\Cc, \Dd)$ is $P$-semiadditive.
\end{proposition}
\begin{proof}
	By \cref{cor:AdjunctionSemiadditiveFunctorsIntoPointedObjects}, we may assume that $\Dd$ is pointed. It follows from \Cref{cor:FunctorCategoryPointed} that $\ulFun^{\Poplus}_T(\Cc, \Dd)$ is pointed and from \cref{lem:SemiadditiveFunctorsClosedUnderLimits} that $\ulFun^{\Poplus}_T(\Cc, \Dd)$ admits finite $P$-products. These are computed pointwise, meaning that for $p\colon A \to B$ in $\finPsets$ the map \[p_*\colon \ulFun^{\Poplus}(\Cc,\ulFun_T(\ul{A},\Dd))\to \ulFun^{\Poplus}(\Cc,\ulFun_T(\ul{B},\Dd))\] is given by post-composition with $p_*\colon \ulFun_T(\ul{A},\Dd) \to \ulFun_T(\ul{B},\Dd)$.

	To show that $\ulFun^{\Poplus}_T(\Cc, \Dd)$ is $P$-semiadditive, we will apply the recognition principle from \cref{prop:CharacterizationPSemiadditivity}. For every morphism $p\colon A \to B$ in $\finPsets$ and every $P$-semiadditive $T_{/B}$-functor $G\colon \pi_B^*\Cc \to \pi_B^*\Dd$, we define a natural transformation $\mu_pG\colon p_*p^*G \to G$. For notational simplicity, we will construct this in the case where $B = 1$ is a terminal object of $T$; the general case is obtained by replacing $T$ by $T_{/B}$. In this case, $\mu_p G$ is defined as the following composite:
	\[
	p_*p^*G \simeq p_*G^Ap^* \xrightarrow{(\Nm^G_{p})^{-1}} G p_!p^* \xrightarrow{G c^!_p} G;
	\]
	here we denote by $G^A\colon \ulFun_T(\ul{A},\Cc) \to \ulFun_T(\ul{A},\Dd)$ the $T$-functor induced by $G$. We need to check that conditions (a) and (b) of \cref{prop:CharacterizationPSemiadditivity} are satisfied. Condition (b) follows directly from the definitions, using \Cref{prop:CompositionSemiadditiveFunctors}\eqref{it:PostcompositionWithLimitPreserving} to compute the norm map of $p_*F$ in terms of the norm map of $F$ and the right base change equivalence $p^*p_* \simeq (\pr_2)_*\pr_1^*$. For condition (a), we need to show that for every $P$-semiadditive $T$-functor $G\colon \Cc \to \Dd$, the composite
	\[
	p^*G \xrightarrow{\Nmadjdual_p p^*G} p^*p_*p^*G \xrightarrow{p^*\mu_pG} p^*G
	\]
	is homotopic to the identity in $\Fun^{{\Pbiprod}}_{T_{/A}}(\pi_A^*\Cc,\pi_A^*\Dd) \simeq \Fun^{{\Pbiprod}}_{T}(\Cc,\ulFun_T(\ul{A},\Dd))$. Observe that pointedness of $\Dd$ guarantees that the transformation $\Nmadjdual_pp^*G\colon p^*G \to p^*p_*p^*G$ is given by whiskering $p^*G$ with the transformation $\Nmadjdual\!{}^{\Dd}_p\colon \id \to p^*p_*$. Spelling out the definitions, we are therefore interested in the composite along the top right in the following diagram:
	\[\hskip-21.66pt\hfuzz=22pt
	\begin{tikzcd}[cramped, column sep = 10]
		{p^*G} & {{\pr_2}_*\Delta_*\Delta^*\pr_1^*p^*G} & {{\pr_2}_*\Delta_!\Delta^*\pr_1^*p^*G} & {{\pr_2}_*\pr_1^*p^*G} & {p^*p_*p^*G} \\
		{G^Ap^*} & {{\pr_2}_*\Delta_*\Delta^*\pr_1^*G^Ap^*} & {{\pr_2}_*\Delta_!\Delta^*\pr_1^*G^Ap^*} & {{\pr_2}_*\pr_1^*G^A} & {p^*p_*G^Ap^*} \\
		&&&& {p^*Gp_!p^*} & {p^*G} \\
		{G^Ap^*} & {G^A{\pr_2}_!\Delta_!\Delta^*\pr_1^*p^*} && {G^A{\pr_2}_!\pr_1^*p^*} & {G^Ap^*p_!p^*} & {G^Ap^*} \\
		{G^Ap^*} & {G^A(\pr_2)_!\Delta_!\Delta^*\pr_2^*p^*} && {G^A{\pr_2}_!\pr_2^*p^*} && {G^Ap^*}
		\arrow["{l.b.c.}", from=4-4, to=4-5]
		\arrow["{c^!_{\Delta}}", from=4-2, to=4-4]
		\arrow["\simeq", from=4-1, to=4-2]
		\arrow["\simeq"', from=4-5, to=3-5]
		\arrow["\simeq", from=2-1, to=2-2]
		\arrow["{c^!_{\Delta}}", from=2-3, to=2-4]
		\arrow["{r.b.c.}", from=2-4, to=2-5]
		\arrow["{\Nm_{\Delta}^{-1}}", from=2-2, to=2-3]
		\arrow[Rightarrow, no head, from=2-1, to=4-1]
		\arrow["{\Nmadjdual_pp^*G}", bend left = 14pt, from=1-1, to=1-5]
		\arrow["\simeq"', from=1-1, to=1-2]
		\arrow["{\Nm_{\Delta}^{-1}}", from=1-2, to=1-3]
		\arrow["\simeq"', from=1-1, to=2-1]
		\arrow["\simeq"', from=1-2, to=2-2]
		\arrow["\simeq"', from=1-3, to=2-3]
		\arrow["\simeq"', from=1-4, to=2-4]
		\arrow["\simeq"', from=1-5, to=2-5]
		\arrow["{c^!_{\Delta}}", from=1-3, to=1-4]
		\arrow["{r.b.c.}"', from=1-4, to=1-5]
		\arrow["{c^!_p}", from=3-5, to=3-6]
		\arrow["{p^*\mu_p G}", bend left = 6pt, from=1-5, to=3-6]
		\arrow["\simeq"', from=4-6, to=3-6]
		\arrow["{c^!_p}", from=4-5, to=4-6]
		\arrow["{(\Nm^G_p)^{-1}}", from=2-5, to=3-5]
		\arrow[Rightarrow, no head, from=4-1, to=5-1]
		\arrow[Rightarrow, no head, from=4-6, to=5-6]
		\arrow["{c^!_{\Delta}}", from=5-2, to=5-4]
		\arrow["\simeq", from=4-2, to=5-2]
		\arrow["\simeq", from=5-1, to=5-2]
		\arrow["\simeq", from=4-4, to=5-4]
		\arrow["{c^!_{\pr_2}}", from=5-4, to=5-6]
		\arrow[bend right = 10pt, Rightarrow, no head, from=5-1, to=5-6]
		\arrow["{(1)}"{description}, draw=none, from=4-1, to=2-5]
		\arrow["{(2)}"{description}, draw=none, from=5-4, to=4-6]
	\end{tikzcd}
	\]
	As the composite along the bottom left is the identity, it remains to show that this diagram commutes. Except for (1) and (2), all squares commute either by definition or by naturality, and the commutativity of square (2) follows from the triangle identity. The commutativity of (1) follows from the following commutative diagram:
	\[\hskip-30.03pt\hfuzz=30.1pt
	\begin{tikzcd}
		{G^A} & {{\pr_2}_*\Delta_*\Delta^*\pr_1^*G^A} & {{\pr_2}_*\Delta_!\Delta^*\pr_1^*G^A} & {{\pr_2}_*\pr_1^*G^A} & {p^*p_*G^A} \\
		{G^A} & {{\pr_2}_*\Delta_*G^A\Delta^*\pr_1^*} & {{\pr_2}_*\Delta_!G^A\Delta^*\pr_1^*} & {{\pr_2}_*\pr_1^*G^A} & {p^*p_*G^A} \\
		& {{\pr_2}_*G^{A \times A}\Delta_!\Delta^*\pr_1^*} & {{\pr_2}_*G^{A \times A}\Delta_!\Delta^*\pr_1^*} & {{\pr_2}_*G^{A \times A}\pr_1^*} & {p^*Gp_!} \\
		{G^A} & {G^A{\pr_2}_!\Delta_!\Delta^*\pr_1^*} && {G^A{\pr_2}_!\pr_1^*} & {G^Ap^*p_!.}
		\arrow["{l.b.c.}", from=4-4, to=4-5]
		\arrow["{c^!_{\Delta}}", from=4-2, to=4-4]
		\arrow["{\Nm^G_p}"', from=3-5, to=2-5]
		\arrow["{r.b.c.}"', from=2-4, to=2-5]
		\arrow["{\Nm^{F}_{\pr_1}}"', from=4-4, to=3-4]
		\arrow["\simeq"', from=3-4, to=2-4]
		\arrow["\simeq"', from=2-1, to=2-2]
		\arrow["{\Nm^{G^{A \times A}}_{\Delta}}", from=3-2, to=2-2]
		\arrow["\simeq", from=4-1, to=4-2]
		\arrow[""{name=0, anchor=center, inner sep=0}, Rightarrow, no head, from=2-1, to=4-1]
		\arrow["\simeq"', from=4-5, to=3-5]
		\arrow["{(4)}"{description}, draw=none, from=3-4, to=3-5]
		\arrow["{\Nm^{G^A}_{\pr_1}}", from=4-2, to=3-2]
		\arrow["\simeq"', from=1-1, to=1-2]
		\arrow[Rightarrow, no head, from=2-5, to=1-5]
		\arrow["{c^!_{\Delta}}", from=1-3, to=1-4]
		\arrow["{r.b.c.}", from=1-4, to=1-5]
		\arrow["{\Nm_{\Delta}^{-1}}", from=1-2, to=1-3]
		\arrow["\simeq"', from=2-2, to=1-2]
		\arrow["{\Nm_{\Delta}}"', from=2-3, to=2-2]
		\arrow["\simeq"', from=2-3, to=1-3]
		\arrow["{c^!_{\Delta}}", from=3-3, to=3-4]
		\arrow["{\BC_!}", from=2-3, to=3-3]
		\arrow[Rightarrow, no head, from=3-2, to=3-3]
		\arrow[Rightarrow, no head, from=2-4, to=1-4]
		\arrow["{(2)}"{description}, draw=none, from=3-2, to=2-3]
		\arrow["{(3)}"{description}, draw=none, from=3-3, to=1-4]
		\arrow[Rightarrow, no head, from=2-1, to=1-1]
		\arrow["{(1)}"{description, pos=0.8}, draw=none, from=0, to=3-2]
	\end{tikzcd}
	\]
	The unlabeled squares commute by naturality. The fact that (1) commutes follows from \cref{cor:NormsVsComposition}, while the commutativity of (4) follows from \cref{cor:NormsVsBaseChange}. The commutativity of (2) and (3) easily follows from the definitions. This finishes the proof.
\end{proof}

\begin{proposition}
	\label{prop:FunctorSemiadditiveVsColimitPreserving}
	Let $\Cc$ be a pointed $T$-$\infty$-category which admits finite $P$-coproducts, and suppose $\Dd$ is $P$-semiadditive. Then a $T$-functor $F\colon \Cc \to \Dd$ is $P$-semiadditive if and only if it preserves finite $P$-coproducts. In particular we get that $\ulFun_T^{\Poplus}(\Cc,\Dd)$ and $\ulFun_T^{{\Pcoprod}}(\Cc,\Dd)$ are the same subcategory of $\ulFun_T(\Cc,\Dd)$.

	Analogously, suppose $\Cc$ is a $P$-semiadditive $T$-$\infty$-category, and suppose $\Dd$ admits finite $P$-products. Then a $T$-functor $G\colon \Cc \to \Dd$ is $P$-semiadditive if and only if it preserves finite $P$-products. In particular $\ulFun_T^{\Poplus}(\Cc,\Dd)$ and $\ulFun_T^{\Pprod}(\Cc,\Dd)$ are the same subcategory of $\ulFun_T(\Cc,\Dd)$.
\end{proposition}

\begin{proof}
	We start with the first case. Observe that in both cases $F$ is pointed so that \Cref{lem:NormVsBeckChevalley} applies. Adjoining over $p^*$ to the right gives a commutative triangle
	\[\begin{tikzcd}
		F_B p_! \drar{\Nm^F_p} \\
		p_!F_A \uar{\BC_!} \rar[swap]{\Nm^{\Dd}_pF_A} & p_*F_A.
	\end{tikzcd}\]
	Since $\Dd$ is a $P$-semiadditive, the bottom map is an equivalence. It thus follows from the two-out-of-three property that $\BC_!\colon p_!F_A \Rightarrow F_Bp_!$ is an equivalence if and only if $\Nm^F_p\colon F_Bp_! \Rightarrow p_*F_A$ is, proving the result.

	Next we consider the second case. Just as before the result follows from the commutativity of the triangle
	\[
	\begin{tikzcd}
		& {p_*F_A} \\
		{F_Bp_!} & {F_Bp_*,}
		\arrow["{F_B \Nm^{\Cc}_{p}}"', from=2-1, to=2-2]
		\arrow["{\BC_*}"', from=2-2, to=1-2]
		\arrow["{\Nm_p^F}", from=2-1, to=1-2]
	\end{tikzcd}
	\]
	which in turn follows from the commutative diagram
	\[\begin{tikzcd}
		{p_*p^*F_Bp_!} & {p_* F_A p^* p_!} & {p_*F_A} \\
		{F_Bp_!} & {F_Bp_*p^*p_!} & {F_Bp_*}
		\arrow["{u_p^*}"', from=2-1, to=1-1]
		\arrow["{{\Nmadj_{p}^{\Cc}}}"', from=2-2, to=2-3]
		\arrow["{{\Nmadj_{p}^{\Cc}}}", from=1-2, to=1-3]
		\arrow["{\BC_*}"', from=2-3, to=1-3]
		\arrow["{\BC_*}"', from=2-2, to=1-2]
		\arrow["{u_p^*}"', from=2-1, to=2-2]
		\arrow["\simeq", from=1-1, to=1-2]
	\end{tikzcd}\]
	The left square commutes by the triangle identity and the right by naturality.
\end{proof}

\begin{corollary}
\label{cor:SemiadditiveFunctorsBetweenSemiadditiveCats}
Let $\Cc$ and $\Dd$ be $P$-semiadditive $T$-$\infty$-categories. Then a $T$-functor $F\colon \Cc \to \Dd$ preserves finite $P$-coproducts if and only if it preserves finite $P$-products. \qednow
\end{corollary}

There exists a characterization of $P$-semiadditivity which does not make reference to the norm maps: it suffices for finite $P$-products to commute with finite $P$-coproducts.

\begin{corollary}
	\label{cor:SemiadditiveIffProductsCommuteWithCoproducts}
	Let $\Cc$ be a pointed $T$-$\infty$-category which admits finite $P$-products and finite $P$-coproducts. Then the following conditions are equivalent:
	\begin{enumerate}[(1)]
		\item The $T$-$\infty$-category $\Cc$ is $P$-semiadditive
		\item For every morphism $p \colon A \to B$ in $\finPsets$, the $T_{/B}$-functor
        \[
        p_*\colon \ulFun_{T_{/B}}(\ul{A},\pi^*_B\Cc) \to \pi^*_B\Cc
        \] preserves finite $P$-coproducts.
	\end{enumerate}
\end{corollary}
\begin{proof}
    Suppose $\Cc$ is $P$-semiadditive. Then so are the $T_{/B}$-$\infty$-categories $\pi^*_B\Cc$ and $\Fun_{T_{/B}}(\ul{A},\pi^*_B\Cc)$. Given a morphism $p\colon A \to B$, the $T_{/B}$-functor $p_*$ is a right adjoint of $p^*$ so preserves finite $P$-products. By \Cref{cor:SemiadditiveFunctorsBetweenSemiadditiveCats}, it follows that $p_*$ also preserves finite $P$-coproducts, proving that (1) implies (2).

    Conversely, applying (2) to the finite $P$-coproduct $p_!$ gives that the double Beck-Chevalley map $p_!{\pr_2}_* \Rightarrow p_*{\pr_1}_!$ associated to the pullback square \eqref{eq:PullbackDiagramNormMap} is an equivalence. It thus follows from \Cref{lem:NormMapInTermsOfExchangeMap} that the norm map $\Nm_p$ is an equivalence, showing that (2) implies (1).
\end{proof}

We finish this subsection by observing that passing to the $T$-$\infty$-category of $P$-semiadditive $T$-functors out of a small $T$-$\infty$-category $\Cc$ preserves presentability.

\begin{proposition}\label{prop:PsemiadditiveFunctorsPresentable}
	Let $\Cc$ be a small pointed $T$-$\infty$-category which admits finite $P$-coproducts. Let $\Dd$ be a presentable $T$-$\infty$-category, so that $\Dd$ in particular admits finite $P$-products by \cref{rmk:PresentableHasLimits}. Then the $T$-$\infty$-category $\ul{\Fun}^{{\Pbiprod}}(\Cc,\Dd)$ is again presentable and the inclusion
	\[
	\ul{\Fun}^{{\Pbiprod}}(\Cc,\Dd) \subset \ul{\Fun}(\Cc,\Dd)
	\]
	admits a left adjoint.
\end{proposition}
\begin{proof}
	We will exhibit $\ul{\Fun}^{{\Pbiprod}}_T(\Cc,\Dd) $ as the $T$-$\infty$-category of $S$-local objects for a parametrized family $S$ of morphisms in $\ul{\Fun}_T(\Cc,\Dd)$ (i.e.\ a set $S(B)$ of morphisms of $\Fun_{T_{/B}}(\pi_B^*\Cc,\pi_B^*\Dd)$ for every $B\in T$ which are closed under restriction). Then \cref{ex:accessible-Bousfield-presentable} implies both statements of the proposition. Since we may prove the statement after pulling back to every slice of $T$, we may assume without loss of generality that $T$ has a final object. We will describe a set $S'(1)$ of morphisms in $\Fun_T(\Cc,\Dd)$ such that $F$ is $P$-semiadditive if and only if $F$ is $S'(1)$-local; the set $S'(B)$ at any other object $B \in T$ is given by the analogous procedure applied to the slice $T_{/B}$. We then define $S(A)$ to be the union of the restriction of $S'(B)$ along every map $A\rightarrow B$ in $T$. Note that an object $F\in\ul\Fun_T(\Cc,\Dd)(A)$ is $S(A)$-local if and only if $f_* F\in\ul\Fun_T(\Cc,\Dd)(B)$ (with $f_*$ denoting the right adjoint to restriction as before) is $S'(B)$-local for every $f\colon A\rightarrow B$ in $T$. By \Cref{lem:SemiadditiveFunctorsClosedUnderLimits} this is equivalent to $F$ being $S'(A)$-local.

	By definition, a $T$-functor $F\colon \Cc\rightarrow \Dd$ is $T$-semiadditive if and only if it preserves $T$-final objects and the norm map $\Nm_p\colon F_B\circ p_!\Rightarrow p_*\circ F_A$ is an equivalence for every $p\colon A\rightarrow B$ in $\finPsets$. By presentability of $\Dd(B)$, there exists a set $\{d_i\}$ of generating objects of $\Dd(B)$ for every $B \in \finTsets$, which we may assume to be closed under restriction along maps in $\finTsets$. It follows that $F$ is semiadditive if and only if for every morphism $p\colon A \to B$ in $\finPsets$, every generator $d_i \in \Dd(B)$ and every $x \in \Cc(A)$ the following two maps of spaces are equivalences:
	\begin{enumerate}[(1)]
		\item $\Hom_{\Dd(B)}(d_i, F_B(*)) \to \Hom_{\Dd(B)}(d_i,*) \simeq *$;
		\item $\Hom_{\Dd(B)}(d_i,F_B(p_!(x)) \rightarrow \Hom_{\Dd(B)}(d_i, p_*(F_A(x))) \simeq \Hom_{\Dd(A)}(p^*(d_i), F_A(x))$.
	\end{enumerate}
	Note that this is a set's worth of conditions. We claim that these maps of spaces are obtained by applying $\Hom_{\Fun_T(\Cc,\Dd)}(-,F)$ to a certain set of maps $S'(1)$ in $\Fun_T(\Cc,\Dd)$. Since the maps are natural in the functor $F$, it suffices to prove that the source and target of each map are corepresented. Note that the functor $F\mapsto \ast$ is corepresented by the initial object of $\Fun_T(\Cc,\Dd)$. Therefore it will suffice to show that functors in $F$ of the form $\Hom_{\Dd(B)}(y,F_B(x))$ are corepresented. First recall the standard fact that the assignment $F \mapsto \Hom_{\Dd(B)}(d_i,F_B(x))$ is corepresented by the functor $y(x)\otimes d_i\colon \Cc(B)\rightarrow \Dd(B)$ in $\Fun(\Cc(B),\Dd(B))$. Here $y(x) = \Hom_{\Cc(B)}(x,-)\colon \Cc(B) \to \Spc$ denotes the Yoneda embedding, while the functor $- \otimes d_i\colon \Spc \to \Dd(B)$ denotes the standard tensoring over spaces in the cocomplete category $\Dd(B)$. To prove the claim, it thus remains to show that the evaluation functor
	\[ \ev_B\colon \Fun_T(\Cc,\Dd)\rightarrow \Fun(\Cc(B),\Dd(B))
	\]
	admits a left adjoint. Note that by \Cref{prop:CoLimitsInFunctorCategories} it preserves colimits and limits. Since both source and target are presentable the existence of the required left adjoint follows immediately from the adjoint functor theorem \cite{HTT}*{Corollary 5.5.2.9}.
\end{proof}

\subsection{Finite pointed \texorpdfstring{\for{toc}{$P$}\except{toc}{$\bm P$}}{P}-sets}

We will now introduce the $T$-$\infty$-category $\ulfinptdPsets$ of \textit{finite pointed $P$-sets} for an atomic orbital subcategory $P \subseteq T$ and prove that it is the free pointed $T$-$\infty$-category admitting finite $P$-coproducts.

\begin{definition}
	\label{def:FinitePointedPSets}
	Let $P \subseteq T$ be an atomic orbital subcategory. We define the subcategory $\ulfinptdPsets \subseteq \ul{\Spc}_{T,*}$ of \textit{finite pointed $P$-sets} as the inverse image of the subcategory $\ulfinPsets \subseteq \ul{\Spc}_{T}$ under the forgetful functor $\ul{\Spc}_{T,*} \to \ul{\Spc}_{T}$: it contains those pointed $T$-spaces $(X,f,s) \in \ul{\Spc}_{T,*}(B)$ whose underlying $T$-space $(f\colon X \to B)$ is in $\ulfinPsets$.

	Note that $\ulfinptdPsets$ is equivalent to $(\ulfinPsets)_*$, the pointed objects in the $T$-$\infty$-category of finite $P$-sets.
\end{definition}

\begin{notation}
	\label{nota:DisjointBasepoint}
	By \Cref{ex:AdjunctionGivesParametrizedAdjunctionOnTObjects}, the forgetful functor $\ul{\Spc}_{T,*} \to \ul{\Spc}_{T}$ admits a left adjoint $(-)_+\colon \ul{\Spc}_T \to \ul{\Spc}_{T,*}$. It is given at $B \in T$ by the functor
	\begin{align*}
		(-)_+\colon \PSh(T)_{/B} \to (\PSh(T)_{/B})_*\colon (X,f) \mapsto (X_+,f_+,s),
	\end{align*}
	where $X_+ := X \sqcup B$, where $f_+ := (f,\id)\colon X \sqcup B \to B$ and where $s\colon B \hookrightarrow X \sqcup B$ is the canonical inclusion. We will often abuse notation and write $X_+$ or $(X,f)_+$ instead of $(X_+,f_+,s)$.

	Observe that the $T$-functor $(-)_+\colon \ul{\Spc}_T \to \ul{\Spc}_{T,*}$ of \Cref{nota:DisjointBasepoint} restricts to a $T$-functor $(-)_+\colon \ulfinPsets \to \ulfinptdPsets$ which is left adjoint to the forgetful functor $\fgt\colon \ulfinptdPsets \to \ulfinPsets$.
\end{notation}

\begin{lemma}
	\label{lem:FinitePtdPSetsDisjointlyBased}
	Let $P \subseteq T$ be an atomic orbital subcategory. Then the $T$-functor $(-)_+\colon \ulfinPsets \to \ulfinptdPsets$ is essentially surjective: any finite pointed $P$-set $(Y,p,s) \in \ulfinptdPsets(B)$  is equivalent to one of the form $X_+$ for some $(X,q) \in \ulfinPsets(B)$.
\end{lemma}
\begin{proof}
	By definition, we may write $Y = \bigsqcup_{i=1}^n A_i$ as a finite disjoint union such that each map $p_i\colon A_i \to B$ is in $P$. The section $s\colon B \to \bigsqcup_{i=1}^n A_i$ must factor as $B \to A_i \hookrightarrow \bigsqcup_{i=1}^n A_i$ for some $i$. But this implies that the map $B \to A_i$ is a section of $p_i\colon A_i \to B$, so by \cref{lem:AtomicVsDisjointDiagonals} it must be an equivalence, exhibiting $B$ as a disjoint summand of $Y$. Defining $X$ as the disjoint union of the remaining summands gives the desired equivalence $Y \simeq X_+$ over $B$.
\end{proof}

\begin{notation}
	We will assume all pointed $P$-set over $B \in T$ are given to us in the form $X_+ = X \sqcup B$ for $(X,q) \in \ulfinPsets(B)$. This convention is justified by \Cref{lem:FinitePtdPSetsDisjointlyBased}. We emphasize that the maps $X_+\rightarrow Y_+$ of finite pointed $P$-sets over $B$ are not assumed to respect this decomposition, i.e.\ they might not be induced by maps in $\finPsets(B)$.
\end{notation}

\begin{lemma}
	\label{lem:FindisjptdPSetsHasFinitePColimits}
	The $T$-$\infty$-category $\ulfinptdPsets$ from \Cref{def:FinitePointedPSets} admits finite $P$-coproducts and the inclusion $\ulfinptdPsets \hookrightarrow \ul{\Spc}_{T,*}$ preserves finite $P$-coproducts. Furthermore, for any other $T$-$\infty$-category $\Dd$ which admits finite $P$-coproducts, a $T$-functor $F\colon \ulfinptdPsets \to \Dd$ preserves finite $P$-coproducts if and only if the composite $F \circ (-)_+$ does.
\end{lemma}
\begin{proof}
	By \Cref{ex:PointedTSpacesPAdjointable}, it suffices to prove that $\ulfinptdPsets$ is closed under finite $P$-coproducts in $\ul{\Spc}_{T,*}$. By \Cref{cor:FinPSetsHasFinitePColimits}, the $T$-category $\ulfinPsets$ admits finite $P$-coproducts and these are preserved by the (left adjoint) $T$-functor $(-)_+\colon \ulfinPsets \to \ulfinptdPsets$. Conversely it follows from \Cref{lem:FinitePtdPSetsDisjointlyBased} that every diagram in $\ulfinptdPsets$ indexed by a finite $P$-set comes from $\ulfinPsets$. The claim follows.
\end{proof}

Let $S^0\colon \ul{1} \to \ulfinptdPsets$ denote the $T$-functor given at $B \in T$ by the object $B_+ \in \ulfinptdPsets(B)$. The goal of the remainder of this subsection is to show that this map exhibits the $T$-$\infty$-category $\ulfinptdPsets$ as the free pointed $T$-$\infty$-category admitting finite $P$-coproducts.

If $\Ee$ is an $\infty$-category admitting a final object $*$, we let $\Ee_+ \subseteq \Ee_*$ denote the full subcategory of pointed objects $* \to Z$ for which there exists a pointed equivalence $Z \simeq X \sqcup *$ for some $X \in \Ee$. If $\Ee$ admits finite coproducts, then $\Ee_+$ also admits finite coproducts and the functor $(-)_+\colon \Ee \to \Ee_+\colon X \mapsto X_+ := X \sqcup *$ preserves finite coproducts. Furthermore $\Ee_+$ is pointed. We will show that the functor $(-)_+\colon \Ee \to \Ee_+$ is universal among coproduct preserving functors from $\Ee$ into a pointed $\infty$-category.

\begin{lemma}
	\label{lem:DisjointlyBasedObjectsFormFreePointing}
	Let $\Ee$ and $\Dd$ be $\infty$-categories admitting finite coproducts. Assume that $\Ee$ admits a final object and that $\Dd$ is pointed. Then precomposition with the functor $(-)_+\colon \Ee \to \Ee_+$ induces an equivalence
	\begin{align*}
		\Fun^{\sqcup,*}(\Ee_+,\Dd) \iso \Fun^{\sqcup}(\Ee,\Dd).
	\end{align*}
\end{lemma}
\begin{proof}
	We claim an inverse is given by sending a finite-coproduct-preserving functor $F\colon \Ee \to \Dd$ to the functor $\widetilde{F}\colon \Ee_+ \to \Dd$ defined by the formula
	\begin{align*}
		\widetilde{F}(X_+) := \cofib(F(*) \to F(X_+)).
	\end{align*}
	Observe that this colimit exists and is equivalent to $F(X)$ by the following pushout diagram:
	\[\begin{tikzcd}
		* \dar \rar \drar[pushout] & F(*) \dar \rar & * \dar \\
		F(X) \rar & F(X_+) \rar & F(X).
	\end{tikzcd}\]
	Here the left square is a pushout since $\Dd$ is pointed and $F$ preserves finite coproducts, and it thus follows from the pasting law of pushout diagrams that the right square is a pushout as well. This proves that the composition $\widetilde{F} \circ (-)_+$ is equivalent to $F$. It is easily observed that $\widetilde{F}$ is pointed and preserves finite coproducts.

	Now assume we are given a pointed functor $\widetilde{F}\colon \Ee_+ \to \Dd$ which preserves finite coproducts. It remains to show that for every object $Z \in \Ee_+$ the canonical map
	\begin{align*}
		\cofib(\widetilde{F}(*_+) \to \widetilde{F}(Z_+)) \to \widetilde{F}(Z)
	\end{align*}
	is an equivalence. This follows from the fact that $Z_+$ is a coproduct in $\Ee_+$ of $Z$ and $*_+$ and that $\widetilde{F}$ preserves coproducts by assumption.
\end{proof}

Let $\Cat^{\sqcup}_{\infty} \subseteq \Cat_{\infty}$ denote the (non-full) subcategory consisting of $\infty$-categories which admit finite coproducts and functors which preserve finite coproducts. Let $\Cat^{\sqcup,\pt}_{\infty} \subseteq \Cat^{\sqcup,*}_{\infty} \subseteq \Cat^{\sqcup}_{\infty}$ denote the full subcategories spanned by those $\infty$-categories with finite coproducts which admit a zero object or admit a final object, respectively.

\begin{corollary}
	\label{cor:DisjointlyBasedObjectsFormFreePointing}
	The inclusion $\Cat^{\sqcup,\pt}_{\infty} \hookrightarrow \Cat^{\sqcup,*}_{\infty}$ admits a left adjoint \[(-)_+\colon \Cat^{\sqcup,*}_{\infty} \to \Cat^{\sqcup,\pt}_{\infty}\] which on objects sends $\Ee$ to $\Ee_+$.
\end{corollary}
\begin{proof}
	We need to show that for any $\Ee \in \Cat^{\sqcup,*}_{\infty}$ and any $\Dd \in \Cat^{\sqcup,\pt}_{\infty}$, the precomposition with the map $(-)_+\colon \Ee \to \Ee_+$ induces an equivalence \[\Hom_{\Cat^{\sqcup}_{\infty}}(\Ee_+,\Dd) \iso \Hom_{\Cat^{\sqcup}_{\infty}}(\Ee,\Dd).\] This is immediate from \Cref{lem:DisjointlyBasedObjectsFormFreePointing}.
\end{proof}

\begin{corollary}
	\label{cor:FinptdPSetsFree}
	Let $\Dd$ be a pointed $T$-$\infty$-category $\Dd$ which admits finite $P$-coproducts. Then composition with $S^0\colon \ul{1} \to \ulfinptdPsets$ induces an equivalence of $T$-$\infty$-categories
	\begin{align*}
		\ulFun_T^{{\Pcoprod},*}(\ulfinptdPsets,\Dd) \to \ulFun_T(\ul{1},\Dd) \simeq \Dd.
	\end{align*}
\end{corollary}
\begin{proof}
	Note that $S^0$ is the composite $\ul{1} \xrightarrow{*} \ulfinPsets \xrightarrow{(-)_+} \ulfinptdPsets$. By \cref{cor:FinPSetsFreeOnFinitePColimits} it thus suffices to show that composition with the $T$-functor $(-)_+\colon \ulfinPsets \to \ulfinptdPsets$ induces an equivalence $\ulFun_T^{{\Pcoprod},*}(\ulfinptdPsets,\Dd) \iso \ulFun_T^{{\Pcoprod}}(\ulfinPsets,\Dd)$. It in fact suffices to show that it induces an equivalence between $T$-$\infty$-categories of \emph{fiberwise coproduct} preserving functors. Namely by the last part of \Cref{lem:FindisjptdPSetsHasFinitePColimits} this equivalence will restrict to the subcategories of $P$-coproduct preserving functors on either side. Replacing $T$ by $T_{/B}$ for every $B \in T$, it suffices to prove this on underlying $\infty$-categories. Note that the subcategory $\Cat_T^{\sqcup} \subseteq \Cat_T$ is closed under cotensoring by $\Cat_{\infty}$ and that there is a canonical equivalence $\Hom_{\Cat_{\infty}}(\Ee,\Fun^{\sqcup}_T(\Cc,\Dd)) \simeq \Hom_{\Cat_T^{\sqcup}}(\Cc,\Dd^{\Ee})$ for $\Ee \in \Cat_{\infty}$ and $\Cc,\Dd \in \Cat_T^{\sqcup}$. By the Yoneda lemma it will thus suffice to show that the functor $(-)_+\colon \ulfinPsets \to \ulfinptdPsets$ induces an equivalence $\Hom_{\Cat_T^{\sqcup}}(\ulfinptdPsets,\Dd) \iso \Hom_{\Cat_T^{\sqcup}}(\ulfinPsets,\Dd)$. This is immediate from \Cref{cor:DisjointlyBasedObjectsFormFreePointing}.
\end{proof}

\subsection{\texorpdfstring{\for{toc}{$P$}\except{toc}{$\bm P$}}{P}-commutative monoids}
Fix an atomic orbital subcategory $P\subseteq T$. In this subsection we will introduce the notion of a \textit{$P$-commutative monoid} in a $T$-$\infty$-category $\Dd$ admitting finite $P$-products. Furthermore we will show that the $T$-$\infty$-category $\ul{\CMon}^{P}(\Dd)$ of $P$-commutative monoids in $\Dd$ is the terminal $P$-semiadditive $T$-$\infty$-category equipped with a finite $P$-product preserving $T$-functor to $\Dd$.

\begin{definition}[$P$-commutative monoids, cf.\ \cite{nardin2016exposeIV}*{Definition~5.9}]
	Let $\Dd$ be a $T$-$\infty$-category which admits finite $P$-products. A \textit{$P$-commutative monoid object of $\Dd$} is a $P$-semiadditive $T$-functor $M\colon \ulfinptdPsets \to \Dd$. We define the $T$-$\infty$-category $\ulPCMon(\Dd)$ of \textit{$P$-commutative monoids} in $\Dd$ as
	\begin{align*}
		\ulPCMon(\Dd) :=\ulFun^{\Pbiprod}_T(\ulfinptdPsets,\Dd).
	\end{align*}
	We define the forgetful functor $\mathbb{U}\colon \ulPCMon(\Dd) \to \Dd$ to be given by precomposition with the $T$-functor $S^0\colon \ul{1} \to \ulfinptdPsets$.

	As a special case, we define the $T$-$\infty$-category $\ul{\mathrm{CMon}}_T^P$ of \textit{$P$-commutative monoids} as
	\begin{align*}
		\ul{\mathrm{CMon}}^{P}_T := \ulPCMon(\ul{\Spc}_T).
	\end{align*}
\end{definition}

Combining our previous results, we can immediately deduce the universal property of $P$-commutative monoids. We spell this out in the following series of statements.

\begin{proposition}
	\label{cor:CommutativeMonoidsAreSemiadditive}
	For every $T$-$\infty$-category $\Dd$ admitting finite $P$-products, the $T$-$\infty$-category $\ulPCMon(\Dd)$ is $P$-semiadditive. Furthermore, the forgetful functor $\ulPCMon(\Dd) \to \Dd$ preserves finite $P$-products.
\end{proposition}
\begin{proof}
	The first statement is a special case of \cref{prop:SemiadditiveFunctorsFormSemiadditiveCategory} for $\Cc = \ulfinptdPsets$. The second statement is a special case of \cref{lem:SemiadditiveFunctorsClosedUnderLimits} combined with \cref{prop:CoLimitsInFunctorCategories}.
\end{proof}

\begin{proposition}
	\label{prop:CommutativeMonoidsInSemiadditiveCategories}
	Given a $T$-$\infty$-category $\Dd$ admitting finite $P$-products, \[\mathbb{U}\colon \ulPCMon(\Dd) \to \Dd\] is an equivalence if and only if $\Dd$ is $P$-semiadditive.
\end{proposition}
\begin{proof}
	As $\ulPCMon(\Dd)$ is $P$-semiadditive by \Cref{cor:CommutativeMonoidsAreSemiadditive}, one direction is immediate. Conversely, if $\Dd$ is $P$-semiadditive, then \Cref{prop:FunctorSemiadditiveVsColimitPreserving} provides an equivalence
	\begin{align*}
		\ulPCMon(\Dd) = \ulFun_T^{\Pbiprod}(\ulfinptdPsets,\Dd) \simeq \ulFun_T^{\Pcoprod}(\ulfinptdPsets,\Dd).
	\end{align*}
	The result thus follows from \Cref{cor:FinptdPSetsFree}.
\end{proof}

\begin{corollary}[cf.\ \cite{nardin2016exposeIV}*{Corollary~5.11.1}]\label{cor:UniversalPropCMon}
	Let $\Cc$ and $\Dd$ be $T$-$\infty$-categories such that $\Cc$ is pointed and admits finite $P$-coproducts and $\Dd$ admits finite $P$-products. Then postcomposition with the forgetful functor $\mathbb{U}\colon \ulPCMon(\Dd) \to \Dd$ induces an equivalence
	\begin{align*}
		\Fun_T^{\Pcoprod}(\Cc,\ulPCMon(\Dd)) \to \Fun_T^{\Pbiprod}(\Cc,\Dd).
	\end{align*}
\end{corollary}
\begin{proof}
	By \Cref{prop:FunctorSemiadditiveVsColimitPreserving}, the left-hand side is equal to the $T$-$\infty$-category of $P$-semiadditive $T$-functors $\Cc \to \ulPCMon(\Dd)$. By \Cref{cor:BilinearFunctors} this is in turn equivalent to $\ulPCMon(\ulFun_T^{\Pbiprod}(\Cc,\Dd))$. The claim thus follows by combining \Cref{prop:SemiadditiveFunctorsFormSemiadditiveCategory} and \Cref{prop:CommutativeMonoidsInSemiadditiveCategories}.
\end{proof}

\begin{corollary}\label{cor:adjunctionCMon}
	The inclusion $\Cat_T^{\Pbiprod} \hookrightarrow \Cat_T^{\Pprod}$ of the $T$-$\infty$-category of $P$-semiadditive $T$-$\infty$-categories and $P$-semiadditive $T$-functors into the $T$-$\infty$-category of $T$-$\infty$-categories admitting finite $P$-products and the finite $P$-product preserving $T$-functors admits a right adjoint given by
	\begin{equation*}
		\ulPCMon(-)\colon \Cat_T^{\Pprod} \to \Cat_T^{\Pbiprod}. \qednow
	\end{equation*}
\end{corollary}

We are also interested in a presentable version of \cref{cor:adjunctionCMon}.

\begin{lemma}\label{lemma:forgetful-RA}
	Let $\Cc$ be a presentable $T$-$\infty$-category. Then $\mathbb U\colon\ul\CMon^P(\Cc)\to\Cc$ admits a left adjoint $\mathbb P$.
	\begin{proof}
		The functor $\ev_{S^0}\colon\ul\Fun_T(\ul{\mathbb F}^P_{T,*},\Cc)\to\Cc$ admits a left adjoint by \cite{martiniwolf2021limits}*{Theorem~6.3.5 and Corollary~6.3.7}. The claim follows as also the inclusion $\ul\CMon^P(\Cc)\hookrightarrow\ul\Fun_T(\ul{\mathbb F}^P_{T,*},\Cc)$ admits a left adjoint by Proposition~\ref{prop:PsemiadditiveFunctorsPresentable}.
	\end{proof}
\end{lemma}

\begin{definition}
	We define $\PrRTsemi$ to be the full subcategory of $\PrRT$ spanned by those presentable $T$-$\infty$-categories which are moreover $P$-semiadditive. Similarly we define $\PrLTsemi$.
\end{definition}

\begin{proposition}\label{prop:CMon_restricts_to_pres}
	The functor $\ulPCMon$ restricts to a functor
	\[\ulPCMon\colon \PrRT\rightarrow \PrRTsemi\]
	right adjoint to the inclusion.
\end{proposition}

\begin{proof}
	Let $\Cc$ be a presentable $T$-$\infty$-category. Note that by \cref{prop:PsemiadditiveFunctorsPresentable}, $\ulPCMon(\Cc)$ is again presentable. Furthermore suppose $G\colon \Cc\rightarrow \Dd$ is a right adjoint between presentable $T$-$\infty$-categories, and denote its left adjoint by $F$. Note that $G$ preserves finite $P$-products, and so induces a functor $\ulPCMon(G)\colon\ulPCMon(\Cc)\rightarrow \ulPCMon(\Dd)$. Because $G$ preserves local objects, the composite
	\[\begin{tikzcd}
		{\ulPCMon(\Dd)} & {\ulFun_T(\ulfinptdPsets,\Dd)} & {\ulFun_T(\ulfinptdPsets,\Cc)} & {\ulPCMon(\Cc)}
		\arrow["F", from=1-2, to=1-3]
		\arrow[hook, from=1-1, to=1-2]
		\arrow["{L^{{\Pbiprod}}(-)}", from=1-3, to=1-4]
	\end{tikzcd}\]
	is left adjoint to $\ulPCMon(R)$, where $L^{{\Pbiprod}}$ refers to the left adjoint of the inclusion $\ulPCMon\subset \ulFun_T(\ulfinptdPsets,\Cc)$ constructed in \Cref{prop:PsemiadditiveFunctorsPresentable}.

	Finally, the unit $\mathbb U$ is a right adjoint by Lemma~\ref{lemma:forgetful-RA} while the counit is even an equivalence by Proposition~\ref{prop:CommutativeMonoidsInSemiadditiveCategories}.
\end{proof}

\begin{corollary}\label{cor:PresentableUniversalPropertyCMon}
	There exists an adjunction
	\[\ulPCMon(-)\colon \PrLT \rightleftarrows \PrLTsemi\noloc \mathrm{incl}.\] Furthermore the unit $\mathbb{P}\colon \Cc\rightarrow \ulPCMon(\Cc)$ is left adjoint to the forgetful functor $\mathbb{U}$.
\end{corollary}

\begin{proof}
	Consider the adjunction constructed in \cref{prop:CMon_restricts_to_pres} and apply the equivalence $\PrLT\simeq (\PrRT)\catop$.
\end{proof}

For ease of reference we record the strongest results obtained above in one omnibus theorem:

\begin{theorem}\label{thm:Semiadd_omnibus}
	Let $\Cc$ be a $T$-$\infty$-category with finite $P$-products. The functor $\mathbb{U}\colon\ulPCMon(\Cc)\to \Cc$ exhibits $\ulPCMon(\Cc)$ as the $P$-semiadditive envelope of $\Cc$, i.e.~for every $P$-semiadditive  $T$-$\infty$-category $\tcat$ postcomposition with $\mathbb{U}$ induces an equivalence
	\begin{equation*}
		\ul\Fun^{\Pprod}(\tcat, \mathbb{U})\colon\ul\Fun^{\Pbiprod}(\tcat,\ulPCMon(\Cc))\to \ul\Fun^{\Pprod}(\tcat, \Cc).
	\end{equation*}
	Suppose now that $\Dd$ is moreover presentable. Then the left adjoint $\mathbb{P}$ of $\mathbb{U}$ exhibits $\ulPCMon(\Cc)$ as the presentable $P$-semiadditive completion of $\Cc$, i.e.~for any presentable $P$-semiadditive $T$-$\infty$-category $\tcat$ precomposition with $\mathbb{P}$ yields an equivalence
	\begin{equation*}
		\ul\Fun^\textup{L}(\mathbb{P}, \tcat)\colon\ul\Fun^\textup{L}(\ulPCMon(\Cc), \tcat)\to\ul\Fun^\textup{L}(\Cc, \tcat).\qednow
	\end{equation*}
\end{theorem}

Combining the result above with the universal property of $\ul{\Spc}_{T}$ already shows that we have for any presentable $P$-semiadditive $T$-$\infty$-category $\mathcal D$ an equivalence $\ulFun^\textup{L}_T(\ul{\CMon}^P_T,\mathcal D) \simeq\mathcal D$ of $T$-$\infty$-categories. As our final result in this subsection we will generalize this to the case where $\mathcal D$ is merely assumed to be $T$-cocomplete:

\begin{theorem}\label{thm:cocomplete-univ-cmon}
	Let $\mathcal D$ be a locally small $T$-cocomplete $P$-semiadditive $T$-$\infty$-category. Then evaluation at $\mathbb P(*)$ defines an equivalence
	\begin{equation}\label{eq:ccucm}
		\ulFun^\textup{L}_T(\ul{\CMon}^P_T,\mathcal D) \xrightarrow{\;\simeq\;} \mathcal D.
	\end{equation}
\end{theorem}
\begin{proof}
	Appealing to the universal property of $\ul\Spc_T$ and passing to adjoints, we see that $(\ref{eq:ccucm})$ agrees up to equivalence with the map
	\begin{equation*}
		\ul\Fun^\textup{R}_T(\mathcal D,\mathbb U)\colon\ul\Fun^\textup{R}_T(\mathcal D,\ul\CMon^P_T)\to\ul\Fun^\textup{R}_T(\mathcal D,\ul\Spc_T)
	\end{equation*}
	between parametrized categories of \emph{right adjoint} functors. In particular it is fully faithful by the first half of Theorem~\ref{thm:Semiadd_omnibus}, so it only remains to prove essential surjectivity.

	Replacing $\Dd$ by $\ul\Fun_T(\ul A,\Dd)$ for $A\in T$, it will be enough to construct for every $X\in\Gamma(\Dd)$ a $T$-left adjoint $F\colon\ul\CMon^P\to\mathcal D$ with $F(\mathbb P(*))\simeq X$.

	For this, we use the universal property of $\ul{\mathbb F}^P_{T,*}$ (Lemma~\ref{cor:FinptdPSetsFree}) to obtain a $P$-coproduct preserving functor $\phi\colon\ul{\mathbb F}^P_{T,*}\to\mathcal D^\op$ sending $S^0$ to $X$, which we may then extend to a left adjoint $\Phi\colon\ul\Fun_T(\ul{\mathbb F}^P_{T,*},\ul\Spc_T)\to\mathcal D$ via Proposition~\ref{prop:Universal_Property_Presheaves}. To complete the proof it suffices now to prove that $\Phi$ factors through the Bousfield localization $L^{{\Pbiprod}}\colon\Fun_T(\ul{\mathbb F}^P_{T,*},\ul\Spc_T)\to\ul\CMon^P$, or equivalently that its right adjoint takes values in $\ul\CMon^P$. However, by Remark~\ref{rk:ra-Yoneda-extension} the value of this right adjoint on $Y\in \mathcal D(A)$ is given by the composite
	\begin{equation*}
		\pi_A^*\ul{\mathbb F}^P_{T,*}\xrightarrow{\pi_A^*\phi} \pi_A^*\mathcal D^\op\xrightarrow{\maps(\blank,Y)}\Spc_{T_{/A}}\simeq\pi_A^*\ul\Spc_T
	\end{equation*}
	and the first functor sends $\pi_A^*P$-coproducts to $\pi_A^*P$-products by construction of $\phi$ and semiadditivity of $\mathcal D$ while the second one even preserves all $\pi_A^*T$-limits that exist in $\pi_A^*\mathcal D^\op$ \cite{martiniwolf2021limits}*{Corollary~4.4.9}.
\end{proof}

For use in future work, we record the following result elaborating on the construction of the inverse to $(\ref{eq:ccucm})$:

\begin{proposition}
	Write $j\colon (\ul{\mathbb F}^P_{T,*})^\op\to\ul\CMon^P_T$ for the unique finite $P$-product preserving functor sending $S^0$ to $\mathbb P(*)$. Then the restriction $j^*\colon\ul\Fun_T(\ul\CMon^P_T,\Dd)\to\ul\Fun_T((\ul{\mathbb F}^P_{T,*})^\op,\Dd)$ admits a left adjoint $j_!$, and $j^*$ and $j_!$ restrict to mutually inverse equivalences $\ul\Fun_T^\textup{L}(\ul\CMon^P_T,\Dd)\simeq\ul\Fun_T^{P\text-\times}((\ul{\mathbb F}^P_{T,*})^\op,\Dd)$.
	\begin{proof}
		Let us write $\bar\jmath$ for the composition of the Yoneda embedding $y\colon(\ul{\mathbb F}^P_{T,*})^\op\to\ul\Fun_T(\ul{\mathbb F}^P_{T,*},\ul\Spc_T)$ with the localization $L^{P\text-\oplus}$. We will first prove the proposition with $\bar\jmath$ in lieu of $j$, and then conclude in the end that in fact $j\simeq\bar\jmath$.

		\cite{martiniwolf2021limits}*{Theorem~7.1.1} shows that for any $T$-cocomplete ($P$-semiadditive) $\Dd$ the restriction $y^*\colon\ul\Fun_T(\ul\Fun_T(\ul{\mathbb F}^P_{T,*},\ul\Spc_T),\Dd)\to\ul\Fun_T((\ul{\mathbb F}^P_{T,*})^\op,\Dd)$ has a left adjoint $y_!$ inducing an equivalence $\ul\Fun_T((\ul{\mathbb F}^P_{T,*})^\op,\Dd)\simeq\ul\Fun^\text{L}_T(\ul\Fun_T(\ul{\mathbb F}^P_{T,*},\ul\Spc_T),\Dd)$, while Theorem~6.3.5 and Corollary~6.3.7 of \emph{op.~cit.}~show that $\bar\jmath^*$ admits a left adjoint $\bar\jmath_!$.

		We claim that $\bar\jmath_!$ and $\bar\jmath^*$ restrict to functors $\ul\Fun_T^{P\text-\times}\hskip 0pt minus 1pt((\ul{\mathbb F}^P_{T,*})^\op\hskip 0pt minus 1.5pt,\hskip0pt minus .25pt\Dd)\hskip0pt minus 1pt\rightleftarrows\hskip 0pt minus 1pt\ul\Fun_T^\text{L}(\ul\CMon^P_T\hskip 0pt minus .5pt,\hskip 0pt minus .5pt\Dd)$, which are then automatically adjoint to each other again. As the right adjoint $\bar\jmath^*$ in this adjunction is then moreover an equivalence by the previous theorem together with Corollary~\ref{cor:FinptdPSetsFree}, they will then be mututally inverse equivalences.

		To prove the claim, note that we have seen in the proof of the previous theorem that $y_!\colon\ul\Fun_T^{P\text-\times}((\ul{\mathbb F}^P_{T,*})^\op,\ul\Spc_T)\to\ul\Fun_T^\text{L}(\ul\Fun_T(\ul{\mathbb F}^P_{T,*},\ul\Spc_T),\Dd)$ factors through the fully faithful functor $(L^{P\text-\oplus})^*\colon \ul\Fun_T^\text{L}(\ul\CMon^P_T,\Dd)\to\ul\Fun_T^\text{L}(\ul\Fun_T(\ul{\mathbb F}^P_{T,*},\ul\Spc_T),\Dd)$. If we write $f$ for the resulting functor $\ul\Fun_T^{P\text-\times}((\ul{\mathbb F}^P_{T,*})^\op,\ul\Spc_T)\to\ul\Fun_T^\text{L}(\ul\CMon^P_T,\Dd)$, then for any $A\in T$, $X\in\ul\Fun_T^{P\text-\times}((\ul{\mathbb F}^P_{T,*})^\op,\Dd)(A)$, and $Y\in\ul\Fun_T(\ul\CMon^P_T,\Dd)(A)$
		\begin{align*}
			\maps(f(X),Y)&\simeq \maps ((L^{P\text-\oplus})^*f(X), (L^{P\text-\oplus})^*Y)\simeq \maps(y_!X,(L^{P\text-\oplus})^*Y)\\&\simeq\maps(X,y^*(L^{P\text-\oplus})^*Y)=\maps(X,\bar\jmath^*Y)\simeq\maps(\bar\jmath_!X,Y)
		\end{align*}
		by full faithfulness, the definition of $f$, and adjunction, respectively. It follows from the ordinary Yoneda Lemma that $f(X)\simeq\bar\jmath_!(X)$, and in particular the latter lives in $\ul\Fun^\text{L}_T(\ul\CMon^P_T,\Dd)(A)$, i.e.~$\bar\jmath_!$ maps $\ul\Fun_T^{P\text-\times}((\ul{\mathbb F}^P_{T,*})^\op,\Dd)$ into $\ul\Fun^\text{L}_T(\ul\CMon^P_T,\Dd)$.

		Conversely, any $T$-cocontinuous $\ul\CMon^P_T\to\Dd$ has to arise via the construction of the previous theorem, i.e.~$f$ is essentially surjective. But $\bar\jmath^*f=y^*(L^{P\text-\oplus})^*f\simeq y^*y_!\simeq\id$, so $\bar\jmath^*$ restricts to
		\begin{equation}\label{eq:barjmath-restriction}
			\bar\jmath^*\colon\ul\Fun^\text{L}_T(\ul\CMon^P_T,\Dd)\to\ul\Fun^{P\text-\times}_T((\ul{\mathbb F}^P_{T,*})^\op,\Dd).
		\end{equation}
		It only remains to show that $\bar\jmath$ agrees with $j$ as constructed in the statement of the proposition. As $\bar\jmath(S^0)\simeq\mathbb P(*)$, we only need to show that $\bar\jmath$ preserves finite $P$-products. But this follows at once by taking $\Dd=\ul\CMon^P_T$ and chasing the identity through $(\ref{eq:barjmath-restriction})$.
	\end{proof}
\end{proposition}

We can now slightly strengthen the second half of Theorem~\ref{thm:Semiadd_omnibus} in the case of $\ul\Spc_T$:

\begin{corollary}\label{cor:cocomplete-univ-cmon-relative}
	Let $\mathcal S$ be a $T$-$\infty$-category equivalent to $\ul\Spc_T$ and let $\mathcal D$ be any locally small $P$-semiadditive $T$-cocomplete $T$-$\infty$-category. Then precomposition with the $T$-functor $\mathbb P \colon \mathcal S \to \ul{\CMon}^P(\mathcal S)$ induces an equivalence
	\begin{equation*}
		\ul\Fun^\textup{L}_T(\ul\CMon^P(\mathcal S),\mathcal D)\xrightarrow{\;\simeq\;}\ul\Fun^\textup{L}_T(\mathcal S,\mathcal D).\qednow
	\end{equation*}
\end{corollary}

\begin{remark}
	We will prove in forthcoming work that \Cref{cor:cocomplete-univ-cmon-relative} in fact holds for any presentable $T$-$\infty$-category $\mathcal S$.
\end{remark}

\subsection{Commutative monoids in \texorpdfstring{\for{toc}{$\ul{\Ee}_T$}\except{toc}{$\bm{\ul{\Ee}_T}$}}{ET}}\label{subsec:P-semiadditive-Un}
Let $\Ee$ be an $\infty$-category. Recall that a $T$-functor $F\colon \ulfinptdPsets \to \ul{\Ee}_T$ corresponds to a functor $\widetilde{F}\colon \smallint \ulfinptdPsets \to \Ee$ of $\infty$-categories, see \Cref{lem:AdjunctionGrothendieckConstructionTObjects}. We will now give a characterization of those functors $\widetilde{F}$ whose associated $T$-functor $F$ is a $P$-semiadditive monoid in $\ul{\Ee}_T$. We start with an explicit description of the adjoint norm map $\Nmadj_p\colon p^*p_! \Rightarrow \id$ associated to $\ulfinptdPsets$.

\begin{lemma}
	\label{lem:UnitMapDisjointSummandInclusion}
	Let $P \subseteq T$ be an atomic orbital subcategory. Consider a map $p\colon A \to B$ in $\finPsets$ and let $f\colon X \to A$ and $g\colon Y \to A$ be a morphisms in $\PSh(T)$. Then the map $1 \times_p 1\colon X \times_A Y \to X \times_B Y$ is a disjoint summand inclusion.
\end{lemma}
\begin{proof}
	Using \Cref{prop:FinitePSetsHasDisjointDiagonalInclusions}, this follows directly from the observation that the map $X \times_A Y \to X \times_B Y$ is a base change of the disjoint summand inclusion $\Delta\colon A \to A \times_B A$ along the map $f \times_B g\colon X \times_B Y \to A \times_B A$.
\end{proof}

\begin{construction}
	\label{cstr:AdjointNormMapFinitePointedPSets}
	Consider a morphism $p\colon A \to B$ in $\finPsets$. For any finite $P$-set $(X,q) \in \ulfinPsets(A)$, the unit map $(1,q)\colon X \to X \times_B A = p^*p_!X$ is a disjoint summand inclusion by \Cref{lem:UnitMapDisjointSummandInclusion}, and thus we may choose an identification
	\begin{align*}
		X \times_B A \simeq X \sqcup J_X
	\end{align*}
	for some finite $P$-set $J_X \in \ulfinPsets(A)$. In particular we obtain a map $p^*p_!(X_+) \to X_+$ in $\ulfinPsets(A)$ defined as the following composite:
	\begin{align*}
		p^*p_!(X_+) \simeq (X \times_B A)_+ \simeq (X \sqcup J_X)_+ \to X_+,
	\end{align*}
	where the last map projects away the disjoint component $J_X$ to the disjoint basepoint.
\end{construction}

\begin{lemma}\label{lem:norm_is_project_away}
	The map $p^*p_!(X_+) \to X_+$ constructed in \Cref{cstr:AdjointNormMapFinitePointedPSets} is homotopic to the adjoint norm map $\Nmadj_p\colon p^*p_!(X_+) \to X_+$ associated to the $T$-$\infty$-category $\ulfinptdPsets$.
\end{lemma}

\begin{proof}
	Choose a map $J_A \hookrightarrow A \times_B A$ exhibiting $J_A$ as a complement of the disjoint summand inclusion $\Delta\colon A \hookrightarrow A \times_B A$. The resulting equivalence $A \times_B A \simeq A \sqcup J_A$ induces an equivalence $\ulfinptdPsets(A \times_B A) \simeq \ulfinptdPsets(A \sqcup J_A) \simeq \ulfinptdPsets(A) \times \ulfinptdPsets(J_A)$. Pulling back the decomposition $A \times_B A \simeq A \sqcup J_A$ along the map $X \times_B A \to A \times_B A$ gives a decomposition $X \times_B A \simeq X \sqcup J_X$, and it follows that the object $\pr_2^*(X_+) \simeq (X \times_B A)_+ \in \ulfinptdPsets(A \times_B A)$ corresponds to the pair $(X_+,{J_X}_+) \in \ulfinptdPsets(A) \times \ulfinptdPsets(J_A)$. By \Cref{lem:AlphaIsDiagonalMatrix}, the transformation $\alpha\colon \pr_2^* \Rightarrow \pr_1^*$ corresponds to a transformation of functors into $\ulfinptdPsets(A) \times \ulfinptdPsets(J_A)$ which on the first component is the identity and on the second component is the zero-map which projects everything onto the disjoint basepoint. The description from \cref{cstr:AdjointNormMapFinitePointedPSets} follows.
\end{proof}

\begin{notation}\label{notation:elements-of-grothendieck}
	We will abuse notation and denote objects of the unstraightening $\smallint \ulfinptdPsets$ by pairs $(A,X_+)$, where $A \in T$ and $(X,q\colon X \to A) \in \ulfinPsets(A)$ is a finite $P$-set. We will specify $q$ explicitly whenever confusion might arise.
\end{notation}

\begin{construction}[Parametrized Segal map]
	\label{cstr:ParametrizedSegalMap}
	Consider a map $p\colon A \to B$ in $P$, a map $C \to B$ in $T$ and a finite pointed $P$-set $X_+ \to A$ in $\ulfinptdPsets(A)$. Since $p$ is in $P$, the pullback $A \times_B C$ of $p$ along $C \to B$ may be written as a disjoint union of maps $p_i\colon C_i \to C$ in $P$:
	\[\begin{tikzcd}
		\bigsqcup_{i=1}^n C_i \rar \dar[swap]{(p_i)_{i=1}^n} \drar[pullback] & A \dar{p} \\
		C \rar & B.
	\end{tikzcd}\]
	We will we construct for each $i \in \{1, \dots, n\}$ a \textit{parametrized Segal map}
	\begin{equation*}
		\rho_{i}\colon  (C, (X \times_B C)_+) \to (C_i,(X \times_A C_i)_+)
	\end{equation*}
	in $\smallint \ulfinptdPsets$. To give such a map, we need to provide a map $C_i \to C$ in $T$, which we simply take to be the map $p_i\colon C_i \to C$, and a map $p_i^*(X \times_B C)_+ \simeq (X \times_B C_i)_+ \to (X \times_A C_i)_+$ in $\ulfinptdPsets(C_i)$. Recall from \Cref{lem:UnitMapDisjointSummandInclusion} that the map $X \times_A C_i \to X \times_B C_i$ is a disjoint summand inclusion, so that we may choose an equivalence
	\begin{align*}
		(X \times_B C_i) \simeq (X \times_A C_i) \sqcup J_i,
	\end{align*}
	where $J_i \to C_i$ is some finite $P$-set. The required map $(X \times_B C_i)_+ \simeq (X \times_A C_i)_+ \vee {J_i}_+ \to (X \times_A C_i)_+$ is now given by projecting away the second summand.
\end{construction}

\begin{proposition}\label{prop:CMon_in_T_objects}
	Let $\Ee$ be an $\infty$-category and consider a $T$-functor $F\colon \ulfinptdPsets \to \ul{\Ee}_T$. Denote by $\widetilde{F}\colon \smallint \ulfinptdPsets \to \Ee$ the functor associated to $F$ under the equivalence of \Cref{lem:AdjunctionGrothendieckConstructionTObjects}. Then $F$ is a $P$-semiadditive monoid in $\ul{\Ee}_T$ if and only if $F$ is fiberwise semiadditive and for every map $p\colon A \to B$ in $P$, every map $f\colon C \to B$ in $T$ and every finite pointed $P$-set $X_+ \in \ulfinptdPsets(A)$, the map
	\begin{equation*}
		(\widetilde{F}(\rho_i))_{i=1}^n\colon \widetilde{F} (C,(X \times_B C)_+) \to \prod_{i=1}^n \widetilde{F}(C_i,(X \times_A C_i)_+)
	\end{equation*}
	induced by the parametrized Segal maps is an equivalence.
\end{proposition}

\begin{proof}
	By \Cref{cor:PSemiadditivity}, the $T$-functor $F$ is $P$-semiadditive if and only if it is fiberwise semiadditive and for all maps $p\colon A \to B$ in $P$ the transformation $\Nm^F_p\colon F_B \circ p_! \Rightarrow p_* \circ F_A$ of functors $\ulfinptdPsets(A) \to \ul{\Ee}_T(B) = \Fun(T_{/B}\catop, \Ee)$ is an equivalence. Since we may check this pointwise, it suffices to show that for every finite $P$-set $X_+ \in \ulfinptdPsets(A)$ and every object $f\colon C \to B$ of $T_{/B}$, the induced map
	\begin{align*}
		F_B(p_!(A,X_+))(C,f)
		\to
		(p_*(F_A(A,X_+))(C,f)
	\end{align*}
	is an equivalence. By definition, this map is given by the composite
	% https://q.uiver.app/?q=WzAsNCxbMCwwLCJGX0IoKEIsWF8rKSkoQyxmKSJdLFsxLDAsInBfKnBeKihGX0IoKEIsWF8rKSkpKEMsZikiXSxbMSwxLCJGX0EocF4qcF8hKEEsWF8rKSkocF4qKEMsZikpIl0sWzIsMSwiXHRGX0EoQSxYXyspKHBeKihDLGYpKS4iXSxbMSwyLCJcXHNpbSJdLFswLDEsInVfcF4qIl0sWzIsMywiXFxObWFkal9wIl1d
	\[\begin{tikzcd}[cramped]
		{F_B((B,X_+))(C,f)} & {p_*p^*(F_B((B,X_+)))(C,f)} \\
		& {F_A(p^*p_!(A,X_+))(p^*(C,f))} & {	F_A(A,X_+)(p^*(C,f)).}
		\arrow["\sim", from=1-2, to=2-2]
		\arrow["{u_p^*}", from=1-1, to=1-2]
		\arrow["{\Nmadj_p}", from=2-2, to=2-3]
	\end{tikzcd}\]
	To make this composite explicit it will be useful to consider the objects of $\ul{\Ee}_T(B)$ as functors from $(\finTsets_{/B})\catop$ to $\Ee$ by limit extending. Similarly it will be useful to consider $F$ as a natural transformation of functors from $\finTsets$ to $\Cat_\infty$ by again limit extending. If we make both of these extensions we may again apply \cref{lem:AdjunctionGrothendieckConstructionTObjects} to conclude that $F$ is induced by a functor $\bar{F}\colon \int_\finTsets \ulfinptdPsets \rightarrow \Ee$. Namely we recall from
	\cref{rmk:unwinding_T_object_adj} that given a $T$-set $X$ and a pointed $P$-set $Y\rightarrow X$ over $X$, $F_X(X,Y_+)(f\colon Z\rightarrow X) = \bar{F}(f^*(X,Y_+)) = \bar{F}(Z,(Y\times_X Z)_+)$. Using this identification we find that the composite above is equivalent to
	\[\hskip-7.83pt\hfuzz=8pt\begin{tikzcd}[cramped]
		{\bar{F}(C,X\times_B C)} & {\bar{F}(C\times_B A,X\times_B(C\times_B A))} & {\bar{F}(C\times_B A, X\times_A (C\times_B A)),}
		\arrow["{\bar{F}(\phi_p)}", from=1-1, to=1-2]
		\arrow["{\bar{F}(\Nmadj_p)}", from=1-2, to=1-3]
	\end{tikzcd}\]
	where $\phi_p$ is a cocartesian edge expressing $X\times_B(C\times_B A)$ as a pullback of $X\times_B C$ along $u_p\colon C\times_B A\rightarrow C$. Now recall that $\bar{F}$ was defined to be the limit extension of $F$, and so given a decomposition $C\times_B A \simeq \coprod C_i$, we find that
	\[\bar{F}(C\times_B A, X\times_A (C\times_B A)) \iso \prod \bar{F}(C_i, X\times_A C_i).\] To conclude we would like to show that projecting the composite above to any factor agrees with the map constructed in \cref{cstr:ParametrizedSegalMap}. For this observe that by definition applying $\bar{F}$ to a cocartesian edge over $\iota\colon C_j \hookrightarrow C\times_B A$ gives the projection \[\pr_j\colon \prod_i \bar{F}(C_i, X\times_A C_i) \rightarrow \bar{F}(C_j,X\times_A C_j) \] Therefore we can compute the top-right way around the following commutative diagram
	% https://q.uiver.app/?q=WzAsNixbMSwxLCJcXGJhcntGfShDXFx0aW1lc19CIEEsWFxcdGltZXNfQihDXFx0aW1lc19CIEEpKSJdLFswLDEsIlxcYmFye0Z9KEMsWFxcdGltZXNfQiBDKSJdLFsyLDEsIlxcYmFye0Z9KENcXHRpbWVzX0IgQSwgWFxcdGltZXNfQSAoQ1xcdGltZXNfQiBBKSkiXSxbMywwXSxbMiwyLCJcXGJhcntGfShDX2osWFxcdGltZXNfQSBDX2opIl0sWzEsMiwiXFxiYXJ7Rn0oQ19qLFhcXHRpbWVzX0IgQ19qKSJdLFsxLDAsIlxcYmFye0Z9KFxcbWF0aGNhbHtVfV9wKSJdLFswLDIsIlxcYmFye0Z9KFxcTm1hZGopIl0sWzIsNCwiXFxwcl9qIl0sWzUsNCwiXFxiYXJ7Rn0oXFxwcl9qKFxcTm1hZGopKSIsMl0sWzAsNSwiXFxwcl9qIiwyXSxbMSw1LCJcXGJhcntGfShcXG1hdGhjYWx7VX1fe3Bfan0pIiwyXV0=
	\[\hskip-6.44pt\hfuzz=7pt\begin{tikzcd}[cramped]
		{\bar{F}(C,X\times_B C)} & {\bar{F}(C\times_B A,X\times_B(C\times_B A))} & {\bar{F}(C\times_B A, X\times_A (C\times_B A))} \\
		& {\bar{F}(C_j,X\times_B C_j)} & {\bar{F}(C_j,X\times_A C_j)}
		\arrow["{\bar{F}(\phi_p)}", from=1-1, to=1-2]
		\arrow["{\bar{F}(\Nmadj_p)}", from=1-2, to=1-3]
		\arrow["{\bar{F}(\phi_{\iota})}", from=1-3, to=2-3]
		\arrow["{\bar{F}(\iota^*(\Nmadj_p))}"', from=2-2, to=2-3]
		\arrow["{\bar{F}(\phi_{\iota})}"', from=1-2, to=2-2]
		\arrow["{\bar{F}(\phi_{p_j})}"', from=1-1, to=2-2]
	\end{tikzcd}\]
	by instead going along the bottom. Once again $\phi_{\iota}$ is our notation for a cocartesian edge over $\iota$. Because cocartesian edges compose we see that $\phi_{p_j}$ is a cocartesian edge witnessing $X\times_B C_j$ as the pullback of $X\times_B C$ along the map $C_j\rightarrow C$. Using the description of $\Nmadj_p$ given in \cref{lem:norm_is_project_away} we find that $\iota^*(\Nmadj_p))$ is equivalent to the map $X\times_B C_i\rightarrow X\times_A C_i$ given in \cref{cstr:ParametrizedSegalMap}. Finally note that by definition $\bar{F}$ agrees with $\widetilde{F}$ on the full subcategory over $T\subset \finTsets$. Therefore the proposition follows.
\end{proof}

We now show that the $P$-semiadditivity of a functor $\tilde{F}\colon \smallint \ulfinptdPsets \to \Ee$ in fact follows from substantially less than the previous proposition suggests.

\begin{observation}\label{obs:unique-segal-map}
	Let $X_+\in \ulfinptdPsets(A)$ be a finite pointed $P$-set, and let $p\colon A\rightarrow B$ be a map in $P$. Furthermore let $C\rightarrow B$ be the identity of $B$. Considering the parametrized Segal maps associated to this data, we note that $A\times_B B = A$, so there is just one. We call this map $\rho_{p,X}$. If $X = A_+$, we simply write $\rho_p$.
\end{observation}

\begin{proposition}\label{prop:parametrized-segal-special-map}
	Let $\Ee$ be an $\infty$-category and consider a $T$-functor $F\colon \ulfinptdPsets \to \ul{\Ee}_T$ which corresponds to a functor $\widetilde{F}\colon \smallint \ulfinptdPsets \to \Ee$ of $\infty$-categories. Then $F$ is a $P$-semiadditive monoid in $\ul{\Ee}_T$ if and only if $F$ is fiberwise semiadditive and for every map $p\colon A \to B$ in $P$, the map
	\begin{equation*}
		\widetilde{F}(\rho_{p})\colon \widetilde{F}(B,A_+) \to  \widetilde{F}(A,A_+)
	\end{equation*}
	is an equivalence.
\end{proposition}

\begin{proof}
	First we observe that $F$ is a $P$-semiadditive monoid in $\ul{\Ee}_T$ if and only if $F$ is fiberwise semiadditive and for every map $p\colon A \to B$ in $P$ and every finite pointed $P$-set $X_+ \in \ulfinptdPsets(A)$, the map
	\begin{equation*}
		\widetilde{F}(\rho_{p,X})\colon \widetilde{F}(B,X_+) \to  \widetilde{F}(A,X_+)
	\end{equation*}
	is an equivalence. For this it suffices to observe that the following triangle commutes
	\[
	\begin{tikzcd}[column sep = large]
		{\widetilde{F} (C,(X \times_B C)_+)} & {\prod_{i=1}^n \widetilde{F}(C_i,(X \times_A C_i)_+)} \\
		& {\prod_{i=1}^n \widetilde{F}(C,(X \times_A C_i)_+).}
		\arrow[from=1-1, to=2-2]
		\arrow["{(\widetilde{F}(\rho_i))_{i=1}^n}", from=1-1, to=1-2]
		\arrow["{(\widetilde{F}(\rho_{p_i,X\times_A C_i}))_{i=1}^n}"', from=2-2, to=1-2]
	\end{tikzcd}
	\]
	Next suppose that $X= \coprod C_i$. We note that by fiberwise semi-additivity of $F$, $\tilde{F}(\rho_{p,X})$ is equal to a product of the $\tilde{F}(\rho_{p,C_i})$, and therefore we can further reduce to the case where $X = C$ is in $T$. Write $q\colon C\rightarrow A$ for the map in $P$ expressing $C$ as a finite $P$-set over $A$. Finally we claim that the following diagram
	\[
	\begin{tikzcd}
		(B,C_+) \arrow[rd, "\rho_{pq}"] \arrow[r, "\rho_{p,C}"] &  (A,C_+) \arrow[d, "\rho_p"]  \\
		&   (C,C_+)
	\end{tikzcd}
	\] commutes in $\int \ulfinptdPsets$. This can readily be checked from the definitions. Therefore after applying $\tilde{F}$, the 2-out-of-3 property implies that it suffices to assume that $\tilde{F}(\rho_p)$ is an equivalence for all $p\in P$.
\end{proof}

\begin{remark}\label{rmk:otherlevels}
	While \Cref{prop:parametrized-segal-special-map} gives an explicit description of the underlying $\infty$-category of $\ulPCMon(\ul{\Ee}_T)$, a similar analysis in fact describes the whole $T$-$\infty$-category $\ulPCMon(\ul{\Ee}_T)$. At an object $B' \in T$, it consists of those $T$-functors $F\colon \ulfinptdPsets \times \ul{B'} \to \ul{\Ee}_T$ whose curried map $F'\colon \ulfinptdPsets \to \ulFun(\ul{B'},\ul{\Ee}_T)$ is $P$-semiadditive, see \Cref{cor:SemiadditiveFunctorsUnderBaseChangeAdjunction}. On the other hand, the $T$-functor $F$ corresponds to a functor $\tilde{F}\colon \smallint (\ulfinptdPsets \times \ul{B'}) \rightarrow \Ee$ by \Cref{lem:AdjunctionGrothendieckConstructionTObjects}. Carrying out the same analysis as in the proofs of \Cref{prop:CMon_in_T_objects} and \Cref{prop:parametrized-segal-special-map} shows that $F$ corresponds to a $P$-semiadditive functor $F'\colon \ulfinptdPsets \to \ulFun_T(\ul{B'},\ul{\Ee}_T)$ if and only if the following conditions are satisfied:
	\begin{itemize}
		\item The $T$-functor $F'$ is fiberwise semiadditive; put differently, for any $f\colon B\to B'$ the restriction of $\tilde F$ to the (non-full) subcategory $\ul{\mathbb F}^P_{T,*}(B)\times\{f\}\subset\ul{\mathbb F}^P_{T,*}(B)\times\ul{B'}(B)\subset\smallint \big(\ulfinptdPsets \times \ul{B'}\big)$ is semiadditive in the usual sense.
		\item For every map $p \colon A \to B$ in $P$ and every map $f\colon B\to B'$ in $T$, the map
		\[
		\tilde{F}(\rho_p,p)\colon \tilde{F}(B,A_+,f) \to \tilde{F}(A,A_+,p \circ f)
		\]
		is an equivalence.
	\end{itemize}
\end{remark}

\section{The universal property of special global \texorpdfstring{$\Gamma$}{Gamma}-spaces}\label{sec:univ-prop-global-Gamma}
In this section we want to identify the global $\infty$-category of $\Orb$-commutative monoids in global spaces with the various models of \emph{globally} and \emph{$G$-globally coherently commutative monoids} studied in \cite{schwede2018global}*{Chapter 2} and \cite{g-global}*{Chapter 2}. In particular, after evaluating at the trivial group, this will yield an equivalence between the underlying ordinary $\infty$-category of $\Orb$-commutative monoids in global spaces with Schwede's \emph{ultra-commutative monoids} with respect to finite groups.

For this, the model based on so-called \emph{(special) $G$-global $\Gamma$-spaces} will be the most convenient; we recall the relevant theory in \ref{subsec:GGlobalGammaSpaces} below and show how $G$-global $\Gamma$-spaces assemble into a global $\infty$-category $\ul{\GammaS}^\text{gl}$. In~\ref{subsec:comp-all-Gamma} we will then identify $\ul{\GammaS}^\text{gl}$ with a certain parametrized functor category, from which we will deduce the desired comparison between \emph{special} $G$-global $\Gamma$-spaces and $\ul\CMon^{\Orb}(\ul\Spc_{\Glo})$ in~\ref{subsec:universal-gamma}. This will then immediately imply various universal properties of global $\Gamma$-spaces, including Theorem~\ref{introthm:universal-prop-gamma} from the introduction.

\subsection{A reminder on \texorpdfstring{\for{toc}{$G$}\except{toc}{$\bm G$}}{G}-global \texorpdfstring{\for{toc}{$\Gamma$}\except{toc}{$\bm\Gamma$}}{Γ}-spaces} \label{subsec:GGlobalGammaSpaces}
Segal \cite{segal} introduced \emph{(special) $\Gamma$-spaces} as a model of commutative monoids in the $\infty$-category of spaces, and an equivariant generalization of his theory was later established by Shimakawa \cite{shimakawa}. We will be concerned with the following $G$-global refinement \cite{g-global}*{Section~2.2} of this story:

\begin{definition}
	We write $\Gamma$ for the category of finite pointed sets and pointed maps. For any $n\ge0$ we let $n^+\mathrel{:=}\{0,\dots,n\}$ with basepoint $0$.

	We moreover write $\cat{$\bm\Gamma$-$\bm{E\mathcal M}$-$\bm G$-SSet}$ for the category of functors $\Gamma\to \cat{$\bm{E\mathcal M}$-$\bm G$-SSet}$. A map $f\colon X\to Y$ in $\cat{$\bm\Gamma$-$\bm{E\mathcal M}$-$\bm G$-SSet}$ (i.e.~a natural transformation) is called a \emph{$G$-global level weak equivalence} if $f(S_+)\colon X(S_+)\to Y(S_+)$ is a $(G\times\Sigma_S)$-global weak equivalence (with respect to the $\Sigma_S$-action induced by the tautological action on $S$) for every finite set $S$.

	Similarly, we write $\cat{$\bm\Gamma$-$\bm G$-$\bm{\mathcal I}$-SSet}$ for the category of functors $X\colon\Gamma\to \cat{$\bm G$-$\bm{\mathcal I}$-SSet}$, and we define $G$-global level weak equivalences in $\cat{$\bm\Gamma$-$\bm G$-$\bm{\mathcal I}$-SSet}$ analogously.
\end{definition}

We will refer to objects of either of these categories as \emph{$G$-global $\Gamma$-spaces}. Beware that \cite{g-global} reserves this name for functors $X$ for which $X(0^+)$ is a terminal object, while for us the above definition will be more useful. However, we will later only be interested in so-called \emph{special} $G$-global $\Gamma$-spaces, for which this technicality will turn out to be irrelevant, see Proposition~\ref{prop:special-pointed-vs-unpointed} below.

\subsubsection{Model categorical properties} Just like in the unstable case we have the following Elmendorf type theorem expressing the homotopy theory of special $G$-global $\Gamma$-spaces in terms of enriched presheaves:

\begin{proposition}\label{prop:elmendorf-Gamma}
	The $G$-global level weak equivalences are part of a simplicial combinatorial model structure on $\cat{$\bm\Gamma$-$\bm{E\mathcal M}$-$\bm G$-SSet}$.

	Moreover, if we write $\OGglGamma\subset \cat{$\bm\Gamma$-$\bm{E\mathcal M}$-$\bm G$-SSet}$ for the full subcategory spanned by the objects $\Gamma_{H,S,\phi}\mathrel{:=}(\Gamma(S_+,\blank)\times E\mathcal M\times G_\phi)/H$ (where $H$ is a finite group, $S$ a finite $H$-set, $\phi\colon H\to G$ a homomorphism, and $G_\phi$ denotes $G$ with $H$ acting from the right via $\phi$), then the enriched Yoneda embedding induces a functor \[ \Phi_\Gamma\colon\cat{$\bm\Gamma$-$\bm{E\mathcal M}$-$\bm G$-SSet}\to\PSH(\OGglGamma)\]
	which is the right half of a Quillen equivalence when we equip the right hand side with the projective model structure.
	\begin{proof}
		For any finite group $H$, any finite $H$-set $S$, and any homomorphism $\phi\colon H\to G$, the functor $X\mapsto X(S_+)^\phi$ preserves filtered colimits, pushouts along injections, and it is corepresented by $\Gamma_{H,S,\phi}$ (via evaluation at $[\id,1,1]$). Thus, the objects of $\OGglGamma$ form a \emph{set of orbits} in the sense of \cite{dwyer-kan-elmendorf}*{2.1}, and the above statements are instances of Theorems~2.2 and~3.1 of \emph{op.~cit.}
	\end{proof}
\end{proposition}

\begin{remark}\label{rk:OGgl-Gamma-morphsim-spaces}
	We can make the morphism spaces in $\OGglGamma$ explicit, analogously to Remark~\ref{rk:OGgl-morphism-spaces}: as observed in the above proof, we have for any $(H,S,\phi)$ as above and any $G$-global $\Gamma$-space $X$ an isomorphism
	\begin{equation*}
		\epsilon\colon\maps(\Gamma_{H,S,\phi},X)\to X(S_+)^\phi
	\end{equation*}
	given by evaluation at $[\id,1,1]$. Specializing this to $X=\Gamma_{K,T,\psi}$, we see that $\OGglGamma$ is a $(2,1)$-category (the quotient $\Gamma_{K,T,\psi}=(\Gamma(T_+,\blank)\times E\mathcal M\times G_\psi)/K$ being the nerve of a groupoid as $K$ acts freely on $E\mathcal M$) and that $n$-simplices of $\maps(\Gamma_{H,S,\phi},\Gamma_{K,T,\psi})$ correspond to $\phi$-fixed classes $[f;u_0,\dots,u_n;g]$ where $f\colon T_+\to S_+$, $u_0,\dots,u_n\in\mathcal M$, and $g\in G$.

	Moreover, a direct computation shows that under the above identification composition is given by
	\begin{equation*}
		[f';u'_0,\dots,u_n';g'][f;u_0,\dots,u_n;g]=[ff';u_0u_0',\dots, u_nu_n';gg']
	\end{equation*}
	and that the following diagram commutes for any $X\in\cat{$\bm\Gamma$-$\bm{E\mathcal M}$-$\bm G$-SSet}$:
	\begin{equation*}
		\begin{tikzcd}[column sep=huge]
			\Phi(X)(\Gamma_{K,T,\psi})\arrow[d,"\epsilon"']\arrow[r, "{\Phi(X)[f;u_0,\dots,u_n;g]}"] &[3em] \Phi(X)(\Gamma_{H,S,\phi})\arrow[d,"\epsilon"]\\
			X(T_+)^\psi\arrow[r, "X(f)\circ \big({[u_0,\dots,u_n;g]\cdot\blank}\big)"'] & X(S_+)^\phi\llap.
		\end{tikzcd}
	\end{equation*}
\end{remark}

\subsubsection{The global $\infty$-category of global $\Gamma$-spaces}
Letting $G$ vary, the categories $\cat{$\bm\Gamma$-$\bm{E\mathcal M}$-$\bm G$-SSet}$ together with the $G$-global weak equivalences assemble into a global relative category with functoriality given by restrictions (apply Lemma~\ref{lemma:alpha-star-EM} with $\alpha$ replaced by $\alpha\times\Sigma_S$). Localizing, we then get a global $\infty$-category $\ul{\Gamma\mathscr S}^\textup{gl}$. Analogously, we obtain a global $\infty$-category $\ul{\Gamma\mathscr S}^\text{gl}_{\mathcal I}$ whose value at a finite group $G$ is the localization of $\cat{$\bm\Gamma$-$\bm G$-$\bm{\mathcal I}$-SSet}$ at the $G$-global weak equivalences, with functoriality given via restrictions.

\begin{proposition}\label{prop:Gamma-I-vs-EM}
	The evaluation functor $\ev_\omega$ induces an equivalence $\ul\GammaS^\textup{gl}_{\mathcal I}\simeq\ul\GammaS^\textup{gl}$.
	\begin{proof}
		Precisely the same argument as in \cite{g-global}*{Theorem~2.2.33} shows that the functor $\ev_\omega\colon\cat{$\bm\Gamma$-$\bm G$-$\bm{\mathcal I}$-SSet}\to \cat{$\bm\Gamma$-$\bm{E\mathcal M}$-$\bm G$-SSet}$ admits a homotopy inverse for any $G$ (given by applying the homotopy inverse of $\cat{$\bm{\mathcal I}$-SSet}\to\cat{$\bm{E\mathcal M}$-SSet}$ levelwise).
	\end{proof}
\end{proposition}

For every $G$-global $\Gamma$-space $X$, evaluating at $1^+$ (with trivial action) yields an \emph{underlying $G$-global space} $X(1^+)$, and this obviously yields a global functor $\mathbb U\colon\ul\GammaS^\text{gl}\to\ul\S^\text{gl}$. For later use we record:

\begin{lemma}
	The global functor $\mathbb U$ admits a left adjoint, which is pointwise induced by $\Gamma(1^+,\blank)\times \blank$.
	\begin{proof}
		By the Yoneda Lemma we have an adjunction
		\begin{equation*}
			\Gamma(1^+,\blank)\times\blank\colon\cat{$\bm{E\mathcal M}$-SSet}\rightleftarrows\cat{$\bm\Gamma$-$\bm{E\mathcal M}$-SSet} :\ev_{1^+},
		\end{equation*}
		and for every finite group $G$ pulling through the $G$-actions yields an adjunction $\cat{$\bm{E\mathcal M}$-$\bm G$-SSet}\rightleftarrows\cat{$\bm\Gamma$-$\bm{E\mathcal M}$-$\bm G$-SSet}$ of $1$-categories such that both functors are homotopical. In particular, $\mathbb U$ admits a pointwise adjoint of the above form.

		For the Beck-Chevalley condition it suffices now to observe that since all functors are homotopical, the Beck-Chevalley comparison map of $\infty$-categorical localizations can be modelled by the $1$-categorical Beck-Chevalley map, and the latter is even the identity by construction.
	\end{proof}
\end{lemma}

\subsubsection{Specialness}
Just like in the non-equivariant case, in the theory of global coherent commutativity one typically isn't interested in \emph{all} $G$-global $\Gamma$-spaces, but only those satisfying a certain `specialness' condition (although the fact that there are \emph{non-special} $G$-global $\Gamma$-spaces is what will make this model so convenient for our comparison):

\begin{definition}[cf.~\cite{g-global}*{Definition~2.2.50}]
	A $G$-global $\Gamma$-space $X\colon\Gamma\to\cat{$\bm{E\mathcal M}$-$\bm G$-SSet}$ is called \emph{special} if for every finite set $S$ the \emph{Segal map}
	\begin{equation*}
		\rho\colon X(S_+)\to\prod_{s\in S} X(1^+)
	\end{equation*}
	induced by the characteristic maps $\chi_s\colon S_+\to 1^+$ of the elements $s\in S$ is a $(G\times\Sigma_S)$-global weak equivalence.

	We write $\ul\GammaS^{\text{gl, spc}}\subset\ul\GammaS^{\text{gl}}$ for the full global subcategory spanned in degree $G$ by the special $G$-global $\Gamma$-spaces, and $\ul\GammaS^{\text{gl, spc}}_*\subset\ul\GammaS^{\text{gl}}$ for those special $\Gamma$-spaces $X$ for which $X(0^+)$ is terminal in the $1$-categorical sense (and not just $G$-globally weakly equivalent to a terminal object).

	Analogously, we define specialness for elements of $\cat{$\bm\Gamma$-$\bm G$-$\bm{\mathcal I}$-SSet}$, yielding nested full global subcategories $\ul\GammaS^{\text{gl, spc}}_{\mathcal I,*}\subset\ul\GammaS^{\text{gl, spc}}_{\mathcal I}\subset\ul\GammaS^{\text{gl}}_{\mathcal I}$.
\end{definition}

\begin{proposition}\label{prop:special-pointed-vs-unpointed}
	All maps in the commutative diagram
	\begin{equation*}
		\begin{tikzcd}
			\ul\GammaS_{\mathcal I,*}^\textup{gl, spc}\arrow[r,hook]\arrow[d, "\ev_\omega"'] & \ul\GammaS_{\mathcal I}^\textup{gl, spc}\arrow[d, "\ev_\omega"]\\
			\ul\GammaS_*^\textup{gl, spc}\arrow[r, hook] & \ul\GammaS^\textup{gl, spc}
		\end{tikzcd}
	\end{equation*}
	of global $\infty$-categories are equivalences.
	\begin{proof}
		For the left hand vertical arrow this is part of \cite{g-global}*{Corollary~2.2.53}. We will now show that the lower horizontal inclusion is an equivalence; the argument for the top inclusion is then similar, and with this established the proposition will follow by $2$-out-of-$3$.

		To prove the claim, we now fix a finite group $G$ and observe that the inclusion $\cat{$\bm\Gamma$-$\bm{E\mathcal M}$-$\bm G$-SSet}_*\hookrightarrow\cat{$\bm\Gamma$-$\bm{E\mathcal M}$-$\bm G$-SSet}$ of those $G$-global $\Gamma$-spaces $X$ with $X(0^+)=*$ admits a left adjoint given by quotienting out $X(0^+)$, i.e.~forming the pushout
		\begin{equation*}
			\begin{tikzcd}
				\const X(0^+)\arrow[r]\arrow[d]\arrow[dr, phantom, "\ulcorner"{very near end}] & X\arrow[d]\\
				*\arrow[r] & X/X(0^+)
			\end{tikzcd}
		\end{equation*}
		where the top map is induced by the unique pointed maps $0^+\to S_+$ for varying $S$.
		It will therefore be enough that the right hand vertical map is a $G$-global level weak equivalence if $X$ is special. But indeed, in this case $\const X(0^+)\to *$ is a $G$-global level weak equivalence (as $X(0^+)$ is $G$-globally and hence also $(G\times\Sigma_T)$-globally weakly contractible for any $T$ by the special case $S=\varnothing$ of the Segal condition), while for \emph{any} $\Gamma$-space the top map is an injective cofibration as $X(0^+)\to X(S_+)$ admits a retraction via functoriality. The claim then follows as pushouts along injective cofibrations preserve $G$-global level weak equivalences by \cite{g-global}*{Lemma~1.1.14} applied levelwise.
	\end{proof}
\end{proposition}

\subsection{Global \texorpdfstring{\for{toc}{$\Gamma$}\except{toc}{$\bm\Gamma$}}{Γ}-spaces as parametrized functors}\label{subsec:comp-all-Gamma}
In this section we will prove the key computational ingredient to the universal property of special global $\Gamma$-spaces in form of the following description of the global $\infty$-category $\ul\GammaS^\text{gl}$ of all global $\Gamma$-spaces:

\begin{theorem}\label{thm:Gamma-presheaves}
	There exists an equivalence of global $\infty$-categories \[\Xi\colon\ul\GammaS^{\textup{gl}}\simeq\ul\Fun_{\Glo}(\ul{\mathbb F}_{\Glo,*}^{\Orb},\ul\Spc)\]
	together with a natural equivalence filling
	\begin{equation*}
		\begin{tikzcd}
			\ul\GammaS^{\textup{gl}}\arrow[r, "\Xi"]\arrow[d, "\mathbb U"'] & \ul\Fun_{\Glo}(\ul{\mathbb F}_{\Glo,*}^{\Orb},\ul\Spc)\arrow[d, "\ev_{(\id_1)_+}"]\\
			\ul\S^\textup{gl}\arrow[r, "\simeq"'] & \ul\Spc_{\Glo}
		\end{tikzcd}
	\end{equation*}
	where the unlabelled arrow on the bottom is `the' essentially unique equivalence (see Theorem~\ref{thm:global-spaces}).
\end{theorem}

\subsubsection{A model of finite $\Orb$-sets} The proof of the theorem will occupy this whole subsection. As the first step, we will recognize $\ul{\mathbb F}^{\Orb}_{\Glo}$ and $\ul{\mathbb F}^{\Orb}_{\Glo,*}$ as some familiar global $1$-categories:

\begin{construction}
	For any finite group $G$, we write $\mathcal F_G$ for the category of finite $G$-sets. The assignment $G\mapsto \mathcal F_G$ becomes a strict $2$-functor in $\sGlo^\op$ via restrictions, and we denote the resulting global category by $\mathcal F_\bullet$.

	We moreover write $\mathcal F^+_\bullet$ for the corresponding global category of pointed finite $G$-sets.
\end{construction}

\begin{lemma}\label{lemma:F-vs-F}
	There is an essentially unique equivalence of global $\infty$-categories $\ul{\mathbb F}_{\textup{Glo}}^{\textup{Orb}}\simeq\nerve\mathcal F_\bullet$. Up to isomorphism, this sends $(H\hookrightarrow G)\in\ul{\mathbb F}^{\Orb}_{\Glo}(G)$ to $G/H\in\mathcal F_G$ for all finite groups $H\subset G$.
	\begin{proof}
		By Corollary~\ref{cor:FinPSetsFreeOnFinitePColimits} there is an essentially unique global functor $\ul{\mathbb F}_{\Glo}^{\Orb}\to\nerve\mathcal F_\bullet$ that preserves $\Orb$-coproducts and the terminal object. It remains to construct any such equivalence and prove that it admits the above description.

		By construction the left hand side is a subcategory of $\ul\Spc_{\Glo}$. On the other hand, we have a fully faithful functor of global $\infty$-categories $\iota\colon\nerve\mathcal F_\bullet\to\ul\S^\textup{gl}$ that is given by sending a finite $G$-set $X$ to $X$ considered as a discrete simplicial set with trivial $E\mathcal M$-action. It then suffices to show that the unique equivalence $F\colon \ul\Spc_{\Glo}\to\ul\S^\text{gl}$ restricts accordingly and admits the above description.

		For this we first observe that indeed $F(i\colon H\hookrightarrow G)\simeq G/H$ for every $H\subset G$: namely, $i$ can be identified with $i_!p^*(*)$ where $p\colon H\to 1$ is the unique homomorphism, and since $F$ is an equivalence it follows that $F(i)\simeq i_!p^*F(*)=i_!p^*(*)$, which can in turn be identified with $G/H$ by Lemma~\ref{lemma:alpha-star-EM}.

		As a consequence of Corollary~\ref{cor:FinPSetsHasFinitePColimits}, each $\mathbb F_{\Glo}^{\Orb}(G)$ is closed under (ordinary) finite coproducts, so $F$ preserves them (as a functor to $\ul\S^\text{gl}$). Together with the above computation, it immediately follows that $F$ restricts to an essentially surjective functor $\ul{\mathbb F}_{\Glo}^{\Orb}\to\textup{ess im}\,\iota$ as claimed.
	\end{proof}
\end{lemma}

\begin{corollary}\label{cor:F-vs-F-pointed}
	There is an essentially unique equivalence $\theta\colon\ul{\mathbb F}_{\Glo,*}^{\Orb}\simeq\nerve\mathcal F_\bullet^+$. Up to isomorphism, this sends $(H\hookrightarrow G)_+$ to $G/H_+$ for all finite groups $H\subset G$.
	\begin{proof}
		The existence of such an equivalence is immediate from the previous lemma. For the uniqueness part, it suffices by Corollary~\ref{cor:FinptdPSetsFree} that any autoequivalence of $\mathcal F_1^+$ preserves $1^+$ up to isomorphism, which is immediate from the observation that this is the only non-zero object without non-trivial automorphisms.
	\end{proof}
\end{corollary}

\subsubsection{Grothendieck constructions}
Thanks to \Cref{Rmk:FunTobjects}, understanding the global functor category $\ul\Fun_{\Glo}(\ul{\mathbb F}_{\Glo,*}^{\Orb},\ul\Spc)$ is equivalent to understanding the unstraightenings $\textstyle\int \ul{\mathbb F}^{\Orb}_{\Glo,*}\times\ul G$
of the diagram $\ul{\mathbb F}^{\Orb}_{\Glo,*}\times\ul G\colon \Glo\catop \to \Cat_{\infty}$ naturally in $G \in \Glo$. However as an upshot of the previous subsection, the functors $\ul{\mathbb F}^{\Orb}_{\Glo,*}\times\ul G$ are modelled by strict $2$-functors of strict $(2,1)$-categories, which will allow us to give a reasonably explicit description in terms of the classical Grothendieck construction:

\begin{construction}\label{constr:grothendieck-2-functor}
	Let $\mathscr C$ be a strict $(2,1)$-category. We recall (see \cite{buckley-grothendieck}*{Construction~2.2.1} or \cite{grothendieck-unstraightening}*{Definition 6.1}) the Grothendieck construction $\tcatUn{C} F$ for a strict $2$-functor $F\colon\mathscr C\to\cat{Cat}_{(2,1)}$ into the $(2,1)$-category of $(2,1)$-categories:
	\begin{enumerate}[(1)]
		\item The objects of $\tcatUn{C}F$ are given by
		pairs $(c,X)$ with $c\in\mathscr C$ and $X\in F(c)$
		\item A morphism from $(c,X)$ to $(d,Y)$ is given by a pair of a map $f\colon c\to d$ and a map $g\colon F(f)(X)\to Y$ in $F(d)$; if $(f',g')\colon (d,Y)\to (e,Z)$ is another morphism, then their composite is
		\begin{equation*}
			(f',g')(f,g)=\big(f'f, F(f'f)(X)=F(f')F(f)(X)\xrightarrow{F(f')(g)} F(f')(Y)\xrightarrow{g'} Z).
		\end{equation*}
		\item A $2$-cell $(f_1,g_1)\Rightarrow (f_2,g_2)$ between maps $(c,X)\to (d,Y)$ is given by a $2$-cell $\sigma\colon f_1\Rightarrow f_2$ in $\mathscr C$ together with a $2$-cell
		\begin{equation*}
			\begin{tikzcd}[row sep=small]
				F(f_1)(X)\arrow[dd, "F(\sigma)"']\arrow[dr, bend left=10pt, "g_1", ""'{name=A}]\\
				& Y\rlap.\\
				F(f_2)(X)\arrow[from=A,Rightarrow, "\tau"]\arrow[ur, bend right=10pt, "g_2"']
			\end{tikzcd}
		\end{equation*}
		in $F(d)$. If $(\rho,\zeta)\colon (f_2,g_2)\Rightarrow (f_3,g_3)$ is another $2$-cell, then the composite $(\rho,\zeta)\circ(\sigma,\tau)$ is given by the composition in $\mathscr C$ and the pasting
		\begin{equation*}
			\begin{tikzcd}
				F(f_1)(X)\arrow[d, "F(\sigma)"']\arrow[dr, bend left=15pt, "g_1", ""'{name=A}]\\
				F(f_2)(X)\arrow[d, "F(\rho)"']\arrow[r, "g_2"{description}, ""'{name=B,yshift=-3pt}]\arrow[Rightarrow,from=A, "\tau"] &[1em] Y\\
				F(f_3)(X)\arrow[ur, bend right=15pt, "g_3"']\arrow[Rightarrow,from=B,"\zeta"']
			\end{tikzcd}
		\end{equation*}
		in $F(d)$. Moreover, if $(\sigma',\tau')\colon (f_1',g_1')\Rightarrow(f_2',g_2')$ is a $2$-cell between maps $(d,Y)\to (e,Z)$, then the horizontal composite $(\sigma',\tau')\odot(\sigma,\tau)$ is given by the horizontal composite $\sigma'\odot\sigma$ and the pasting
		\begin{equation*}
			\begin{tikzcd}[column sep=large]
				F(f_1'f_1)(X)\arrow[d, "F(f_1')F(\sigma)(X)"']\arrow[dr, bend left=15pt, "F(f_1')(g_1)", ""'{name=A}]\\[1.5em]
				F(f_1'f_2)(X)\arrow[r, "F(f_1')(g_2)"{description}]\arrow[from=A, Rightarrow, "F(f_1')(\tau)"']\arrow[d, "F(\sigma')(F(f_2)(X))"'] &[2em] F(f_1')(Y)\arrow[d, "F(\sigma')(Y)"']\arrow[dr, bend left=15pt, "g_1'", ""'{name=B}]\\[1.5em]
				F(f_2'f_2)(X)\arrow[r, "F(f_2')(g_2)"'] & F(f_2')(Y) \arrow[r, "g_2'"']\arrow[from=B,Rightarrow, "\tau'"] & Z
			\end{tikzcd}
		\end{equation*}
		where the square commutes as $F(\sigma')$ is a natural transformation $F(f_1')\Rightarrow F(f_2')$.

	\end{enumerate}
	This comes with a natural strict $2$-functor $\pi\colon\tcatUn{C} F\to\mathscr C$ given by projecting onto the first coordinate. By~\cite{grothendieck-unstraightening}*{Proposition~2.15} the homotopy coherent nerve of this functor is a cocartesian fibration representing $\nerve_\Delta\circ F$. Put differently, there is a natural equivalence $\int(\nerve_\Delta\circ F)\simeq\nerve_\Delta(\tcatUn{C}F)$ over $\nerve_\Delta(C)$ from the usual marked unstraightening to the homotopy coherent nerve of the 2-categorical Grothendieck construction which preserves cocartesian edges.

	We can also describe the behaviour of this equivalence on fibers as follows: for any $c\in\mathscr C$ the composition
	\begin{equation*}
		\nerve_\Delta F(c)\hookrightarrow \nerve_\Delta(\textstyle\tcatUn{C}F)\simeq \int(\nerve_\Delta \circ F)
	\end{equation*}
	of the natural embedding with the above equivalence agrees with the usual identification of $\nerve_\Delta F(c)$ with the fiber of the unstraightening $\int(\nerve_\Delta\circ F)$ over $c$, see~\cite{grothendieck-unstraightening}*{proof of Proposition~6.25}. In particular, for the cocartesian fibration $\nerve_\Delta(\tcatUn{C}F)$ the notation $(c,X)$ (with $X\in F(c)$) for vertices is compatible with Notation~\ref{notation:elements-of-grothendieck}. As the above equivalence moreover preserves cocartesian edges, we also immediately deduce the analogous statement for $1$-simplices.
\end{construction}

\begin{construction}
	For every finite group $G$, we define a strict $(2,1)$-category $\combFinOrbSets{G}$ as follows: Sending a finite group $H$ to the product of the strict $(2,1)$-category $\mathcal{F}^+_H$ of finite pointed $H$-sets and the groupoid $\sGlo(H,\hskip0pt minus .5ptG)\hskip0pt minus 1pt\mathrel{:=}\hskip0pt minus 1pt\HOM_{\sGlo}(H,\hskip0pt minus .5ptG)$ of group homomorphisms $H \to G$ and conjugations defines a strict 2-functor
	\[
	\mathcal{F}^+_\bullet\times {\sGlo}(-,G)\colon \sGlo\catop \to \cat{Cat}_{(2,1)}.
	\]
	Composing this functor with the equivalence of strict $(2,1)$-categories $\gamma\colon \rOgl \iso \sGlo$ from \cref{cons:gamma_rOglvsGlo}, we obtain a strict 2-functor
	\[
	\FinOrbSets{G} := \left( \mathcal{F}^+_\bullet\times {\sGlo}(-,G)\right) \circ \gamma \colon (\rOgl)\catop \to \cat{Cat}_{(2,1)}.
	\]
	As before, we let $\combFinOrbSets{G}$ denote the 2-categorical Grothendieck construction of $\FinOrbSets{G}$. The assignment $G \mapsto \combFinOrbSets{G}$ then becomes a strict $2$-functor $\combFinOrbSets{\bullet}\colon \sGlo \to \cat{Cat}_{(2,1)}$ via (post)composition in $\sGlo$.
\end{construction}

As promised we can now prove:

\begin{proposition}\label{prop:Theta-grothendieck}
	There exists an equivalence
	\begin{equation*}
		\Theta_G\colon\nerve_\Delta\big(\combFinOrbSets{G}\big)=\nerve_\Delta\big(\tcatUn{}(\mathcal F_\bullet^+\times\cat{Glo}(\blank,G))\circ\gamma\big)\xrightarrow{\;\simeq\;}\textstyle\int \ul{\mathbb F}_{\Glo,*}^{\Orb}\times\ul G
	\end{equation*}
	of $\infty$-categories natural in $G\in\Glo$ with the following properties:
	\begin{enumerate}
		\item For all $H\in\rOgl$ and $\phi\colon H\to G$ in $\sGlo$, the following diagram commutes up to equivalence:
		\begin{equation}\label{diag:Theta-G-on-fibers}
			\begin{tikzcd}
				\nerve_\Delta(\mathcal F^+_H\times\{\phi\})\arrow[d,hook] \arrow[r, "\theta_H"] & \ul{\mathbb F}^{\Orb}_{\Glo,*} \times\{\phi\}\arrow[d,hook]\\
				\nerve_\Delta(\mathcal F^+_H\times\sGlo(H,G))\arrow[d] & \mathbb F^{\Orb}_{\Glo}\times\ul{G}(H)\arrow[d]\\
				\nerve_\Delta(\combFinOrbSets{G})\arrow[r, "\Theta_G"'] & \int\ul{\mathbb F}^{\Orb}_{\Glo,*}\times\ul G
			\end{tikzcd}
		\end{equation}
		where $\theta$ is the equivalence from Corollary~\ref{cor:F-vs-F-pointed} and the bottom vertical arrows are the chosen identifications of the fibers over $H$.

		In particular, $\Theta_G$ restricts to an equivalence between the non-full subcategory $\nerve_\Delta\big(\combFinOrbSets{G}\big)_\phi\subset \nerve_\Delta\big(\combFinOrbSets{G}\big)$ with objects of the form $(H;X,\phi)$ (for $X\in\mathcal F_H^+$) and morphisms only those that are the identities in $H$ and $\phi$ (i.e.~the image of $\mathcal F^+_H\times\{\phi\}$ under the chosen identification) and the analogous full subcategory on the right.

		\item For all maps $(\alpha,u)\colon K\to H$ in $\rOgl$ and $f\colon \alpha^*X\to Y$ in $\mathcal F^+_K$, the map $\Theta_G(\alpha,u;f,\id_{\phi\alpha})$ agrees up to equivalence with $(\alpha;\theta_K(f),\id_{\phi\alpha})$.
	\end{enumerate}
	\begin{proof}
		Specializing the above discussion we have a natural equivalence
		\begin{equation*}
			\nerve_\Delta\big(\combFinOrbSets{G}\big)\simeq\textstyle\int\nerve_\Delta\circ\FinOrbSets{G}=\int \nerve_\Delta\circ (\mathcal F_\bullet^+\times\sGlo(\blank,G))\circ\gamma
		\end{equation*}
		between the $2$-categorical and $\infty$-categorical Grothendieck construction sending the map $(\alpha,u;f,\id_{\phi\alpha})$ on the left to the map of the same name on the right and such that for every $H\in\rOgl$ the induced map on fibers respects the identifications with $\mathcal F^+_H\times\sGlo(H,G)$.

		On the other hand, as $\gamma\colon\rOgl\to\sGlo$ is an equivalence, the right hand side is in turn naturally equivalent to the unstraightening $\int\nerve_\Delta\circ(\mathcal F^+_\bullet\times\cat{Glo}(\blank,G))$ over $\Glo^\op$ by an equivalence sending $(\alpha,u;f,\id_{\phi\alpha})$ to $(\alpha;f,\id_{\phi\alpha})$ up to equivalence; again, under our chosen identifications this is just the identity on fibers.

		Finally, by construction of the $\infty$-categorical Yoneda embedding we have an equivalence $\upsilon\colon\nerve_\Delta(\cat{Glo}(L,G))\simeq\Glo(L,G)=\ul{G}(L)$ natural in both variables sending $\psi\colon L\to G$ to $\psi$, which together with the global equivalence $\theta$ from Corollary~\ref{cor:F-vs-F-pointed} induces an equivalence $\int\nerve_\Delta(\mathcal F^+_\bullet\times\sGlo(\blank,G))\simeq\int\ul{\mathbb F}^{\Orb}_{\Glo,*}\times\ul G$ sending $(\alpha;f,\id_{\phi\alpha})$ to $(\alpha;\theta_K(f),\id_{\phi\alpha})$ and that is given under the chosen identifications of the fibers over $H$ by $\theta_H\times \upsilon$. The commutativity of $(\ref{diag:Theta-G-on-fibers})$ follows immediately, which completes the proof of the proposition.
	\end{proof}
\end{proposition}

\subsubsection{Global $\Gamma$-spaces as enriched functors}
Thanks to the above proposition, we can replace the somewhat mysterious $\infty$-categorical unstraightenings $\textstyle\int \ul{\mathbb F}^{\Orb}_{\Glo,*}\times\ul G$ by the homotopy coherent nerves of the much more explicit $(2,1)$-categories $\combFinOrbSets{G}$. These are suitably combinatorial to in turn admit a comparison to the $\OGglGamma$'s:

\begin{construction}\label{constr:ogglgamma-vs-grothendieck}
	Let $G$ be a finite group. We define $\delta\colon (\combFinOrbSets{G})\catop\to\OGglGamma$ as follows:
	\begin{enumerate}[(1)]
		\item An object $(H;S_+,\phi)$ consisting of a universal subgroup $H\subset\mathcal M$, a finite pointed $H$-set $S_+$ and a homomorphism $\phi\colon H\to G$ is sent to $(\Gamma(S_+,\blank)\times E\mathcal M\times G_\phi)/H$.
		\item A morphism $(u\in\mathcal M,\sigma\colon H\to K;f\colon \sigma^*T_+\to S_+;g\in G)$ is sent to the map induced by $\Gamma(f,\blank)\times (\blank \cdot (u,g))$, i.e.~the map corresponding to $[f;u;g]$ under the identification from Remark~\ref{rk:OGgl-Gamma-morphsim-spaces}.
		\item A $2$-cell $k\colon (u,\sigma;f,g)\Rightarrow (u',\sigma';f',g')$ (for $k\in K\subset\mathcal M$) is sent to the $2$-cell corresponding to $[f;u'k, u;g]$.
	\end{enumerate}
\end{construction}

\begin{proposition}\label{prop:delta-equivalence}
	The assignment $\delta$ is well-defined (i.e.~the above indeed represent morphisms and $2$-cells in $\OGglGamma$) and is an equivalence of $(2,1)$-categories.
	\begin{proof}
		We break this up into several steps.

		\emph{It is well-defined on morphisms and a full 1-functor:} If $(u,\sigma;f;g)$ is a morphism $(H,S_+,\phi)\to (K,T_+,\psi)$ in the opposite of the Grothendieck construction, then $hu=u\sigma(h)$ for all $h\in H$ as $(u,\sigma)$ is a morphism $H\to K$ in $\rOgl$; moreover, $c_g\psi\sigma=\phi$ as $g$ is a morphism $\phi\to\psi\sigma$ in $\sGlo(K,G)$, while $(h\cdot \blank)\circ f=f\circ(\sigma(h)\cdot \blank)$ for all $h\in H$ as $f$ is a map of (pointed) $H$-sets. Thus,
		\begin{equation*}
			(h,\phi(h))\cdot[f;u;g]=[(h\cdot \blank)\circ f; hu; \phi(h)g]=
			[f\circ(\sigma(h)\cdot \blank); u\sigma(h); g\psi(\sigma(h))]
			=[f;u;g],
		\end{equation*}
		i.e.~$[f;u;g]$ is indeed $\phi$-fixed. Note that we can also deduce this statement from (the easy direction of) \cite{g-global}*{Lemma~1.2.38}: namely, if we consider $\Gamma(S_+,T_+)\times G_\psi$ as a $(G\times H)\times K$-biset, where $G$ acts on $G$ from the left, $H$ acts on $S_+$ from the left, and $K$ acts from the right via its given action on $T_+$ and its action on $G$ via $\psi$, then swapping the factors defines an isomorphism
		\begin{equation*}
			(\Gamma(T_+,S_+)\times E\mathcal M\times G_\psi)/K\big)^\phi \cong
			\big(E\mathcal M\times_K (\Gamma(T_+,S_+)\times G)\big)^{(\id_H,\phi)}
		\end{equation*}
		where the right hand side is the usual balanced product; \emph{loc.~cit.} then says that $[u;f;g]$ defines a vertex of the right hand side if \emph{and only if} there exists a homomorphism $\sigma\colon H\to K$ (necessarily unique) such that $hu=u\sigma(h)$ for all $h\in H$ and moreover $(h,\phi(h))\cdot(f,g)=(f,g)\cdot \sigma(h)$, i.e.~$f$ is equivariant as a map $\sigma^*T_+\to S_+$ and $\phi=c_g\psi\sigma$. From the `only if' part we then immediately deduce that the above is surjective on morphisms: a preimage of $[u;f;g]$ is given by $(u,\sigma;f;g)$.

		The equality $\delta(1,1;\id_{S_+},1)=[1;\id_{S_+};1]$ shows that $\delta$ preserves identities. To see that it is also compatible with composition of $1$-morphisms (whence a $1$-functor), we let $(u',\sigma';f',g')$ be a map $(K;T_+;\psi)\to(L;U_+;\zeta)$ in the opposite category (so that $\sigma'\colon K\to L$ is a homomorphism and $f'\colon (\sigma')^*U_+\to T_+$ an equivariant map). Then indeed
		\begin{align*}
			\delta\big((u',\sigma';f',g')(u,\sigma;f,g)\big)&\overset{\!(*)\!}{=}\delta(uu',\sigma'\sigma;ff';gg')\\
			&=[ff';uu';gg']=\delta(u';f';g')\delta(u;f;g)
		\end{align*}
		where the somewhat surprising formula $(*)$ for the composition in the Grothendieck construction comes from the fact that $\sigma^*$ does not change underlying maps of sets nor the group elements representing maps in $\sGlo(\blank,G)$.

		\emph{It is well-defined on $2$-cells and a locally fully faithful $2$-functor:}
		First, let us show that $\delta$ defines fully faithful functors
		\begin{equation}\label{diag:delta-on-homs}
			\maps\big((H;S_+,\phi),(K;T_+,\psi)\big)\to\maps\big(\Gamma_{H,S,\phi}, \Gamma_{K,T,\psi}\big)
		\end{equation}
		for all objects $(H;S_+,\phi)$ and $(K;T_+,\psi)$. For this it will be enough to prove this after postcomposing with the isomorphism $\epsilon$ to $(\Gamma_{K,T;\psi})^\phi$.

		If now $(u_1,\sigma_1;f_1;g_1)$ and $(u_2,\sigma_2;f_2;g_2)$ are morphisms $(H;S_+;\phi)\rightrightarrows (K;T_+;\psi)$, then \cite{g-global}*{Lemma~1.2.74} shows that we have a bijection between morphisms $[f_1;u_1;g_1]\to[f_2;u_2;g_2]$ in $\Gamma_{K,T,\psi}^\phi$ and elements $k\in K$ such that $f_1=f_2(k\cdot \blank)$, $g_1=g_2\psi(k)$, and $\sigma_2=c_k\sigma_1$, which is explicitly given by $k\mapsto[f_1;u_2k,u_1;g_1]$. The last condition precisely says that $k$ is a $2$-cell $(u_1,\sigma_1)\Rightarrow (u_2,\sigma_2)$ in $\rOgl$, while the remaining two conditions say that $(f_2;g_2)\circ \FinOrbSets{G}(k)= (f_1;g_1)$, which is precisely the compatibility condition for $2$-cells in the Grothendieck construction. Thus, $(\ref{diag:delta-on-homs})$ is well-defined and bijective on morphisms. To see that it is indeed a functor, we observe that $\delta(1)=\id$ by design, and that for any further $2$-cell $k'\colon (u_2,\sigma_2;f_2;g_2)\Rightarrow (u_3,\sigma_3;f_3;g_3)$ we have
		\begin{align*}
			\delta(k')\circ\delta(k)&=[f_2;u_3k',u_2;g_2]\circ[f_1;u_2k;u_1;g_1]\\
			&\stackrel{\kern-3pt(*)\kern-3pt}{=}[f_1;u_3k'k,u_2k;g_1]\circ[f_1;u_2k;u;g_1]\\
			&=[f_1;u_3k'k,u_1;g_1]=\delta(k'k)=\delta(k'\circ k),
		\end{align*}
		where the equality $(*)$ uses $[f_2;u_3k',u_2;g_2]=[f_2(k\cdot \blank),u_3k'k,u_2k;g_2\psi(k)]$ together with the above relations.

		To complete the current step, it now only remains to show that $\delta$ is compatible with horizontal composition of $2$-cells, i.e.~if $(u_1',\sigma_1';f_1';g_1'),(u_2',\sigma_2';f_2';g_2')\colon (K;T_+;\psi)\rightrightarrows (L;U_+;\zeta)$ are parallel morphisms and $\ell\colon (u_1',\sigma_1';f_1';g_1')\Rightarrow(u_2',\sigma_2';f_2';g_2')$, then $\delta(\ell\odot k)=\delta(\ell)\odot\delta(k)$. Plugging in the definitions, the left hand side is given by $\delta(\ell\sigma_2(k))=[f_1f_1'; u_2u_2'\ell\sigma_2(k), u_1u_1'; g_1g_1']$ while the right hand side evaluates to $[f_1'; u_2'\ell, u_1'; g_1']\odot[f_1; u_2k, u_1; g_1]=[f_1f_1'; u_2k u_2'\ell; u_1u_1'; g_1g_1']$. But $ ku_2'=u_2'\sigma_2'(k)$ as $(u_2',\sigma_2')$ is a morphism, while $\sigma_2'(k)\ell=\ell\sigma_2(k)$ as $\ell$ is a $2$-cell, whence $u_2ku_2'\ell=u_2u_2'\ell\sigma_2(k)$ as desired.

		\emph{The $2$-functor $\delta$ is an equivalence:} We have shown above that $\delta$ is a $2$-functor, surjective on $1$-cells, and bijective on $2$-cells. As it is clearly surjective on objects, the claim follows immediately.
	\end{proof}
\end{proposition}

Together with the Elmendorf Theorem for $G$-global $\Gamma$-spaces, we can now describe the global relative category of global $\Gamma$-spaces in terms of suitable simplicially enriched functor categories. The structure of the argument is very similar to the arguments following \Cref{cstr:PsiUnstable}.

\begin{construction}
	We define \[\Psi_\Gamma\colon \cat{$\bm\Gamma$-$\bm{E\mathcal M}$-$\bm G$-SSet}\to\FUN(\combFinOrbSets{G}, \cat{SSet})\] as follows:
	\begin{enumerate}[(1)]
		\item If $X$ is any $G$-global $\Gamma$-space, then $\Psi_\Gamma(X)(H;S_+;\phi) = X(S_+)^\phi$. If $(K;T_+;\psi)$ is another object, then an $n$-simplex
		\begin{equation}\label{eq:psi-Gamma-n-simpl}
			(u_0,\sigma_0;f_0;g_0)\xRightarrow{k_1}(u_1,\sigma_1;f_1;g_1)\xRightarrow{k_2}\cdots\xRightarrow{k_n}(u_n,\sigma_n;f_n;g_n)
		\end{equation}
		of $\maps((K;T_+;\psi),(H;S_+;\phi))$ is sent to the composition \[\big((u_nk_n\cdots k_1,\dots, u_1k_1,u_0;g_0)\cdot \blank\big)\circ X(f_0).\]
		\item If $f\colon X\to Y$ is any map of $G$-global $\Gamma$-spaces, then \[\Psi_\Gamma(f)(H;S_+;\phi)=f(S_+)^\phi\colon X(S_+)^\phi\to Y(S_+)^\phi.\]
	\end{enumerate}
\end{construction}

\begin{proposition}\label{prop:Psi-Gamma-equivalence}
	The assignment $\Psi_\Gamma$ is well-defined and it descends to an equivalence when we localize the source at the $G$-global level weak equivalences and the target at the levelwise weak homotopy equivalences.
	\begin{proof}
		One argues precisely as in the proof of Proposition~\ref{prop:Psi-equivalence} that $\Psi_\Gamma$ is well-defined and isomorphic (via corepresentability) to the composite
		\begin{equation*}
			\cat{$\bm\Gamma$-$\bm{E\mathcal M}$-$\bm G$-SSet}\xrightarrow{\Phi_\Gamma}\FUN((\OGglGamma)^\op,\cat{SSet})\xrightarrow{\delta^*}\FUN(\combFinOrbSets{G}, \cat{SSet}).
		\end{equation*}
		The claim now follows from Proposition~\ref{prop:elmendorf-Gamma} together with Proposition~\ref{prop:delta-equivalence}.
	\end{proof}
\end{proposition}

\begin{proposition}\label{prop:Psi-Gamma-natural}
	The maps $\Psi_\Gamma$ are strictly $2$-natural in $\sGlo$ (where the right hand side is a $2$-functor in $G$ as before).
	\begin{proof}
		We again break this up into two steps:

		\textit{The $\Psi_\Gamma$'s are $1$-natural:} Let $\alpha\colon G\to G'$ be a group homomorphism. We will first show that we have for every $G$-global $\Gamma$-space $X$ an equality of enriched functors $\Psi_\Gamma(\alpha^*X)=\Psi_\Gamma(X)\circ\big(\tcatUn{}(\mathcal F_\bullet^+\times\sGlo(\blank,\alpha))\circ\gamma\big)$. To prove this, we first observe that this holds on objects as $X(S_+)^{\alpha\phi}=(\alpha^*X)(S_+)^\phi$ for all universal $H\subset\mathcal M,\phi\colon H\to G$. Given now an $n$-simplex $(\ref{eq:psi-Gamma-n-simpl})$ of $\maps((H;S_+;\phi),(K;T_+;\psi))$, it is straight-forward to check that both $\Psi_\Gamma(\alpha^*X)$ and $\Psi_\Gamma(X)\circ(\tcatUn{}(\mathcal F_\bullet^+\times\sGlo(\blank,\alpha))\circ\gamma)$ send this to the restriction of the composite
		\begin{equation*}
			\big((u_nk_n\cdots k_1,\dots, u_1k_1,u_0;\alpha(g))\cdot\blank\big)\circ X(f).
		\end{equation*}
		With this established, naturality on morphisms can be checked levelwise, i.e.~after evaluating at each $(H;S_+;\phi)$. However, for any map $f$ both $\Psi(\alpha^*f)(H;S_+;\phi)$ and $\Psi(f)(\tcatUn{}(\mathcal F_\bullet^+\times\sGlo(\blank,\alpha))\circ\gamma)(H;S_+;\phi)$ are simply given by a restriction of $f(S_+)$.

		\textit{The $\Psi_\Gamma$'s are $2$-natural}: It only remains to show that for each $\alpha,\beta\colon G\to G'$ and $g'\colon\alpha\Rightarrow\beta$ the two pastings
		\begin{equation*}
			\begin{tikzcd}
				\cat{$\bm\Gamma$-$\bm{E\mathcal M}$-$\bm{G'}$-SSet}\arrow[r, bend left=15pt, shift left=2.5pt, "\alpha^*"{name=a}]\arrow[r, bend right=15pt, shift right=2.5pt, "(\alpha')^*"'{name=b}] & \cat{$\bm\Gamma$-$\bm{E\mathcal M}$-$\bm{G}$-SSet} \arrow[r, "\Psi_\Gamma"] & \FUN({\combFinOrbSets{G}}, \cat{SSet})
				\twocell[from=a,to=b, "\scriptstyle g'"{xshift=7.5pt}]
			\end{tikzcd}
		\end{equation*}
		and
		\begin{equation*}
			\begin{tikzcd}
				\cat{$\bm\Gamma$-$\bm{E\mathcal M}$-$\bm{G'}$-SSet}\arrow[r,"\Psi_\Gamma"] & \FUN(\combFinOrbSets{G'}, \cat{SSet}) \arrow[r, bend left=15pt, shift left=2.5pt, "\alpha^*"{name=a}]\arrow[r, bend right=15pt, shift right=2.5pt, "(\alpha')^*"'{name=b}] & \FUN\big({\combFinOrbSets{G}}, \cat{SSet}\big)\twocell[from=a,to=b, "\scriptstyle g'"{xshift=7.5pt}]
			\end{tikzcd}
		\end{equation*}
		agree. However, as we have already established $1$-naturality, this can be again checked pointwise in $\cat{$\bm\Gamma$-$\bm{E\mathcal M}$-$\bm{G'}$-SSet}$ and levelwise in $\combFinOrbSets{G}$, where both are simply given by restriction of the action of $g'$.
	\end{proof}
\end{proposition}

\subsubsection{The comparison} Putting everything together we now get:

\begin{proof}[Proof of Theorem~\ref{thm:Gamma-presheaves}]
	Arguing precisely as in the proof of Theorem~\ref{thm:global-spaces}, we deduce from Propositions~\ref{prop:Psi-Gamma-equivalence} and~\ref{prop:Psi-Gamma-natural} that we have an equivalence of global $\infty$-categories
	\begin{equation*}
		\ul\GammaS^\text{gl}\simeq \Fun(\nerve_\Delta\combFinOrbSets{\bullet},\Spc)
	\end{equation*}
	given on objects in degree $G$ by sending a $G$-global $\Gamma$-space $X$ to $\nerve_\Delta(P\circ\Psi_\Gamma(X))$ where $P$ is our favourite simplically enriched Kan fibrant replacement functor.

	On the other hand, Proposition~\ref{prop:Theta-grothendieck} provides an equivalence between the right hand side and $\Fun(\int\ul{\mathbb F}^{\Orb}_{\Glo}\times\ul{(\blank)},\Spc)$. The desired equivalence now follows as \Cref{Rmk:FunTobjects} also gives a natural equivalence
	\begin{equation}
		\label{eq:glo-functors-vs-unstraightening}
		\ul\Fun_{\Glo}(\ul{\mathbb F}^{\Orb}_{\Glo,*},\ul{\Spc}_{\Glo})\simeq \Fun(\textstyle\int(\ul{\mathbb F}^{\Orb}_{\Glo,*}\times\ul{(\blank)}), \Spc).
	\end{equation}

	It remains to construct an equivalence filling the diagram on the left in
	\begin{equation*}
		\begin{tikzcd}
			\ul\GammaS^\text{gl}\arrow[d, "\mathbb U"'] \arrow[r, "\Xi"] & \ul\Fun_{\Glo}(\ul{\mathbb F}^{\Orb}_{\Glo,*},\ul\Spc_{\Glo})\arrow[d, "\ev_{\id_+}"]\\
			\ul\S^\text{gl}\arrow[r, "\simeq"'] & \ul\Spc_{\Glo}
		\end{tikzcd}
		\quad
		\begin{tikzcd}
			\ul\GammaS^\text{gl}\arrow[r, "\Xi"] & \ul\Fun_{\Glo}(\ul{\mathbb F}^{\Orb}_{\Glo,*},\ul\Spc_{\Glo})\\
			\ul\S^\text{gl}\arrow[r, "\simeq"']\arrow[u, "{\Gamma(1^+,\blank)\times\blank}"] & \ul\Spc_{\Glo}\arrow[u, "\textup{left Kan ext.}"']
		\end{tikzcd}
	\end{equation*}
	for which it is enough by passing to vertical left adjoints (as the horizontal maps are equivalences) to construct an equivalence filling the diagram on the right. By the universal property of $\ul\Spc_{\Glo}$ it is in turn enough for this to chase through the terminal object. Now the forgetful functor $\cat{$\bm{E\mathcal M}$-SSet}\to\cat{SSet}$ sending an $E\mathcal M$-simplicial set to its underlying non-equivariant homotopy type is obviously homotopical right Quillen with left adjoint given by $E\mathcal M\times\blank$; passing to associated $\infty$-categories, we obtain an adjunction $\ul\S^\text{gl}(1)\rightleftarrows\Spc$ and as $E\mathcal M\simeq*$ by \cite{g-global}*{Example~1.2.35}, we see that the left adjoint preserves the terminal objects. On the other hand, as $1$ is a terminal object of $\Glo$, the evaluation functor $\ev_1\colon\ul\Spc_{\Glo}(1)\to\Spc$ similarly admits a left adjoint given by $\const\colon\Spc\to\ul\Spc_{\Glo}(1)$, which again preserves the terminal object. In particular, we see by another application of the universal property of $\ul\Spc_{\Glo}$ that the equivalence $\ul\S^\text{gl}\simeq\ul\Spc_{\Glo}$ is compatible with these adjunctions.

	We are therefore reduced to constructing a natural equivalence filling the diagram on the left in
	\begin{equation*}\hfuzz=30pt\hskip-29.48pt
		\begin{tikzcd}
			\ul\GammaS^\text{gl}(1)\arrow[r, "\Xi"] & \Fun_{\Glo}(\ul{\mathbb F}^{\Orb}_{\Glo,*},\ul{\Spc}_{\Glo})\\
			\Spc\arrow[u, "{\Gamma(1^+,\blank)\times E\mathcal M\times\blank}"] \arrow[r, equal] & \Spc\arrow[u]
		\end{tikzcd}
		\qquad
		\begin{tikzcd}
			\ul\GammaS^\text{gl}(1)\arrow[r, "\Xi"]\arrow[d, "\textup{forget}\circ\mathbb U"'] & \Fun_{\Glo}(\ul{\mathbb F}^{\Orb}_{\Glo,*},\ul{\Spc}_{\Glo})\arrow[d, "\ev_1\circ\ev_{\id_+}"]\\
			\Spc \arrow[r, equal] & \Spc,
		\end{tikzcd}
	\end{equation*}
	for which it is then by the same argument as before enough to construct a natural equivalence filling the diagram on the right. By Remark~\ref{rk:grothendieck-vs-evaluation}, the composite of the right hand vertical map with the equivalence $(\ref{eq:glo-functors-vs-unstraightening})$ from the construction of $\Xi$ is given by evaluating at $(1;1,1)$. However, by the description of $\Theta_1$ from Proposition~\ref{prop:Theta-grothendieck}, $\Theta_1(1;1,1)=(1;1,1)$, so it follows by construction of $\Xi$ that the upper path through this diagram is induced by the homotopical functor $P\circ\Psi_\Gamma(\blank)(1;1,1)\colon \cat{$\bm\Gamma$-$\bm{E\mathcal M}$-SSet}\to\cat{Kan}$. However, by definition $\Psi_\Gamma(\blank)(1;1,1)$ is precisely the functor sending a global $\Gamma$-space $X$ to $X(1^+)$ considered as a non-equivariant space, so the claim follows.
\end{proof}

\subsection{Proof of Theorem~\ref{introthm:universal-prop-gamma}}\label{subsec:universal-gamma} Building on the above we will now prove a comparison between \emph{special} $G$-global $\Gamma$-spaces and $\ul\CMon^{\Orb}_{\Glo}(\ul\Spc_{\Glo}).$ Recall from \Cref{ex:eqsemi} the notion of equivariant semiadditivity.

\begin{theorem}\label{thm:Xi-restr}
	There exists an essentially unique pair of an equivariantly semiadditive functor $\Xi\colon\ul\GammaS^\textup{gl, spc}\to\ul\CMon^{\Orb}(\ul\Spc_{\Glo})$ together with a natural equivalence filling
	\begin{equation}\label{diag:defining-special-Xi}
		\begin{tikzcd}
			\ul\GammaS^\textup{gl, spc}\arrow[d, "\mathbb U"']\arrow[r, "\Xi"] & \ul\CMon^{\Orb}(\ul\Spc_{\Glo})\arrow[d, "\mathbb U=\ev_{\id_+}"]\\
			\ul\S^\textup{gl}\arrow[r, "\simeq"'] & \ul\Spc_{\Glo}\rlap.
		\end{tikzcd}
	\end{equation}
	Moreover, $\Xi$ is an equivalence.
\end{theorem}

As the notation suggest, we will in fact show that the equivalence $\Xi$ from Theorem~\ref{thm:Gamma-presheaves} restricts accordingly and is still an equivalence. For this let us first translate our definition of specialness into something that is more akin to the characterization of equivariant semiadditivity given in Subsection~\ref{subsec:P-semiadditive-Un}:

\begin{proposition}\label{prop:special-rephrased}
	A $G$-global $\Gamma$-space $X$ is special if and only if the following conditions are satisfied for every universal subgroup $H\subset\mathcal M$ and every homomorphism $\phi\colon H\to G$:
	\begin{enumerate}[(1)]
		\item For all finite $H$-sets $S,T$ the collapse maps $S_+\gets S_+\vee T_+\to T_+$ induce a weak homotopy equivalence $X(S_+\vee T_+)^\phi\to X(S_+)^\phi\times X(T_+)^\phi$.
		\item For all $K\subset H$ the composite map
		\begin{equation*}
			X(H/K_+)^\phi\hookrightarrow X(H/K_+)^{\phi|_K}\xrightarrow{X(\chi)^{\phi|_K}} X(1^+)^{\phi|_K},
		\end{equation*}
		is a weak homotopy equivalence, where $\chi\colon H/K_+\to 1^+$ is the characteristic map of $[1]=K\in H/K$.
	\end{enumerate}
	\begin{proof}
		Let us first assume that $X$ is special. Then we have a commutative diagram
		\begin{equation*}
			\begin{tikzcd}
				X(S_+\vee T_+)\arrow[r]\arrow[d,"\rho"'] & X(S_+)\times X(T_+)\arrow[d,"\rho\times\rho"]\\
				\prod_{S\sqcup T} X(1^+)\arrow[r, "\cong"'] & \prod_{S} X(1^+)\times \prod_T X(1^+)
			\end{tikzcd}
		\end{equation*}
		where the top horizontal map is again induced by the collapse maps. By assumption, the left hand vertical map is a $(G\times\Sigma_{S\sqcup T})$-global weak equivalence, hence also a $(G\times H)$-global weak equivalence with respect to the $H$-action on $S\sqcup T$. Similarly, one shows that the right hand vertical map is a $(G\times H)$-global weak equivalence, and hence so is the top horizontal map by $2$-out-of-$3$. Taking fixed points with respect to $(\phi,\id)\colon H\to G\times H$ then establishes Condition $(1)$.

		In order to verify Condition~$(2)$, we first note that we have for any $H$-space $Y$ an isomorphism $\big(\prod_{H/K} Y\big)^H \cong Y^K$ via projection to the factor indexed by $[1]$. Applying this to $Y=(\phi,\id_H)^*X(1^+)$ we then get a commutative diagram
		\begin{equation}\label{diag:segal-vs-characteristic}
			\begin{tikzcd}[row sep=small]
				& \left(\prod\nolimits_{H/K} X(1^+)\right)^\phi\arrow[dd,"\cong"',"\pr_{[1]}"]\\
				X(H/K_+)^\phi\arrow[ur, bend left=15pt, "\rho"]\arrow[dr, bend right=15pt, "X(\chi)"']\\
				& X(1^+)^{\phi|_K}
			\end{tikzcd}
		\end{equation}
		in which the top map is a weak homotopy equivalence by specialness. The claim follows by 2-out-of-3.

		Conversely, assume $X$ is a $G$-global $\Gamma$-space satisfying Conditions $(1)$ and $(2)$. We want to show that for every finite set $S$ the Segal map $X(S_+)\to\prod_S X(1^+)$ is a $(G\times\Sigma_S)$-global weak equivalence, i.e.~for every universal subgroup $H\subset\mathcal M$, every $H$-action on $S$ (i.e.~homomorphism $\rho\colon H\to\Sigma_S)$, and every homomorphism $\phi\colon H\to G$ it induces a weak homotopy equivalence $X(S_+)^\phi\to\big(\prod_{S}X(1^+)\big)^\phi$. Using Condition $(1)$ one readily reduces to the case that $S$ is transitive, i.e.~$S=H/K$ for some $K\subset H$; however, in this case the claim again follows by applying $2$-out-of-$3$ to the commutative diagram $(\ref{diag:segal-vs-characteristic})$.
	\end{proof}
\end{proposition}

In order to relate this to our characterization of equivariant semiadditive functors into $\ul{\Spc}_{\Glo}$ we note:

\begin{lemma}\label{lemma:comparison-of-collapse-maps}
	Let $p\colon K\hookrightarrow H$ be an inclusion of finite groups (hence a map in $\Orb$). Then the essentially unique equivalence $\theta\colon\ul{\mathbb F}_{\Glo,*}^{\Orb}\simeq\nerve\mathcal F^+_\bullet$ (see Corollary~\ref{cor:F-vs-F-pointed}) sends the map $\rho_{p}\colon p^*p_!(\id_+)\to\id_+$ in $\ul{\mathbb F}^{\Orb}_{\Glo,*}(K)$ from Observation~\ref{obs:unique-segal-map} up to isomorphism to the map $\chi\colon H/K_+\to 1^+$ in $\mathcal F_K^+$ from Proposition~\ref{prop:special-rephrased}.
	\begin{proof}
		By construction, $\rho_{p}$ is characterized by the properties that $\rho_{p}\eta=\id$ and $\rho_{p,L}j=0$ for some (hence any) complement $j\colon C\to p^*p_!(\iota)$ of $\eta$. Now the inclusion $1^+\to H/K_+$ of the coset $[1]$ qualifies as a unit $1^+\to p^*p_!1^+$, and with respect to this choice of $\eta$ the map $\chi\colon H/K_+\to1^+$ obviously admits the analogous description.

		If we now assume for ease of notation that $\theta(\id_1)_+=1^+$ (instead of them just being isomorphic), then the calculus of mates provides us with an isomorphism $\alpha\colon  H/K_+\cong\theta(p^*p_!(\id_+))$ in $\mathcal F^+_K$ fitting into a commutative diagram
		\begin{equation}\label{diag:units-f-sets}
			\begin{tikzcd}
				1^+\arrow[d,"\eta"']\arrow[r,equal] & \theta(\id_+)\arrow[d, "\theta(\eta)"]\\
				H/K_+\arrow[r, "\cong", "\alpha"'] & \theta(p^*p_!\id_+)\rlap,
			\end{tikzcd}
		\end{equation}
		and we claim that $\chi$ is actually equal to $\theta(\rho_{p})\alpha$. Indeed,
		\begin{equation*}
			\chi\eta=\id_{1^+}=\theta(\id_{\id_+})=\theta(\rho_{p}\eta)=\theta(\rho_{p})\theta(\eta)= \theta(\rho_{p})\alpha\eta,
		\end{equation*}
		where the last equation uses the commutativity of $(\ref{diag:units-f-sets})$. On the other hand, if $j\colon C\to p^*p_!\id_+$ is a complement of $\eta$, then $\theta(j)$ is a complement of $\theta(\eta)$ (as $\theta$ preserves coproducts), so $\alpha^{-1}\theta(j)$ is a complement of $\eta\colon 1^+\to H/K_+$ in $\mathcal F_K^+$ by commutativity of $(\ref{diag:units-f-sets})$ again. But then $\chi(\alpha^{-1}\theta(j))=0=\theta(0)=\theta(\rho_{p}j)=\theta(\rho_{p,L})\alpha(\alpha^{-1}\theta(j))$, which finishes the proof.
	\end{proof}
\end{lemma}

\begin{proof}[Proof of Theorem~\ref{thm:Xi-restr}]
	By the universal property of $\ul\CMon^{\Orb}(\ul\Spc_{\Glo})$ it will suffice to construct such an equivalence, for which we will show that the equivalence $\Xi$ from Theorem~\ref{thm:Gamma-presheaves} restricts accordingly, i.e.~that a $G$-global $\Gamma$-space $X$ is special if and only if $\Xi(X)\colon \pi_G^*\ul{\mathbb F}_{\Glo,*}^{\Orb}\to\pi_G^*\ul\Spc_{\Glo}$ is $\pi^*_G\Orb$-semiadditive.

	For this, let us write $\hat\Xi(X)$ for the functor $\int\ul{\mathbb F}_{\Glo,*}^{\Orb}\times\ul G\to\Spc$ corresponding to $\Xi(X)$. Plugging in the construction of $\Xi$, this is simply given by the restriction of $\nerve_\Delta(P\circ\Psi_\Gamma(X))\colon\nerve_\Delta(\combFinOrbSets{})\to \nerve_\Delta(\cat{Kan})=\Spc$ (where $P$ is a fixed fibrant replacement again) along the inverse of the equivalence $\Theta_G\colon\int\ul{\mathbb F}_{\Glo,*}^{\Orb}\simeq\nerve_\Delta\combFinOrbSets{}$ from Proposition~\ref{prop:Theta-grothendieck}. On the other hand, \Cref{rmk:otherlevels} shows that $\Xi(X)$ is semiadditive if and only if $\hat\Xi(X)$ is fiberwise semiadditive and sends the Segal maps (defined there) to equivalences.

	\textit{Fiberwise semiadditivity.} We will first show that $X$ satisfies Condition~$(1)$ of Proposition~\ref{prop:special-rephrased} if and only if $\hat\Xi(X)$ is fiberwise semiadditive. Namely, $\hat\Xi(X)$ is fiberwise semiadditive if and only if its restriction to the non-full subcategories spanned by the objects $(H;X,\phi)$ and the maps of the form $(\id;f,\id)$ for each \emph{universal} $H\subset\mathcal M$ and $\phi\colon H\to G$ is semiadditive (as the universal subgroups of $\mathcal M$ account for all objects of $\Glo$ up to isomorphism). As $\Theta_G$ identifies this with the corresponding full subcategory $\nerve_\Delta(\combFinOrbSets{G})_\phi\subset\nerve_\Delta(\combFinOrbSets{G})$ via an equivalence by Proposition~\ref{prop:Theta-grothendieck}, we conclude that $\hat\Xi(X)$ is fiberwise semiadditive if and only if $\Theta_G^*\hat\Xi(X)$ is semiadditive when restricted to each $\nerve_\Delta(\combFinOrbSets{G})_\phi$. But by the explicit construction of $\Psi_\Gamma$, we immediately see that the latter condition for $\Theta_G^*\hat\Xi(X)\simeq \nerve_\Delta(P\circ\Psi_\Gamma(X))$ is equivalent for every fixed $\phi$ to $X(\blank)^\phi$ sending coproducts of finite pointed $H$-sets to products, which is precisely what we wanted to prove.

	\textit{Segal maps.} To complete the proof, it will now suffice to show that $X$ satisfies Condition~$(2)$ of Proposition~\ref{prop:special-rephrased} if and only if $\hat\Xi(X)$ sends the parametrized Segal maps $\rho\colon (H;\iota_+,\phi)\to (K;\id_+,\phi\iota)$ (where $\iota\colon K\hookrightarrow H$ is an inclusion of universal subgroups and $\phi\colon H\to G$ is a homomorphism) in $\int\ul{\mathbb F}_{\Glo,*}^{\Orb}\times\ul G$ to equivalences. However, by the description of $\Theta_G$ given in Proposition~\ref{prop:Theta-grothendieck} together with the computation in Lemma~\ref{lemma:comparison-of-collapse-maps}, we conclude that $\Theta_G^{-1}(\rho)$ is given up to equivalence by $(\iota,1;\chi,\id_{\phi\iota})\colon (H;H/K_+,\phi)\to (K;1^+,\phi\iota)$, and by the explicit construction of $\Psi_\Gamma$ we see that $P\circ \Psi_\Gamma$ sends this up to weak equivalence to the map $X(H/K_+)^\phi\to X(1^+)^{\phi|_K}$ from Proposition~\ref{prop:special-rephrased} as desired.
\end{proof}

We can now leverage the above comparison in order to deduce a universal property of $\ul\GammaS^{\text{gl, spc}}$.

\begin{theorem}\label{thm:global-gamma-relative}
	The functor $\mathbb U\colon\ul\GammaS^\textup{gl, spc}\to\ul\S^\textup{gl}$ exhibits $\ul\GammaS^\textup{gl, spc}$ as the equivariantly semiadditive envelope of $\ul\S^\textup{gl}$, i.e.~for every equivariantly semiadditive global $\infty$-category $\Cc$ we have an equivalence
	\begin{equation*}
		\ul\Fun_{\Glo}^{\Pprod}(\Cc,\mathbb U)\colon\ul\Fun_{\Glo}^{\Pbiprod}(\Cc,\ul\GammaS^\textup{gl, spc})\xrightarrow{\;\simeq\;} \ul\Fun_{\Glo}^{\Pprod}(\Cc,\ul\S^\textup{gl}).
	\end{equation*}
	Moreover, $\mathbb U$ admits a left adjoint $\mathbb P$ which exhibits $\ul\GammaS^\textup{gl, spc}$ as the equivariantly semiadditive completion in the following sense: for every globally cocomplete equivariantly semiadditive global $\infty$-category $\mathcal D$ we have an equivalence
	\begin{equation*}
		\ul\Fun_{\Glo}^\textup{L}(\mathbb P,\mathcal D)\colon\ul\Fun^\textup{L}_{\Glo}(\ul\GammaS^\textup{gl, spc},\mathcal D)\xrightarrow{\;\simeq\;}\ul\Fun^\textup{L}_{\Glo}(\ul\S^\textup{gl},\mathcal D).
	\end{equation*}
	\begin{proof}
		The existence of the left adjoint follows formally from Theorem~\ref{thm:Xi-restr} and the fact that $\mathbb U\colon\ul\CMon^{\Orb}_{\Glo}\to\ul\Spc_{\Glo}$ admits a left adjoint (see~Corollary~\ref{cor:PresentableUniversalPropertyCMon}).

		Now the free-forgetful adjunction $\ul\S^\textup{gl}\rightleftarrows\ul\CMon^{\Orb}_{\Glo}(\ul\S^\text{gl})$ has both of the above universal properties by \cref{thm:Semiadd_omnibus} and Corollary~\ref{cor:cocomplete-univ-cmon-relative}), so it suffices to show that the equivalence $\Xi$ from Theorem~\ref{thm:Xi-restr} is compatible with the free-forgetful adjunctions in the sense that there are natural equivalences filling
		\begin{equation*}
			\begin{tikzcd}[cramped]
				\ul\GammaS^\textup{gl, spc}\arrow[d, "\mathbb U"']\arrow[r, "\Xi"] & \ul\CMon^{\Orb}(\ul\Spc_{\Glo})\arrow[d, "\mathbb U=\ev_{\id_+}"]\\
				\ul\S^\textup{gl}\arrow[r, "\simeq"'] & \ul\Spc^{\Glo}.
			\end{tikzcd}
			\quad\text{and}\quad
			\begin{tikzcd}[cramped]
				\ul\GammaS^\textup{gl, spc}\arrow[from=d, "\mathbb P"]\arrow[r, "\Xi"] & \ul\CMon^{\Orb}(\ul\Spc_{\Glo})\arrow[from=d, "\mathbb P"']\\
				\ul\S^\textup{gl}\arrow[r, "\simeq"'] & \ul\Spc^{\Glo}\llap.
			\end{tikzcd}
		\end{equation*}
		However, as $\Xi$ is an equivalence it suffices to prove the first statement, which is simply the defining property of $\Xi$.
	\end{proof}
\end{theorem}

Together with Theorem~\ref{thm:cocomplete-univ-cmon} we moreover get Theorem~\ref{introthm:universal-prop-gamma} from the introduction:

\begin{theorem}\label{thm:global-gamma}
	Let $\mathcal D$ be a globally cocomplete and equivariantly semiadditive global $\infty$-category. Then evaluation at $\mathbb P(*)$ provides an equivalence
	\[\ul\Fun^\textup{L}_{\Glo}(\ul\GammaS^\textup{gl, spc},\mathcal D)\xrightarrow{\;\sim\;} \mathcal D.\] Put differently, $\ul{\GammaS}^\textup{gl, spc}$ is the free globally cocomplete (or presentable) equivariantly semiadditive global $\infty$-category on one generator (namely, the free global special $\Gamma$-space $\mathbb P(*)$).\qed
\end{theorem}

Using Propositions~\ref{prop:Gamma-I-vs-EM} and~\ref{prop:special-pointed-vs-unpointed} we can deduce several variants of the above theorems. Let us make two of them explicit:

\begin{corollary}
	There exist equivalences \[\ul\GammaS^\textup{gl}_{\mathcal I}\simeq\ul\Fun_{\Glo}(\ul{\mathbb F}^{\Orb}_{\Glo,*},\ul\Spc_{\Glo})\qquad\text{and}\qquad \ul\GammaS^\textup{gl, spc}_{\mathcal I,*}\simeq\ul\CMon^{\Orb}_{\Glo}\] fitting into a commutative diagram
	\begin{equation*}
		\begin{tikzcd}
			\ul\GammaS^\textup{gl, spc}_{\mathcal I,*}\arrow[r, hook]\arrow[d,"\simeq"'] & \ul\GammaS^\textup{gl}_{\mathcal I}\arrow[d, "\simeq"{description}] \arrow[r, "\ev_{1^+}"] &[1em] \ul\S^\textup{gl}_{\mathcal I}\arrow[d,"\simeq"]\\
			\ul\CMon^{\Orb}_{\Glo}\arrow[r,hook] & \ul\Fun_{\Glo}(\ul{\mathbb F}^{\Orb}_{\Glo,*},\ul\Spc_{\Glo})\arrow[r, "\ev_{(\id_1)_+}"'] & \ul\Spc_{\Glo}
		\end{tikzcd}
	\end{equation*}
	where the equivalence on the right is the unique one (see Corollary~\ref{cor:SglI-univ-prop}).\qed
\end{corollary}

\begin{corollary}\label{cor:univ-property-Gamma-I}
	The forgetful functor $\mathbb U\colon\ul\GammaS^\textup{gl, spc}_{\mathcal I,*}\to\ul\S_{\mathcal I}^\textup{gl}$ exhibits $\ul\GammaS^\textup{gl, spc}_{\mathcal I,*}$ as the universal equivariantly semiadditive envelope of $\ul\S_{\mathcal I}^\textup{gl}$. Moreover, it admits a left adjoint $\mathbb P$, exhibiting $\ul\GammaS^\textup{gl, spc}_{\mathcal I,*}$ as the equivariantly semiadditive completion of $\ul\S_{\mathcal I}^\textup{gl}$.\qed
\end{corollary}

\begin{remark}
	\cite{g-global} also discusses various other models of `$G$-globally coherently commutative monoids,' for example \emph{$G$-ultra-commutative monoids} (Definition~2.1.25 of \emph{op.~cit.}) or \emph{$G$-parsummable simplicial sets} (Definition~2.1.10). Similarly, \cite{global-operads}*{Definition~3.9} introduces a notion of \emph{global $E_\infty$-operads}, and for any global $E_\infty$-operad $\mathcal O$, considering $\mathcal O$-algebras in $\cat{$\bm{E\mathcal M}$-$\bm G$-SSet}$ (with respect to the trivial $G$-action on $\mathcal O$) yields a concept of \emph{$G$-global $E_\infty$-algebras}.

	All of these models are related via suitable zig-zags of Quillen equivalences by \cite{g-global}*{Chapter~2} and \cite{global-operads}*{Section~4}, and while these can be somewhat complicated (especially on the operadic side of things), in each case they are by design strictly compatible with restrictions along group homomorphisms and moreover at least one of the adjoints is homotopical, so that they lift to equivalences of associated global $\infty$-categories in the same way as before. As moreover each of them is readily seen to be compatible with the respective forgetful functors, we obtain universal properties in the above spirit for each of these models.

	Conversely, while each of these comparisons comes from a concrete (and sometimes ad-hoc) model categorical construction, this tells us that \emph{a posteriori}, once we have passed to parametrized $\infty$-categories, these comparisons are actually canonical and completely characterized by lying over the forgetful functors.
\end{remark}

\section{Parametrized stability}
In this section, we will introduce the notion of a $P$-stable $T$-$\infty$-category: a $T$-$\infty$-category which is both $P$-semiadditive and fiberwise stable.

\subsection{Fiberwise stable \texorpdfstring{\for{toc}{$T$-$\infty$}\except{toc}{$\bm T$-$\bm\infty$}}{T-∞}-categories}
\begin{definition}
	We say a $T$-$\infty$-category $\Cc$ is \textit{fiberwise stable}
	if the following conditions are satisfied:
	\begin{enumerate}[(1)]
		\item For every object $B \in T$, the $\infty$-category $\Cc(B)$ is stable;
		\item For every morphism $\beta\colon B' \to B$, the restriction functor $\beta^*\colon \Cc(B) \to \Cc(B')$ is exact.
	\end{enumerate}
	Equivalently, $\Cc$ is fiberwise stable if the functor $\Cc\colon T\catop \to \Cat_{\infty}$ factors through the (non-full) subcategory $\Cat^{\st}_{\infty} \subseteq \Cat_{\infty}$ of stable $\infty$-categories and exact functors.
	We let $\Cat^{\st}_T$ denote the $\infty$-category $\Fun(T\catop,\Cat^{\st}_{\infty})$ of fiberwise stable $T$-$\infty$-categories.
\end{definition}

\begin{definition}
	Denote by $\Cat_{\infty}^{\lex} \subseteq \Cat_{\infty}$ the (non-full) subcategory spanned by the $\infty$-categories admitting finite limits and the finite-limit-preserving functors between them. We let $\Cat^{\lex}_T$ denote the functor $\infty$-category $\Fun(T\catop,\Cat^{\lex}_{\infty})$ of $T$-$\infty$-categories $\Cc$ admitting fiberwise finite limits (cf.\ \cref{def:fiberwiseCocompleteness}) and $T$-functors preserving fiberwise finite limits.
\end{definition}

\begin{definition}
	Let $\Cc$ and $\Dd$ be two $T$-$\infty$-categories with finite limits. We write $\ulFun_T^{\lex}(\Cc,\Dd)$ for the full subcategory of $\ulFun_T(\Cc,\Dd)$ spanned on level $B\in T$ by those functors $F\colon \pi_B^* \Cc\rightarrow \pi_B^* \Dd$ which preserve fiberwise finite limits.

	When $\Cc$ and $\Dd$ are both fiberwise stable, we will write $\ulFun_T^{\ex}(\Cc,\Dd)$ for $\ulFun_T^{\lex}(\Cc,\Dd)$.
\end{definition}

\begin{construction}[Fiberwise stabilization]
	Let $\Cc \in \Cat^{\lex}_T$ be a $T$-$\infty$-category which has fiberwise finite limits. We define the $T$-$\infty$-category $\ul{\Sp}^{\fib}(\Cc)$, called the \textit{fiberwise stabilization of $\Cc$}, as the composite
	\begin{align*}
		T\catop \xrightarrow{\Cc} \Cat_{\infty}^{\lex} \xrightarrow{\Sp} \Cat_\infty^{\st}.
	\end{align*}
	This construction assembles into a functor $\ul{\Sp}^{\fib}\colon \Cat^{\lex}_T \to \Cat^{\st}_T$.
\end{construction}

\begin{example}
	The $T$-$\infty$-category $\ul{\Sp}_T$ of naive $T$-spectra is the fiberwise stabilization of the $T$-$\infty$-category $\ul{\Spc}_T$ of $T$-spaces.

	More generally, if $\Ee$ is an $\infty$-category admitting finite limits, then the fiberwise stabilization of the $T$-$\infty$-category $\ul{\Ee}_T$ of $T$-objects in $\Ee$ is the $T$-$\infty$-category $\ul{\Sp(\Ee)}_T$ of $T$-objects in the stabilization $\Sp(\Ee)$. Indeed, this follows easily from the equivalence $\Sp(\Fun(-,\Ee)) \simeq \Fun(-,\Sp(\Ee))$ from \cite{HA}*{Remark 1.4.2.9}.
\end{example}

\begin{remark}\label{rem:limit_ext_of_SpC}
	As a right adjoint, the stabilization functor $\Sp\colon \Cat_{\infty}^{\lex} \to \Cat_\infty^{\st}$ preserves limits, which in both the source and target are computed in $\Cat_\infty$. It follows that the limit extension of $\ul{\Sp}^{\fib}(\Cc)$ to the presheaf category $\PSh(T)$ is given by
	postcomposing the limit extension of $\Cc$ to $\PSh(T)$ with the functor $\Sp$.
\end{remark}

\begin{remark}
	We will use that the functor $\Sp\colon \Cat_{\infty}^{\lex} \to \Cat_{\infty}^{\st}$ is in fact functorial in natural transformations of finite limit preserving functors, i.e.\ that $\Sp$ refines to a 2-functor between homotopy 2-categories. Given that taking functor categories forms such a functor, this immediately follows from the definition of $\Sp(\Cc)$ as a full subcategory of $\Fun(\Spc^{\mathrm{fin}}_\ast,\Cc)$, see \cite{HA}*{Definition 1.4.2.8}. (Using the same argument, one can in fact show that $\Sp$ is an $(\infty,2)$-functor.)

	It follows in particular that stabilization preserves adjunctions between left exact functors.
\end{remark}

\begin{proposition}\label{prop:SpfibRightAdjoint}
	The functor $\ul{\Sp}^{\fib}\colon \Cat^{\lex}_T \to \Cat^{\st}_T$ is right adjoint to the fully faithful inclusion $\Cat^{\st}_T \subseteq \Cat^{\lex}_T$.
\end{proposition}
\begin{proof}
	Since $\Fun(T\catop,-)\colon \Cat_{\infty} \to \Cat_{\infty}$ preserves adjunctions, this is immediate from the fact that the stabilization functor $\Sp\colon \Cat^{\lex}_\infty \to \Cat^{\st}_\infty$ is right adjoint to the fully faithful inclusion $\Cat^{\st}_\infty \subseteq \Cat^{\lex}_\infty$ by \cite{HA}*{Corollary~1.4.2.23}.
\end{proof}

\begin{lemma}\label{lem:FunintoSpfib}
	Consider $\Cc \in \Cat^{\st}_T$ and $\Dd \in \Cat^{\lex}_T$. Composition with the adjunction counit $\Omega^{\infty}\colon \ul{\Sp}^{\fib}(\Dd) \to \Dd$ induces an equivalence of $T$-$\infty$-categories
	\begin{align*}
		\ulFun_T^{\ex}(\Cc,\ul{\Sp}^{\fib}(\Dd)) \iso \ulFun_T^{\lex}(\Cc,\Dd).
	\end{align*}
\end{lemma}

\begin{proof}
	It immediately follows from \cref{prop:SpfibRightAdjoint} that the map
	\begin{equation*}
		\core\Fun_T(\Cc,\Omega^\infty)\colon\core\Fun^\text{ex}_T(\Cc,\ul{\Sp}^{\fib}(\Dd))\to\core\Fun^\text{lex}_T(\Cc,\Dd)
	\end{equation*}
	is an equivalence. We will now show that this already holds before passing to cores. Replacing $T$ by $T_{/B}$ for varying $B\in T$ then yields the proof of the proposition. For this it will be enough to show that for every small $\infty$-category $K$ the induced map $\iota\big(\Fun^\text{ex}_T(\Cc,\ul{\Sp}^{\fib}(\Dd))^K)\to\iota\big(\Fun^\text{lex}_T(\Cc,\Dd)^K)$ is an equivalence. But this agrees up to equivalence with the map induced by $(\Omega^\infty)^K\colon\ul{\Sp}^{\fib}(\Dd)^K\to\Dd^K$; the claim follows as this is again the stabilization of $\Dd^K$.
\end{proof}

The fiberwise stabilization of a $T$-$\infty$-category $\Cc$ inherits certain parametrized limits from $\Cc$. Since this is clear for limits along constant $T$-$\infty$-categories, we focus on limits along $T$-$\infty$-groupoids.

\begin{lemma}
	\label{lem:AdjointabilityInSpFib}
	Let $\bbU$ be a class of $T$-$\infty$-groupoids, and let $\Cc$ be a $\bbU$-complete $T$-$\infty$-category which admits fiberwise finite limits. Then $\ul{\Sp}^{\fib}(\Cc)$ is $\bbU$-complete and the $T$-functor $\ul{\Sp}^{\fib}(\Cc) \to \Cc$ preserves $\bbU$-limits.
\end{lemma}
\begin{proof}
	We will use the characterization of \cref{lem:UColimitsVsAdjointable}. Given a morphism $p\colon A \to B$ in $\bbU$, applying the functor $\Sp\colon \Cat_{\infty}^{\lex} \to \Cat_{\infty}^{\st}$ to the adjunction
	\[
	p^*\colon \Cc(B) \rightleftarrows \Cc(A) \noloc p_*
	\]
	shows that the functor $\Sp(p^*)\colon \Sp(\Cc(B)) \to \Sp(\Cc(A))$ admits a right adjoint given by $\Sp(p_*)\colon \Sp(\Cc(A)) \to \Sp(\Cc(B))$. Furthermore, for a pullback square
	\[
	\begin{tikzcd}
		A' \dar[swap]{p'} \rar{\alpha} \drar[pullback] & A \dar{p} \\
		B' \rar{\beta} & B
	\end{tikzcd}
	\]
	in $\PSh(T)$ with $p\colon A \to B$ in $\bbU$ and $\beta\colon B' \to B$ in $T$, the resulting Beck-Chevalley transformation $\Sp(p_*) \circ \Sp(\beta^*) \Rightarrow \Sp(\alpha^*) \circ \Sp(p'_*)$ is given by applying $\Sp$ to the Beck-Chevalley transformation $p_* \circ \beta^* \Rightarrow \alpha^* \circ p'_*$, and thus is again an equivalence. This shows that $\ul{\Sp}^{\fib}(\Cc)$ is again $\bbU$-complete. It is immediate from this construction that the $T$-functor $\ul{\Sp}^{\fib}(\Cc) \to \Cc$ preserves $\bbU$-limits, finishing the proof.
\end{proof}

Fiberwise stabilization preserves parametrized presentability.

\begin{definition}
	We define $\PrRTst$ to be the full subcategory of $\PrRT$ spanned by those presentable $T$-$\infty$-categories which are also fiberwise stable. The subcategory $\PrLTst \subseteq \PrLT$ is defined similarly.
\end{definition}

\begin{proposition}
	The functor $\ul{\Sp}^{\fib}\colon \Cat^{\lex}_T\rightarrow \Cat^{\st}_T$ restricts to a functor
	\[
	\ul{\Sp}^{\fib}\colon \PrRT\rightarrow \PrRTst
	\]
	which is right adjoint to the inclusion $\PrRTst \hookrightarrow \PrRT$.
\end{proposition}
\begin{proof}
	We first show that the fiberwise stabilization of a presentable $T$-$\infty$-category $\Cc$ is again presentable. By \cite{HA}*{Proposition~1.4.4.4, Example 4.8.1.23}, $\ul{\Sp}^{\fib}(\Cc)$ is given by the composite
	\[
	T\catop \xrightarrow{\Cc} \PrL \xrightarrow{- \otimes \Sp} \PrL,
	\]
	proving that $\ul{\Sp}^{\fib}(\Cc)$ is again fiberwise presentable. Since the functor $- \otimes \Sp\colon \PrL \to \PrL$ preserves adjunctions, one deduces the existence of left adjoints $f_!$ for all morphisms $f\colon A \to B$ in $\PSh(T)$ satisfying the Beck-Chevalley conditions, similar to the proof of \cref{lem:AdjointabilityInSpFib}. This shows that $\ul{\Sp}^{\fib}(\Cc)$ is again a presentable $T$-$\infty$-category.	One can similarly show that if $L\dashv R$ is an adjunction between presentable $T$-$\infty$-categories, then $L\otimes \Sp \dashv \ul\Sp^{\fib}(R)$ is again an adjunction. This shows that $\ul{\Sp}^{\fib}$ restricts to a functor $\PrRT\rightarrow \PrRTst$. It is right adjoint to the inclusion $\PrRTst \hookrightarrow \PrRT$ by \Cref{prop:SpfibRightAdjoint}.
\end{proof}

Applying the equivalence $(\PrRT)\catop \simeq \PrLT$, we obtain:

\begin{corollary}\label{cor:SpfibPresentableAdjunction}
	The construction $\Cc \mapsto \ul{\Sp}^{\fib}(\Cc)$ defines a functor
	\[
	\ul{\Sp}^{\fib} \colon \PrLT \to \PrLTst
	\]
	which is left adjoint to the inclusion functor $\mathrm{incl}\colon \PrLTst \hookrightarrow \PrLT$. \qednow
\end{corollary}

\subsection{\texorpdfstring{\for{toc}{$P$}\except{toc}{$\bm P$}}{P}-stable \texorpdfstring{\for{toc}{$T$-$\infty$}\except{toc}{$\bm T$-$\bm\infty$}}{T-∞}-categories} \label{subsec:PStableTCats}
Let us fix an atomic orbital subcategory $P\subseteq T$.

\begin{definition}
	We say a $T$-$\infty$-category $\Cc$ has \textit{finite $P$-limits} if it has fiberwise finite limits and finite $P$-products. We say a functor $F\colon \Cc\to \Dd$ between $T$-$\infty$-categories with finite $P$-limits \textit{preserves finite $P$-limits} if it preserves fiberwise finite limits and finite $P$-products. We define $\Cat_T^{\Plex}$ to be the (non-full) subcategory of $\Cat_T$ spanned by the $T$-$\infty$ categories which admit finite $P$-limits and those functors which preserve finite $P$-limits.

	Let $\Cc$ and $\Dd$ be two $T$-$\infty$-categories  with finite $P$-limits, we define $\Fun_T^{\Plex}(\Cc,\Dd)$ to be the full subcategory of $\Fun_T(\Cc,\Dd)$ spanned on level $B$ by those functors $F\colon\pi_B^* \Cc\rightarrow \pi_B^* \Dd$ which preserve finite $P$-limits. This is a $T$-subcategory by the dual of Lemma~\ref{lemma:restriction-preserves-cocompl}.
\end{definition}

\begin{remark}
	Unravelling the above definition, a $T$-$\infty$-category $\Cc$ has finite $P$-limits if and only if it has $\bbU$-limits in the sense of Definition~\ref{def:UColimits} for $\bbU$ the union of $\ul{\mathbb F}^P_T\subseteq\ul{\Spc}_T\subseteq\ul{\text{cat}}_T$ and the full subcategory spanned by the constant finite $T$-$\infty$-categories. Analogously, a $T$-functor preserves finite $P$-limits if and only if it preserves $\bbU$-limits in the sense of \emph{loc.\ cit.}
	
	It is possible to express the condition of having $P$-limits alternatively as the existence of limits indexed by a class of \textit{$P$-finite} $T$-$\infty$-categories. To this end, let $\ul{\Cat}_T^{P\text{-fin}}$ be the smallest subcategory of $\ul{\Cat}_T$ which is closed under $P$-finite colimits and contains $[0]$ and $[1]$. Note that this class of $T$-$\infty$-categories contains $\bbU$. It follows from \cite{martiniwolf2022presentable}*{Proposition~A.3.1} that a $T$-$\infty$-category admits finite $P$-limits if and only if it admits $I$-indexed limits for every $I \in \ul{\Cat}_T^{P\text{-finite}}$. While this alternative perspective on finite $P$-limits is the inspiration for our terminology, it is more convenient for the purposes of this paper to work with the previous simpler definition.
\end{remark}

\begin{definition}[cf.\ \cite{nardin2016exposeIV}*{Definition~7.1}]
	\label{def:PStability}
	A $T$-$\infty$-category $\Cc$ is \textit{$P$-stable} if it is fiberwise stable and $P$-semiadditive. We define $\Cat_T^{\Pst}$ to be the full subcategory of $\Cat_T^{\Plex}$ spanned by the $P$-stable $T$-$\infty$-categories.

	When $\Cc$ and $\Dd$ are both $P$-stable $T$-$\infty$-categories, we will write $\Fun_T^{\Pex}(\Cc,\Dd)$ for $\Fun_T^{\Plex}(\Cc,
	\Dd)$.
\end{definition}

\begin{example}\label{ex:eq_stability}
	Applied to the pair $\Orb\subset \Glo$ we obtain a notion of $\Orb$-stability for global $\infty$-categories. We will refer to this as \emph{equivariant stability}.
\end{example}

\begin{lemma}
	Let $\Cc$ be a $T$-$\infty$-category. If $\Cc$ admits finite $P$-limits, then so does $\ulPCMon(\Cc)$.
\end{lemma}
\begin{proof}
	This is a special case of \cref{lem:SemiadditiveFunctorsClosedUnderLimits}.
\end{proof}

\begin{definition}[\cite{nardin2016exposeIV}*{Definition~7.3}]
	\label{def:PStabilization}
	\label{def:PGenuineTSpectra}
	Let $\Cc$ be a $T$-$\infty$-category which admits finite $P$-limits. Then the \textit{$P$-stabilization of $\Cc$} is the $T$-$\infty$-category $\ul{\Sp}^{P}(\Cc)$ defined as
	\begin{align*}
		\ul{\Sp}^{P}(\Cc) := \ul{\Sp}^{\fib}(\ulPCMon(\Cc)),
	\end{align*}
	the fiberwise stabilization of the $T$-$\infty$-category of $P$-commutative monoids in $\Cc$. As a special case, we define the $T$-$\infty$-category $\ul{\Sp}^{P}_T$ of \textit{$P$-genuine $T$-spectra} as
	\begin{align*}
		\ul{\Sp}^{P}_T := \ul{\Sp}^{P}(\ul{\Spc}_T),
	\end{align*}
	the $P$-stabilization of the $T$-$\infty$-category of $T$-spaces.
\end{definition}

The next lemma shows that the $P$-stabilization of a $T$-$\infty$-category with finite $P$-limits is indeed $P$-stable.

\begin{lemma}
	\label{lem:SpfibAgainPSemi}
	Let $\Cc$ be a $P$-semiadditive $T$-$\infty$-category with finite $P$-limits. Then $\ul{\Sp}^{\fib}(\Cc)$ is again $P$-semiadditive, and thus in particular $P$-stable.
\end{lemma}

\begin{proof}
	The $T$-$\infty$-category $\ul{\Sp}^{\fib}(\Cc)$ is fiberwise pointed and admits finite $P$-products by \cref{lem:AdjointabilityInSpFib}. By \cref{lem:CharacterizationPSemiadditiveTCategories}, it will suffice to show that for every morphism $p\colon A \to B$ in $\finPsets$ the dual adjoint norm map $\Nmadjdual_p\colon \id \to \Sp(p^*)\Sp(p_*)$ exhibits $\Sp(p^*)$ as a right adjoint of $\Sp(p_*)$. Since the adjunction data for $\ul{\Sp}^{\fib}(\Cc)$ is inherited from $\Cc$ by fiberwise stabilizing, the dual adjoint norm map for $\ul{\Sp}^{\fib}(\Cc)$ is obtained by applying the stabilization functor to the map $\Nmadjdual^{\Cc}_p\colon \id \to p^*p_*$. As stabilization preserves adjunctions, the claim thus follows from $P$-semiadditivity of $\Cc$.
\end{proof}

\begin{corollary}\label{cor:adjunctionPstable}
	The functor $\ul{\Sp}^{P}\colon \Cat_T^{\Plex} \to \Cat_T^{\Pst}$ is right adjoint to the inclusion $\Cat_T^{\Pst} \hookrightarrow \Cat_T^{\Plex}$.
\end{corollary}

\begin{proof}
	\cref{lem:SpfibAgainPSemi} shows that the adjunction of \cref{prop:SpfibRightAdjoint} restricts to an adjunction
	\[\mathrm{incl} \colon \Cat_T^{P-\st} \rightleftarrows \Cat_T^{\lex,{\Pbiprod}}\noloc \ul{\Sp}^{\fib}(-).\] Composing this with the adjunction of \cref{cor:adjunctionCMon} gives the statement.
\end{proof}

From the adjunction of $\infty$-categories from \cref{cor:adjunctionPstable}, we may immediately deduce an equivalence at the level of $T$-$\infty$-categories of functors.

\begin{definition}
	We define the $T$-functor $\Omega^\infty\colon \ul{\Sp}^{P}(\Cc)\rightarrow \Cc$ to be the counit of the adjunction from \cref{cor:adjunctionPstable}. Explicitly it is given by the composite
	\[\ul{\Sp}^{\textup{fib}}(\ulPCMon(\Cc))\xrightarrow{\Omega^\infty} \ulPCMon(\Cc) \xrightarrow{\mathbb{U}} \Cc,\]
	where the first functor is the infinite loop space functor and the second functor is given by evaluation at $S^0\colon \ul{1} \to \ulfinptdPsets$.
\end{definition}

\begin{proposition}\label{prop:charact-stabilization-semiadditive}
	Let $\Dd$ be a $T$-$\infty$-category with finite $P$-limits. For every $P$-stable $T$-$\infty$-category $\Cc$, composition with $\Omega^\infty\colon \ul{\Sp}^{P}(\Cc)\rightarrow \Cc$ induces an equivalence of $T$-$\infty$-categories
	\begin{equation*}
		\ul\Fun_T(\Cc,\Omega^\infty)\colon\ul\Fun_T^{\textup{$P$-ex}}(\Cc,\ul{\Sp}^{P}(\Dd))\to\ul\Fun_T^{\textup{$P$-lex}}(\Cc,\Dd).
	\end{equation*}
	\begin{proof}
		This follows by combining \cref{cor:UniversalPropCMon} and \cref{lem:FunintoSpfib}.
	\end{proof}
\end{proposition}

\begin{lemma}
	Let $\bbU$ be a family of $T$-$\infty$-groupoids, and let $\Cc$ be a $\bbU$-complete $T$-$\infty$-category which admits finite $P$-limits. Then also $\ul{\Sp}^P(\Cc)$ is $\bbU$-complete and the $T$-functor $\Omega^{\infty}\colon \ul{\Sp}^P(\Cc) \to \Cc$ preserves $\bbU$-limits.
\end{lemma}
\begin{proof}
	This follows immediately from \Cref{lem:AdjointabilityInSpFib} and \Cref{lem:SemiadditiveFunctorsClosedUnderLimits}.
\end{proof}

As before, $P$-stabilization restricts to an adjunction on presentable $T$-$\infty$-categories.

\begin{lemma}
\label{lem:SpPresentableAdjunction}
	The construction $\Cc \mapsto \ul{\Sp}^{P}(\Cc)$ defines a functor
	\[
	\ul{\Sp}^{P}\colon \PrLT\rightarrow \PrLTPst
	\]
	which is left adjoint to the inclusion $\PrLTPst \hookrightarrow \PrLT$.
\end{lemma}

\begin{proof}
	Combine \cref{cor:SpfibPresentableAdjunction} and \cref{cor:PresentableUniversalPropertyCMon}.
\end{proof}

\begin{definition}
	If $\Cc$ is a presentable $T$-$\infty$-category, we write $\Sigma^\infty_+\colon \Cc\rightarrow \ul{\Sp}^P(\Cc)$ for the left adjoint of the forgetful functor $\Omega^\infty\colon \ul{\Sp}^P(\Cc) \rightarrow \Cc$. It is the unit of the adjunction in \Cref{lem:SpPresentableAdjunction}.
\end{definition}

We record the results of this section in the following theorem for easy reference:

\begin{theorem}\label{thm:Stable_omnibus}
	Let $\Cc$ be a $T$-$\infty$-category with finite $P$-limits. The functor $\Omega^\infty\colon\ul\Sp^P(\Cc)\to \Cc$ exhibits $\ul\Sp^P(\Cc)$ as the $P$-stable envelope of $\Cc$, i.e.~for every $P$-stable  $T$-$\infty$-category $\tcat$ postcomposition with $\Omega^\infty$ induces an equivalence
	\begin{equation*}
		\ul\Fun^{{\Plex}}(\tcat, \Omega^\infty)\colon\ul\Fun^{\Pex}(\tcat,\ul\Sp^P(\Cc))\to \ul\Fun^{\Plex}(\tcat, \Cc).
	\end{equation*}
	Suppose now that $\Cc$ is moreover presentable. Then the left adjoint $\Sigma^\infty_+$ of $\Omega^\infty$ exhibits $\ul\Sp^P(\Cc)$ as the presentable $P$-stable completion of $\Cc$, i.e.~for any presentable $P$-stable $T$-$\infty$-category $\tcat$ precomposition with $\Sigma^\infty_+$ yields an equivalence
	\begin{equation*}
		\ul\Fun^\textup{L}(\Sigma^\infty_+, \tcat)\colon\ul\Fun^\textup{L}(\ul\Sp^P(\Cc), \tcat)\to\ul\Fun^\textup{L}(\Cc, \tcat).\qednow
	\end{equation*}
\end{theorem}

As a simple consequence, we get that the $T$-$\infty$-category $\ul{\Sp}^{P}_T$ of $P$-genuine $T$-spectra is the free presentable $P$-stable $T$-$\infty$-category on a single generator. As in the $P$-semiadditive setting of \Cref{subsec:P-semiadditive-Un}, we can strengthen this to the $T$-cocomplete setting:

\begin{theorem}\label{thm:cocomplete-univ-sp}
	Let $\mathcal D$ be a locally small $T$-cocomplete $P$-stable $T$-$\infty$-category. Then evaluating at $\Sigma^\infty_+(*)$ yields an equivalence
	\begin{equation*}
		\ul\Fun_T^\textup{L}(\ul\Sp^P_T,\mathcal D)\xrightarrow{\;\simeq\;}\mathcal D.
	\end{equation*}
\end{theorem}

For the proof we will first consider the following non-parametrized version strengthening \cite{HA}*{Corollary~1.4.4.5}:

\begin{lemma}\label{lemma:non-param-cocompl-univ-sp}
	Let $\mathcal C$ be a presentable $\infty$-category and let $\mathcal D$ be cocomplete and stable. Then we have equivalences
	\begin{align*}
		\Fun^\textup{L}(\Sigma^\infty_+,\mathcal D)\colon \Fun^\textup{L}(\Sp(\mathcal C),\mathcal D) &\xrightarrow{\;\simeq\;} \Fun^\textup{L}(\mathcal C,\mathcal D)\\
		\Fun^\textup{R}(\mathcal D,\Omega^\infty)\colon \Fun^\textup{R}(\mathcal D,\Sp(\mathcal C)) &\xrightarrow{\;\simeq\;} \Fun^\textup{R}(\mathcal D,\mathcal C)
	\end{align*}
	of categories of left adjoint and categories of right adjoint functors, respectively.
 \end{lemma}

\begin{proof}
		It suffices to prove the second statement. Since full faithfulness follows from the usual universal property of spectrification \cite{HA}*{Corollary~1.4.2.23}, it only remains to prove essential surjectivity, i.e.~for every right adjoint $G\colon\mathcal D\to\mathcal C$ we can find a \emph{right adjoint} $G_\infty\colon\mathcal D\to\Sp(\mathcal C)$ such that $\Omega^\infty G_\infty\simeq G$.

		For this we first observe that $G$ lifts to a functor $G_*\colon \mathcal D\simeq\mathcal D_*\to\mathcal C_*$ as $\mathcal D$ is pointed and $G$ preserves terminal objects; moreover, this is again a right adjoint functor by the dual of \cite{HTT}*{Proposition~5.2.5.1}. Replacing $\mathcal C$ by $\mathcal C_*$ if necessary, we may therefore assume without loss of generality that $\mathcal C$ is pointed.

		We now define $G_i\mathrel{:=} G\Sigma^i\colon\mathcal D\to\mathcal C$ for all $i\ge 0$. Then we have equivalences
		\begin{equation*}
			\Omega G_{i+1} = \Omega G\Sigma^{i+1}\simeq G\Omega\Sigma^{i+1}\simeq G\Sigma^i=G_i,
		\end{equation*}
		and so we get
		\begin{equation*}
			G_\infty\colon \mathcal D\to\Sp(\mathcal C)=\lim\big(\cdots\xrightarrow{\Omega}\mathcal C\xrightarrow\Omega\mathcal C\big)
		\end{equation*}
		with $\Omega^\infty G_\infty\simeq G_0=G$ by passing to limits. However, each $G_i$ for $i<\infty$ is a right adjoint (as $G$ is and since $\Sigma^i$ is even an equivalence by stability), whence so is the limit map $G_\infty$ by \cite{descent-lim}*{Theorem~B}.
\end{proof}

\begin{corollary}
	In the above situation, let $G\colon\mathcal D\to\Sp(\mathcal C)$ be an exact functor. Then $G$ admits a left adjoint if and only if $\Omega^\infty\circ G$ does.\qed
\end{corollary}

\begin{proposition}
	Let $\mathcal C$ be a presentable $T$-$\infty$-category and let $\mathcal D$ be a $T$-cocomplete fiberwise stable $T$-$\infty$-category. Then we have equivalences
	\begin{align*}
		\ul\Fun^\textup{L}_T(\Sigma^\infty_+,\mathcal D)\colon\ul\Fun^\textup{L}_T(\ul\Sp^{\fib}(\mathcal C),\mathcal D)&\xrightarrow{\;\simeq\;} \ul\Fun^\textup{L}_T(\mathcal C,\mathcal D)\\
		\ul\Fun^\textup{R}_T(\mathcal D,\Omega^\infty)\colon\ul\Fun^\textup{R}_T(\mathcal D,\ul\Sp^{\fib}(\mathcal C))&\xrightarrow{\;\simeq\;} \ul\Fun^\textup{R}_T(\mathcal D,\mathcal C).
	\end{align*}
\end{proposition}
\begin{proof}
	Arguing as before, it suffices to show that any right adjoint $g\colon\mathcal D\to\mathcal C$ lifts to a \emph{right adjoint} $G\colon\mathcal D\to\ul\Sp^{\fib}(\mathcal C)$. However, by Lemma~\ref{lem:FunintoSpfib} there exists a fiberwise left exact functor $G$ lifting $g$, and by the previous corollary this admits a \emph{pointwise} left adjoint $F$; it only remains to show that for every $t\colon A\to B$ in $T$ the Beck-Chevalley map $Ft^*\Rightarrow t^*F$ is an equivalence.

	However, for the diagram
	\begin{equation*}
	\begin{tikzcd}
		\mathcal D(A)\arrow[r, "G"] & \ul\Sp^{\fib}(\Cc)(A)\arrow[r, "\Omega^\infty"] & \ul\Cc(A)\\
		\mathcal D(B)\arrow[u, "t^*"]\arrow[r, "G"'] & \ul\Sp^{\fib}(\Cc)(B)\arrow[u, "t^*"']\arrow[r, "\Omega^\infty"'] & \Cc(B)\arrow[u, "t^*"']
	\end{tikzcd}
	\end{equation*}
	both the mate of the total square as well as the mate of the right hand square are equivalences as $g$ and $\Omega^\infty$ are parametrized right adjoints. By the compatiblity of mates with pasting we conclude that $Ft^*\Rightarrow t^*F$ becomes an equivalence after precomposition with $\Sigma^\infty_+\colon\ul\Cc(B)\to\ul\Sp^{\fib}(\Cc)(B)$. Therefore the claim follows by the first half of Lemma~\ref{lemma:non-param-cocompl-univ-sp}.
\end{proof}

\begin{proof}[Proof of Theorem~\ref{thm:cocomplete-univ-sp}]
	By the same reduction as in the semiadditive case (Theorem~\ref{thm:cocomplete-univ-cmon}), we only have to construct for each given $X\in\Gamma(\Dd)$ a left adjoint functor $F\colon\ul\Sp^P_T\to\mathcal D$ with $F(\Sigma^\infty_+(1))\simeq X$.

	To this end, we simply observe that Theorem~\ref{thm:cocomplete-univ-cmon} provides us with a left adjoint $f\colon\ul\CMon^P_T\to \mathcal D$ with $f(\mathbb P(1))\simeq X$, and by the previous proposition
	$f$ factors as
	\begin{equation*}
		\ul\CMon^P_T\xrightarrow{\Sigma^\infty} \ul\Sp^{\fib}(\ul\CMon^P_T)=\ul\Sp^P_T\xrightarrow{F}\mathcal D
	\end{equation*}
	for some left adjoint $F$, which is then the desired functor.
\end{proof}

\begin{corollary}\label{cor:cocomplete-univ-sp-relative}
	Let $\mathcal S$ be a $T$-$\infty$-category equivalent to $\ul\Spc_T$ and let $\mathcal D$ be a locally small $T$-cocomplete $P$-stable $T$-$\infty$-category. Then we have an equivalence
	\begin{equation*}
		\ul\Fun^\textup{L}_T(\Sigma^\infty_+,\mathcal D)\colon\ul\Fun^\textup{L}_T(\ul\Sp^P(\mathcal S),\mathcal D)\xrightarrow{\;\simeq\;}\ul\Fun^\textup{L}_T(\mathcal S,\mathcal D).\qednow
	\end{equation*}
\end{corollary}

\section{The universal property of global spectra}
\label{sec:UniversalPropertyGlobalSpectra}
In this section, we will prove the main result of this article: an interpretation of the global $\infty$-category of global spectra, defined via certain localizations of symmetric $G$-spectra generalizing \cites{schwede2018global,hausmann-global}, in terms of the abstract stabilization procedure we have described in the previous section.

\subsection{Stable \texorpdfstring{\for{toc}{$G$}\except{toc}{$\bm G$}}{G}-global homotopy theory}
\label{subsec:GGlobalSpectra}
We start by recalling the $\infty$-category of $G$-global spectra for a finite group $G$, and then show how these assemble for varying $G$ into a global $\infty$-category $\ul\mySp^\text{gl}$.

\begin{definition}
	We write $\cat{Spectra}$ for the category of \emph{symmetric spectra} in the sense of \cite{hss}*{Definition~1.2.2}. We will as usual evaluate symmetric spectra more generally at all finite sets (and not only at the standard sets $\{1,\dots,n\}$ for $n\ge0$), see~e.g.~\cite{hausmann-equivariant}*{2.4}.

	We write $\cat{$\bm G$-Spectra}$ for the category of $G$-objects in $\cat{Spectra}$ and call its objects \emph{(symmetric) $G$-spectra}.
\end{definition}

For a finite group $G$, we refer the reader to \cite{hausmann-equivariant}*{Definition~2.35} for the definition of \emph{$G$-stable equivalences} of symmetric $G$-spectra, to which we will refer as \emph{$G$-weak equivalences} below.

\begin{definition}
	Let $G$ be a finite group and let $f\colon X\to Y$ be a map of symmetric $G$-spectra. We call $f$ a \emph{$G$-global weak equivalence} if $\phi^*f$ is an $H$-weak equivalence for every group homomorphism $\phi\colon H\to G$ (not necessarily injective).
\end{definition}

\begin{theorem}[See~\cite{g-global}*{Proposition~3.1.20 and Theorem~3.1.41}]\label{thm:spectra-global-model-structure}
	There is a unique (combinatorial) model structure on $\cat{$\bm G$-Spectra}$ with
	\begin{itemize}
		\item weak equivalences the $G$-global weak equivalences \emph{and}
		\item acyclic fibrations those maps $f$ such that $f(A)^\phi$ is an acyclic Kan fibration for all finite sets $A$, all $H\subset\Sigma_A$, and all $\phi\colon H\to G$.
	\end{itemize}
	We call this the \emph{projective $G$-global model structure}.\qed
\end{theorem}

\begin{remark}
	For $G=1$ the above was first studied by Hausmann \cite{hausmann-global}, who also exhibited it as a Bousfield localization of Schwede's \emph{global orthogonal spectra} \cite{schwede2018global}*{4.1} at certain `$\mathcal Fin$-global weak equivalences,' see \cite{hausmann-global}*{Theorem~5.3}.
\end{remark}

\begin{lemma}[See~\cite{g-global}*{Lemma~3.1.49}]\label{lemma:alpha-star-rQ-spectra}
	Let $\alpha\colon G\to H$ be a homomorphism. Then the adjunction
	\begin{equation*}
		\alpha_!\colon\cat{$\bm G$-Spectra}_{\projGgl{G}}\rightleftarrows\cat{$\bm{H}$-Spectra}_{\projGgl{H}} \noloc \alpha^*
	\end{equation*}
	is a Quillen adjunction with homotopical right adjoint.\qed
\end{lemma}

There are also \emph{injective} analogues of the above model structures that will become useful below:

\begin{theorem}[See \cite{g-global}*{Corollary~3.1.46}]\label{thm:spectra-injective-global-model-structure}
	There is a unique (combinatorial) model structure on $\cat{$\bm G$-Spectra}$ with
	\begin{itemize}
		\item weak equivalences the $G$-global weak equivalences \emph{and}
		\item cofibrations the injective cofibrations (i.e.~levelwise injections).
	\end{itemize}
	We call this the \emph{injective $G$-global model structure}.\qed
\end{theorem}

\subsubsection{Relation to unstable $G$-global homotopy theory}
Passing to pointwise localizations as before, we get a global $\infty$-category $\ul\mySp^\text{gl}$ such that $\ul\mySp^\text{gl}(G)=\mySp^\text{gl}_G$ is the $\infty$-category of $G$-global spectra, with functoriality given via restriction. Let us now relate this to the unstable models from \ref{subsec:reminder-global-spaces}.

\begin{construction}
	Let $X$ be an $\mathcal I$-space (or, more generally, an $I$-space). Then we define its \emph{unbased suspension spectrum} $\Sigma^\bullet_+X$, cf.~\cite{sagave-schlichtkrull}*{discussion before Proposition 3.19}, via
	\begin{equation*}
		(\Sigma^\bullet_+X)(A)\mathrel{:=}S^A\smashp X(A)_+=\Sigma^A_+ X(A)
	\end{equation*}
	with the diagonal $\Sigma_A$-action and with structure maps given by
	\begin{align*}
		S^A\smashp (\Sigma^\bullet_+X)(B)&=S^A\smashp \big(S^B\smashp X(B)_+\big)\cong S^{A\amalg B}\smashp X(B)_+\\
		&\xrightarrow{S^{A\amalg B}\smashp X(\textup{incl})} S^{A\amalg B}\smashp X(A\amalg B)_+=(\Sigma^\bullet_+X)(A\amalg B)
	\end{align*}
	where the unlabelled isomorphism is the canonical one.

	This has a right adjoint $\Omega^\bullet$ (e.g.~by the Special Adjoint Functor Theorem); for any finite group $G$, we get an induced adjunction $\cat{$\bm G$-$\bm{\mathcal I}$-SSet}\rightleftarrows\cat{$\bm G$-Spectra}$ by pulling through the $G$-actions, which we again denote by $\Sigma^\bullet_+\dashv\Omega^\bullet$.
\end{construction}

\begin{warn}
	Beware that \cite{g-global} uses different (more elaborate) notation for the right adjoint, reserving the above for the right adjoint of $\Sigma^\bullet_+\colon\cat{$\bm G$-$\bm I$-SSet}\to \cat{$\bm G$-Spectra}$. However, as the latter adjoint will play no role here, we have decided to use the above, simpler notation.
\end{warn}

\begin{lemma}[See \cite{g-global}*{Proposition~3.2.2, Corollary~3.2.6, and Remark~3.2.7}]\label{lemma:sigma-qa}
	The above functor $\Sigma^\bullet_+$ preserves $G$-global weak equivalences and it is part of a Quillen adjunction
	\begin{equation*}
		\Sigma^\bullet_+\colon\cat{$\bm G$-$\bm{\mathcal I}$-SSet}_\textup{G-gl}\rightleftarrows\cat{$\bm G$-Spectra}_{\projGgl{G}} \noloc\Omega^\bullet.\qednow
	\end{equation*}
\end{lemma}

In particular, we get a global functor $\Sigma^\bullet_+\colon\ul\S^\text{gl}\to\ul\mySp^\text{gl}$, and this admits a pointwise adjoint $\cat{R}\Omega^\bullet$ as Quillen adjunctions induce adjunctions of $\infty$-categories. In fact we have:

\begin{proposition}\label{prop:looping-parametrized-ra}
	The global functor $\Sigma^\bullet_+\colon\ul\S^\textup{gl}\to\ul\mySp^\textup{gl}$ admits a parametrized right adjoint, given pointwise by the right derived functors $\cat{R}\Omega^\bullet$.
\end{proposition}

We will denote this right adjoint simply by $\cat{R}\Omega^\bullet$ again.

\begin{proof}
	As we already know that these form pointwise right adjoints, it only remains to verify the Beck-Chevalley condition, i.e.~that for every $\alpha\colon H\to G$ the canonical mate $\alpha^*\cat{R}\Omega^\bullet\Rightarrow\cat{R}\Omega^\bullet\alpha^*$ is an equivalence. This can be checked on the level of homotopy categories, for which we pick a fibrant replacement functor for the projective $H$-global model structure on $\cat{$\bm H$-Spectra}$, i.e.~an endofunctor $P$ taking values in projectively fibrant objects together with a natural transformation $\iota\colon\id\Rightarrow P$ that is levelwise an $H$-global weak equivalence. As $\Sigma^\bullet_+$ and $\alpha^*$ are homotopical (Lemma~\ref{lemma:sigma-qa} and Lemma~\ref{lemma:alpha-star-rQ-spectra}, respectively) and $\Omega^\bullet$ is right Quillen (Lemma~\ref{lemma:sigma-qa} again), the mate is then represented for any fibrant $G$-spectrum $X$ by the lower composite $\alpha^*\Omega^\bullet X\to\Omega^\bullet P\alpha^*X$ in the diagram
	\begin{equation*}
		\begin{tikzcd}
			\alpha^*\Omega^\bullet X\arrow[r, "\eta"] & \Omega^\bullet\Sigma^\bullet_+\alpha^*\Omega^\bullet X\arrow[d, "\iota"'] \arrow[r, equal] & \Omega^\bullet \alpha^*\Sigma^\bullet_+\Omega^\bullet X \arrow[r, "\epsilon"]\arrow[d,"\iota"{description}] & \Omega^\bullet\alpha^*X\arrow[d, "\iota"]\\
			& \Omega^\bullet P\Sigma^\bullet_+\alpha^*\Omega^\bullet X\arrow[r, equal] &\Omega^\bullet P\alpha^*\Sigma^\bullet_+\Omega^\bullet X\arrow[r, "\epsilon"'] & \Omega^\bullet P\alpha^*X
		\end{tikzcd}
	\end{equation*}
	in which the two squares commute simply by naturality. However, the top composite is simply the identity (as the adjunction was defined by pulling through the actions); on the other hand, $\iota\colon\alpha^*X\to P\alpha^*X$ is an $H$-global weak equivalence of fibrant objects ($\alpha^*$ being right Quillen), hence $\Omega^\bullet\iota\colon\Omega^\bullet \alpha^*X\to\Omega^\bullet P\alpha^*X$ is an $H$-global weak equivalence by Ken Brown's Lemma ($\Omega^\bullet$ being right Quillen). The claim now follows by $2$-out-of-$3$.
\end{proof}

\subsubsection{A t-structure} The model structures from Theorems~\ref{thm:spectra-global-model-structure} and~\ref{thm:spectra-injective-global-model-structure} are stable \cite{g-global}*{Proposition~3.1.48}, and so $\mySp^\text{gl}_G$ is a stable $\infty$-category. We will close this discussion by establishing a t-structure on it which generalizes Schwede's t-structure on the global stable homotopy category from \cite{schwede2018global}*{Theorem~4.4.9}. For this we first introduce:

\begin{construction}
	Let $H$ be a finite group, let $\phi\colon H\to G$ be a homomorphism, and let $k\in\mathbb Z$. If $X$ is any $G$-global spectrum, then the \emph{$k$-th $\phi$-equivariant homotopy group} $\pi^\phi_k(X)$ is the usual equivariant homotopy group $\pi_k^H(\phi^*X)$, i.e.~the hom set $[\Sigma^k\mathbb S,\phi^*X]$ in the $H$-equivariant stable homotopy category, with the group structure coming from semiadditivity.
\end{construction}

Equivalently (but more intrinsically), we can also describe $\pi^\phi_k(X)$ as the hom set $[\Sigma^{\bullet+k}_+ I(H,\blank)\times_\phi G,X]$ in the homotopy category of $\mySp^{\text{gl}}_G$, see~\cite{g-global}*{Corollary~3.3.4}.

\begin{theorem}\label{thm:g-global-t-structure}
	The stable $\infty$-category $\mySp^\textup{gl}_G$ is compactly generated by the objects $\Sigma^\bullet_+I(H,\blank)\times_\phi G$ for homomorphisms $\phi\colon H\to G$ from finite groups to $G$. Moreover, it admits a right complete t-structure with
	\begin{enumerate}[(1)]
		\item connective part $(\mySp^\textup{gl}_G)_{\ge0}$ those $G$-global spectra that are \emph{$G$-globally connective}, i.e.~$\pi^\phi_kX=0$ for all $k<0$,
		\item coconnective part $(\mySp^\textup{gl}_G)_{\le0}$ those $G$-global spectra that are \emph{$G$-globally coconnective}, i.e.~$\phi^\phi_kX=0$ for all $k>0$.
	\end{enumerate}
\end{theorem}

Here we recall \cite{HA}*{p.~44} that a t-structure on a stable $\infty$-category $\mathscr C$ is called \emph{right complete} if the truncations exhibit $\mathscr C$ as the inverse limit
\begin{equation*}
	\cdots\xrightarrow{\tau_{\ge-2}}\mathscr C_{\ge-2}\xrightarrow{\tau_{\ge-1}}\mathscr C_{\ge-1}\xrightarrow{\tau_{\ge0}}\mathscr C_{\ge0}.
\end{equation*}
By \cite{HA}*{Proposition~1.2.1.19} this is equivalent to demanding that $\bigcap_{n} \mathscr C_{\le-n}$ consist only of zero objects.

\begin{proof}
	We first observe that the $G$-global spectra $\Sigma^\bullet_+ I(H,\blank)\times_\phi G$ for finite groups $H$ (up to isomorphism) and homomorphisms $\phi\colon H\to G$ form a set of compact generators. Indeed, the $\phi$-equivariant homotopy groups for varying $\phi$ detect zero objects as the $H$-equivariant homotopy groups for every $H$ do \cite{hausmann-equivariant}*{Proposition~4.9-(iii)}, and they moreover commute with coproducts as the classical equivariant homotopy groups do (by the same argument) and since $\phi^*\colon\mySp^\text{gl}_G\to\mySp_H$ is a left adjoint by \cite{g-global}*{Corollary~3.3.4}.

	With this established, \cite{HA}*{Proposition 1.4.4.11} yields a t-structure on $\mySp^\text{gl}_G$ with connective part $(\mySp^\text{gl}_G)_{\ge0}$ the smallest subcategory closed under small colimits and extensions containing all the $\Sigma_+^\bullet I(H,\blank)\times_\phi G$. We claim that this has the desired properties.

	To see this, we let $Y$ be a $G$-global spectrum. Then the non-negative homotopy groups of $Y$ vanish if and only if $\maps(\Sigma^\bullet_+I(H,\blank)\times_\phi G,Y)\simeq0$ for all $\phi\colon H\to G$. On the other hand, the class of objects $X$ for which $\maps(X,Y)\simeq0$ is closed under colimits and extensions, so it has to contain all of $(\mySp^\text{gl}_G)_{\ge0}$ in this case, i.e.~$(\mySp^\text{gl}_G)_{\le-1}$ consists precisely of those objects with trivial non-negative homotopy groups. As suspension shifts ($H$-equivariant, hence $G$-global) homotopy groups, this proves the characterization of the coconnective objects.

	On the other hand, the connective $G$-global spectra contain all the $\Sigma^\bullet_+I(H,\blank)\times_\phi G$'s and they are closed under small coproducts (see above) as well as cofibers and extensions (by the long exact sequence), i.e.~every object in $(\mySp^\text{gl}_G)_{\ge0}$ is $G$-globally connective. Conversely, if $X$ is $G$-globally connective, then there is a cofiber sequence $X_{\ge0}\to X\to X_{\le -1}$ with $X_{\ge0}\in (\mySp^\text{gl}_G)_{\ge0}$ and $X_{\le -1}\in(\mySp^\text{gl}_G)_{\le-1}$ by what it means to be a t-structure. But then $X_{\ge0}$ is $G$-globally connective by the above, whence so is the cofiber $X_{\le -1}$. But on the other hand $X_{\le -1}$ has trivial non-negative homotopy groups, so $X_{\le -1}\simeq0$ and hence $X\simeq X_{\ge0}\in (\mySp^\text{gl}_G)_{\ge0}$ as claimed.

	This finishes the construction of the desired t-structure. Right completeness is immediate as any object in $\bigcap_{n\ge0} (\mySp^\text{gl}_G)_{\le-n}$ has trivial homotopy groups.
\end{proof}

\subsection{From global \texorpdfstring{\for{toc}{$\Gamma$}\except{toc}{$\bm\Gamma$}}{Γ}-spaces to global spectra}
Segal's classical Delooping Theorem \cite{segal} relates (very special) $\Gamma$-spaces to connective spectra. We will now recall a $G$-global refinement of this from \cite{g-global} and interpret it in the parametrized context.

\begin{construction}
	We define a functor $\mathcal E^\otimes\colon\cat{$\bm\Gamma$-$\bm{\mathcal I}$-SSet}_*\to\cat{Spectra}$ from the $1$-category of global $\Gamma$-spaces $X$ satisfying $X(0^+)=*$ to symmetric spectra via the $\cat{SSet}_*$-enriched coend
	\begin{equation*}
		\mathcal E^\otimes X\mathrel{:=}\int^{T_+\in\Gamma}\mathbb S^{\times T}\otimes X(T_+)
	\end{equation*}
	with the evident functoriality in $X$; here $\otimes$ denotes the pointwise smash product of a spectrum with a pointed $\mathcal I$-simplicial set, see \cite{g-global}*{Construction~3.2.9}.

	For any finite group $G$, pulling through the $G$-actions yields a functor \[\mathcal E^\otimes\colon\cat{$\bm\Gamma$-$\bm G$-$\bm{\mathcal I}$-SSet}_*\to\cat{$\bm G$-Spectra}\] that we again denote by $\mathcal E^\otimes$.
\end{construction}

\begin{proposition}
	For any finite $G$, there is a model structure on $\cat{$\bm\Gamma$-$\bm G$-$\bm{\mathcal I}$-SSet}_*$ in which a map $f$ is a weak equivalence or fibration if and only if $f(S_+)$ is a weak equivalence or fibration in the model structure on $\cat{$\bm{(G\times\Sigma_S)}$-$\bm{\mathcal I}$-SSet}$ from Theorem~\ref{thm:global-model-I} for every finite set $S$; in particular, its weak equivalences are precisely the $G$-global level weak equivalences.

	Moreover, the above functor $\mathcal E^\otimes$ is homotopical and part of a Quillen adjunction
	\begin{equation*}
		\mathcal E^\otimes\colon\cat{$\bm\Gamma$-$\bm G$-$\bm{\mathcal I}$-SSet}_*\rightleftarrows\cat{$\bm G$-Spectra}_{\injGgl{G}} \noloc\Phi^\otimes.
	\end{equation*}
\end{proposition}
\begin{proof}
	The existence of the model structure is part of \cite{g-global}*{Theorem~2.2.31}, while the remaining statements appear as Corollaries~3.4.20 and~3.4.24 of \emph{op.~cit.}
\end{proof}

\begin{remark}\label{rk:associated-gamma-U}
	While the precise form of the above right adjoint will not be relevant below, we record that there is a natural isomorphism $(\Phi^\otimes X)(1^+)\cong\Omega^\bullet X$, see \cite{g-global}*{Construction~3.2.17}. Restricting to injectively fibrant objects, we in particular immediately obtain an equivalence $\mathbb U\cat{R}\Phi^\otimes\simeq\cat{R}\Omega^\bullet$ of derived functors for any fixed $G$.
\end{remark}

Passing to localizations, $\mathcal E^\otimes$ induces a global functor $\ul\GammaS_{\mathcal I,*}^\text{gl}\to\ul\mySp^\text{gl}$.

\begin{lemma}
	The global functor $\mathcal E^\otimes\colon\ul\GammaS_{\mathcal I,*}^\textup{gl}\to\ul\mySp^\textup{gl}$ admits a parametrized right adjoint which is pointwise given by the $\cat{R}\Phi^\otimes$.
\end{lemma}

We will denote this parametrized right adjoint simply by $\cat{R}\Phi^\otimes$ again.

\begin{proof}
	It only remains to prove that for every $\alpha\colon H\to G$ the mate transformation $\alpha^*\cat{R}\Phi^\otimes\Rightarrow\cat{R}\Phi^\otimes \alpha^*$ at the level of homotopy categories is an equivalence. By the same computation as in Proposition~\ref{prop:looping-parametrized-ra} this reduces to showing that for any injectively fibrant $G$-global spectrum $X$ and some (hence any) injectively fibrant replacement $\iota\colon\alpha^*X\to Y$ of $G$-global spectra the induced map $\Phi^\otimes \iota\colon \Phi^\otimes\alpha^*X\to \Phi^\otimes Y=\cat{R}\Phi^\otimes\alpha^* X$ is an $H$-global level weak equivalence. This is precisely the content of \cite{g-global}*{claim in proof of Proposition~3.4.30}.
\end{proof}

\begin{definition}
	A special $G$-global $\Gamma$-space $X\in\cat{$\bm\Gamma$-$\bm G$-$\bm{\mathcal I}$-SSet}_*$ is called \emph{very special} if for every finite group $H$, every homomorphism $\phi\colon H\to G$, and some (hence any) complete $H$-set universe $\mathcal U_H$ the induced monoid structure on $\pi_0^\phi(X)\mathrel{:=}\pi_0\big((\phi^* X)(1^+)(\mathcal U_H)\big)$ coming from the zig-zag
	\begin{equation*}
		X(1^+)\times X(1^+)\xleftarrow[\sim]{\;\rho\;} X(2^+)\xrightarrow{X(\mu)} X(1^+),
	\end{equation*}
	where $\mu$ is defined by $\mu(1)=\mu(2)=1$, is a group structure. We write $\ul\GammaS^\textup{gl, vspc}_{\mathcal I,*}\subset\ul\GammaS_{\mathcal I,*}^\textup{gl}$ for the full global subcategory of very special objects.
\end{definition}

\begin{remark}
	The above condition is equivalent to $\phi^*X(\mathcal U_H)$ being very special as an $H$-equivariant $\Gamma$-space in the sense of \cite{ostermayr}*{Definition~4.5} for every $H$ and $\phi$ as above, see \cite{g-global}*{discussion after Definition~3.4.12}.
\end{remark}

We can now rephrase the $G$-global delooping theorem \cite{g-global}*{Theorem~3.4.21} in our setting as follows:

\begin{theorem}\label{thm:g-global-delooping}
	The parametrized adjunction $\mathcal E^\otimes\dashv\cat{R}\Phi^\otimes$ restricts to an equivalence $\ul\GammaS^\textup{gl, vspc}_{\mathcal I,*}\simeq \ul\mySp^\textup{gl}_{\ge0}$.\qed
\end{theorem}

Finally, we want to reinterpret this in terms of equivariant stabilizations, in the sense of \Cref{ex:eq_stability}

\begin{theorem}\label{thm:pw-stabilization-cmon}
	The global $\infty$-category $\ul\mySp^\textup{gl}$ is equivariantly stable and the functor
	\begin{equation}\label{eq:universtal-stab}
		\cat{R}\Phi^\otimes\colon\ul{\mySp}^\textup{gl}\to\ul\GammaS^\textup{gl, spc}_{\mathcal I,*}
	\end{equation}
	is the universal equivariant stabilization.
\end{theorem}

For the proof of the theorem we will need two lemmas:

\begin{lemma}\label{lemma:stabilization-of-connective}
	The adjunction $\textup{incl}\colon(\mySp_G^\textup{gl})_{\ge0}\rightleftarrows\mySp_G^\textup{gl}\noloc\tau_{\geq 0}$ is the universal stabilization in the world of presentable $\infty$-categories.
\end{lemma}
\begin{proof}
	By Theorem~\ref{thm:g-global-t-structure}, $(\mySp_G^\textup{gl})_{\ge0}$ is the connective part of a right complete t-structure. As mentioned without proof in the introduction of \cite{SAG}*{Appendix C}, this formally implies the statement of the lemma. Let us give the argument in this generality for completeness: given a right complete t-structure on a stable $\infty$-category $\mathscr C$ we consider the diagram
	\begin{equation*}
	\begin{tikzcd}
		\cdots\arrow[r, "\Omega"] &\arrow[d, "\Omega^2", "\simeq"']\mathscr C_{\ge0} \arrow[r, "\Omega"] & \mathscr C_{\ge0}\arrow[r, "\Omega", "\simeq"']\arrow[d, "\Omega", "\simeq"'] &\mathscr C_{\ge0}\arrow[d,"\id", "\simeq"']\\
		\cdots\arrow[r, "\tau_{\ge-2}"'] &\mathscr C_{\ge-2} \arrow[r, "\tau_{\ge-1}"'] & \mathscr C_{\ge-1}\arrow[r, "\tau_{\ge0}"'] &\mathscr C_{\ge0}
	\end{tikzcd}
	\end{equation*}
	where the little squares are filled by the total mates of the identity transformations $\Sigma^n\circ\textup{incl} = \Sigma^{n-1}\circ\Sigma$. Passing to row-wise homotopy limits we then get a commutative diagram
	\begin{equation*}
	\begin{tikzcd}[row sep=small]
		\Sp(\mathscr C_{\ge0})=\lim_n \mathscr C_{\ge0}\arrow[dr, bend left=15pt, "\Omega^\infty=\pr_0"]\arrow[dd, "\lim_n\Omega^{n}"']\\
		& \mathscr C_{\ge0}\\
		\lim_n\mathscr C_{\ge -n}\arrow[ur, bend right=15pt, "\pr_0"']
	\end{tikzcd}
	\end{equation*}
	in which the vertical map on the left is an equivalence as a homotopy limit of equivalences. On the other hand, by right completeness the lower map agrees up to equivalence with $\tau_{\geq 0}\colon \mathscr C\to\mathscr C_{\ge0}$; the claim follows immediately as $\Omega^\infty\colon\Sp(\mathscr C_{\ge0})\to\mathscr C_{\ge0}$ is the universal stabilization by \cite{HA}*{Remark~1.4.2.25}.
\end{proof}

\begin{lemma}
	Let $T$ be an $\infty$-category and let $U\colon\mathcal D\to\mathcal C$ be a $T$-functor such that $\mathcal D$ is fiberwise stable, $\mathcal C$ has fiberwise finite limits, and each $U(A)\colon \mathcal D(A)\to\mathcal C(A)$ is a stabilization in the non-parametrized sense. Then $U$ is a fiberwise stabilization.
\end{lemma}

Put differently, if we already know fiberwise stability of the source, then fiberwise stabilizations can be checked pointwise without regards to any homotopies or higher structure.

\begin{proof}
	In the naturality square
	\begin{equation*}
	\begin{tikzcd}
		\ul\Sp^\textup{fib}(\mathcal D)\arrow[d, "\Omega^\infty"']\arrow[r, "\ul\Sp^\textup{fib}(U)"] &[1.5em] \ul\Sp^\textup{fib}(\mathcal C)\arrow[d, "\Omega^\infty"]\\
		\mathcal D\arrow[r, "U"'] &\mathcal C
	\end{tikzcd}
	\end{equation*}
	the left hand vertical arrow is an equivalence as $\mathcal D$ is fiberwise stable, and so is the top horizontal map as $\big(\ul\Sp^\textup{fib}(U)\big)(A)=\Sp\big(U(A)\big)$ and each $U(A)$ was assumed to be a stabilization. Finally, the right hand vertical map is a fiberwise stabilization by construction, so the claim follows immediately.
\end{proof}

\begin{proof}[Proof of Theorem~\ref{thm:pw-stabilization-cmon}]
	As each $\mySp^\text{gl}_G$ is stable and all restriction maps between them are exact (being right adjoints), it will suffice by the previous lemma that
	\begin{equation*}
		\cat{R}\Phi^\otimes\colon\mySp^\text{gl}_G\to\ul{\GammaS}^\text{gl, spc}_{\mathcal I,*}(G)
	\end{equation*}
	is a stabilization in the non-parametrized sense for every fixed $G$, for which it suffices by stability of the source that this induces an equivalence after applying spectrification. By Lemma~\ref{lemma:stabilization-of-connective}, it suffices to show this for the restriction to $(\mySp_G^\text{gl})_{\ge0}$, for which it is then in turn enough by Theorem~\ref{thm:g-global-delooping} that the inclusion $\textup{incl}\colon\ul{\GammaS}^\text{gl, vspc}_{\mathcal I,*}(G)\hookrightarrow\ul{\GammaS}^\text{gl, spc}_{\mathcal I,*}(G)$ of very special $G$-global $\Gamma$-spaces induces an equivalence after spectrification.

	For this we observe that the loop space functor $\Omega\colon \ul{\GammaS}^\text{gl, spc}_{\mathcal I,*}(G)\to \ul{\GammaS}^\text{gl, spc}_{\mathcal I,*}(G)$ factors through $\ul{\GammaS}_{\mathcal I,*}^\text{gl, vspc}(G)$ as for any special $G$-global $\Gamma$-space $X$ the commutative monoid structure on $\pi_0^\phi(\Omega X)$ coming from the $\Gamma$-space structure agrees with the group structure coming from $\Omega$ by the Eckmann-Hilton argument. It is then clear that for the induced functor $\Sp(\Omega)\colon\Sp(\ul{\GammaS}_{\mathcal I,*}^\text{gl, spc}(G))\to \Sp(\ul{\GammaS}_{\mathcal I,*}^\text{gl, vspc}(G))$ the composites $\Sp(\textup{incl})\Sp(\Omega)$ and $\Sp(\Omega)\Sp(\textup{incl})$ are given by the respective loop functors, so they are equivalences by stability. The claim follows by $2$-out-of-$6$.
\end{proof}

\subsection{Proof of Theorem~\ref{introthm:universal-prop-spectra}} Using the above we now easily get:

\begin{theorem}\label{thm:global-spectra-relative}
	The functor $\cat{R}\Omega^\bullet\colon\ul\mySp^\textup{gl}\to\ul\S^\textup{gl}$ exhibits $\ul\mySp^\textup{gl}$ as the equivariantly stable envelope of $\ul\S^\textup{gl}$, i.e.~for every equivariantly stable global $\infty$-category $\Cc$ postcomposition with $\cat{R}\Omega^\bullet$ induces an equivalence
	\begin{equation*}
		\ul\Fun^\textup{{\textup{$\Orb$-lex}}}_{\Glo}(\Cc, \cat{R}\Omega^\bullet)\colon\ul\Fun^{\textup{$\Orb$-ex}}_{\Glo}(\Cc, \ul\mySp^\textup{gl})\to \ul\Fun^{\textup{$\Orb$-lex}}_{\Glo}(\Cc, \ul\S^\textup{gl}).
	\end{equation*}
	Moreover, the left adjoint $\Sigma^\bullet_+$ exhibits $\ul\mySp^\textup{gl}$ as the equivariantly stable completion in the following sense: for any globally cocomplete equivariantly stable global $\infty$-category $\tcat$ precomposition with $\Sigma^\bullet_+$ yields an equivalence
	\begin{equation*}
		\ul\Fun^\textup{L}_{\Glo}(\Sigma^\bullet_+, \tcat)\colon\ul\Fun_{\Glo}^\textup{L}(\ul\mySp^\textup{gl}, \tcat)\to\ul\Fun^\textup{L}_{\Glo}(\ul\S^\textup{gl}, \tcat).
	\end{equation*}
\end{theorem}
\begin{proof}
	By \Cref{thm:Stable_omnibus} and Corollary~\ref{cor:cocomplete-univ-sp-relative}, respectively, together with Corollary \ref{cor:univ-property-Gamma-I} it will suffice to show that the diagrams
	\begin{equation*}
	\begin{tikzcd}[column sep=small]
		\ul\mySp^\text{gl}\arrow[rr, "\cat{R}\Omega^\bullet"]\arrow[dr, bend right=15pt, "\cat{R}\Phi^\otimes"'] && \ul\S^\text{gl}\\
		& \ul\GammaS^\text{gl, spc}_{\mathcal I,*}\arrow[ur, bend right=15pt, "\mathbb U"']
	\end{tikzcd}
	\qquad\text{and}\qquad
	\begin{tikzcd}[column sep=small]
		\ul\mySp^\text{gl}\arrow[from=rr, "\Sigma^\bullet_+"']\arrow[from=dr, bend left=15pt, "\mathcal E^\otimes"] && \ul\S^\text{gl}\\
		& \ul\GammaS^\text{gl, spc}_{\mathcal I,*}\arrow[from=ur, bend left=15pt, "\mathbb P"]
	\end{tikzcd}
	\end{equation*}
	of global functors commute up to equivalence.

	By uniqueness of adjoints, it suffices to prove this for the second diagram, for which it is enough by the universal property of global spaces to chase through $*\in\S^\text{gl}_1$; in particular, it suffices to show that this commutes after evaluation at the trivial group. But by uniqueness of adjoints again, it is then enough to prove this for the diagram on the left instead, where this is immediate from Remark~\ref{rk:associated-gamma-U}.
\end{proof}

Together with Theorem~\ref{thm:global-spaces} we then immediately get Theorem~\ref{introthm:universal-prop-spectra} from the introduction:

\begin{theorem}\label{thm:global-spectra}
	Let $\mathcal D$ be any globally cocomplete equivariantly stable global $\infty$-category. Then evaluation at the global sphere spectrum $\mathbb S\cong\Sigma^\bullet_+(*)\in\mySp^\textup{gl}_1$ defines an equivalence
	\begin{equation*}
		\ul\Fun_{\Glo}^\textup{L}(\ul\mySp^\textup{gl},\mathcal D)\xrightarrow{\;\simeq\;}\mathcal D.
	\end{equation*}
	Put differently, $\ul\mySp^\textup{gl}$ is the free globally cocomplete (or presentable) $\Orb$-stable global $\infty$-category on one generator (namely, the global sphere spectrum $\mathbb S$).\qed
\end{theorem}

Comparing universal properties we can also reformulate this as follows:

\begin{corollary}\label{cor:comparison-stable-global}
	The essentially unique left adjoint functor $\ul\Sp_{\Glo}^{\Orb}\to\ul\mySp^\textup{gl}$ sending $\Sigma^\infty_+(*)$ to $\mathbb S$ is an equivalence.\qed
\end{corollary}

\appendix
\numberwithin{theorem}{section}
\section{Slices of \texorpdfstring{$(2,1)$}{(2,1)}-categories}\label{appendix:slices}
In this short appendix we will prove that for a strict $(2,1)$-category the $\infty$-categorical and $2$-categorical slices agree. More precisely:

\begin{proposition}\label{prop:comparison-slices}
Let $\mathscr C$ be a strict $(2,1)$-category. Then the cocartesian fibration $\ev_1\colon\nerve_\Delta(\mathscr C)^{\Delta^1}\to\nerve_\Delta(\mathscr C)$ classifies the homotopy coherent nerve of the composition
\begin{equation*}
	\mathscr C\xrightarrow{\mathscr C_{/\bullet}} \cat{Cat}_{(2,1)}\xrightarrow{\nerve_\Delta}\cat{Cat}_\infty.
\end{equation*}
\end{proposition}
\begin{proof}
We begin by making the $2$-categorical Grothendieck construction $\pi\colon\mathscrGr\to\mathscr C$ (Construction~\ref{constr:grothendieck-2-functor}) of the functor $\mathscr C_{/\bullet}\colon\mathscr C\to\cat{Cat}_{(2,1)}$ explicit, which, upon passing to homotopy coherent nerves, will then yield a concrete model of the unstraightening:
\begin{enumerate}[(1)]
	\item An object of $\mathscrGr$ is a morphism $f\colon X\to Y$ in $\mathscr C$.
	\item A morphism $f\to g$ is a diagram
	\begin{equation}\label{diag:morphism-G}
	\begin{tikzcd}
		X_1\arrow[r, "x"]\arrow[d, "f"'] & X_2\arrow[d, "g"]\\
		Y_1\arrow[r, "y"']\arrow[ur, Rightarrow, "\theta"] & Y_2
	\end{tikzcd}
	\end{equation}
	(the pair $(x,\theta)$ being a morphism from the pushforward $\mathscr C_{/y}(f)$ to $g$ in $\mathscr C_{/Y_2}$). Composition of morphisms is given by composition of $1$-cells and pasting of $2$-cells in $\mathscr C$.
	\item A $2$-cell between two such morphisms $(x,\theta,y)$, $(x',\theta',y')$ is a pair of a $2$-cell $\sigma\colon x\Rightarrow x'$ and a $2$-cell $\tau\colon y\Rightarrow y'$ such that the pastings
	\begin{equation*}
	\begin{tikzcd}[column sep=large, row sep=large]
		X_1\arrow[d, "f"']\arrow[r, bend left=15pt, "x'", ""'{yshift=2pt,name=b},yshift=2pt]\arrow[r, bend right=15pt,yshift=-2pt, "x"', ""{yshift=-2pt,name=a}]\arrow[from=a,to=b,Rightarrow,"\sigma\;"] & X_2\arrow[d, "g"]\\
		Y_1\arrow[r, "y"']\arrow[ur, Rightarrow, "\theta"'] & Y_2
	\end{tikzcd}
	\qquad\text{and}\qquad
	\begin{tikzcd}[column sep=large, row sep=large]
		X_1\arrow[d, "f"']\arrow[r, "x'"] & X_2\arrow[d, "g"]\\
		Y_1\arrow[r, bend left=15pt, "y'", ""'{yshift=0pt,name=b},yshift=2pt]\arrow[r, bend right=15pt,yshift=-2pt, "y"', ""{yshift=-2pt,name=a}]\arrow[ur, Rightarrow, "\theta'", yshift=3pt] & Y_2\arrow[from=a,to=b,Rightarrow,"\tau\;"]
	\end{tikzcd}
	\end{equation*}
	agree. Horizontal and vertical composition of $2$-cells is given by horizontal and vertical composition, respectively, in $\mathscr C$.
\end{enumerate}
The projection $\pi\colon\mathscrGr\to \mathscr C$ sends an object $f\colon X\to Y$ to $Y$, a morphism $(\ref{diag:morphism-G})$ to $y$, and a $2$-cell $(\sigma,\tau)$ to $\tau$.

The homotopy coherent nerve $\nerve_\Delta(\mathscrGr)$ is then a strictly $3$-coskeletal simplicial set, hence it suffices to describe the $2$-truncation and to characterize which diagrams $\partial\Delta^3\to\nerve_\Delta(\mathscr G)$ extend over $\Delta^3$. Unravelling the definitions, we get:
\begin{enumerate}[(1)]
	\item A vertex of $\nerve_\Delta(\mathscr G)$ is a morphism $f\colon X\to Y$ in $\mathscr C$.
	\item An edge $f\to g$ in $\nerve_\Delta(\mathscr G)$ is a diagram $(\ref{diag:morphism-G})$.
	\item A $2$-simplex with boundary
	\begin{equation*}
	\begin{tikzcd}
		X_0\arrow[d, "f_0"'] \arrow[r, "x_{01}"] & X_1\arrow[d, "f_1"]\arrow[from=dl, Rightarrow, "\theta_{01}"]\\
		Y_0\arrow[r, "y_{01}"'] & Y_1
	\end{tikzcd}\qquad
	\begin{tikzcd}
		X_1\arrow[d, "f_1"'] \arrow[r, "x_{12}"] & X_2\arrow[d, "f_2"]\arrow[from=dl, Rightarrow, "\theta_{12}"]\\
		Y_1\arrow[r, "y_{12}"'] & Y_2
	\end{tikzcd}\qquad
	\begin{tikzcd}
		X_0\arrow[d, "f_0"'] \arrow[r, "x_{02}"] & X_2\arrow[d, "f_2"]\arrow[from=dl, Rightarrow, "\theta_{02}"]\\
		Y_0\arrow[r, "y_{02}"'] & Y_2
	\end{tikzcd}
	\end{equation*}
	amounts to the data of a natural transformation $\sigma\colon x_{02}\Rightarrow x_{12}x_{01}$ and a natural transformation $\tau\colon y_{02}\Rightarrow y_{12}y_{01}$ such that the two pastings
	\begin{equation*}
	\begin{tikzcd}
		\phantom{X_0}\\[-1ex]
		X_0\arrow[d, "f_0"'] \arrow[r, "x_{01}"] & X_1\arrow[d, "f_1"{description}]\arrow[from=dl, Rightarrow, "\theta_{01}"] \arrow[r, "x_{12}"] & X_2\arrow[d, "f_2"]\arrow[from=dl, Rightarrow, "\theta_{12}"]\\
		Y_0\arrow[rr, bend right=1.8cm, "y_{02}"', ""{name=X}]\arrow[r, "y_{01}"{description}] & \arrow[from=X, Rightarrow, "\tau"']Y_1 \arrow[r, "y_{12}"{description}] & Y_2
	\end{tikzcd}\qquad\text{and}\qquad
	\begin{tikzcd}
		& X_1\arrow[dr, bend left=15pt, "x_{12}"]\\[-1ex]
		X_0\arrow[ur, bend left=15pt, "x_{01}"]\arrow[d, "f_0"'] \arrow[rr, "x_{02}"{description}] &\strut\arrow[u, "\sigma"', Rightarrow]& X_2\arrow[d, "f_2"]\arrow[from=dll, Rightarrow, "\theta_{02}"]\\
		Y_0\arrow[rr, "y_{02}"'] && Y_2
	\end{tikzcd}
	\end{equation*}
	agree.
	\item A diagram $\partial\Delta^3\to\nerve_\Delta(\mathscrGr)$ corresponding to
	\begin{equation}\label{diag:del-Delta-3}
	\hskip-15pt
	\begin{aligned}
	&\begin{tikzcd}[column sep=.9em]
		& X_1\arrow[dr, bend left=15pt, "x_{12}"]\\
		X_0\arrow[ur, bend left=15pt, "x_{01}"]\arrow[rr, "x_{02}"'] &\phantom{x}\arrow[u, Rightarrow, "\sigma_{012}"']& X_2
	\end{tikzcd}\kern.5em
	\begin{tikzcd}[column sep=.9em]
		& X_1\arrow[dr, bend left=15pt, "x_{13}"]\\
		X_0\arrow[ur, bend left=15pt, "x_{01}"]\arrow[rr, "x_{03}"'] &\phantom{x}\arrow[u, Rightarrow, "\sigma_{013}"']& X_3
	\end{tikzcd}\kern.5em
	\begin{tikzcd}[column sep=.9em]
		& X_2\arrow[dr, bend left=15pt, "x_{23}"]\\
		X_0\arrow[ur, bend left=15pt, "x_{02}"]\arrow[rr, "x_{03}"'] &\phantom{x}\arrow[u, Rightarrow, "\sigma_{023}"']& X_3
	\end{tikzcd}\kern.5em
	\begin{tikzcd}[column sep=.9em]
		& X_2\arrow[dr, bend left=15pt, "x_{23}"]\\
		X_1\arrow[ur, bend left=15pt, "x_{12}"]\arrow[rr, "x_{13}"'] &\phantom{x}\arrow[u, Rightarrow, "\sigma_{123}"']& X_3
	\end{tikzcd}
	\\
	&\begin{tikzcd}[column sep=1.275em]
		& Y_1\arrow[dr, bend left=15pt, "y_{12}"]\\
		Y_0\arrow[ur, bend left=15pt, "y_{01}"]\arrow[rr, "y_{02}"'] &\phantom{y}\arrow[u, Rightarrow, "\tau_{012}"']& Y_2
	\end{tikzcd}\kern.5em
	\begin{tikzcd}[column sep=1.275em]
		& Y_1\arrow[dr, bend left=15pt, "y_{13}"]\\
		Y_0\arrow[ur, bend left=15pt, "y_{01}"]\arrow[rr, "y_{03}"'] &\phantom{y}\arrow[u, Rightarrow, "\tau_{013}"']& Y_3
	\end{tikzcd}\kern.5em
	\begin{tikzcd}[column sep=1.275em]
		& Y_2\arrow[dr, bend left=15pt, "y_{23}"]\\
		Y_0\arrow[ur, bend left=15pt, "y_{02}"]\arrow[rr, "y_{03}"'] &\phantom{y}\arrow[u, Rightarrow, "\tau_{023}"']& Y_3
	\end{tikzcd}\kern.5em
	\begin{tikzcd}[column sep=1.275em]
		& Y_2\arrow[dr, bend left=15pt, "y_{23}"]\\
		Y_1\arrow[ur, bend left=15pt, "y_{12}"]\arrow[rr, "y_{13}"'] &\phantom{y}\arrow[u, Rightarrow, "\tau_{123}"']& Y_3
	\end{tikzcd}
	\hskip-10pt
	\end{aligned}
	\end{equation}
	extends to $\Delta^3$ if and only if the pastings
	\begin{equation*}
	\begin{tikzcd}[column sep=large, row sep=large]
		X_0\arrow[dr, "x_{02}"{description,name=y},""{name=a}, ""'{name=b}]\arrow[r, "x_{01}"]\arrow[d, "x_{03}"', ""{name=x}] & X_1\arrow[d, "x_{12}"]\arrow[from=a, Rightarrow, "\sigma_{012}"] \arrow[from=x,to=y,Rightarrow, "\sigma_{023}"']\\
		X_3 & \arrow[l, "x_{23}"] X_2
	\end{tikzcd}\qquad\text{and}\qquad
	\begin{tikzcd}[column sep=large, row sep=large]
		X_0\arrow[r, "x_{01}"]\arrow[d, "x_{03}"',""{name=x}] & X_1\arrow[d, "x_{12}"]\arrow[dl, "x_{13}"{description,name=y}, ""{name=a}]\\
		X_3 & \arrow[l, "x_{23}"] X_2\arrow[from=a, Rightarrow, "\!\sigma_{123}"]\arrow[from=x,to=y,Rightarrow, "\sigma_{013}"]
	\end{tikzcd}
	\end{equation*}
	agree, and likewise for the $\tau$'s. Put differently, $\partial\Delta^3\to\nerve_\Delta(\mathscrGr)$ extends over $\Delta^3$ if and only if the two maps $\partial\Delta^3\to\nerve_\Delta(\mathscr C)$ defined by $(\ref{diag:del-Delta-3})$ extend over $\Delta^3$.
\end{enumerate}

The degeneracy map $\nerve_\Delta(\mathscrGr)_0\to\nerve_\Delta(\mathscrGr)_1$ is given by sending $f\colon X\to Y$ to the square
\begin{equation*}
	\begin{tikzcd}
		X\arrow[r, "\id"]\arrow[d,"f"'] & X\arrow[d,"f"]\\
		Y\arrow[r, "\id"']\arrow[ur, Rightarrow, "\id"] & Y
	\end{tikzcd}
\end{equation*}
and similarly the degeneracies $\nerve_\Delta(\mathscrGr)_1\to\nerve_\Delta(\mathscrGr)_2$ are given by inserting identity arrows and identity $2$-cells.

The map $\nerve_\Delta(\pi)\colon\nerve_\Delta(\mathscrGr)\to\nerve_\Delta(\mathscr C)$ is the evident forgetful map. It then remains to construct an equivalence $\nerve_\Delta(\mathscrGr)\simeq \nerve_\Delta(\mathscr C)^{\Delta^1}$ of cocartesian fibrations over $\nerve_\Delta(\mathscrGr)$.

For this we observe that $\nerve_\Delta(\mathscr C)^{\Delta^1}$ is again strictly $3$-coskeletal (as $\nerve_\Delta(\mathscr C)$ is), and that unravelling definitions it can be described as follows:
\begin{enumerate}[(1)]
	\item A vertex of $\nerve_\Delta(\mathscr C)^{\Delta^1}$ is a morphism $f\colon X\to Y$ in $\mathscr C$.
	\item An edge $f\to g$ in $\nerve_\Delta(\mathscr C)^{\Delta^1}$ is a diagram
	\begin{equation}\label{diag:edge-U}
		\begin{tikzcd}[column sep=large, row sep=large]
			X_1\arrow[d, "f"']\arrow[dr, ""{name=c}, ""'{name=e}] \arrow[r, "x"] & X_2\arrow[d, "g"]\arrow[Rightarrow, from={c}, "\theta"]\\
			Y_1\arrow[Rightarrow, from={e}, "\theta'"]\arrow[r, "y"'] & Y_2\rlap.
		\end{tikzcd}
	\end{equation}
	\item A $2$-simplex in $\nerve_\Delta(\mathscr C)^{\Delta^1}$ with boundary
	\begin{equation*}
		\begin{tikzcd}
			X_0\arrow[d, "f_0"'] \arrow[r, "x_{01}"] & X_1\arrow[d, "f_1"]\arrow[from=dl, Rightarrow, "\theta_{01}"]\\
			Y_0\arrow[r, "y_{01}"'] & Y_1
		\end{tikzcd}\qquad
		\begin{tikzcd}
			X_1\arrow[d, "f_1"'] \arrow[r, "x_{12}"] & X_2\arrow[d, "f_2"]\arrow[from=dl, Rightarrow, "\theta_{12}"]\\
			Y_1\arrow[r, "y_{12}"'] & Y_2
		\end{tikzcd}\qquad
		\begin{tikzcd}
			X_0\arrow[d, "f_0"'] \arrow[r, "x_{02}"] & X_2\arrow[d, "f_2"]\arrow[from=dl, Rightarrow, "\theta_{02}"]\\
			Y_0\arrow[r, "y_{02}"'] & Y_2
		\end{tikzcd}
	\end{equation*}
	(where we have pasted the two natural isomorphisms and omitted the middle arrow) amounts to the data of a natural transformation $\sigma\colon x_{02}\Rightarrow x_{12}x_{01}$ and a transformation $\tau\colon y_{02}\Rightarrow y_{12}y_{01}$ satisfying the same conditions as for $\nerve_\Delta(\mathscr G)$.

	\item A diagram $\partial\Delta^3\to\nerve_\Delta(\mathscr C)^{\Delta^1}$ corresponding to $(\ref{diag:del-Delta-3})$ extends to $\Delta^3$ if and only if it satisfies the same pasting condition as for $\nerve_\Delta(\mathscr G)$, i.e.~if and only if the two maps $\partial\Delta^3\to\nerve_\Delta(\mathscr C)$ defined by the above extend to $\Delta^3$.
\end{enumerate}
In each case, the degeneracy maps are again given by inserting identity arrows and $2$-cells.

It is then straight-forward to check that we have a unique map $\Phi\colon\nerve_\Delta(\mathscr C)^{\Delta^1}\to\nerve_\Delta(\mathscrGr)$ that is the identity on vertices, sends an edge $(\ref{diag:edge-U})$ to the edge given by pasting of $\theta$ and $(\theta')^{-1}$, and that sends a $2$-simplex of $\nerve_\Delta(\mathscr C)^{\Delta^1}$ corresponding to $\sigma\colon x_{02}\Rightarrow x_{12}x_{01}$, $\tau\colon y_{02}\Rightarrow y_{12}y_{01}$ to the $2$-simplex of $\nerve_\Delta(\mathscrGr)$ corresponding to the same transformations. This is clearly a map over $\nerve_\Delta(\mathscr C)$ and so by \cite{HTT}*{Proposition 3.1.3.5} it only remains to show that it induces equivalences on fibers.

It is bijective on objects by definition, so it only remains to prove that for all $f\colon X_1\to Y$, $g\colon X_2\to Y$ the induced map
\begin{equation}\label{eq:induced-mapping-spaces}
	\Hom^{\text L}_{(\nerve_\Delta(\mathscr C)^{\Delta^1})_Y}(f,g)\to\Hom^{\text L}_{\nerve_\Delta(\mathscrGr)_Y}(f,g)
\end{equation}
is a weak homotopy equivalence. However, both sides are nerves of groupoids, so it is enough to show that it is surjective on vertices and that for any two vertices $\alpha,\beta$ on the left hand side it induces a bijection between edges $\alpha\to\beta$ and edges between their images.

For the first statement, it suffices to observe that by definition $(\ref{eq:induced-mapping-spaces})$ is given on vertices by the effect of $\Phi$ on edges $f\to g$; thus, given any edge $(x,\id_Y,\sigma)$ of $\nerve_\Delta(\mathscr G)_Y$, a preimage is given by
\begin{equation*}
	\begin{tikzcd}[column sep=large, row sep=large]
		X_1\arrow[d, "f"']\arrow[dr, ""{name=c}, ""'{name=e}, "f"{description}] \arrow[r, "x"] & X_2\arrow[d, "g"]\arrow[Rightarrow, from={c}, "\sigma"]\\
		Y\arrow[Rightarrow, from={e}, "\id"]\arrow[r, "\id"'] & Y\rlap.
	\end{tikzcd}
\end{equation*}
Similarly, the effect of $(\ref{eq:induced-mapping-spaces})$ on edges is induced by the effect of $\Phi$ on $2$-cells, so it follows immediately from the above description that it induces bijections between edges $\alpha\to\beta$ and edges between their images.
\end{proof}

\begin{remark}
	Let $\mathscr I$ be a (say, strict) $(2,1)$-category; as announced in \cite{duskin}, the $\infty$-categorical functor category $\nerve_\Delta(\mathscr C)^{\nerve_\Delta(\mathscr I)}$ can be identified with the homotopy coherent nerve of the strict $(2,1)$-category $\Fun^\text{pseudo}(\mathscr I,\mathscr C)$ of normal (i.e.~strictly unital) pseudofunctors $\mathscr I\to\mathscr C$, pseudonatural transformations, and modifications. If one is willing to take this for granted, the proof of the proposition can be significantly shortened, as the above Grothendieck construction $\mathscrGr$ is canonically \emph{isomorphic} to  $\Fun^\text{pseudo}([1],\mathscr C)$.

	However, the authors are unaware of any place in the literature where such a comparison is actually proven: in particular, the announced sequel to \cite{duskin} apparently never appeared. On the level of objects (i.e.~that maps $\nerve_\Delta(\mathscr I)\to\nerve_\Delta(\mathscr C)$ correspond to normal pseudofunctors $\mathscr I\to\mathscr C$) a detailed proof is given as \cite{kerodon}*{Tag {\ttfamily 00AU}}. The statement that at least every pseudonatural transformation of functors $\mathscr I\to\mathscr C$ gives rise to a transformation of maps $\nerve_\Delta(\mathscr I)\to\nerve_\Delta(\mathscr C)$ appears as \cite{bfb}*{Proposition~4.4}, but its proof is left as an exercise.
\end{remark}

\bibliographystyle{amsalpha}
\bibliography{reference}

\end{document}